\documentclass[arxiv]{alt2026} 

\title[Adaptive Acceleration without Strong Convexity Priors or Restarts]{Adaptive Acceleration without Strong Convexity Priors or Restarts}
\usepackage{times}

\usepackage{graphicx}
\usepackage{jmlrutils}

\usepackage{siunitx}
\usepackage{adjustbox}
\usepackage{booktabs}
\usepackage{colortbl}
\usepackage[most]{tcolorbox}
\usepackage{algorithmic}
\usepackage{algorithm}

\usepackage{mathtools}
\usepackage{enumitem}

\newtheorem{assumption}{Assumption}[section]
\newenvironment{sketch}{\renewcommand{\proofname}{Proof sketch}\begin{proof}}{\end{proof}}

\newcommand{\bigO}{\mathcal{O}}

\newcommand{\T}{\mathsf{T}}
\newcommand{\ct}{\mathsf{H}}
\newcommand{\rgd}{r_{\textup{GD}}}
\newcommand{\rnag}{r_{\textup{NAG}}}

\DeclarePairedDelimiter\ceil{\lceil}{\rceil}

\DeclareMathOperator{\diag}{diag}

\usepackage[capitalize,noabbrev]{cleveref}
\crefname{assumption}{Assumption}{Assumptions}
\Crefname{assumption}{Assumption}{Assumptions}
\crefname{proposition}{Proposition}{Propositions}
\Crefname{proposition}{Proposition}{Propositions}

\newtheorem{cor}{Corollary}[section]
\crefname{cor}{Corollary}{Corollaries}
\Crefname{cor}{Corollary}{Corollaries}

\altauthor{%
    \Name{Joao V. Cavalcanti} \Email{caval@mit.edu}\\
    \addr MIT
    \AND
    \Name{Laurent Lessard} \Email{l.lessard@northeastern.edu}\\
    \addr Northeastern University%
    \AND
    \Name{Ashia C. Wilson} \Email{ashia07@mit.edu}\\
    \addr MIT
}

\begin{document}

\maketitle

\begin{abstract}%
A longstanding challenge in optimization is achieving optimal performance when the strong convexity parameter \(m\) is unknown. 
In this paper, we propose \emph{NAG-free}, a simple extension of Nesterov’s accelerated gradient (NAG) which is the first method capable of estimating \(m\) directly, without priors or restarts. 
Our estimator is inexpensive: it requires no additional function or gradient evaluations, only the storage of one extra iterate and gradient already computed by NAG. 
We prove that, by estimating the smoothness parameter \(L\) via backtracking, NAG-free converges globally at least as fast as gradient descent. 
We also prove that, given an upper bound on \(L\), NAG-free achieves accelerated convergence locally near the minimum under local smoothness of the Hessian and some mild additional assumptions. 
Finally, we present experiments with smooth and nonsmooth Hessians on both synthetic and real-world data which demonstrate that NAG-free is competitive with restart methods, and naturally adapts to favorable local curvature conditions.

\end{abstract}


\section{Introduction}
\label{introduction}

Accelerated methods are special for achieving optimal convergence rates among first-order optimization algorithms on key problem classes~\citep{Nesterov2018}.
A notable example is \(\mathcal{S}(L,m)\), the class of Lipschitz-smooth, strongly convex functions characterized by the smoothness parameter \(L\) and the strong convexity parameter \(m\), which finds applications in signal processing \cite{Combettes2005}, imaging \cite{Chambolle2016} and machine learning \cite{Sra2011}.
To apply accelerated methods effectively in this setting, both \(L\) and \(m\) must be known; yet, as noted by \citet[p.463]{Boyd2004}, these parameters ``are known only in rare cases.'' 
While \(L\) can be bounded via backtracking~\citep{Tseng2008,Beck2009}, ``estimating the strong convexity parameter is much more challenging'' \citet[p.3]{ODonoghue2015}.
As \citet[p.21]{Su2016} put it: ``while it is relatively easy to bound the Lipschitz constant \(L\) by the use of backtracking, estimating the strong convexity parameter \(m\), if not impossible, is very challenging.'' 
In this context, applying the general idea of restarting an algorithm \citep{Pokutta2020} to accelerated methods has emerged as the only viable approach to handle unknown \(m\)~\citep[Sec.~6]{dAspremont2021}.
Roughly speaking, there are two ways in which robustness to unknown \(m\) can be achieved, depending on the purpose of restarting.
When restarting is used to accelerate a base method, the optimal restarting schedule is a function of \(m\).
Then, the first way to handle unknown \(m\) is by devising schedules without explicit knowledge of \(m\).
For example, \cite{Roulet2017} propose to run independent schedules in parallel, each of which corresponds to a guess of the true \(m\).
Then, by efficiently selecting the \(m\) guesses, one of the schedules is guaranteed to be close to optimal, resulting in a total logarithmic suboptimality factor when accounting for all the schedules.7
In contrast, in \citep{Aujol2024}, the schedule is determined one batch of iterations at a time, by estimating the condition number of the problem.
These are predetermined restarting criteria, in the sense that base methods are restarted after a fixed number of iterations, which cannot be changed during the actual execution of the base method.
The second way to achieve robustness to unknown \(m\) is when restarting is used to correct accelerated algorithms that underperform, according to some restarting criteria checked at runtime \citep{Nesterov2013,Lin2015a,ODonoghue2015,Liang2022,Renegar2022,Lan2023,Sujanani2025}.
For example, in \citep{Sujanani2025}, if the base method falls behind accelerated convergence predicted using an estimate of \(m\), then the estimate is updated and the base method is restarted.
A simpler way to preemptively restart, proposed by \cite{ODonoghue2015}, is to check for nonmonotonicity in function values, which is shown to indicate an underestimate of \(m\).

\subsection*{Contributions}

The main contribution of this paper is \emph{NAG-free}, a simple extension of Nesterov’s accelerated gradient method (NAG) which is the first method capable of estimating the strong convexity parameter \(m\) directly, without priors or restarts.
Our estimator is inexpensive: it requires no additional function or gradient evaluations, only the storage of one extra iterate and gradient already computed by NAG. 
We prove that, by estimating the smoothness parameter \(L\) via backtracking, NAG-free converges globally at least as fast as gradient descent.
A byproduct of this result, and a secondary contribution of this paper, is that NAG converges globally at least as fast as GD even if it is parameterized with an overestimate of \(m\).
Moreover, given an upper bound on \(L\), we show that NAG-free achieves accelerated convergence locally near the minimum under local smoothness of the Hessian and some mild additional assumptions.
Smooth and strongly convex problems with locally smooth Hessians include regularized logistic loss, exponential-family negative log-likelihoods with bounded natural parameters, and Moreau-envelope smoothing of any smooth and strongly convex function.
The main class of mainstream machine learning problems that are smooth and strongly convex but lack locally smooth Hessians are those with max-squared objectives plus \(\ell_{2}\) regularization--notably squared-hinge support vector machine and Huber regression.
We present experiments with smooth and nonsmooth Hessians on both synthetic and real-world data which demonstrate that NAG-free is competitive with restart-based methods, and naturally adapts to favorable local curvature conditions.

\section{Preliminaries}
\label{preliminaries}

Consider the task of finding \(x^{\star}(f)\), the unique solution of the unconstrained optimization problem 
\begin{align}
    \min_x f(x), 
    \label{prob:min}
\end{align}
where \(f \in \mathcal{S}(L,m)\), the set of Lipschitz-smooth strongly convex functions, defined below.

\begin{definition}[Lipschitz-Smooth and Strongly Convex Functions.]\label{def:lssc}
    We say that a differentiable function \(f : \mathbb{R}^d \to \mathbb{R}\) belongs to \(\mathcal{S}(L,m)\), the set of Lipschitz-smooth and strongly convex functions, if there exist \(L>0\) and \(m>0\) such that for all \(x, y \in \mathbb{R}^d\)
    \begin{align}
        f(y)
        \leq 
        f(x) 
        + \langle \nabla f(x), y - x \rangle
        + (L/2) \Vert y - x \Vert^{2}, \text{ and }
        \label{ineq:descent_lemma}
    \end{align}
    \begin{align}
        f(x) 
        + \langle \nabla f(x), y - x \rangle
        + (m/2) \Vert y - x \Vert^{2}
        \leq f(y).
        \label{ineq:strong_convexity}
    \end{align}
\end{definition}

The following will be useful for analyzing NAG-free locally.

\begin{definition}[Locally H\"{o}lder-smooth Hessian]\label{def:locally-holder-hessian}
    Let \(f \in \mathcal{S}(L,m)\) be twice differentiable at \(x^{\star}=x^{\ast}(f)\).
    Then, \(\nabla^{2}f\) is called locally H\"{o}lder-smooth at \(x^{\star}\) if there are \(\delta_{H}\), \(L_{H}\) and \(\alpha_{H}\) such that
    \begin{align}
        \Vert \nabla^{2}f(x) - \nabla^{2}f(x^{\star}) \Vert \leq L_{H}\Vert x - x^{\star} \Vert^{\alpha_{H}},
        &&
        \forall \Vert x - x^{\star} \Vert \leq \delta_{H}.
        \label{ineq:locally-holder-hessian}
    \end{align}
\end{definition}
\section{The NAG-free algorithm}
\label{nag-free}

If \(f \in \mathcal{S}(L,m)\), then for all \(x \neq y\)
\begin{align}
    m 
    \leq c(x,y) 
    \coloneqq \Vert \nabla f (x) - \nabla f(y) \Vert / \Vert x - y \Vert
    \leq L,
    \label{ineq:effective-curvature}
\end{align}
which follows from standard results on smooth and strongly convex functions~\cite[theorems 2.1.5 and 2.1.10]{Nesterov2018}. 
The quantity \(c(x,y)\) captures a local notion of curvature between two points and lies in the interval \([m, L]\).
Given iterates \(x_{t}\) and \(x_{t-1}\) produced by Nesterov's accelerated gradient method (NAG) and letting \(c_{t} = c(x_{t}, x_{t-1})\) and \(\gamma>1\), we propose to estimate \(m\) online via the following update: if \(c_{t} < m_{t-1}\), then \(m_{t} = \min(m_{t-1}/\gamma,c_{t})\), otherwise \(m_{t}=m_{t-1}\).
This update guarantees substantial improvement when the estimate decreases.
Moreover, since \(c_{t}\geq m\), it follows that \(m_{t}\) can only take finitely many values, which is important for theoretical reasons that will become clearer later.
Then, \(m_{t}\) parameterizes NAG to produce a new iterate, which in turn feeds \(c_{t+1}\), to update \(m_{t+1}\).
The resulting procedure is computationally lightweight: it reuses gradients already computed by NAG and only requires storing one additional iterate and gradient.
To initialize \(m_{0}\), we simply set \(m_{0}=L_{0}\).
The input \(L_{0}\) is an initial estimate of smoothness parameter \(L\), which is refined in every iteration through classical backtracking.
\cref{alg:nag-free} summarizes the complete procedure, which we call \emph{NAG-free}.

\begin{algorithm2e}
    \caption{NAG-free, an extension of NAG that estimates the strong convexity parameter.}
    \label{alg:nag-free}
    \SetAlgoLined
    \KwData{\(T>0, x_{0}=y_{0}, L_{0} > 0, \gamma > 1, \gamma_{L} > 1\)}
    \KwResult{\(x_{T}, y_{T}\)}
    $m_{0} \gets L_{0}$
    \tcp*{initialization}
    \For{$t=0,1,\ldots,T-1$}{
        $y_{t+1} \gets x_{t} - (1/L_{t})\nabla f(x_{t})$ \tcp*{NAG \#1}
        \lWhile{$f(y_{t+1}) > f(x_{t}) -(1/2L_{t})\Vert \nabla f(x_{t}) \Vert^{2}$}{
            \tcp*{BLS}
            \Indp 
            $L_{t} \gets \gamma_{L} L_{t}$\\
            $y_{t+1} \gets x_{t} - (1/L_{t})\nabla f(x_{t})$
        }
        \Indm
        $L_{t+1} \gets L_{t}$\\
        $x_{t+1} \gets y_{t+1} + \frac{\sqrt{L_{t}}-\sqrt{m_{t}}}{\sqrt{L_{t}}+\sqrt{m_{t}}}(y_{t+1}-y_{t})$\tcp*{NAG \#2}
        $c_{t+1} \gets \Vert \nabla f(x_{t+1}) - \nabla f(x_{t}) \Vert/\Vert x_{t+1}-x_{t} \Vert$\tcp*{estimate \(m\)}
        \eIf{$c_{t+1} < m_{t}$}{
            $m_{t+1} \gets \min(m_{t}/\gamma,c_{t+1})$\\
        }
        {
            $m_{t+1} \gets m_{t}$\\
        }
    }
\end{algorithm2e}

\subsection*{Convergence intuition}

Two key features underlie the convergence of NAG-free:

\begin{enumerate}[left=0pt]
    \item \textbf{Adaptive interpolation between GD and NAG.} The NAG-free iteration is a convex combination of GD and NAG, for \(m_{t}\geq m\).
    If \(m_{t} \to L\), then its momentum coefficient becomes zero, and NAG-free iterates as GD.
    Otherwise, if \(m_{t} \to m\), then NAG-free iterates as NAG.
    Backtracking preserves the convergence guarantees of both GD and NAG, up to a suboptimality factor of \(\gamma_{L}\).
    Thus, NAG-free converges globally at least as fast as GD, the slowest between GD and NAG.
    
    \item \textbf{Power iteration-like behavior near the optimum.} Near the optimum, the curvature estimate \(c_{t}\) evolves similarly to a power method applied to the Hessian, with some additional dynamics.
    As a result, the iterate \(x_{t}\) rapidly concentrates in the eigenspace corresponding to the least eigenvalue of the Hessian, \(m\), which translates into \(c_{t}\) quickly approaching \(m\), accelerating NAG-free.
\end{enumerate}

\section{Summary of convergence guarantees}
\label{convergence_guarantees}

In this section, we summarize the most important convergence results for NAG-free.
The full derivation of global convergence and local acceleration can be found in \cref{app:global_convergence,app:local_acceleration}, respectively.

\subsection{Global convergence}

The main global convergence result for NAG-free is that it always converges at least as fast as GD.

\begin{theorem}\label{thm:gc}
    Let \(f\in\mathcal{S}(L,m)\) and suppose that \(\kappa=L/m \geq 2\).
    If \(y_{t}\) are generated by \cref{alg:nag-free} for given \(L_{0}\geq m\), \(\gamma \leq 2\) and \(\gamma_{L}\leq 2\), then letting \(\bar{\kappa}=\max(L_{0},2L)/m\), we have that
    \begin{align}
        f(y_{t+1}) - f(x^{\star})
        \leq
        \Bigl( 1 - \frac{1}{\bar{\kappa}} \Bigr)^{t}
        8\max(L_{0},L)\bar{\kappa}^{3}
        \Vert x_{0} - x^{\star} \Vert^{2}.
        \label{ineq:gc}
    \end{align}
\end{theorem}

\begin{sketch}
    We split the proof in two cases: \(m_{t}\geq m\) and \(m_{t}< m\).
    The second case can arise if \(c_{t} \in [m, \gamma m)\) and \(c_{t} < m_{t}\), which implies that \(m_{t+1} \leq c_{t}/\gamma < m\).
    
    The key observation underpinning the analysis of the first case is that the NAG-free iterates are convex combination of GD and NAG iterates.
    Indeed, given  some \(x_{t}\) and an estimate \(L_{t}\), GD, NAG and NAG-free compute the same gradient descent step:
    \begin{align*}
        y_{t+1} = x_{t} - (1/L_{t})\nabla f(x_{t}).
    \end{align*}
    Thus, the NAG-free gradient step is a convex combination of GD's and NAG's for every \(\alpha_{t}\in [0,1]\).
    Moreover, if \(m_{t}\geq m\), then defining the momentum coefficients
    \begin{align*}
        \beta_{t}
        = \frac{\sqrt{L_{t}} - \sqrt{m_{t}}}{\sqrt{L_{t}} + \sqrt{m_{t}}},
        &&
        \text{and}
        &&
        \theta_{t}
        = \frac{\sqrt{L_{t}} - \sqrt{m}}{\sqrt{L_{t}} + \sqrt{m}},
    \end{align*}
    we can express the momentum step \(x_{t+1}\) of NAG-free as
    \begin{align*}
        x_{t+1} 
        = (1+\beta_{t})y_{t+1} -\beta_{t}y_{t}
        =&\ \Bigl( 1+\theta_{t}\frac{\beta_{t}}{\theta_{t}} + \frac{\beta_{t}}{\theta_{t}} - \frac{\beta_{t}}{\theta_{t}} \Bigr)y_{t+1} -\theta_{t}\frac{\beta_{t}}{\theta_{t}} y_{t}
        \\
        =&\ ( 1-\alpha_{t} ) y_{t+1} + \alpha_{t}((1+\theta_{t})y_{t+1} -\theta_{t} y_{t}),
    \end{align*}
    where \(\alpha_{t} = \beta_{t}/\theta_{t} \in [0,1]\).
    Now, \(y_{t+1}\) is exactly the (trivial) momentum step computed by GD, and \((1+\theta_{t})y_{t+1} -\theta_{t} y_{t}\) is the momentum step computed by NAG, parameterized by the true value of \(m\).
    Hence, the NAG-free momentum step is a convex combination of GD's and NAG's for \(\alpha_{t} = \beta_{t}/\theta_{t}\).
    Therefore, the entire NAG-free iteration is a convex combination of GD's and NAG's for this \(\alpha_{t}\).

    To conclude the analysis of the first case, we construct a convex Lyapunov function \(V^{\textup{GD}}_{t}\) with two components: \(W_{t}\)  and \(U_{t}\).
    The first is a common convex Lyapunov function \(W_{t}\) for GD and NAG.
    The second serves to "patch" the mismatch \(x_{t}\neq y_{t}\), which does not occur for GD.
    Then, we obtain \eqref{ineq:gc} by combining the convex property of NAG-free iterations with convexity of \(V^{\textup{GD}}_{t}\).

    In the second case, \eqref{ineq:gc} follows by constructing a Lyapunov function for the NAG-free iteration.
    
\end{sketch}

The full proof of \cref{thm:gc} can be found in \cref{app:global_convergence}.
A byproduct of this proof is that even if NAG uses an overestimate of \(m\), it still converges at least as fast as GD.
The precise statement is the following corollary, which is also proven in \cref{app:global_convergence}.

\begin{cor}\label{cor:nag}
    Let \(f\in\mathcal{F}(L,m)\), \(m' \in\mathopen{[}m,L\mathclose{]}\) and \(\kappa=L/m\).
    If \(y_{t}\) are generated by NAG parameterized with \(m'\) instead of \(m\), then
    \begin{align}
        f(y_{t+1}) - f(x^{\star})
        \leq
        \Bigl( 1 - \frac{1}{\kappa} \Bigr)^{t}
        2L\kappa^{2}
        \Vert x_{0} - x^{\star} \Vert^{2}.
        \label{ineq:nag}
    \end{align}
\end{cor}

\subsection{Local acceleration}

The main local convergence result for NAG-free is that it achieves acceleration near the minimum, under some additional assumptions.

\begin{assumption}\label{ass:locally-smooth-hessian}
    The Hessian of \(f\) is locally H\"{o}lder-smooth at \(x^{\star}\). 
That is, there exist \(\delta_{H} > 0\), \(L_{H} > 0\), and \(\alpha_{H} > 0\) such that if \(\Vert x - x^{\star} \Vert \leq \delta_{H}\), then \(\Vert \nabla^{2}f(x) - \nabla^{2}f(x^{\star}) \Vert \leq L_{H}\Vert x - x^{\star} \Vert^{\alpha_{H}}\).
\end{assumption}

\begin{assumption}\label{ass:known-lipschitz-upper-bound}
    Given \(f\in\mathcal{S}(L,m)\), there is some \(L_{0}>L\) that can be used by \cref{alg:nag-free}.
\end{assumption}

To present the remaining assumptions, we introduce some notation.

\paragraph{Notation}
\cref{ass:locally-smooth-hessian} implies that \(f\) is twice differentiable.
Then, the Hessian \(\nabla^{2}f(x^{\star})\) of \(f\) is real symmetric, by Schwarz's Theorem.
In turn, the eigenvectors \(v_{i}\) associated with the \(d\) eigenvalues \(\lambda_{i}\) of \(\nabla^{2}f(x^{\star})\) can be chosen to form an orthonormal basis for \(\mathbb{R}^{d}\), by the Spectral Theorem.
Hence, for each \(x_{t}\), there exist \(d\) unique coordinates \(x_{t,i}\) such that \(x_{t}-x^{\star}=\sum_{i=1}^{d}x_{t,i}v_{i}\).
We refer to \(x_{t,i}\) as the \(\lambda_{i}\)-coordinate of \(x_{t} - x^{\star}\).
Moreover, if \(f\in\mathcal{S}(L,m)\), then \(\lambda_{i}\in\mathopen{[}m,L\mathclose{]}\).
Without loss of generality, we assume \(m=\lambda_{1} \leq \ldots \leq \lambda_{d}=L\).
Then, for example, \(x_{0,1}\) denotes the \(m\)-coordinate of \(x_{0}-x^{\star}\).

\begin{assumption}\label{ass:mt-separation-from-eigs}
    There exists some \(\delta_{\lambda}\in\mathopen{(}0,1\mathclose{)}\) such that \(\vert m_{t} -\lambda_{i} \vert > \delta_{\lambda}L\) for every \(\lambda_{i} > m\).
\end{assumption}

\begin{assumption}\label{ass:xi0-lower-bound}
    There exists some \(\omega>0\) such that \(\omega x_{0,1}^{2} \geq  \Vert x_{0}-x^{\star} \Vert^{2}\).
\end{assumption}

\cref{ass:mt-separation-from-eigs} simplifies the analysis and is not strictly necessary.
\cref{ass:xi0-lower-bound} prevents pathological cases in which \(x_{0,1}\), the \(m\)-coordinate of \(x_{0}-x^{\star}\), is arbitrarily small compared with the other coordinates.
In fact, we shall see in \cref{numerical_experiments} that violations of this assumption translate into more favorable problem conditioning, actually improving the performance of NAG-free.

\begin{theorem}\label{thm:la}
    Let \(f\in\mathcal{S}(L,m)\), suppose that \cref{ass:locally-smooth-hessian,ass:known-lipschitz-upper-bound,ass:mt-separation-from-eigs,ass:xi0-lower-bound} hold and \(\bar{\kappa} = L_{0}/m > L/m \geq 4\).
    There is \(\epsilon>0\) such that if \(\Vert x_{0} - x^{\star} \Vert \leq \epsilon\), then the iterates \(x_{t}\) produced by NAG-free satisfy
    \begin{align}
        \Vert x_{t+1}-x^{\star} \Vert
        \leq 
        C
        \bigg(
        1 - \frac{1}{\sqrt{\sigma\bar{\kappa}}}
        \bigg)^{t} 
        \Vert x_{0} - x^{\star} \Vert,
        \label{ineq:la}
    \end{align}
    where \(\sigma\) depends on \(\gamma\), \(C\) depends on \(\bar{\kappa}\) and \(\omega\), with \(\omega\) given by \cref{ass:xi0-lower-bound}.
\end{theorem}

\begin{sketch}
    The dynamics of NAG-free around \(x^{\star}\) can be analyzed as two consecutive regimes.
    In the initial regime, \(m_{t}\) approaches \(m\) at an accelerated rate.
    After \(m_{t}\) becomes sufficiently accurate, the final regime begins and \(x_{t}\) converges to \(x^{\star}\) at an accelerated rate.

    We consider \(m_{t}\) to be sufficiently accurate when \(m_{t}\in[m/\gamma, (1+\delta_{m})m]\), for some appropriate \(\delta_{m}>0\).
    By design, \(m_{t}\) is nonincreasing, and decreases by a factor of at least \(\gamma>1\) every time it is updated to a new value.
    Moreover, updates can only occur if \(m_{t} > m\), since they are triggered when \(c_{t} < m_{t}\), and we have that \(c_{t}\geq m\), by \eqref{ineq:effective-curvature}.
    Therefore, \(m_{t}\) can take at most \(\log_{\gamma}(L/m)\) distinct values.
    In addition, \(m_{t}\) never leaves an interval \([m/\gamma,(1+\delta_{m})m]\) after reaching it.
    That is, NAG-free remains in the final regime after entering it.

    First, we analyze the two regimes considering the case where \(f\in \mathcal{S}(L,m)\) is quadratic.
    In this case, \(x_{t}-x^{\star}=\sum_{i=1}^{d}x_{t,i}v_{i}\) behaves as a linear system, and the dynamics of each coordinate \(x_{t,i}\) is determined by \(\lambda_{i}\), \(L_{t}\) and \(m_{t}\).
    By \cref{ass:known-lipschitz-upper-bound}, \(L_{0}>L\), for which the Descent Lemma is always satisfied, so that \(L_{t}\equiv L_{0}\) for all \(t \geq 0\).
    Since \(m_{t}\) can only take finitely many values, it follows that \(x_{t,i}\) evolve as a sequence of finitely many linear time-invariant (LTI) systems.
    Then, we note that \(\Vert \nabla^{2}f(x^{\star})(x_{t+1}-x_{t})\Vert^{2} = \sum_{i=1}^{d}\lambda_{i}^{2}(x_{t+1,i}-x_{t,i})^{2}\), since \(v_{i}\) are orthonormal.
    That is, \(c_{t+1}\) amounts to an average of eigenvalues weighted by \((x_{t+1,i}-x_{t,i})^{2}\), which can be seen as a power iteration with the additional dynamics of \(x_{t+1}-x_{t}\).
    We characterize the dynamics of the weights \((x_{t+1,i}-x_{t,i})^{2}\), using \cref{ass:mt-separation-from-eigs} to obtain slightly more convenient solutions.
    Then, we show that \(x_{t+1,i}-x_{t,i}\) for which \(\lambda_{i} > m_{t}\) decrease much faster than those for which \(\lambda_{i} \leq m_{t}\) while \(m_{t} > (1+\delta_{m})m\), for appropriate \(\delta_{m}\).
    We use this fact to further show that \(m_{t}\) reaches \([m/\gamma,(1+\delta_{m})m]\) at a rate \(1-1/\sqrt{\sigma_{\phi}\bar{\kappa}}\), where \(\sigma_{\phi}\) represents a suboptimality factor with respect to the nominal accelerated rate, \(1-1/\sqrt{\bar{\kappa}}\).
    Once \(m_{t}\) reaches \([m/\gamma,(1+\delta_{m})m]\), the final regime begins.
    The convergence of \(x_{t}-x^{\star}\) is dominated by the slowest coordinate \(x_{t,i}\) of \(x_{t}\), which we show is the \(m\)-coordinate.
    In turn, \(x_{t,1}\) converges at a rate determined by the worst case \(m\) estimates in the interval \([m/\gamma,(1+\delta_{m})m]\), which are precisely its endpoints, \(m/\gamma\) and \((1+\delta_{m})m\).
    These endpoints translate into suboptimality factors \(\gamma\) and \(\sigma_{m}=1+2\delta_{m}+2\sqrt{\delta_{m}(1+\delta_{m})}\), respectively.
    We conclude the analysis of the final regime with a Lyapunov argument, in which we use quadratic Lyapunov functions defined by positive definite matrices \(P_{i}\) for each \(x_{t,i}\).
    As in any argument involving a quadratic Lyapunov function for linear systems, there is a trade-off between the conditioning of \(P_{i}\) and the rate at which these functions decrease along \(x_{t,i}\) trajectories.
    A further suboptimality factor, \(\sigma_{P}\), accounts for this trade-off.
    Then, the overall convergence rate involves a suboptimality factor of \(\sigma_{P}\max\{\sigma_{\phi}, \sigma_{m}, \gamma\}\).

    Next, we consider the general case in which \(f\in \mathcal{S}(L,m)\) need not be quadratic.
    By \cref{thm:gc}, \(\Vert x_{t} - x^{\star} \Vert = \bigO(\Vert x_{0} - x^{\star} \Vert)\).
    Hence, for every \(\delta_{H}>0\), there is some \(\epsilon>0\) such that if \(\Vert x_{0} - x^{\star} \Vert \leq \epsilon\), then \(\Vert x_{t} - x^{\star} \Vert \leq \delta_{H}\) for all \(t\geq 0\).
    Thus, under \cref{ass:locally-smooth-hessian}, \(\epsilon\) can be chosen sufficiently small so that \(x_{t}\) remains in the region around \(x^{\star}\) where \(\nabla^{2}f(x^{\star})\) is H\"{o}lder-smooth.
    For such \(\epsilon\), \cref{thm:gc} also implies that the error \(\Vert \nabla^{2}f(x_{t}) - \nabla^{2}f(x^{\star}) \Vert\) converges to 0 at a linear rate.
    We can express the dynamics of \(x_{t}-x^{\star}\) as a perturbation of the dynamics in the quadratic case, where the perturbation error is \(\bigO(\Vert \nabla^{2}f(x_{t}) - \nabla^{2}f(x^{\star}) \Vert)\).
    Since the perturbation converges to 0 at a linear rate, the solutions in the general case correspond to perturbed solutions of the quadratic case, a result that follows from \cite[theorem 3.4]{Bodine2015} under \cref{ass:mt-separation-from-eigs}.
    Moreover, the size of the solution perturbation can be controlled by \(\epsilon\).
    Hence, under \cref{ass:xi0-lower-bound}, \(\epsilon\) can be taken sufficiently small such that the coordinate dynamics of \(x_{t,1}\) is roughly preserved from the quadratic case.
    Then, we can analyze the initial regime in the general case using similar arguments from the quadratic case.
    Finally, we analyze the final regime, again using a Lyapunov argument and controlling \(\epsilon\) to recover roughly the same suboptimality factors from the quadratic case.
    
\end{sketch}

The full proof of \cref{thm:la} and a thorough discussion of \(C\) and \(\sigma\) can be found in \cref{app:local_acceleration}.
For now, we mention that the suboptimality factor \(\sigma\) is comparable to those of restart schemes, where ``the convergence rate is slowed down by roughly a factor four'' \cite[page 167]{dAspremont2021}.
To see this, consider the terms involved in \(\sigma = \sigma_{P}\max\{\gamma, \sigma_{m}, \sigma_{\phi}\}\).
The first term, \(\sigma_{P}\), arises from the analysis of the second convergence regime of NAG-free, which is based on a Lyapunov function defined by a matrix \(P \succ 0\).
Specifically, \(\sigma_{P}\) reflects the trade-off between the condition number of \(P\) and the rate of convergence guaranteed by the Lyapunov function.
This trade-off is typically ignored in analyses of this kind \cite{Lessard2016}, since the condition number of \(P\) captures transient factors that quickly become negligible.
A realistic estimate for the remaining suboptimality factors can be obtained by taking \(\gamma=2\), \(\delta_{u}=0.01\), \(\delta_{m}=0.2\) and \(\delta_{\ell}=(1+\delta_{m})/(1+\delta_{u})-1\), which results in \(\sigma_{m}\approx 2.19\) and \(\sigma_{\phi}\lesssim 2.38\).
Therefore, \(\sigma \leq 3\).
\section{Numerical experiments}
\label{numerical_experiments}

In this section, we present numerical experiments to validate NAG-free.
Further results can be found in \cref{app:ne}.
Given problem-specific bounds \(\bar{L} \geq L\), we consider two instances of NAG-free: one with \(L_{0}=\bar{L}\) and the other with \(L_{0}=0.01\bar{L}\), which we respectively refer to as \(\text{NAG-free } (\bar{L})\) and \(\text{NAG-free } (0.01\bar{L})\).
Otherwise, both instances use the same values for \(\gamma=\gamma_{L}=1.5\).
Thus, \(\text{NAG-free } (\bar{L})\) satisfies the assumptions of both \cref{thm:gc} and \cref{thm:la}, whereas \(\text{NAG-free } (0.01\bar{L})\) only satisfies the assumptions of \cref{thm:gc}.
On the other hand, backtracking is active for NAG-free \((0.01\bar{L})\), but not NAG-free \((\bar{L})\).

To assess how well the two NAG-free instances perform, we consider four additional methods.
To initialize the methods that require \(m\) to be passed as an input, we use the regularization parameter \(\eta \leq m\).
The four methods are:
\begin{itemize}
    \item NAG using \(L=\bar{L}\) and \(m=\eta\);
    \item Triple momentum method \cite[TM]{VanScoy2017} using \(L=\bar{L}\) and \(m=\eta\);
    \item Restarting NAG \cite[NAG+R]{ODonoghue2015} when \(f(x_{t}) > f(x_{t-1})\), using \(L=\bar{L}\);
    \item Restarting NAG \cite[NAG+RB]{ODonoghue2015} when \(f(x_{t}) > f(x_{t-1})\), where an initial estimate \(L_{0}=\bar{L}/100\) is given and then refined through backtracking with \(\gamma_{L}=1.5\).
\end{itemize}
In the following, we collectively refer to NAG-free and restarting methods as \emph{adaptive}.
Out of the two restarting criteria proposed by \cite{ODonoghue2015}, we picked \(f(x_{t}) > f(x_{t-1})\) because it worked best in experiments.
Another remark is that the Descent Lemma condition can be unreliable numerically.
Therefore, in practice we enforce the relaxed condition that
\begin{align*}
    f(y_{t+1}) \leq (1+\texttt{tol})(f(x_{t}) -(1/2L_{t})\Vert \nabla f(x_{t})\Vert^{2}),
\end{align*}
where \(\texttt{tol}=\num{1e-3}\) for the last experiment (\cref{ne:svm}) and \(\texttt{tol}=\num{1e-6}\) for the rest.

\subsection{Logistic regression}

First, we consider the logistic regression problem, defined by the objective
\begin{align} 
    f(x) 
    = -\frac{1}{n}\sum_{i=1}^{n}
    \Bigl(
    b_{i}\log(s(A_{i}^{\T}x)
    + (1-b_{i})\log(s(A_{i}^{\T}x))
    \Bigr)
    + \frac{\eta}{2}\Vert x \Vert^{2},
    \label{prob:logreg}
\end{align}
where \(s(z)=1/(1+\exp(-z))\) is the sigmoid function, \(A_{i}\in\mathbb{R}^{d}\) and \(b_{i}\in\{0,1\}\) denote the features and labels of \(n\) datapoints, and \(\eta>0\) is a regularization parameter.
The datapoints are taken from LIBSVM datasets \cite{Chang2011}.
For \eqref{prob:logreg}, we have that \(L \leq \bar{L}=(1/4)\lambda_{\max}(A^{\T}A) + \eta\), where \(A\) is the matrix with rows given by \(A_{i}^{\T}\), and \(\lambda_{\max}(A^{\T}A)\) denotes the top eigenvalue of \(A^{\T}A\).

\begin{figure}[tb]
    \centering
    \includegraphics[width=.9\linewidth]{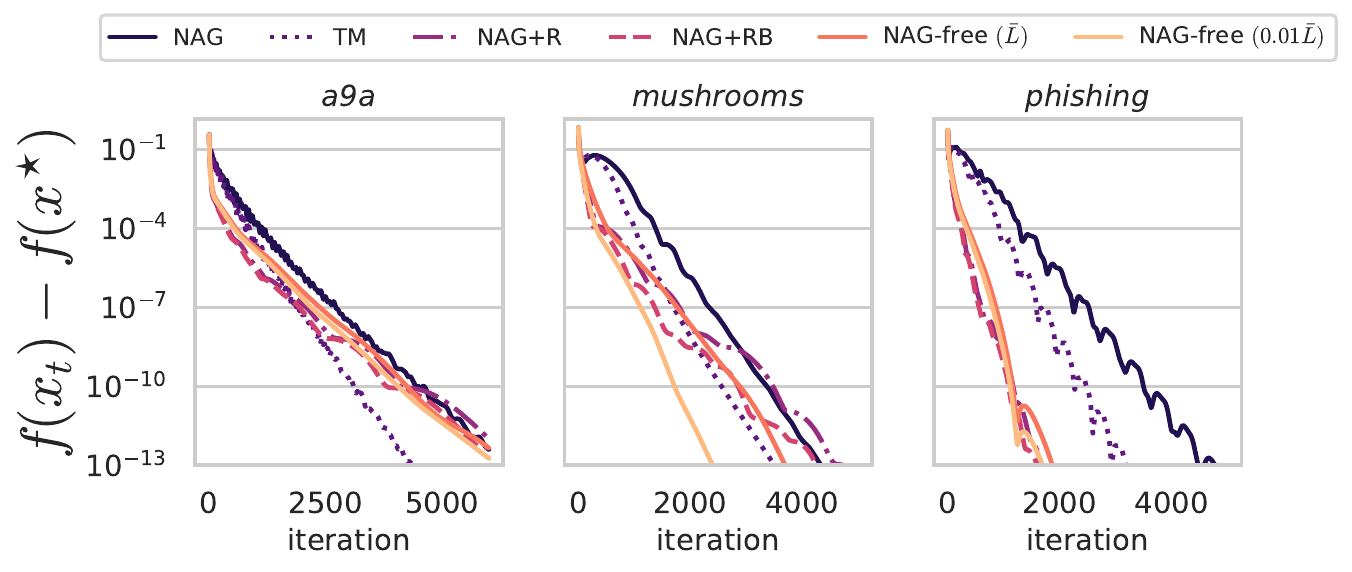}
    \caption{Suboptimality gap for logistic regression on three datasets.}
    \label{fig:logreg_summary}
\end{figure}

\cref{fig:logreg_summary} shows the suboptimality gap for three datasets: \textsc{a9a}, \textsc{mushrooms} and \textsc{phishing}.
A similar plot for three more datasets can be found in \cref{app:ne}.
For \textsc{a9a} and \textsc{mushrooms}, TM and \(\text{NAG-free } (0.01\bar{L})\) perform the best.
For \textsc{phishing}, the adaptive methods lead.
To interpret these results, we take a look at \(m_{\infty}\) and \(L_{\infty}\), the last \(L_{t}\) and \(m_{t}\) estimates held by \(\text{NAG-free } (0.01\bar{L})\).

We have \(L_{\infty}=\num{1.39}\) and \(m_{\infty}=\num{4.73e-6}\) for \textsc{a9a}, indicating that \(\bar{L}=\num{1.57}\) and \(\eta=4.83e-6\) are good estimates of \(L\) and \(m\).
Since TM has the best theoretical convergence rates among the methods above, we indeed expect it to outperform the other methods.
In contrast, \(\bar{L}=\num{2.59}\) and \(L_{\infty}=\num{0.80}\) for \textsc{mushrooms}, suggesting that \(\bar{L}\) is a loose estimate of \(L\).
Hence, even though \(m_{\infty}=\num{2.62e-5} \leq \num{3.18e-5} = \eta\) indicates that \(\eta\) is a good estimate of \(m\), \(\text{NAG-free } (0.01\bar{L})\) outperforms TM.
Finally, for \textsc{phishing} we have \(\bar{L}=\num{1.63e-1}\), \(L_{\infty}=\num{1.41e-1}\), \(\eta=\num{1.47e-6}\) and \(m_{\infty}=\num{8.35e-6}\).
That is, \(\bar{L}\) is a good estimate of \(L\), but \(m_{\infty}\) is over five times greater than \(\eta\).
Similarly, \(m_{\infty}=\num{8.02e-6}\) for \(\text{NAG-free } (\bar{L})\) is again over five times greater than \(\eta\).
This explains why backtracking only marginally improves the performance of NAG-free, which outperforms TM thanks to better estimates of \(m\).
Similarly, backtracking only marginally improves the performance of the restarting methods.

To confirm that \(\eta\) really is a loose estimate of the true strong convexity parameter for \textsc{phishing}, we compute the least eigenvalue of \(\nabla^{2}f(x^{\star})\).
But surprisingly, we find that they match.
At first, this seems to be at odds with the theory presented above, as we expect that at least \(m_{\infty}(\bar{L})\) to be an accurate estimate of \(m\).
To investigate this puzzle, we inspect \(c_{t}\) and \(m_{t}\) produced by \(\text{NAG-free } (\bar{L})\).

\begin{figure}[tb]
    \centering
    \subfigure[\(c_{t}\) and \(m_{t}\) for \(\text{NAG-free } (\bar{L})\) with two \(x_{0}\).]{
        \includegraphics[width=0.45\linewidth]{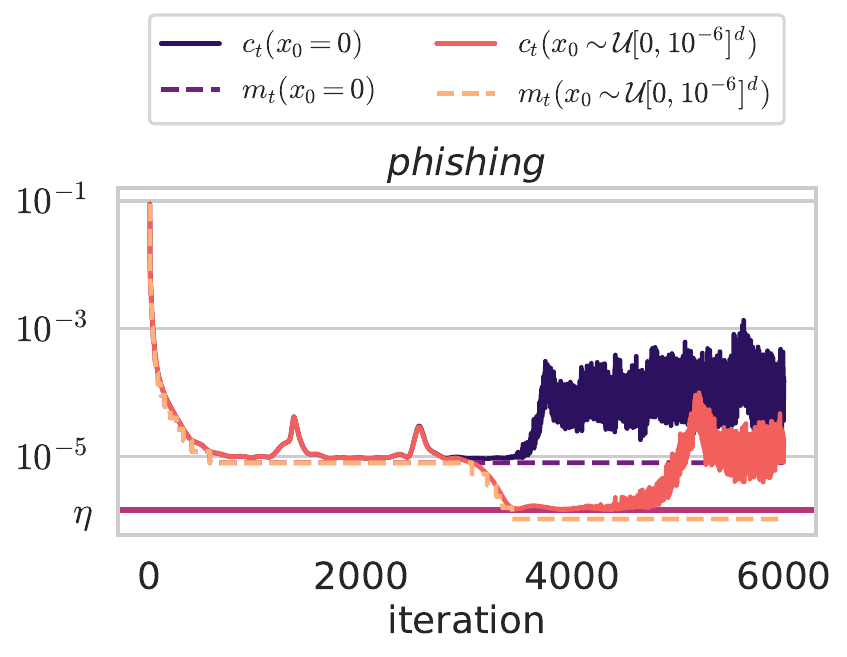}
        \label{fig:logreg_phishing_UB_m_estimates}
    }
    \hfill
    \subfigure[A simple quadratic problem.]{
        \centering
        \includegraphics[width=0.5\linewidth]{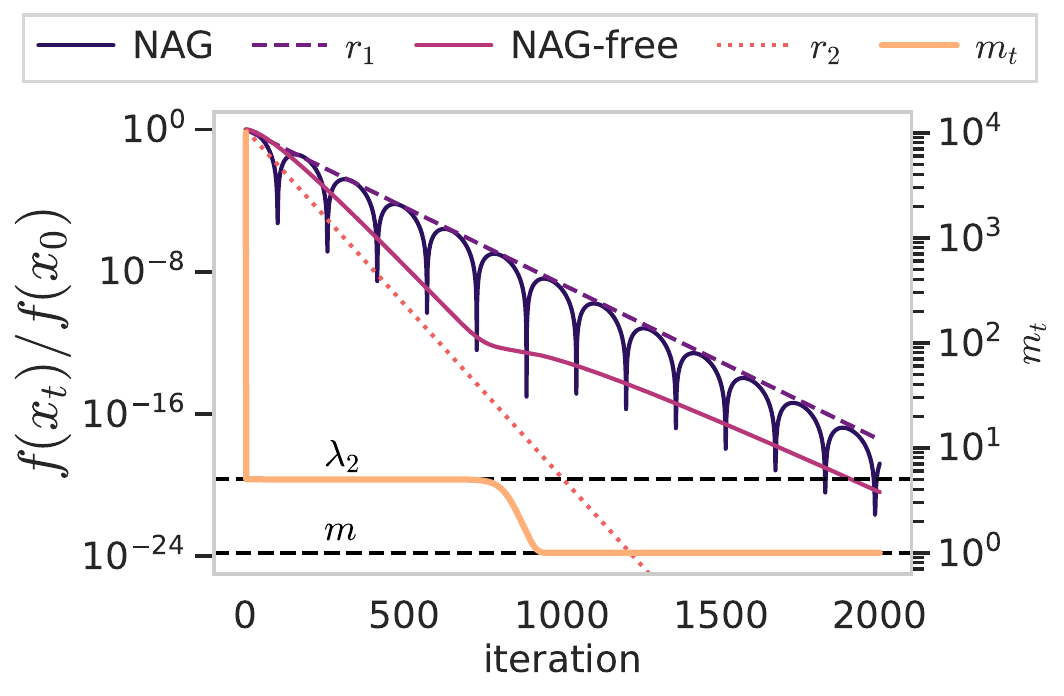}
        \label{fig:quadratics}
    }
    \caption{Left: \(c_{t}\) and \(m_{t}\) produced by \(\text{NAG-free } (\bar{L})\) for logistic regression on \textsc{phishing} dataset, when \(x_{0}=0\) and \(x_{0}\sim 10^{6}\times\mathcal{U}[0,10^{-6}]^{d}\).
    Right: left \(y\)-axis shows the normalized suboptimality gap for NAG and \(\text{NAG-free } (\bar{L})\) on a quadratic problem and the gaps corresponding to accelerated rates for \(m=1\) and \(m=5\), \(r_{\textup{1}}\) and \(r_{\textup{5}}\); right \(y\)-axis shows \(m_{t}\).}
\end{figure}

\cref{fig:logreg_phishing_UB_m_estimates} shows \(c_{t}\) and \(m_{t}\) when \(x_{0}=0\) and \(x_{0}\sim\mathcal{U}[0,10^{-6}]^{d}\).
In addition, a solid horizontal line shows \(\eta\).
When \(x_{0}=0\), \(m_{t}\) and \(c_{t}\) converge linearly to \(m_{\infty}\) at first, and then \(c_{t}\) begins to jitter, at which point the suboptimality gap has essentially reached machine precision.
In this light, we hypothesize that the \(m\)-coordinates of \(x_{t}\) are initially very small, and by the time they would become large enough for \(m_{t}\) to converge to \(m\), numerical errors corrupt \(m_{t}\).
Indeed, \(m_{t}\) eventually converge to \(\eta\) just before \(c_{t}\) start to jitter when \(x_{0}\sim \mathcal{U}[0,10^{-6}]^{d}\), which corroborates this hypothesis.

\subsection{Adapting to better local coordinates: a simple quadratic example}

In the \textsc{phishing} example above, we hypothesized that if the \(m\)-coordinates are relatively small, then the effective condition number improves, which NAG-free adapts to.
To reproduce this behavior, we consider a quadratic problem \(f(x)=(1/2)x^{\T}Hx\) with \(H=\diag(1,5,10^{4})\), and deliberately fix \(x_{0}=(1,10^{3},1)^{\T}\).
That is, the \(m\)-coordinate of \(x_{0}\) is relatively small, making \(\lambda_{2}=5\) the effective strong convexity parameter.
The left-hand y-axis on \cref{fig:quadratics} shows the normalized suboptimality gap obtained by NAG and \(\text{NAG-free } (\bar{L})\), while the right-hand y-axis shows the corresponding \(m_{t}\) for \(\text{NAG-free } (\bar{L})\).
The lines \(r_{1}\) and \(r_{2}\) represent the gaps for methods converging at accelerated rates \(r_{i}=(1-\sqrt{\lambda_{i}/L})^{2}\), where \(\lambda_{1}=1\), \(\lambda_{2}=5\) and \(L=10^{4}\).
When \(m_{t}=\lambda_{2}\), \(\text{NAG-free } (\bar{L})\) roughly converges at rate \(r_{2}\), and then at the nominal rate for this problem, \(r_{1}\), when \(m_{t}=\lambda_{1}\).

Therefore, as \textsc{phishing} shows, in practice \cref{ass:xi0-lower-bound} need not hold for \(x_{1,0}\), but it must for some other \(x_{i,0}\) with \(i\geq 1\).
In turn, the effective condition number improves to \(L/\lambda_{i} \leq L/m\).
Hence, the local strong convexity parameter may be better than the global one in two ways:

\begin{tcolorbox}[colback=gray!10, colframe=gray!80!black, boxrule=0.5pt, arc=0mm]
\begin{center}\textit{  when \(\lambda_{\min}(\nabla^{2}f(x^{\star})) > m\) and when the \(m\)-coordinate of \(x_{t}\) is relatively small.}%
\end{center}
\end{tcolorbox}

NAG-free is able to adapt to both scenarios, achieving better convergence rates.

\subsection{Log-sum-exp}

Next, we consider the smoothed and regularized log-sum-exp problem, defined by the objective
\begin{align}
    f(x) 
    = \theta \log\Biggl( \sum\limits_{i=1}^{n}\exp\biggl( \frac{A_{i}^\top x-b_{i}}{\theta} \biggr) \Biggr) 
    + \frac{\eta}{2}\Vert x \Vert^{2},
    \label{prob:log-sum-exp}
\end{align}
where \((A_{i},b_{i})\in\mathbb{R}^{d}\times\mathbb{R}\) are \(n\) datapoints sampled as \(A_{i}\sim \mathcal{U}[-1,1]^{d}\) and \(b_{i}\sim\mathcal{N}(-1,1)\) independently, while \(\eta > 0\) and \(\theta > 0\) are regularization and smoothing parameters.
The first term of \(f\) approximates the \(\max\) function, and the smaller \(\theta\), the tighter the approximation.
For this problem, we have \(L \leq \bar{L} = (1 + 1/\theta)\sigma_{\max}(A)^{2} + \eta\) and \(m\geq \eta\), where \(A\) and \(\sigma_{\max}(A)^{2}\) denote the matrix with rows given by \(A_{i}^{T}\) and the largest singular value of \(A\).

We fix \(n=600\) and \(d=100\), and take \(\theta \in \{\num{1e-2},\num{1e-1},\num{1e0}\}\) and \(\eta \in \{\num{1e-2},\num{1e-1},\num{1e0}\}\).
\cref{fig:log-sum-exp} shows the suboptimality gap for the methods described above under different \((\theta,\eta)\) settings.
The adaptive methods that use backtracking outperform all the others.
When \(m_{\infty}\) are significantly greater than \(\eta\), the other adaptive methods outperform TM.
For example, \(m_{\infty}=\num{6.68e-2}\) is over six times greater than \(\eta=0.01\) when \(\theta=1\).
As expected, increasing \(\theta\) makes \eqref{prob:log-sum-exp} harder to solve, and the opposite is true for \(\eta\).
Full plots and results can be found in \cref{app:ne}.

\begin{figure}[tb]
    \centering
    \includegraphics[width=\linewidth]{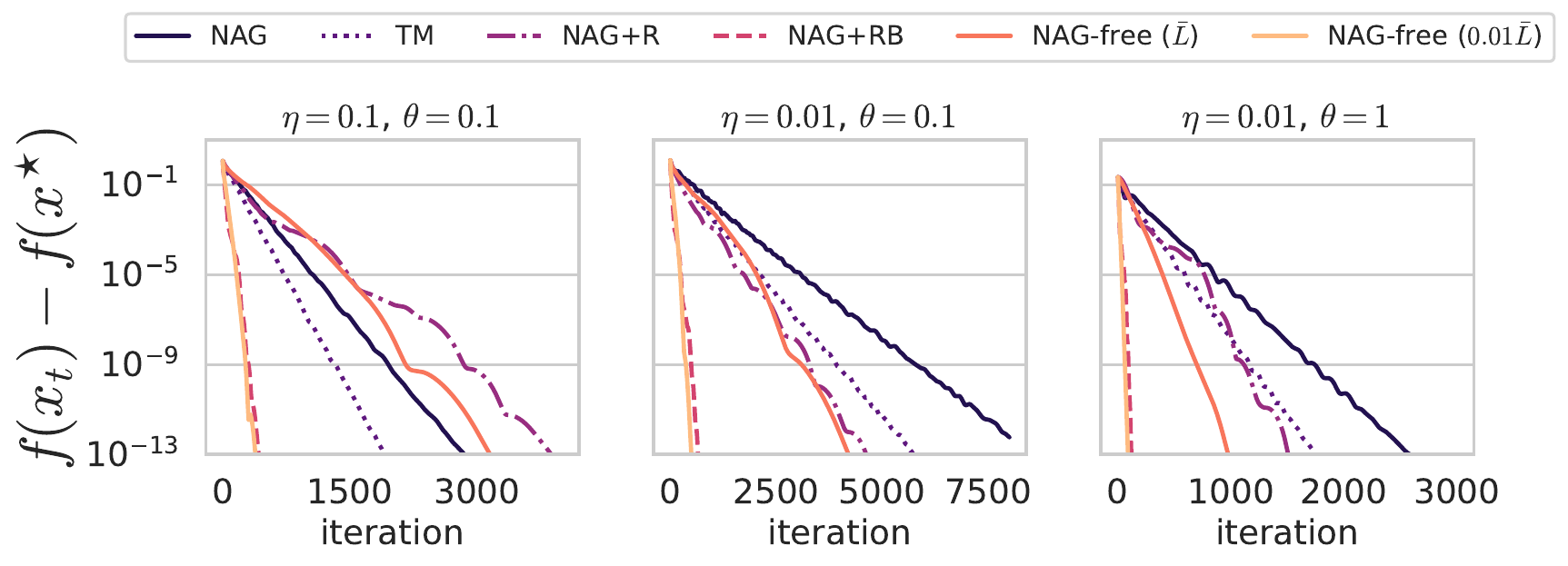}
    \caption{Suboptimality gap \(f(x_{t})-f(x^{\star})\) for log-sum-exp under different \((\eta,\theta)\) settings.}
    \label{fig:log-sum-exp}
\end{figure}

\subsection{Smooth support vector machine}
\label{ne:svm}

Our last experiment is the smooth support vector machine problem (SVM), defined by the objective
\begin{align}
    f(x) 
    = \frac{1}{n}\sum_{i=1}^{n}\max(0,1 - b_{i}A_{i}^{\T}x)^{2} 
    + \frac{\eta}{2}\Vert x \Vert^{2},
    \label{prob:svm}
\end{align}
where \(A_{i}\in \mathbb{R}^{d}\) and \(b_{i}\in \{-1,1\}\) are taken from LIBSVM datasets \cite{Chang2011}, and \(\eta>0\) is a regularization parameter.
Unlike standard SVM, the objective \eqref{prob:svm} is smooth, with \(L\leq \bar{L} = (2/n)\lambda_{\max}(A^{\T}A) + \eta\), because the \(\max\) terms in \eqref{prob:svm} are squared.
But \eqref{prob:svm} is not twice differentiable, therefore the assumptions of \cref{thm:gc} are satisfied, but not those of \cref{thm:la}.

\cref{fig:svm} shows the suboptimality gap for three datasets: \textsc{a9a}, \textsc{mushrooms} and \textsc{phishing}.
Compared with the logistic regression results, shown in \cref{fig:logreg_summary}, the most important difference is that backtracking does not improve performance, and may even degrade it.
The values of \(L_{\infty}\) produced by NAG+RB and \(\text{NAG-free } (0.01\bar{L})\) are similar to \(\bar{L}\), as shown in \cref{app:ne}.
Otherwise, we observe similar trends from the logistic regression experiment: NAG-free remains competitive with restart-based methods, and tighter \(m\) estimates lead to improved performance.
In particular, we note that the lack of Hessian smoothness does not impact NAG-free more severely than it impacts restart-based methods.
The corresponding values of \(m_{\infty}\) can be found in \cref{app:ne}.

\begin{figure}[tb]
    \centering
    \includegraphics[width=\linewidth]{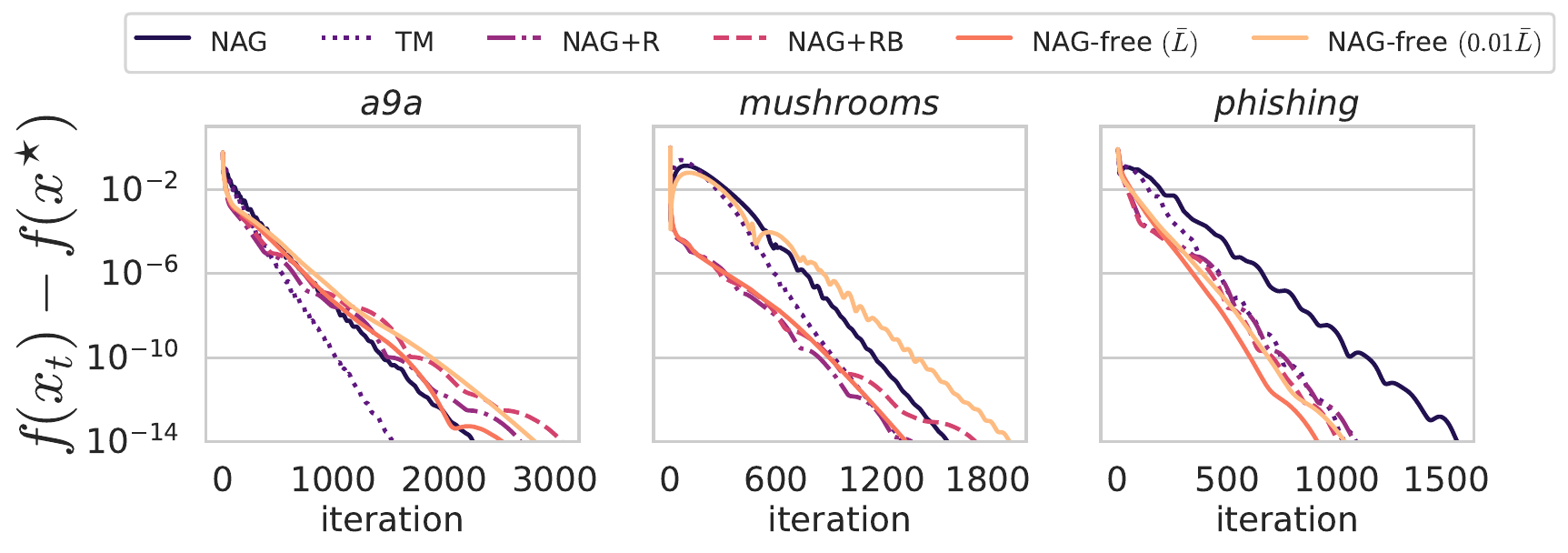}
    \caption{Suboptimality gap \(f(x_{t})-f(x^{\star})\) for SVM on three datasets.}
    \label{fig:svm}
\end{figure}
\section{Conclusion}
\label{conclusion}

In this paper, we proposed \emph{NAG-free}, a simple extension of Nesterov’s accelerated gradient (NAG) which is the first method capable of estimating \(m\) directly, without priors or restarts. 
Roughly speaking, the NAG-free iterations are convex combinations of gradient descent (GD) and NAG iterations controlled by the estimates \(m_{t}\).
In particular, if \(m_{t}=L\) then NAG-free reduces to GD, and if \(m_{t}=m\), then NAG-free reduces to NAG.
Using this property, we proved that, by estimating the smoothness parameter \(L\) via backtracking, NAG-free converges globally at least as fast as gradient descent. 
A byproduct of this analysis is that NAG converges as fast as GD even if \(m\) is overestimated.
The \(m\) estimator is simple and inexpensive, requiring no additional function or gradient evaluations, only the storage of one extra iterate and gradient already computed by NAG.
Essentially, NAG-free exploits a mechanism analogous to a power iteration hidden in NAG.
Building on this idea, we also proved that, given an upper bound on \(L\), NAG-free achieves accelerated convergence locally near the minimum under local smoothness of the Hessian and some mild additional assumptions.
Finally, we presented experiments with smooth and nonsmooth Hessians on both synthetic and real-world data which demonstrate that NAG-free is competitive with restart methods, and naturally adapts to favorable local curvature conditions.

\begin{figure}
    \centering
    \includegraphics[width=\linewidth]{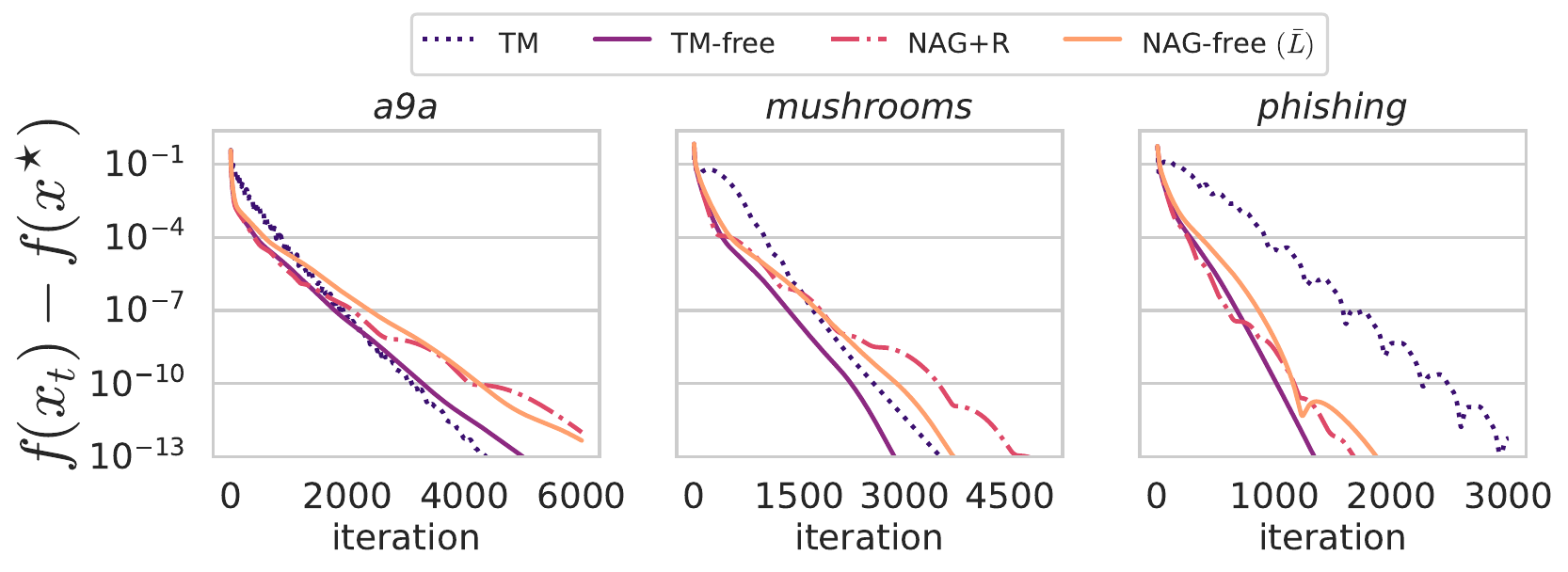}
    \caption{Suboptimality gap \(f(x_{t})-f(x^{\star})\) for logistic regression on three datasets.}
    \label{fig:logreg_tm_free}
\end{figure}

By definition, every first-order method can be endowed with a similar \(m\) estimator at no additional computational cost and a modest memory overhead.
The underlying power iteration mechanism is also rather general, but whether formal guarantees can be given for other extended methods is a question for future work.
Base methods with nominal convergence guarantees better than those of NAG are natural candidates for extension.
For the sake of illustration, we extended the triple momentum method (TM) with an analogous \(m\) estimator, which
\cref{fig:logreg_tm_free} shows can significantly boost the performance of TM.
And even in the worst case represented by \textsc{a9a}, when \(m\) can essentially be estimated offline, the TM extension (TM-free) recovers an accurate estimate of \(m\) without compromising performance.
More details can be found in \cref{app:ne:tm}.

\acks{We thank our colleagues and funding agencies;
  \texttt{\textbackslash documentclass[anon]\{alt2026\}}
automatically hides this text.}

\FloatBarrier           
\bibliography{bib/CTRL, bib/MATH, bib/ML, bib/NA, bib/OPT, bib/SC}

\appendix

\crefalias{section}{appendix} 

\newpage
\clearpage
\section{Global Convergence}
\label{app:global_convergence}

In this section, we prove \cref{thm:gc}, which establishes that \cref{alg:nag-free} converges globally at least as fast as gradient descent (GD).
To this end, we first analyze iterations in which \(m_{t}\geq m\).
Then, we analyze iterations in which \(m_{t}<m\), and the transition from the first kind of iteration to the second.

\subsection{Case 1: \texorpdfstring{\(m_{t} \geq m\)}{mt >= m}}
\label{app:nag-free:gc:case1}

The iterations in which \(m_{t} \geq m\) can be expressed as a convex combination of appropriate GD and NAG iterations. 
We exploit this property to prove that \cref{alg:nag-free} converges at least as fast as GD.
We use an argument based on a Lyapunov function that we denote by \(V^{\textup{GD}}_{t}\). 
The superscript ``GD'' indicates that \(V^{\textup{GD}}_{t}\) decreases at a gradient-descent type of rate along iterations in which \(m_{t} \geq m\).

The Lyapunov function \(V^{\textup{GD}}_{t}\) is the sum of two functions \(W_{t}\) and \(U_{t}\). 
First, we show \(W_{t}\) is a common Lyapunov function for GD and NAG, then we analyze \(U_{t}\) and finally combine all results to give \cref{alg:nag-free} the same type of convergence guarantees of GD. 
To analyze GD and NAG through a common Lyapunov function, we add a trivial momentum step \(x^{\textup{GD}}_{t+1}\) to GD, as in
\begin{align}
    y^{\textup{GD}}_{t+1} 
    =&\ x^{\textup{GD}}_{t} -(1/L_{t})\nabla f(x^{\textup{GD}}_{t}),
    \label{def:sm-gc-gd-descent}
    \\
    x^{\textup{GD}}_{t+1}
    =&\ y^{\textup{GD}}_{t+1},
    \label{def:sm-gc-gd-momentum}
\end{align}
conforming GD to the algorithmic structure of NAG:
\begin{align}
    y^{\textup{NAG}}_{t+1}
    =&\ x^{\textup{NAG}}_{t} -(1/L_{t})\nabla f(x^{\textup{NAG}}_{t}),
    \label{def:nag-descent}
    \\
    x^{\textup{NAG}}_{t+1}
    =&\ y^{\textup{NAG}}_{t+1} + \theta_{t}(y^{\textup{NAG}}_{t+1}-y^{\textup{NAG}}_{t}),
    \label{def:nag-momentum}
\end{align}
where the coefficient \(\theta\) defining the momentum step in \eqref{def:nag-momentum} is given by
\begin{align}
    \theta_{t} = (\sqrt{p_{t}}-1)/(\sqrt{p_{t}}+1),
    &&
    p_{t} = (L_{t}/m).
    \label{def:sm-gc-theta-p}
\end{align}
Similarly, the iterates of \cref{alg:nag-free} are given by
\begin{align}
    y_{t+1}
    =&\ x_{t} -(1/L_{t})\nabla f(x_{t}),
    \label{def:sm-gc-NEST-descent}
    \\
    x_{t+1}
    =&\ y_{t+1} + \beta_{t}(y_{t+1}-y_{t}),
    \label{def:sm-gc-NEST-momentum}
\end{align}
where \(y_{t+1}\) and \(L_{t}\) are such that
\begin{align}
    f(y_{t+1}) - f(x_{t}) 
    \leq -(1/2L_{t})\Vert \nabla f(x_{t}) \Vert^{2}
    \label{ineq:sm-gc-NEST-dl},
\end{align}
and \(\beta_{t}\) is the affine coefficient given by
\begin{align}
    \beta_{t} = (\sqrt{q_{t}}-1)/(\sqrt{q_{t}}+1),
    &&
    q_{t} = (L_{t}/m_{t}),
    \label{def:sm-gc-beta-q}
\end{align}
	
Expressed in the common structure of \labelcref{def:sm-gc-gd-descent,def:sm-gc-gd-momentum,def:nag-descent,def:nag-momentum}, GD and NAG can be analyzed with a common Lyapunov function very similar to the one used in \cite[section~5.5]{Bansal2019}, and given by
\begin{align}
    W_{t}(s_{t}) = \tilde{f}(y_{t}) + (m/2)\Vert z^{\star}_{t} \Vert^{2},
    \label{def:sm-gc-Lyap1}
\end{align}
where \(s_{t}\) stacks the descent and momentum steps into a single pair, as in 
\begin{align}
    s_{t}= (x_{t},y_{t}),
    &&
    s^{\textup{GD}}_{t}= (x^{\textup{GD}}_{t},y^{\textup{GD}}_{t}),
    &&
    \text{and}
    &&
    s^{\textup{NAG}}_{t}=(x^{\textup{NAG}}_{t},y^{\textup{NAG}}_{t}),
    \label{def:sm-gc-sk}
\end{align}
\(\tilde{f}\) denotes the objective function with minimum shifted to 0, meaning that
\begin{align}
    \tilde{f}= f-f(x^{\star}),
    \label{def:sm-gc-f-tilde}
\end{align}
and \(z^{\star}_{s} = z^{\star}_{s}(x_{t},y_{t})\) is the pseudo-state defined as
\begin{align}
    z^{\star}_{s}
    = z_{s}-x^{\star},
    &&
    z_{s} 
    = z_{s}(x_{t},y_{t})
    = 
    \begin{cases}
        x_{0} + \sqrt{p_{0}}(x_{0}-y_{0}), & s = 0,
        \\
        x_{t} + \sqrt{p_{s
        -1}}(x_{t}-y_{t}), & s \geq 1.
    \end{cases}
    \label{def:sm-gc-pseudo-z}
\end{align}

\begin{remark}
    In the definition of \(W_{t}\), we note that the subscript \(t\) determines the subscript of \(p_{0}\) or \(p_{t-1}\) in \(z_{t}\) independently of the subscript of \(x_{t}\) and \(y_{t}\).
    So, for example, we have that
    \begin{align*}
        W_{t+1}(s_{t})
        &=\tilde{f}(y_{t}) + (m/2)\Vert x_{t} + \sqrt{p_{t}}(x_{t}-y_{t})\Vert^{2}
        \\
        &\neq \tilde{f}(y_{t+1}) + (m/2)\Vert x_{t+1} + \sqrt{p_{t}}(x_{t+1}-y_{t+1})\Vert^{2}
        = W_{t+1}(s_{t+1}).
    \end{align*}
\end{remark}

\begin{remark}\label{app:nag-free:gd:rem:w-convex}
    By assumption \(f\in\mathcal{S}(L,m)\) is convex, thus so is \(\tilde{f}\).
    Moreover, the affine transformation that defines \(z^{\star}_{t}\) composed with the 2-norm yields a convex function. 
    Thus, \(V^{\textup{GD}}_{t}\) is the sum of convex functions and is therefore convex.
\end{remark}

In the following, we often use \(g_{t} = \nabla f(x_{t})\).
For brevity, we also define
\begin{align}
    x^{\star}_{t} = x_{t}-x^{\star}
    &&
    \text{and}
    && x^{y}_{t} = x_{t}-y_{t}.
    \label{def:sm-gc-pseudo-x}
\end{align}

\begin{remark}
    Superscripts carry over from \labelcref{def:sm-gc-gd-descent,def:sm-gc-gd-momentum,def:nag-descent,def:nag-momentum} to the notation above in the natural way. 
    For example, by \(g^{\textup{NAG}}_{t}\) we mean \(\nabla f(x^{\textup{NAG}}_{t})\) and by \(x^{\textup{GD},\star}_{t}\) we mean \(x^{\textup{GD}}_{t} -x^{\star}\).
\end{remark}

Although \(W_{t}\) serves as a Lyapunov function for GD and NAG, we should expect \(W_{t}\) to decrease at a faster rate along NAG iterations than along GD iterations.
We now show that \(W_{t}\) decreases at the expected rate for each of the two methods, namely \((1+\delta(p_{t}))^{-1}\) for GD and \((1+\delta(\sqrt{p_{t}}))^{-1}\) for NAG, where the rate increment \(\delta\) is defined by
\begin{align}
    \delta(p)
    = 1/(p-1).
    \label{def:sm-gc-deltas}
\end{align}
The following rate increments will also be convenient:
\begin{align}
    \delta^{\textup{GD}}_{t}
    = \delta(p_{t-1})
    = 1/(p_{t-1}-1)
    &&
    \text{and}
    &&
    \delta^{\textup{ACC}}_{t}
    = \delta(\sqrt{p_{t-1}})
    = 1/(\sqrt{p_{t-1}}-1).
    \label{def:sm-gc-deltas-k}
\end{align}
 
\begin{lemma}\label[lemma]{lem:sm-gc-Lyap1-descent-GD}
    Let \(f\in\mathcal{S}(L,m)\) and \(y^{\textup{GD}}_{t}=x^{\textup{GD}}_{t}\in \mathbb{R}^{d}\). 
    If \(y^{\textup{GD}}_{t+1}\) given by \eqref{def:sm-gc-gd-descent} and \(L_{t}>0\) are such that
    \begin{align}
        f(y^{\textup{GD}}_{t+1}) - f(x^{\textup{GD}}_{t}) 
        \leq -(1/2L_{t})\Vert g^{\textup{GD}}_{t} \Vert^{2},
        \label{ineq:sm-gc-Lyap1-descent-GD-dl}
    \end{align}
    and \(x^{\textup{GD}}_{t+1}\) is given by \eqref{def:sm-gc-gd-momentum}, then
    \begin{align}
        (1+\delta^{\textup{GD}}_{t+1})W_{t+1}(s^{\textup{GD}}_{t+1})
        - W_{t}(s^{\textup{GD}}_{t})
        \leq -(1/2L_{t})\Vert g^{\textup{GD}}_{t} \Vert^{2}.
        \label{ineq:sm-gc-Lyap1-descent-GD}
    \end{align}
\end{lemma}

\begin{proof}
    Let \(f\in \mathcal{S}(L,m)\) and \(\bar{L}\geq L\).
    Following the procedure described above, we start by expressing \((1+\delta^{\textup{GD}}_{t+1})\tilde{f}(y^{\textup{GD}}_{t+1}) - \tilde{f}(y^{\textup{GD}}_{t})\) as the sum of two differences:
    \begin{align*}
        (1+\delta^{\textup{GD}}_{t+1})\tilde{f}(y^{\textup{GD}}_{t+1})
        - \tilde{f}(y^{\textup{GD}}_{t})
        = (1+\delta^{\textup{GD}}_{t+1})(f(y^{\textup{GD}}_{t+1})-f(y^{\textup{GD}}_{t}))
        +\delta^{\textup{GD}}_{t+1}(f(y^{\textup{GD}}_{t})-f(x^{\star})).
    \end{align*}
    If \(y^{\textup{GD}}_{t} = x^{\textup{GD}}_{t}\), \(y^{\textup{GD}}_{t+1}\) is given by \eqref{def:sm-gc-gd-descent} and \eqref{ineq:sm-gc-Lyap1-descent-GD-dl} holds, then the first difference is bounded as
    \begin{align}
        (1+\delta^{\textup{GD}}_{t+1})(f(y^{\textup{GD}}_{t+1})-f(y^{\textup{GD}}_{t}))
        \leq -(1+\delta^{\textup{GD}}_{t+1})(1/2L_{t})\Vert g^{\textup{GD}}_{t} \Vert^{2}.
        \label{ineq:sm-gc-Lyap1-descent-GD-DL}
    \end{align}
    Applying \eqref{ineq:strong_convexity} with \(x=y^{\textup{GD}}_{t}=x^{\textup{GD}}_{t}\) and \(y=x^{\star}\), we bound the second difference as
    \begin{align}
        \delta^{\textup{GD}}_{t+1}(f(y^{\textup{GD}}_{t})-f(x^{\star}))
        \leq \delta^{\textup{GD}}_{t+1}\langle g^{\textup{GD}}_{t},x^{\textup{GD},\star}_{t} \rangle -\delta^{\textup{GD}}_{t+1}(m/2)\Vert x^{\textup{GD},\star}_{t} \Vert^{2}.
        \label{ineq:sm-gc-Lyap1-descent-GD-SC}
    \end{align}
    Also from \(x^{\textup{GD}}_{t}=y^{\textup{GD}}_{t}\), it follows that \(z^{\textup{GD}}_{t}=x^{\textup{GD}}_{t}\) and, likewise, \(z^{\textup{GD}}_{t+1}=y^{\textup{GD}}_{t+1}\). 
    Therefore
    \begin{align}
        (1+\delta^{\textup{GD}}_{t+1})\Vert z^{\textup{GD},\star}_{t+1} \Vert^{2}
        -\Vert z^{\textup{GD},\star}_{t} \Vert^{2}
        =&\ (1+\delta^{\textup{GD}}_{t+1})\Vert y^{\textup{GD},\star}_{t+1} \Vert^{2}-\Vert x^{\textup{GD},\star}_{t} \Vert^{2}
        \nonumber\\
        =&\ (1+\delta^{\textup{GD}}_{t+1})\Bigl( \frac{\Vert g^{\textup{GD}}_{t} \Vert^{2}}{L_{t}^{2}} -\frac{2\langle g^{\textup{GD}}_{t},x^{\textup{GD},\star}_{t} \rangle}{L_{t}} \Bigr)
        +\delta^{\textup{GD}}_{t+1}\Vert x^{\textup{GD},\star}_{t} \Vert^{2}.
        \label{ineq:sm-gc-Lyap1-descent-GD-2-norms}
    \end{align}
    To simplify the above and conclude the proof, we use the identities
    \begin{align*}
        (1+\delta^{\textup{GD}}_{t+1})\Bigl( 1-\frac{1}{p_{t}} \Bigr)
        = \frac{p_{t}}{p_{t}-1}\frac{p_{t}-1}{p_{t}}
        = 1
        &&
        \text{and}
        &&
        \frac{1+\delta^{\textup{GD}}_{t+1}}{p_{t}} 
        = \frac{p_{t}/(p_{t}-1)}{p_{t}}
        = \delta^{\textup{GD}}_{t+1}.
    \end{align*}
    Multiplying \eqref{ineq:sm-gc-Lyap1-descent-GD-2-norms} by \(m/2\), summing the result with \labelcref{ineq:sm-gc-Lyap1-descent-GD-DL} and \eqref{ineq:sm-gc-Lyap1-descent-GD-SC}, then using the identities above, we obtain
    \begin{align*}
        (1+\delta^{\textup{GD}}_{t})W_{t+1}(s^{\textup{GD}}_{t+1})-W_{t}(s^{\textup{GD}}_{t})
        \leq& 
        -(1+\delta^{\textup{GD}}_{t+1})\Bigl( 1-\frac{1}{p_{t}} \Bigr)\frac{1}{2L_{t}}\Vert g^{\textup{GD}}_{t} \Vert^{2}
        \\
        &- \Bigl(\delta^{\textup{GD}}_{t+1} - \frac{1+\delta^{\textup{GD}}_{t+1}}{p_{t}} \Bigr)\langle g^{\textup{GD}}_{t},x^{\textup{GD},\star}_{t} \rangle
        \\
        \leq& -(1/{2L_{t}})\Vert g^{\textup{GD}}_{t} \Vert^{2},
    \end{align*}
    proving \eqref{ineq:sm-gc-Lyap1-descent-GD}.
\end{proof}

The analysis of Lyapunov functions like \(W_{t}\) is challenging because it varies with \(t\), so we consider two types of changes for \(W_{t}\) and subsequent Lyapunov functions: the decrease in a fixed \(W_{t+1}\) from one iteration \(s_{t}\) to the next \(s_{t+1}\) and the increase from \(W_{t}\) to \(W_{t+1}\) for the same iteration \(s_{t}\).
For GD, the steps \(y^{\textup{GD}}_{t}\) and \(x^{\textup{GD}}_{t}\) coincide, hence both also coincide with \(z^{\textup{GD}}_{t}\). 
Therefore, \(W_{t}\) is effectively the same for all \(t\geq 0\) when evaluated at \(s^{\textup{GD}}_{t}\), that is
\begin{align*}
    W_{t}(s^{\textup{GD}}_{t}) = \tilde{f}(y^{\textup{GD}}_{t}) 
    + (m/2)\Vert x^{\textup{GD},\star}_{t} \Vert^{2}
    = W_{t+1}(s^{\textup{GD}}_{t}).
\end{align*}
In contrast, the steps \(y^{\textup{NAG}}_{t}\) and \(x^{\textup{NAG}}_{t}\) need not coincide. Therefore, as \(\sqrt{p_{t}}\) change with \(t\), each \(z^{\textup{NAG},\star}_{t}\) turns into a different affine combination of \(x^{\textup{NAG},\star}_{t}\) and \(y^{\textup{NAG},\star}_{t}\). 
That is, for \(t\geq 1\)
\begin{align*}
    z^{\textup{NAG}}_{t+1}(x_{t},y_{t})
    &= x^{\textup{NAG}}_{t} + \sqrt{p_{t}}(x^{\textup{NAG}}_{t}-y^{\textup{NAG}}_{t})
    \\
    &\neq x^{\textup{NAG}}_{t} + \sqrt{p_{t-1}}(x^{\textup{NAG}}_{t}-y^{\textup{NAG}}_{t})
    = z^{\textup{NAG}}_{t}(x_{t},y_{t})
\end{align*}
due to mismatching \(\sqrt{p_{t}}\) and \(\sqrt{p_{t-1}}\). 
To handle this mismatch, instead of analyzing the difference \((1+\delta^{\textup{ACC}}_{t+1})W_{t+1}(s^{\textup{NAG}}_{t+1})- W_{t}(s^{\textup{NAG}}_{t})\), in the next two results we analyze the difference \((1+\delta^{\textup{ACC}}_{t+1})W_{t+1}(s^{\textup{NAG}}_{t+1})- W_{t+1}(s^{\textup{NAG}}_{t})\), and then bound \(W_{t+1}(s^{\textup{NAG}}_{t+1})\) in terms of \(W_{t}(s^{\textup{NAG}}_{t})\).

\begin{lemma}\label[lemma]{lem:sm-gc-Lyap1-descent-N}
    Let \(f\in\mathcal{S}(L,m)\). 
    If \(y^{\textup{NAG}}_{t+1}\) given by \eqref{def:nag-descent} and \(L_{t}>0\) are such that
    \begin{align}
        f(y^{\textup{NAG}}_{t+1})-f(x^{\textup{NAG}}_{t}) 
        \leq -(1/2L_{t})\Vert g^{\textup{NAG}}_{t} \Vert^{2},
        \label{ineq:sm-gc-Lyap1-descent-N-dl}
    \end{align} 
    and \(x^{\textup{NAG}}_{t+1}\) is given by \eqref{def:nag-momentum}, then
    \begin{align}
        (1+\delta^{\textup{ACC}}_{t+1})W_{t+1}(s^{\textup{NAG}}_{t+1})-W_{t+1}(s^{\textup{NAG}}_{t})
        \leq 0.
        \label{ineq:sm-gc-Lyap1-descent-N}
    \end{align}
\end{lemma}
	
\begin{proof}
    To prove \eqref{ineq:sm-gc-Lyap1-descent-N}, we start by expressing \((1+\delta^{\textup{ACC}}_{t+1})\tilde{f}(y^{\textup{NAG}}_{t+1})-\tilde{f}(y^{\textup{NAG}}_{t})\) as the sum of three further differences:
    \begin{align*}
        (1+\delta^{\textup{ACC}}_{t+1})\tilde{f}(y^{\textup{NAG}}_{t+1})-\tilde{f}(y^{\textup{NAG}}_{t})
        =&\ (1+\delta^{\textup{ACC}}_{t+1})(f(y^{\textup{NAG}}_{t+1}) -f(x^{\textup{NAG}}_{t}))
        \\
        &+ f(x^{\textup{NAG}}_{t})-f(y^{\textup{NAG}}_{t})
        \\
        &+ \delta^{\textup{ACC}}_{t+1}(f(x^{\textup{NAG}}_{t})-f(x^{\star})).
    \end{align*}
    If \eqref{ineq:sm-gc-Lyap1-descent-N-dl} holds, then we bound the first difference as
    \begin{align*}
        (1+\delta^{\textup{ACC}}_{t+1})(f(y^{\textup{NAG}}_{t+1})-f(x^{\textup{NAG}}_{t})) 
        \leq&\ 
        -(1+\delta^{\textup{ACC}}_{t+1})(1/2L_{t})\Vert g^{\textup{NAG}}_{t} \Vert^{2}.
    \end{align*}
    Using convexity and applying \eqref{ineq:strong_convexity} with \(x=x^{\textup{NAG}}_{t}\) and \(y=x^{\star}\), we bound the second and third differences as
    \begin{align*}
        f(x^{\textup{NAG}}_{t})-f(y^{\textup{NAG}}_{t})
        \leq&\ \langle g^{\textup{NAG}}_{t},x^{\textup{NAG},y}_{t} \rangle,
        \\
        f(x^{\textup{NAG}}_{t})-f(x^{\star})
        \leq&\ \langle g^{\textup{NAG}}_{t},x^{\textup{NAG},\star}_{t} \rangle 
        -(m/2)\Vert x^{\textup{NAG},\star}_{t} \Vert^{2}.
    \end{align*}
    To address the rest of \((1+\delta^{\textup{ACC}}_{t+1})W_{t+1}(s^{\textup{NAG}}_{t+1})-W_{t+1}(s^{\textup{NAG}}_{t})\), we expand \(z^{\textup{NAG},\star}_{t+1}\), and then use the definition of \(\theta_{t}\) to simplify the resulting expression, as in
    \begin{align*}
        z^{\textup{NAG},\star}_{t+1}
        =&\ x^{\textup{NAG}}_{t+1} 
        + \sqrt{p_{t}}(x^{\textup{NAG}}_{t+1}-y^{\textup{NAG}}_{t+1}) 
        - x^{\star}
        \\
        =&\ y^{\textup{NAG}}_{t+1}+\theta_{t}(y^{\textup{NAG}}_{t+1}-y^{\textup{NAG}}_{t}) 
        + \sqrt{p_{t}}\theta_{t}(y^{\textup{NAG}}_{t+1}-y^{\textup{NAG}}_{t}) 
        -x^{\star}
        \\
        =&\ -(1/L_{t})(1+\theta_{t}(1+\sqrt{p_{t}}))g^{\textup{NAG}}_{t} 
        + \theta_{t}(1+\sqrt{p_{t}})x^{\textup{NAG},y}_{t} 
        + x^{\textup{NAG},\star}_{t}
        \\
        =&\ -(1/L_{t})\sqrt{p_{t}}g^{\textup{NAG}}_{t} 
        + (\sqrt{p_{t}}-1)x^{\textup{NAG},y}_{t} 
        + x^{\textup{NAG},\star}_{t}.
    \end{align*}
    Next, we note that the 2-norm term that goes into \(W_{t+1}(s^{\textup{NAG}}_{t})\) is \((m/2)\Vert x^{\textup{NAG},\star}_{t} + \sqrt{p_{t}}x^{\textup{NAG},\star}_{t} \Vert^{2}\), aand then we write the 2-norm difference in \((1+\delta^{\textup{ACC}}_{t+1})W_{t+1}(s^{\textup{NAG}}_{t+1})-W_{t+1}(s^{\textup{NAG}}_{t})\) as
    \begin{align*}
        (1+\delta^{\textup{ACC}}_{t+1})\frac{m}{2}\Vert z^{\textup{NAG},\star}_{t+1} \Vert^{2} 
        -\frac{m}{2}\Vert x^{\textup{NAG},\star}_{t} + \sqrt{p_{t}}x^{\textup{NAG},\star}_{t} \Vert^{2}
        =&\ \frac{1+\delta^{\textup{ACC}}_{t+1}}{2L_{t}}\Vert g^{\textup{NAG}}_{t} \Vert^{2}
        \\
        &-\langle g^{\textup{NAG}}_{t},x^{\textup{NAG},y}_{t} \rangle
        \\
        &-\delta^{\textup{ACC}}_{t+1}\langle g^{\textup{NAG}}_{t},x^{\textup{NAG},\star}_{t} \rangle
        \\
        &-\frac{m}{2}\sqrt{p_{t}}\Vert x^{\textup{NAG},y}_{t} \Vert^{2}			
        \\
        &+ \delta^{\textup{ACC}}_{t+1}\frac{m}{2}\Vert x^{\textup{NAG},\star}_{t} \Vert^{2},
    \end{align*}
    where we used the following identities after colons to simplify the coefficients of the terms before colons:
    \begin{align*}
        \langle g^{\textup{NAG}}_{t},x^{\textup{NAG},y}_{t} \rangle&:
        &
        (1+\delta^{\textup{ACC}}_{t+1})\sqrt{p_{t}}(\sqrt{p_{t}}-1)/p_{t}
        =&\ 1,
        \\
        \langle g^{\textup{NAG}}_{t},x^{\textup{NAG},\star}_{t} \rangle &:
        &
        (1+\delta^{\textup{ACC}}_{t+1})\sqrt{p_{t}}/p_{t}
        =&\ \delta^{\textup{ACC}}_{t+1},
        \\
        \Vert x^{\textup{NAG},y}_{t} \Vert^{2}&:
        &
        (1+\delta^{\textup{ACC}}_{t+1})(\sqrt{p_{t}}-1)^{2}
        =&\ \sqrt{p_{t}}(\sqrt{p_{t}}-1),
        \\
        \langle x^{\textup{NAG},y}_{t}, x^{\textup{NAG},\star}_{t} \rangle &:
        &
        (1+\delta^{\textup{ACC}}_{t+1})(\sqrt{p_{t}}-1)
        =&\ \sqrt{p_{t}}.
    \end{align*}
    Finally, we put everything together and then cancel several terms to obtain
    \begin{align*}
        (1+\delta^{\textup{ACC}}_{t+1})W_{t+1}(s^{\textup{NAG}}_{t+1})-W_{t+1}(s^{\textup{NAG}}_{t})
        \leq 
        -(m/2)\sqrt{p_{t}}\Vert x^{\textup{NAG},y}_{t} \Vert^{2} 
        \leq 0,
    \end{align*}
    proving \eqref{ineq:sm-gc-Lyap1-descent-N}.
\end{proof}

By pairing \(W_{t+1}(s^{\textup{NAG}}_{t+1})\) with \(W_{t+1}(s^{\textup{NAG}}_{t})\), we deferred the problem of mismatching \(\sqrt{p_{t}}\) and \(\sqrt{p_{t-1}}\) to obtain \eqref{ineq:sm-gc-Lyap1-descent-N}.
We now handle the mismatch problem by bounding \(W_{t+1}(s^{\textup{NAG}}_{t})\) in terms of \(W_{t}(s^{\textup{NAG}}_{t})\).
In contrast with \eqref{ineq:sm-gc-Lyap1-descent-N}, the inequality we prove next holds for arbitrary \(x_{t}\) and \(y_{t}\), not necessarily generated by NAG. 
We therefore drop the ``\textit{NAG}'' superscript.

\begin{lemma}\label[lemma]{lem:sm-gc-Lyap1-ascent}
    Let \(f\in\mathcal{S}(L,m)\), \(y_{t},x_{t}\in\mathbb{R}^{d}\). If \(L_{t}\geq L_{t-1}\geq m\), then 
    \begin{align}
        W_{t+1}
        &\leq \frac{p_{t}^{2}}{p_{t-1}^{2}}W_{t}.
        \label{ineq:sm-gc-Lyap1-ascent}
    \end{align}
\end{lemma}

\begin{proof}
    The key to prove \eqref{ineq:sm-gc-Lyap1-ascent} is to analyze the difference between the mismatching terms
    \begin{align}
        \Vert x^{\star}_{t} + \sqrt{p_{t}}x^{y}_{t} \Vert^{2}-\Vert z^{\star}_{t} \Vert^{2}
        =&\ 2(\sqrt{p_{t}}-\sqrt{p_{t-1}})\langle x^{\star}_{t},x^{y}_{t} \rangle
        + (p_{t}-p_{t-1})\Vert x^{y}_{t} \Vert^{2}.
        \label{id:sm-gc-Lyap1-ascent-mismatch}
    \end{align}
    We split the analysis in two cases, according to the sign of \(\langle x^{\star}_{t},x^{y}_{t} \rangle\). 
    First, we consider the case \(\langle x^{\star}_{t},x^{y}_{t} \rangle\geq 0\). 
    Assuming \(L_{t}\geq L_{t-1}\), then \(p_{t}\geq p_{t-1}\), which in turn implies
    \begin{align}
        \sqrt{p_{t}}-\sqrt{p_{t-1}}
        \leq \frac{p_{t}}{\sqrt{p_{t}}} - \sqrt{p_{t-1}}\frac{\sqrt{p_{t-1}}}{\sqrt{p_{t}}}
        = \frac{p_{t}-p_{t-1}}{\sqrt{p_{t}}}.
        \label{id:sm-gc-Lyap1-ascent-cp-coeff}
    \end{align}
    Hence, adding \((p_{t}-p_{t-1})p_{t}^{-1}\Vert x^{\star}_{t} \Vert^{2}\geq 0\) to \eqref{id:sm-gc-Lyap1-ascent-mismatch} and then using \eqref{id:sm-gc-Lyap1-ascent-cp-coeff}, we obtain
    \begin{align}
        \Vert x^{\star}_{t} + \sqrt{p_{t}}x^{y}_{t} \Vert^{2}-\Vert z^{\star}_{t} \Vert^{2}
        \leq&\ 2\frac{p_{t}-p_{t-1}}{\sqrt{p_{t}}}\langle x^{\star}_{t},x^{y}_{t} \rangle
        + (p_{t}-p_{t-1})\Vert x^{y}_{t} \Vert^{2}
        + \frac{p_{t}-p_{t-1}}{p_{t}}\Vert x^{\star}_{t} \Vert^{2}
        \nonumber\\
        =&\ \frac{p_{t}-p_{t-1}}{p_{t}}\Vert x^{\star}_{t} + \sqrt{p_{t}}x^{y}_{t} \Vert^{2}.
        \label{ineq:sm-gc-Lyap1-ascent-mismatch-case1}
    \end{align}
    In turn, multiplying right and left-hand side of \eqref{ineq:sm-gc-Lyap1-ascent-mismatch-case1} by \(m/2\), it follows from the definition \eqref{def:sm-gc-Lyap1} that
    \begin{align*}
        W_{t+1}(s_{t})-W_{t}(s_{t})
        = \frac{m}{2}(\Vert x^{\star}_{t} + \sqrt{p_{t}}x^{y}_{t} \Vert^{2}-\Vert z^{\star}_{t} \Vert^{2})
        \leq \frac{p_{t}-p_{t-1}}{p_{t}}W_{t+1}(s_{t}).
    \end{align*}
    Moving terms around then multiplying both sides by \(p_{t}/p_{t-1}\), we obtain
    \begin{align*}
        W_{t+1}(s_{t})
        \leq \frac{p_{t}}{p_{t-1}} W_{t}(s_{t})
        \leq \frac{p_{t}^{2}}{p_{t-1}^{2}}W_{t}(s_{t}),
    \end{align*}
    where the second inequality follows from the fact that \(p_{t}\geq p_{t-1}\).
    
    Now, suppose \(\langle x^{\star}_{t},x^{y}_{t} \rangle<0\). Expressing \((p_{t}-p_{t-1})\Vert x^{y}_{t} \Vert^{2}\) in \eqref{id:sm-gc-Lyap1-ascent-mismatch} as
    \begin{align*}
        (p_{t}-p_{t-1})\Vert x^{y}_{t} \Vert^{2}
        = (\sqrt{p_{t}}(\sqrt{p_{t}}-\sqrt{p_{t-1}}) + \sqrt{p_{t-1}}(\sqrt{p_{t}}-\sqrt{p_{t-1}}))\Vert x^{y}_{t} \Vert^{2}
    \end{align*}
    and then adding \(\pm (\sqrt{p_{t}}-\sqrt{p_{t-1}})\Vert x^{\star}_{t} \Vert^{2}/\sqrt{p_{t}}\) to \eqref{id:sm-gc-Lyap1-ascent-mismatch} to complete a square, we obtain
    \begin{align}
        \Vert x^{\star}_{t} + \sqrt{p_{t}}x^{y}_{t} \Vert^{2}-\Vert z^{\star}_{t} \Vert^{2}
        =&\ 2\frac{\sqrt{p_{t}}-\sqrt{p_{t-1}}}{\sqrt{p_{t}}}\langle x^{\star}_{t},\sqrt{p_{t}} x^{y}_{t} \rangle		
        + \sqrt{p_{t}}(\sqrt{p_{t}}-\sqrt{p_{t-1}})\Vert x^{y}_{t} \Vert^{2}
        \nonumber\\
        &+ \sqrt{p_{t-1}}(\sqrt{p_{t}}-\sqrt{p_{t-1}})\Vert x^{y}_{t} \Vert^{2}
        \pm \frac{\sqrt{p_{t}}-\sqrt{p_{t-1}}}{\sqrt{p_{t}}}\Vert x^{\star}_{t} \Vert^{2}
        \nonumber\\
        =&\ \frac{\sqrt{p_{t}}-\sqrt{p_{t-1}}}{\sqrt{p_{t}}}\Vert x^{\star}_{t} + \sqrt{p_{t}}x^{y}_{t} \Vert^{2}
        + \sqrt{p_{t-1}}(\sqrt{p_{t}}-\sqrt{p_{t-1}})\Vert x^{y}_{t} \Vert^{2}
        \nonumber\\
        &- \frac{\sqrt{p_{t}}-\sqrt{p_{t-1}}}{\sqrt{p_{t}}}\Vert x^{\star}_{t} \Vert^{2}.
        \label{id:sm-gc-Lyap1-ascent-aux}
    \end{align}
    Next, we bound the \(\Vert x^{y}_{t} \Vert^{2}\) term on \eqref{id:sm-gc-Lyap1-ascent-aux} using \(\Vert z^{\star}_{t} \Vert^{2}\) and \(\Vert x^{\star}_{t} \Vert^{2}\) terms. 
    To this end, we use an elementary inequality for 2-norms. 
    If \(a,b\in\mathbb{R}^{d}\) and \(c\in\mathbb{R}\setminus \{0\}\), then
    \begin{align*}
        (1/c^{2})\Vert a \Vert^{2} +2\langle a,b \rangle + c^{2}\Vert b \Vert^{2}
        =\Vert a/c+bc \Vert^{2}
        \geq 0,
    \end{align*}
    so that \(-2\langle a,b \rangle\leq (1/c^{2})\Vert a \Vert^{2} + c^{2}\Vert b \Vert^{2}\), which implies
    \begin{align}
        \Vert a-b \Vert^{2}
        =\Vert a \Vert^{2} -2\langle a,b \rangle + \Vert b \Vert^{2}
        \leq (1+1/c^{2})\Vert a \Vert^{2} + (1+c^{2})\Vert b \Vert^{2}.
        \label{ineq:sm-gc-Lyap1-ascent-elementary}
    \end{align}
    Applying \eqref{ineq:sm-gc-Lyap1-ascent-elementary} with \(a=z^{\star}_{t}, b=x^{\star}_{t}\) and some \(c\neq 0\), we obtain
    \begin{align}
        \Vert x^{y}_{t} \Vert^{2}
        = \Vert x^{y}_{t} \pm x^{\star}_{t}/\sqrt{p_{t-1}} \Vert^{2}
        = \frac{1}{p_{t-1}}\Vert z^{\star}_{t} -x^{\star}_{t}\Vert^{2}
        \leq \frac{1+1/c^{2}}{p_{t-1}}\Vert x^{\star}_{t} \Vert^{2}
        +\frac{1+c^{2}}{p_{t-1}}\Vert z^{\star}_{t} \Vert^{2}.
        \label{ineq:sm-gc-Lyap1-ascent-xyk}
    \end{align}
    We then choose \(c\) such that
    \begin{align}
        \frac{\sqrt{p_{t}}-\sqrt{p_{t-1}}}{\sqrt{p_{t-1}}}(1+c^{2})
        = \frac{p_{t}-p_{t-1}}{p_{t-1}}
        = \frac{\sqrt{p_{t}}-\sqrt{p_{t-1}}}{\sqrt{p_{t-1}}}
        \frac{\sqrt{p_{t}}+\sqrt{p_{t-1}}}{\sqrt{p_{t-1}}}.
        \label{id:sm-gc-Lyap1-ascent-c-id}
    \end{align}
    Cancelling \((\sqrt{p_{t}}-\sqrt{p_{t-1}})/\sqrt{p_{t-1}}\) on both sides of \eqref{id:sm-gc-Lyap1-ascent-c-id} yields
    \begin{align*}
        c^{2} 
        = \frac{\sqrt{p_{t}}+\sqrt{p_{t-1}}}{\sqrt{p_{t-1}}}-1
        = \frac{\sqrt{p_{t}}}{\sqrt{p_{t-1}}}.
    \end{align*}
    Having fixed \(c\) as above, it follows that the coefficient multiplying \(\Vert x^{\star}_{t} \Vert^{2}\) in \eqref{ineq:sm-gc-Lyap1-ascent-xyk} is
    \begin{align}
        \frac{1+1/c^{2}}{p_{t-1}}
        = \frac{1}{p_{t-1}}\Bigl( 1 + \frac{\sqrt{p_{t-1}}}{\sqrt{p_{t}}} \Bigr)
        = \frac{\sqrt{p_{t}} + \sqrt{p_{t-1}}}{p_{t-1}\sqrt{p_{t}}}.
        \label{id:sm-gc-Lyap1-ascent-xastk-coeff}
    \end{align}
    Plugging \eqref{id:sm-gc-Lyap1-ascent-c-id} and \eqref{id:sm-gc-Lyap1-ascent-xastk-coeff} back into \eqref{ineq:sm-gc-Lyap1-ascent-xyk}, we obtain
    \begin{align}
        \sqrt{p_{t-1}}(\sqrt{p_{t}}-\sqrt{p_{t-1}})\Vert x^{y}_{t} \Vert^{2}
        \leq \frac{p_{t}-p_{t-1}}{\sqrt{p_{t}}\sqrt{p_{t-1}}}\Vert x^{\star}_{t} \Vert^{2}
        +\frac{p_{t}-p_{t-1}}{p_{t-1}}\Vert z^{\star}_{t} \Vert^{2}.
        \label{ineq:sm-gc-Lyap1-ascent-xyk-2}
    \end{align}
    In turn, plugging \eqref{ineq:sm-gc-Lyap1-ascent-xyk-2} back into \eqref{id:sm-gc-Lyap1-ascent-aux} yields
    \begin{align}
        \Vert x^{\star}_{t} + \sqrt{p_{t}}x^{y}_{t} \Vert^{2}-\Vert z^{\star}_{t} \Vert^{2}
        &\leq\ \frac{\sqrt{p_{t}}-\sqrt{p_{t-1}}}{\sqrt{p_{t}}}\Vert x^{\star}_{t} + \sqrt{p_{t}}x^{y}_{t} \Vert^{2}
        + \frac{p_{t}-p_{t-1}}{p_{t-1}}\Vert z^{\star}_{t} \Vert^{2}
        \nonumber\\
        &+ \frac{\sqrt{p_{t}}-\sqrt{p_{t-1}}}{\sqrt{p_{t-1}}}\Vert x^{\star}_{t} \Vert^{2},
        \label{ineq:sm-gc-Lyap1-ascent-aux2}
    \end{align}
    where the coefficient multiplying \(\Vert x^{\star}_{t} \Vert^{2}\) is the result of summing that in \eqref{id:sm-gc-Lyap1-ascent-aux} and the one in \eqref{ineq:sm-gc-Lyap1-ascent-xyk-2}
    \begin{align*}
        \frac{p_{t}-p_{t-1}}{\sqrt{p_{t}}\sqrt{p_{t-1}}}
        -\frac{\sqrt{p_{t}}-\sqrt{p_{t-1}}}{\sqrt{p_{t}}}
        = \frac{\sqrt{p_{t}} - \sqrt{p_{t-1}}}{\sqrt{p_{t}}}\Bigl( \frac{\sqrt{p_{t}} + \sqrt{p_{t-1}}}{\sqrt{p_{t-1}}} -1 \Bigr)
        = \frac{\sqrt{p_{t}} - \sqrt{p_{t-1}}}{\sqrt{p_{t-1}}}.
    \end{align*}
    Multiplying both sides of \eqref{ineq:sm-gc-Lyap1-ascent-aux2} by \(m/2\) and using the definition \eqref{def:sm-gc-Lyap1}, we obtain
    \begin{align*}
        W_{t+1}(s_{t})-W_{t}(s_{t})
        \leq&\ \frac{\sqrt{p_{t}}-\sqrt{p_{t-1}}}{\sqrt{p_{t}}}W_{t+1}(s_{t})
        +\frac{p_{t}-p_{t-1}}{p_{t-1}}W_{t}(s_{t})
        \\
        & +\frac{\sqrt{p_{t}}-\sqrt{p_{t-1}}}{\sqrt{p_{t-1}}}\frac{m}{2}\Vert x^{\star}_{t} \Vert^{2}.
    \end{align*}
    Moving all \(W_{t+1}(s_{t})\) terms above to the left-hand side, all \(W_{t}(s_{t})\) above to the right-hand side and then multiplying both sides by \(\sqrt{p_{t}}/\sqrt{p_{t-1}}\), we get
    \begin{align}
        W_{t+1}(s_{t})
        &\leq \frac{p_{t}}{p_{t-1}}\frac{\sqrt{p_{t}}}{\sqrt{p_{t-1}}}W_{t}(s_{t})
        + \frac{\sqrt{p_{t}}}{\sqrt{p_{t-1}}}\frac{\sqrt{p_{t}}-\sqrt{p_{t-1}}}{\sqrt{p_{t-1}}}\frac{m}{2}\Vert x^{\star}_{t} \Vert^{2}
        \nonumber\\
        &\leq \frac{p_{t}}{p_{t-1}}\frac{\sqrt{p_{t}}}{\sqrt{p_{t-1}}}W_{t}(s_{t})
        + \frac{p_{t}}{p_{t-1}}\frac{\sqrt{p_{t}}-\sqrt{p_{t-1}}}{\sqrt{p_{t-1}}}\frac{m}{2}\Vert y^{\star}_{t} \Vert^{2},
        \label{ineq:sm-gc-Lyap1-ascent-aux3}
    \end{align}
    where the second inequality follows from \(\sqrt{p_{t}}\geq \sqrt{p_{t-1}}\) and \(\langle x^{\star}_{t},x^{y}_{t} \rangle<0\), the assumption underpinning the case we are analyzing, which implies
    \begin{align*}
        \Vert y^{\star}_{t} \Vert^{2}
        =\Vert y^{\star}_{t}\pm x^{\star}_{t} \Vert^{2}
        =\Vert x^{\star}_{t}-x^{y}_{t} \Vert^{2}
        =\Vert x^{\star}_{t} \Vert^{2} -2\langle x^{\star}_{t},x^{y}_{t} \rangle+\Vert x^{y}_{t} \Vert^{2}
        \geq \Vert x^{\star}_{t} \Vert^{2} + \Vert x^{y}_{t} \Vert^{2}
        \geq \Vert x^{\star}_{t} \Vert^{2}.
    \end{align*}
    Finally, bounding \((m/2)\Vert y^{\star}_{t} \Vert^{2}\) by \(\tilde{f}(y_{t})\) on \eqref{ineq:sm-gc-Lyap1-ascent-aux3} using \eqref{ineq:strong_convexity} with \(x=x^{\star}\) and \(y=y_{t}\), then bounding \(\tilde{f}(y_{t})\) by \(W_{t}\) directly from the definition of \(W_{t}\), we obtain
    \begin{align*}
        W_{t+1}(s_{t})
        \leq \frac{p_{t}}{p_{t-1}}\Bigl( \frac{\sqrt{p_{t}}}{\sqrt{p_{t-1}}}
        + \frac{\sqrt{p_{t}}-\sqrt{p_{t-1}}}{\sqrt{p_{t-1}}} \Bigr)W_{t}(s_{t})
        \leq \frac{p_{t}^{2}}{p_{t-1}^{2}}W_{t}(s_{t}),
    \end{align*}
    where the last inequality follows from
    \begin{align*}
        \frac{\sqrt{p_{t}}}{\sqrt{p_{t-1}}} + \frac{\sqrt{p_{t}}-\sqrt{p_{t-1}}}{\sqrt{p_{t-1}}}
        \leq \frac{p_{t}}{p_{t-1}},
    \end{align*}
    which holds because \(\sqrt{p_{t}}\geq \sqrt{p_{t-1}}\) implies
    \begin{align*}
        \sqrt{p_{t-1}}(\sqrt{p_{t}}-\sqrt{p_{t-1}})
        \leq \sqrt{p_{t}}(\sqrt{p_{t}}-\sqrt{p_{t-1}}).
    \end{align*}
    Therefore, both when \(\langle x^{\star}_{t},x^{y}_{t} \rangle\geq 0\) and when \(\langle x^{\star}_{t},x^{y}_{t} \rangle< 0\), the inequality
    \begin{align*}
        W_{t+1}(s_{t}) \leq \frac{p_{t}^{2}}{p_{t-1}^{2}}W_{t}(s_{t})
    \end{align*}
    holds generically for all \(s_{t}\), which proves \eqref{ineq:sm-gc-Lyap1-ascent}.
\end{proof}

Now that we have shown that \(W_{t}\) is a common Lyapunov function for GD and NAG, we introduce the second piece of \(V^{\textup{GD}}_{t}\), the function \(U_{t}\) defined by
\begin{align}
    U_{t}(s_{t}) =
    \begin{cases}
        \tilde{f}(y_{0}) 
        + (L_{0}/2)\Vert y^{\star}_{0} \Vert^{2}, & t = 0,
        \\[10pt]
        \tilde{f}(y_{t}) 
        + (L_{t-1}/2)\Vert y^{\star}_{t} \Vert^{2}, & t \geq 1,
    \end{cases}
    \label{def:sm-gc-Lyap2}
\end{align}
where \(\tilde{f}= f-f(x^{\star})\) and \(y^{\star}_{t}\) is a pseudo-state defined by
\begin{align}
    y^{\star}_{t} = y_{t}-x^{\star}.
    \label{def:sm-gc-pseudo-y}
\end{align}

\begin{remark}
    The subscript \(t\) of \(U_{t}\) determines the subscript of \(L_{0}\) or \(L_{t-1}\) independently of the argument of \(U_{t}\).
    So, for example, if \(t\geq 1\), then
    \begin{align*}
        U_{t+1}(s_{t})
        &= \tilde{f}(y_{t}) + (L_{t}/2)\Vert y^{\star}_{t} \Vert^{2}
        \\
        &\neq \tilde{f}(y_{t}) + (L_{t-1}/2)\Vert y^{\star}_{t} \Vert^{2}
        = U_{t}(s_{t}).
    \end{align*}
\end{remark}

Analogously to \(W_{t+1}(s_{t+1})\) and \(W_{t}(s_{t})\), \(U_{t+1}(s_{t+1})\) and \(U_{t}(s_{t})\) have a mismatch in the coefficients of their 2-norm terms, in this case \((L_{t}/2)\Vert y^{\star}_{t+1} \Vert^{2}\) and \((L_{t-1}/2)\Vert y^{\star}_{t} \Vert^{2}\).
Hence, as with \(W_{t}\), we pair \(U_{t+1}(s_{t+1})\) with \(U_{t+1}(s_{t})\) instead of \(U_{t}(s_{t})\) to avoid the mismatch and then address the mismatch problem immediately after. In contrast with the first piece, however, we analyze \(U_{t}\) along NEST iterations explicitly.

\begin{lemma}\label[lemma]{lem:sm-gc-Lyap2-descent}
    Let \(f\in\mathcal{S}(L,m)\).
    If \(s_{t+1}=(x_{t+1},y_{t+1})\) is given by \labelcref{def:sm-gc-NEST-descent,def:sm-gc-NEST-momentum}, then
    \begin{align}
        (1+\delta^{\textup{GD}}_{t+1})U_{t+1}(s_{t+1})
        -U_{t+1}(s_{t})
        \leq L_{t}\langle x^{y}_{t},x^{\star}_{t} \rangle 
        -(L_{t}/2)\Vert x^{y}_{t} \Vert^{2}.
        \label{ineq:sm-gc-Lyap2-descent}
    \end{align}
\end{lemma}

\begin{proof}
    First, we address the difference \((1+\delta^{\textup{GD}}_{t+1})\tilde{f}(y_{t+1}) -\tilde{f}(y_{t})\). By definition, \(f(x^{\star})\leq f(y_{t})\), thus \(-\tilde{f}(y_{t})\leq 0\). 
    Hence, adding \(\pm(1+\delta^{\textup{GD}}_{t+1})f(x_{t})\) to \((1+\delta^{\textup{GD}}_{t+1})\tilde{f}(y_{t+1}) -\tilde{f}(y_{t})\) and discarding \(-\tilde{f}(y_{t})\), we get
    \begin{align*}
        (1+\delta^{\textup{GD}}_{t+1})\tilde{f}(y_{t+1}) 
        - \tilde{f}(y_{t})
        \leq&\ 
        (1+\delta^{\textup{GD}}_{t+1})(f(y_{t+1})-f(x_{t}))
        +(1+\delta^{\textup{GD}}_{t+1})(f(x_{t})-f(x^{\star})).
    \end{align*}
    By assumption \(y_{t+1}\) is given by \eqref{def:sm-gc-NEST-descent}, with \(y_{t+1}\) and \(L_{t}\) such that \eqref{ineq:sm-gc-NEST-dl} holds.
    Therefore
    \begin{align*}
        (1+\delta^{\textup{GD}}_{t+1})(f(y_{t+1})-f(x_{t}))
        \leq
        -(1+\delta^{\textup{GD}}_{t+1})(1/2L_{t})\Vert g_{t} \Vert^{2}.
    \end{align*}
    To address the second difference above, we apply \eqref{ineq:strong_convexity} with \(x=x_{t}\) and \(y=x^{\star}\), obtaining
    \begin{align*}
        (1+\delta^{\textup{GD}}_{t+1})(f(x_{t})-f(x^{\star})) 
        \leq&\ (1+\delta^{\textup{GD}}_{t+1})\Bigl( \langle g_{t},x^{\star}_{t} \rangle 
        -(m/2)\Vert x^{\star}_{t} \Vert^{2} \Bigr).
    \end{align*}
    Then, we put the two bounds together to get
    \begin{align}
        (1+\delta^{\textup{GD}}_{t+1})\tilde{f}(y_{t+1}) -\tilde{f}(y_{t})
        \leq&\ -\frac{1+\delta^{\textup{GD}}_{t+1}}{2L_{t}}\Vert g_{t} \Vert^{2}
        +(1+\delta^{\textup{GD}}_{t+1})\langle g_{t},x^{\star}_{t} \rangle 
        -\delta^{\textup{GD}}_{t+1}\frac{L_{t}}{2}\Vert x^{\star}_{t} \Vert^{2},
        \label{ineq:sm-gc-Lyap2-descent-f-diff}
    \end{align}
    where the coefficient multiplying \(\Vert x^{\star}_{t} \Vert^{2}\) on the right-hand side above follows from the identity
    \begin{align*}
        (1+\delta^{\textup{GD}}_{t+1})\frac{m}{2}
        = \frac{p_{t}}{p_{t}-1}\frac{m}{2}
        = \delta^{\textup{GD}}_{t+1}\frac{L_{t}}{2}.
    \end{align*}
    To address the 2-norm difference in \((1+\delta^{\textup{GD}}_{t+1})U_{t+1}(s_{t+1})-U_{t}(s_{t})\), we expand pseudo-states inside 2-norms as:
    \begin{align*}
        (1+\delta^{\textup{GD}}_{t+1})\Vert y^{\star}_{t+1} \Vert^{2}
        =&\ (1+\delta^{\textup{GD}}_{t+1})\Bigl( \frac{1}{L_{t}^{2}}\Vert g_{t} \Vert^{2} 
        -\frac{2}{L_{t}}\langle g_{t},x^{\star}_{t} \rangle +\Vert x^{\star}_{t} \Vert^{2} \Bigr),
        \\
        \Vert y^{\star}_{t} \Vert^{2}
        =&\ \Vert x^{y}_{t} \Vert^{2} -2\langle x^{y}_{t},x^{\star}_{t} \rangle + \Vert x^{\star}_{t} \Vert^{2}.
    \end{align*}
    Expanding \(\Vert y^{\star}_{t+1} \Vert^{2}\) and \(\Vert y^{\star}_{t} \Vert^{2}\) as above, we get
    \begin{align}
        \frac{L_{t}}{2}((1+\delta^{\textup{GD}}_{t+1})\Vert y^{\star}_{t+1} \Vert^{2}-\Vert y^{\star}_{t} \Vert^{2})
        =&\ (1 + \delta^{\textup{GD}}_{t+1})\Bigl( \frac{\Vert g_{t} \Vert^{2} }{2L_{t}}
        -\langle g_{t},x^{\star}_{t} \rangle \Bigr)
        \nonumber\\
        &+\frac{L_{t}}{2}(
        -\Vert x^{y}_{t} \Vert^{2} 
        + 2\langle x^{y}_{t},x^{\star}_{t} \rangle 
        + \delta^{\textup{GD}}_{t+1}\Vert x^{\star}_{t} \Vert^{2}).
        \label{ineq:sm-gc-Lyap2-descent-y-diff}
    \end{align}
    Finally, combining \eqref{ineq:sm-gc-Lyap2-descent-f-diff} and \eqref{ineq:sm-gc-Lyap2-descent-y-diff}, several terms cancel each other and we are left with
    \begin{align*}
        (1+\delta^{\textup{GD}}_{t+1})U_{t+1}(s_{t+1})-U_{t}(s_{t})
        \leq&\ L_{t}\langle x^{y}_{t},x^{\star}_{t} \rangle 
        -\frac{L_{t}}{2}\Vert x^{y}_{t} \Vert^{2},
    \end{align*}
    proving \eqref{ineq:sm-gc-Lyap2-descent}.
\end{proof}

\begin{lemma}\label[lemma]{lem:sm-gc-Lyap2-ascent}
    Let \(f\in\mathcal{S}(L,m)\). 
    If \(L_{t}\geq L_{t-1}\), then
    \begin{align}
        U_{t+1} \leq \frac{p_{t}}{p_{t-1}}U_{t}.
        \label{ineq:sm-gc-Lyap2-ascent}
    \end{align}
\end{lemma}

\begin{proof}
    Expanding \(U_{t+1}(s_{t})-U_{t}(s_{t})\), multiplying the result by \(L_{t-1}/L_{t-1}\), using that \(\frac{1}{2}L_{t-1}\Vert y^{\star}_{t} \Vert^{2}\leq U_{t}(s_{t})\) and assuming \(L_{t}\geq L_{t-1}\), we obtain
    \begin{align*}
        U_{t+1}(s_{t})-U_{t}(s_{t})
        = \frac{L_{t}-L_{t-1}}{2}\Vert y^{\star}_{t} \Vert^{2}
        = \frac{L_{t}-L_{t-1}}{L_{t-1}}\frac{L_{t-1}}{2}\Vert y^{\star}_{t} \Vert^{2}
        \leq \frac{L_{t}-L_{t-1}}{L_{t-1}}U_{t}(s_{t}).
    \end{align*}
    Multiplying the right-hand side by \(m/m\) to substitute \(p_{t}\) and \(p_{t-1}\) for \(L_{t}\) and \(L_{t-1}\), and then moving \(-U_{t}(s_{t})\) to the right-hand side, we get
    \begin{align*}
        U_{t+1}(s_{t}) \leq \frac{p_{t}}{p_{t-1}}U_{t}(s_{t}).
    \end{align*}
    Since \(s_{t}\) is arbitrary, \eqref{ineq:sm-gc-Lyap2-ascent} follows.
\end{proof}

With \cref{lem:sm-gc-Lyap1-descent-GD,lem:sm-gc-Lyap1-descent-N,lem:sm-gc-Lyap1-ascent,lem:sm-gc-Lyap2-descent,lem:sm-gc-Lyap2-ascent}, we can prove the main result for NEST iterations in which \(m_{t}>m\), using the Lyapunov function \(V^{\textup{GD}}_{t}\) given by
\begin{align}
    V^{\textup{GD}}_{t}
    =
    \begin{cases}
        W_{0} + (\bar{\alpha}_{0}/\sqrt{p_{0}}) U_{0}, & t = 0,
        \\[10pt]
        W_{t} + (\bar{\alpha}_{t-1}/\sqrt{p_{t-1}}) U_{t}, & t \geq 1,
    \end{cases}
    &&
    \text{with}
    && 
    \bar{\alpha}_{t}= 1-\alpha_{t}
    && 
    \text{and}
    &&
    \alpha_{t} = \beta_{t}/\theta_{t}.
    \label{def:sm-gc-Lyapgd}
\end{align}

\begin{remark}
    The subscript \(t\) on \(V^{\textup{GD}}_{t}\) determines the subscripts on \(W_{0} + (\bar{\alpha}_{0}/\sqrt{p_{0}})U_{0}\) or \(W_{t} + (\bar{\alpha}_{t-1}/\sqrt{p_{t-1}})U_{t}\) independently of the argument.
\end{remark}

First, we show that \(V^{\textup{GD}}_{t}\geq 0\) for iterations in which \(m_{t}>m\).
For future reference, we also show that \(\bar{\alpha}_{j}\) is nonincreasing for all \(0\leq j\leq t\).

\begin{lemma}\label[lemma]{lem:sm-gc-bar-alpha}
    If \(m_{t} \geq m\) and, in addition, \(L_{t}\) and \(m_{t}\) are respectively nondecreasing and nonincreasing, then \(V^{\textup{GD}}_{j}\geq 0\) and \(\bar{\alpha}_{j}\) is nonincreasing for all \(0\leq j\leq t\).
\end{lemma}

\begin{proof}
    The assumptions that \(m_{t} \geq m\) and that \(m_{t}\) is nonincreasing imply that \(m_{j} \geq m\) for all \(0\leq j\leq t\).
    Moreover, \(m_{t} \geq m\) implies that \(q_{t} = L_{t}/m_{t} \leq L_{t}/m=p_{t}\), therefore
    \begin{align*}
        \beta_{t}
        = \frac{\sqrt{q_{t}}-1}{\sqrt{q_{t}}+1}
        \leq \frac{\sqrt{q_{t}}-1}{\sqrt{q_{t}}+1}
        = \theta_{t}.
    \end{align*}
    Hence, \(\beta_{j} \leq \theta_{j}\) for all \(0\leq j\leq t\). 
    Therefore, \(\alpha_{j},\bar{\alpha}_{j}\in \mathopen{[}0,1\mathclose{]}\) and, in turn, \(V^{\textup{GD}}_{j}\geq 0\) for all \(0\leq j\leq t\).
    Then, expanding \(\beta_{t}\) and \(\theta_{t}\) in \(\alpha_{t}\), we obtain
    \begin{align}
        \frac{\beta_{t}}{\theta_{t}}
        = \frac{\sqrt{q_{t}}-1}{\sqrt{q_{t}}+1}
        \frac{\sqrt{p_{t}}+1}{\sqrt{p_{t}}-1}
        = \frac{\sqrt{L_{t}}-\sqrt{m_{t}}}{\sqrt{L_{t}}+\sqrt{m_{t}}}
        \frac{\sqrt{L_{t}}+\sqrt{m}}{\sqrt{L_{t}}-\sqrt{m}}
        = \frac{L_{t} -(\sqrt{m}_{t}-\sqrt{m})\sqrt{L_{t}} -\sqrt{m_{t}m}}{L_{t} +(\sqrt{m}_{t}-\sqrt{m})\sqrt{L_{t}} -\sqrt{m_{t}m}}.
        \label{id:sm-gc-Lyapgd-main-cvx-coeff}
    \end{align}
    Letting \(l=L_{t}\), \(d=m_{t}-m\geq 0\) and \(a=\sqrt{m_{t}m}\), then after simplifying several terms, we obtain
    \begin{align*}
        \frac{\partial}{\partial l}\eqref{id:sm-gc-Lyapgd-main-cvx-coeff}
        = \frac{\partial}{\partial l}\frac{l-d\sqrt{l} -a}{l+d\sqrt{l}-a}
        &= \frac{(1-d/2\sqrt{l})(l+d\sqrt{l}-a)-(1+d/2\sqrt{l})(l-d\sqrt{l}-a)}{(l+d\sqrt{l}-a)^{2}}
        \\
        &= \frac{d\sqrt{l} + ad/\sqrt{l}}{(l+d\sqrt{l}-a)^{2}}
        \geq 0.
    \end{align*}
    That is, \(\alpha_{t}\) is nondecreasing in \(L_{t}\) while \(\alpha_{t}\) is decreasing in \(m_{t}\), because \(\beta_{t}\) is decreasing in \(m_{t}\) and \(\theta_{t}\) is not a function of \(m_{t}\). 
    By assumption \(L_{t}\) and \(m_{t}\) are respectively nondecreasing and nonincreasing, therefore \(\alpha_{j}\) is nondecreasing, so that \(\bar{\alpha}_{j}\) is nonincreasing for all \(0\leq j\leq t\).
\end{proof}

\begin{lemma}\label{lem:sm-gc-Lyapgd-main-descent}
    Let \(f\in\mathcal{S}(L,m)\), \(L_{t} \geq m_{t}\) and let \(s_{t+1}\) denote the iterate generated by \cref{alg:nag-free} from \(s_{t}\). 
    If \(m_{t}\geq m\), then
    \begin{align}
        (1+\delta^{\textup{GD}}_{t+1})V^{\textup{GD}}_{t+1}(s_{t+1})
        \leq V^{\textup{GD}}_{t+1}(s_{t}).
        \label{ineq:sm-gc-Lyapgd-main-descent}
    \end{align}
\end{lemma}

\begin{proof}
    To bound \((1+\delta^{\textup{GD}}_{t+1})V^{\textup{GD}}_{t+1}(s_{t+1})\) in terms of \(V^{\textup{GD}}_{t+1}(s_{t})\), we analyze their difference, which is the sum of one difference involving \(W_{t}\) and another one involving \(U_{t}\). 
    We address the one involving \(W_{t}\) first.
    To this end, we use the assumption that \(m_{t} \geq m\) to show that \cref{alg:nag-free} iterations can be expressed as a convex combination of appropriate GD and NAG iterations and then we exploit the fact that \(W_{t}\) is convex to bound the corresponding difference.
    
    To show that \cref{alg:nag-free} iterations are a convex combination of GD and NAG iterations, we consider fictitious ``one-shot'' GD and NAG iterations taking the value of \(T\) into account and appropriately initialized at a given iteration of \cref{alg:nag-free}.
    We let \(y^{\textup{GD}}_{t}=x^{\textup{GD}}_{t}=x^{\textup{NAG}}_{t}=x_{t}\) and \(y^{\textup{NAG}}_{t}=y_{t}\). 
    We initialize \(x^{\textup{GD}}_{t}\) and \(y^{\textup{GD}}_{t}\) ``backwards'' from \(x_{t}\) to conform them to the GD iteration constraint that \(y^{\textup{GD}}_{t}=x^{\textup{GD}}_{t}\).
    On the other hand, since NAG works with arbitrary initial points, we initialize NAG at the \(t\)-th NEST iteration exactly.
    With these initial points in mind, let \(y^{\textup{GD}}_{t+1}\), \(x^{\textup{GD}}_{t+1}\), \(y^{\textup{NAG}}_{t+1}\) and \(x^{\textup{NAG}}_{t+1}\) be the GD and NAG iterations produced by \labelcref{def:sm-gc-gd-descent,def:sm-gc-gd-momentum,def:nag-descent,def:nag-momentum}. Then, GD, NAG and \cref{alg:nag-free} produce the same descent step:
    \begin{align*}
        y^{\textup{GD}}_{t+1} 
        = x^{\textup{GD}}_{t} -(1/L_{t})\nabla f(x^{\textup{GD}}_{t})
        = \underbrace{x_{t} -(1/L_{t})\nabla f(x_{t})}_{y_{t+1}}
        = x^{\textup{NAG}}_{t} -(1/L_{t})\nabla f(x^{\textup{NAG}}_{t})
        = y^{\textup{NAG}}_{t+1}.
    \end{align*}
    In turn, \(x^{\textup{NAG}}_{t+1}\) reduces to an affine combination of the \cref{alg:nag-free} descent steps \(y_{t+1}\) and \(y_{t}\):
    \begin{align*}
        x^{\textup{NAG}}_{t+1} = (1+\theta_{t})y^{\textup{NAG}}_{t+1} -\theta_{t}y^{\textup{NAG}}_{t}
        = (1+\theta_{t})y_{t+1} -\theta_{t}y_{t}.
    \end{align*}
    It follows that, for all \(t\geq 0\) such that \(m_{t} \geq m\), \(x_{t+1}\) is a convex combination of \(x^{\textup{GD}}_{t+1}=y^{\textup{GD}}_{t+1}=y_{t+1}\) and \(x^{\textup{NAG}}_{t+1}\), as in
    \begin{align*}
        x_{t+1} 
        = (1+\beta_{t})y_{t+1} -\beta_{t}y_{t}
        =&\ \Bigl( 1+\theta_{t}\frac{\beta_{t}}{\theta_{t}} \pm \frac{\beta_{t}}{\theta_{t}} \Bigr)y_{t+1} -\theta_{t}\frac{\beta_{t}}{\theta_{t}} y_{t}
        \\
        =&\ \Bigl( 1-\frac{\beta_{t}}{\theta_{t}} \Bigr)y_{t+1} + \frac{\beta_{t}}{\theta_{t}}((1+\theta_{t})y_{t+1} -\theta_{t} y_{t})
        \\
        =&\ \Bigl( 1-\frac{\beta_{t}}{\theta_{t}} \Bigr)y^{\textup{GD}}_{t+1} + \frac{\beta_{t}}{\theta_{t}}((1+\theta_{t})y^{\textup{NAG}}_{t+1} -\theta_{t} y^{\textup{NAG}}_{t})
        \\
        =&\ \bar{\alpha}_{t}x^{\textup{GD}}_{t+1} + \alpha_{t}x^{\textup{NAG}}_{t+1},
    \end{align*}
    where, as defined in \eqref{def:sm-gc-Lyapgd}, the coefficients defining the convex combination are given by
    \begin{align*}
        \alpha_{t} = \beta_{t}/\theta_{t} \in \mathopen{[}0,1\mathclose{]},
        &&
        \bar{\alpha}_{t}= 1-\alpha_{t}\in\mathopen{[]}0,1\mathclose{]}.
    \end{align*}
    Likewise, \(y^{\textup{GD}}_{t+1}=y^{\textup{NAG}}_{t+1}=y_{t+1}\) implies \(y_{t+1}=\bar{\alpha}_{t}y^{\textup{GD}}_{t+1} + \alpha_{t}y^{\textup{NAG}}_{t+1}\) so, in fact, the entire iteration of \cref{alg:nag-free} can be expressed as a convex combination of GD and NAG iterations, as in
    \begin{align*}
        s_{t+1} = \bar{\alpha}_{t}s^{\textup{GD}}_{t+1} + \alpha_{t}s^{\textup{NAG}}_{t+1},
    \end{align*}
    where \(s_{t+1}\), \(s^{\textup{GD}}_{t+1}\) and \(s^{\textup{NAG}}_{t+1}\) comprise the iterations of \cref{alg:nag-free}, GD and NAG:
    \begin{align*}
        s_{t+1} 
        &= (x_{t+1}, y_{t+1}),
        &&
        \\
        s^{\textup{GD}}_{t+1} 
        &= (x^{\textup{GD}}_{t+1}, y^{\textup{GD}}_{t+1})
        =(y_{t+1},y_{t+1}),
        \\
        s^{\textup{NAG}}_{t+1} 
        &= (x^{\textup{NAG}}_{t+1},y^{\textup{NAG}}_{t+1})
        =(x^{\textup{NAG}}_{t+1},y_{t+1}).
    \end{align*}
    Hence, since \(W_{t}\) is convex (see \cref{app:nag-free:gd:rem:w-convex}), we can bound \(W_{t+1}(s_{t+1})\) in terms of GD and NAG iterations, as in
    \begin{align*}
        W_{t+1}(s_{t+1})
        \leq&\ \bar{\alpha}_{t} W_{t+1}(s^{\textup{GD}}_{t+1}) + \alpha_{t} W_{t+1}(s^{\textup{NAG}}_{t+1}).
    \end{align*}
    Then, it follows from \(W_{t+1}(s_{t})=\bar{\alpha}_{t}W_{t+1}(s_{t}) + \alpha_{t}W_{t+1}(s_{t})\) that
    \begin{align*}
        (1+\delta^{\textup{GD}}_{t+1})W_{t+1}(s_{t+1}) -W_{t+1}(s_{t})
        =&\ \bar{\alpha}_{t}((1+\delta^{\textup{GD}}_{t+1})W_{t+1}(s^{\textup{GD}}_{t+1})-W_{t+1}(s_{t})) 
        \\
        &+ \alpha_{t}
        \Bigl( (1+\delta^{\textup{ACC}}_{t+1})W_{t+1}(s^{\textup{NAG}}_{t+1})
        -W_{t+1}(s_{t}) 
        \Bigr)
        \\
        &+ \alpha_{t}(\delta^{\textup{GD}}_{t+1}-\delta^{\textup{ACC}}_{t+1})W_{t+1}(s^{\textup{NAG}}_{t+1}).
    \end{align*}
    Since \(y^{\textup{NAG}}_{t+1}=y_{t+1}\) and \(x^{\textup{NAG}}_{t}=x_{t}\), then the fact that \(y_{t+1}\) and \(L_{t}\) satisfy \eqref{ineq:sm-gc-NEST-dl} implies that
    \begin{align*}
        f(y^{\textup{NAG}}_{t+1})-f(x^{\textup{NAG}}_{t})
        =f(y_{t+1})-f(x_{t})
        \leq -(1/2L_{t})\Vert g_{t} \Vert^{2}
        =-(1/2L_{t})\Vert g^{\textup{NAG}}_{t} \Vert^{2}.
    \end{align*}		
    Moreover, \(L_{t}>0\) is nondecreasing and \(m_{t}>0\) is nonincreasing.
    Therefore, \cref{lem:sm-gc-Lyap1-descent-N} applies because \mbox{\cref{lem:sm-gc-Lyap1-descent-N}} imposes no restrictions on neither \(y^{\textup{NAG}}_{t}\) nor \(x^{\textup{NAG}}_{t}\). 
    So, letting
    \begin{align*}
        s^{\textup{NAG}}_{t} 
        = (x^{\textup{NAG}}_{t}, y^{\textup{NAG}}_{t})
        = (x_{t},y_{t})
        = s_{t},
    \end{align*}
    then \cref{lem:sm-gc-Lyap1-descent-N} combined with both the fact that \(\delta^{\textup{ACC}}_{t+1}\geq \delta^{\textup{GD}}_{t+1}\) and that \(\alpha_{t}>0\), imply
    \begin{align}
        \alpha_{t}\Bigl( (1+\delta^{\textup{ACC}}_{t+1})W_{t+1}(s^{\textup{NAG}}_{t+1})-W_{t+1}(s_{t})
        +(\delta^{\textup{GD}}_{t+1}-\delta^{\textup{ACC}}_{t+1})W_{t+1}(s^{\textup{NAG}}_{t+1}) \Bigr)
        \leq 0.
        \label{ineq:sm-gc-Lyapgd-main-Wk-dl-N}
    \end{align}
    The natural next move would be to address \((1+\delta^{\textup{GD}}_{t+1})W_{t+1}(s^{\textup{GD}}_{t+1})-W_{t+1}(s_{t})\) in an analogous way.
    The caveat, however, is that although \cref{lem:sm-gc-Lyap1-descent-N} applies to NAG iterations with arbitrary \(x^{\textup{NAG}}_{t}\) and \(y^{\textup{NAG}}_{t}\), the same is not true of \cref{lem:sm-gc-Lyap1-descent-GD}.
    That is, \cref{lem:sm-gc-Lyap1-descent-GD} applies to consecutive GD iterations, requiring that \(y^{\textup{GD}}_{t}=x^{\textup{GD}}_{t}\). 
    Hence, to be able to apply \cref{lem:sm-gc-Lyap1-descent-GD}, we add \(\mp W_{t+1}(s^{\textup{GD}}_{t})\) to the difference involving \(W_{t+1}\), using a GD iteration \(s^{\textup{GD}}_{t}\) such that \(y^{\textup{GD}}_{t}=x^{\textup{GD}}_{t}\). 
    That is, we define a fictitious GD iteration \(s^{\textup{GD}}_{t}\) ``backwards'' from \(x_{t}\) using the points that we already defined as \(y^{\textup{GD}}_{t}=x^{\textup{GD}}_{t}\), as in
    \begin{align}
        s^{\textup{GD}}_{t}
        = (x^{\textup{GD}}_{t},y^{\textup{GD}}_{t})
        = (x_{t},x_{t}).
        \label{app:nag-free:def:st-gd}
    \end{align}
    Although \(s^{\textup{GD}}_{t}\) need not equal \(s_{t}\), \(y^{\textup{GD}}_{t+1}=y_{t+1}\) and \(x^{\textup{GD}}_{t}=x_{t}\), thus
    \begin{align*}
        f(y^{\textup{GD}}_{t+1})-f(x^{\textup{GD}}_{t})
        = f(y_{t+1})-f(x_{t})
        \leq -(1/2L_{t})\Vert g_{t} \Vert^{2}
        = -(1/2L_{t})\Vert g^{\textup{GD}}_{t} \Vert^{2}.
    \end{align*}
    Therefore, since \(L_{t}>0\), \cref{lem:sm-gc-Lyap1-descent-GD} applies with \(s^{\textup{GD}}_{t}\) given by \eqref{app:nag-free:def:st-gd}, and implies that
    \begin{align}
        (1+\delta^{\textup{GD}}_{t+1})W_{t+1}(s^{\textup{GD}}_{t+1})-W_{t+1}(s^{\textup{GD}}_{t}) 
        \leq -(1/2L_{t})\Vert g^{\textup{GD}}_{t} \Vert^{2}
        = -(1/2L_{t})\Vert g_{t} \Vert^{2}.
        \label{ineq:sm-gc-Lyapgd-main-dl-gd}
    \end{align}
    Moreover, \(y^{\textup{GD}}_{t}=x^{\textup{GD}}_{t}=x_{t}\) implies that \(x^{\textup{GD},y}_{t}=0\), thus
    \begin{align*}
        z^{\textup{GD},\star}_{t}
        = x^{\textup{GD},\star}_{t} + \sqrt{p_{t-1}}x^{\textup{GD},y}_{t} 
        = x^{\textup{GD},\star}_{t}
        = x^{\star}_{t}
        &&
        \text{and}
        &&
        f(y^{\textup{GD}}_{t}) =f(x_{t}),
    \end{align*}
    and it follows that
    \begin{align}
        W_{t+1}(s^{\textup{GD}}_{t}) -W_{t+1}(s_{t})
        =&\ f(x_{t}) - f(y_{t}) 
        + (m/2)(\Vert x^{\star}_{t} \Vert^{2} -\Vert x^{\star}_{t} + \sqrt{p_{t}}x^{y}_{t} \Vert^{2}).
        \label{id:sm-gc-Lyapgd-main-Wk1-gap}
    \end{align}
    Applying \eqref{ineq:strong_convexity} with \(x=x_{t}\) and \(y=y_{t}\), then using the fact that \(2\langle g_{t},x^{y}_{t} \rangle \leq (1/L_{t})\Vert g_{t} \Vert^{2} + L_{t}\Vert x^{y}_{t} \Vert^{2}\), we obtain
    \begin{align*}
        f(x_{t}) -f(y_{t})
        \leq \langle g_{t},x^{y}_{t} \rangle 
        - (m/2)\Vert x^{y}_{t} \Vert^{2}
        \leq (1/2L_{t})\Vert g_{t} \Vert^{2} 
        + ((L_{t} - m)/2)\Vert x^{y}_{t} \Vert^{2}.
    \end{align*}
    Hence, expanding \(\Vert x^{\star}_{t} + \sqrt{p_{t}}x^{y}_{t} \Vert^{2}\) on \eqref{id:sm-gc-Lyapgd-main-Wk1-gap} and then using the above inequality, we get
    \begin{align}
        W_{t+1}(s^{\textup{GD}}_{t}) -W_{t+1}(s_{t})
        \leq&\ 
        \frac{1}{L_{t}}\Vert g_{t} \Vert^{2} 
        + \frac{L_{t}-m}{2}\Vert x^{y}_{t} \Vert^{2}
        +\frac{m}{2}(-2\sqrt{p_{t}}\langle x^{y}_{t},x^{\star}_{t} \rangle -p_{t}\Vert x^{y}_{t} \Vert^{2})
        \nonumber\\
        =&\ \frac{1}{2L_{t}}\Vert g_{t} \Vert^{2} 
        -\frac{m}{2}\Vert x^{y}_{t} \Vert^{2}
        -\sqrt{L_{t}m}\langle x^{y}_{t},x^{\star}_{t} \rangle,
        \label{ineq:sm-gc-Lyapgd-main-Wk-gap}
    \end{align}
    where \(m\sqrt{p_{t}}=\sqrt{L_{t}m}\) follows directly \eqref{def:sm-gc-theta-p}.
    Then, combining \eqref{ineq:sm-gc-Lyapgd-main-dl-gd} and \eqref{ineq:sm-gc-Lyapgd-main-Wk-gap} yields
    \begin{align}
        (1+\delta^{\textup{GD}}_{t+1})W_{t+1}(s^{\textup{GD}}_{t+1})\mp W_{t+1}(s^{\textup{GD}}_{t})-W_{t+1}(s_{t})
        \leq
        -\sqrt{L_{t}m}\langle x^{y}_{t},x^{\star}_{t} \rangle.
        \label{ineq:sm-gc-Lyapgd-main-Wk-gd}
    \end{align}
    Therefore, since \(\delta^{\textup{GD}}_{t+1}\leq \delta^{\textup{ACC}}_{t+1}\), combining \eqref{ineq:sm-gc-Lyapgd-main-Wk-dl-N} and \eqref{ineq:sm-gc-Lyapgd-main-Wk-gd}, we obtain
    \begin{align}
        (1+\delta^{\textup{GD}}_{t+1})W_{t+1}(s_{t+1})-W_{t+1}(s_{t})
        \leq&\ \bar{\alpha}_{t}((1+\delta^{\textup{GD}}_{t+1})W_{t+1}(s^{\textup{GD}}_{t+1})\mp W_{t+1}(s^{\textup{GD}}_{t}) -W_{t+1}(s_{t}))
        \nonumber\\
        &+ \alpha_{t}
        \Bigl( 
        (1+\delta^{\textup{ACC}}_{t+1})W_{t+1}(s^{\textup{NAG}}_{t+1})-W_{t+1}(s_{t})
        \Bigr)
        \nonumber\\
        &+ \alpha_{t}
        (\delta^{\textup{GD}}_{t+1}-\delta^{\textup{ACC}}_{t+1})W_{t+1}(s^{\textup{NAG}}_{t+1})
        \nonumber\\
        \leq&\ -\bar{\alpha}_{t}\sqrt{L_{t}m}\langle x^{y}_{t},x^{\star}_{t} \rangle.
        \label{ineq:sm-gc-Lyapgd-main-Wk-diff}
    \end{align}
    Next, we address the difference on \((1+\delta^{\textup{GD}}_{t+1})V^{\textup{GD}}_{t+1}(s_{t+1})-V^{\textup{GD}}_{t+1}(s_{t})\) involving \(U_{t}\). 
    \cref{lem:sm-gc-Lyap2-descent} implies that
    \begin{align}
        (\bar{\alpha}_{t}/\sqrt{p_{t}})((1+\delta^{\textup{GD}}_{t+1})U_{t+1}(s_{t+1})-U_{t+1}(s_{t}))
        \leq 
        (\bar{\alpha}_{t}/\sqrt{p_{t}})L_{t}\langle x^{y}_{t},x^{\star}_{t} \rangle
        = \bar{\alpha}_{t}\sqrt{L_{t}m}\langle x^{y}_{t},x^{\star}_{t} \rangle,
        \label{ineq:sm-gc-Lyapgd-main-Uk-diff}
    \end{align}
    since \(L_{t}/\sqrt{p_{t}}=\sqrt{L_{t}m}\).
    Then, combining \eqref{ineq:sm-gc-Lyapgd-main-Wk-diff} with \eqref{ineq:sm-gc-Lyapgd-main-Uk-diff} yields
    \begin{align*}
        (1+\delta^{\textup{GD}}_{t+1})V^{\textup{GD}}_{t+1}(s_{t+1})
        \leq V^{\textup{GD}}_{t+1}(s_{t}),
    \end{align*}
    proving \eqref{ineq:sm-gc-Lyapgd-main-descent}.
\end{proof}

\begin{lemma}\label{lem:sm-gc-Lyapgd-main-ascent}
    Let \(f\in\mathcal{F}(,m)\).
    If \(L_{t}\geq L_{t-1} \geq m\) and \(m_{t}\leq m_{t-1} \leq L\), then
    \begin{align}
        V^{\textup{GD}}_{t+1}
        \leq \frac{p_{t}^{2}}{p_{t-1}^{2}}V^{\textup{GD}}_{t}.
        \label{ineq:sm-gc-Lyapgd-main-ascent}
    \end{align}
\end{lemma}

\begin{proof}
    If \(L_{t}\geq L_{t-1} \geq m\) and \(m_{t}\leq m_{t-1} \leq L\), then \cref{lem:sm-gc-Lyap1-ascent} and \cref{lem:sm-gc-Lyap2-ascent} apply.
    It also follows that \(\sqrt{p_{t}} \geq \sqrt{p_{t-1}}\) and, by \cref{lem:sm-gc-bar-alpha}, that \(\bar{\alpha}_{t} \leq \bar{\alpha}_{t-1}\) 
    Thus, \(\bar{\alpha}_{t}/\sqrt{p_{t}} \leq \bar{\alpha}_{t-1}/\sqrt{p_{t-1}}\).
    Hence, combining \cref{lem:sm-gc-Lyap1-ascent} and \cref{lem:sm-gc-Lyap2-ascent} and then using the definition of \(V^{\textup{GD}}_{t}\), we obtain
    \begin{align*}
        V^{\textup{GD}}_{t+1}
        = W_{t+1} + \frac{\bar{\alpha}_{t}}{\sqrt{p}_{t}}U_{t+1}
        \leq \frac{p_{t}^{2}}{p_{t-1}^{2}}W_{t} + \frac{\bar{\alpha}_{t-1}}{\sqrt{p}_{t-1}}\frac{p_{t}}{p_{t-1}}U_{t}
        \leq \frac{p_{t}^{2}}{p_{t-1}^{2}}\Bigl( W_{t} + \frac{\bar{\alpha}_{t-1}}{\sqrt{p}_{t-1}}U_{t} \Bigr)
        = \frac{p_{t}^{2}}{p_{t-1}^{2}}V^{\textup{GD}}_{t},
    \end{align*}
    proving \eqref{ineq:sm-gc-Lyapgd-main-ascent}.
\end{proof}

\begin{theorem}\label{thm:sm-gc-Lyapgd-main}
    Let \(f\in\mathcal{S}(L,m)\) and let \(s_{t}\) denote the iterates generated by \cref{alg:nag-free}.
    If \(m_{t}\geq m\), then
    \begin{align}
        V^{\textup{GD}}_{t+1}(s_{t+1})
        \leq 2\max(L, L_{0})\frac{p_{t}^{2}}{p_{0}^{2}}\Vert x_{0} - x^{\star} \Vert^{2}
        \prod\limits_{i=1}^{t+1}(1+\delta^{\textup{GD}}_{i})^{-1}.
        \label{ineq:sm-gc-Lyapgd-main}
    \end{align}
\end{theorem}

\begin{proof}
    Under the above assumptions, \cref{lem:sm-gc-Lyapgd-main-descent} and \cref{lem:sm-gc-Lyapgd-main-ascent} hold.
    Hence, combining \eqref{ineq:sm-gc-Lyapgd-main-descent} and \eqref{ineq:sm-gc-Lyapgd-main-ascent}, for all \(s_{t+1}\) and \(s_{t}\) such that \(m_{t} \geq m\) we have that
    \begin{align}
        V^{\textup{GD}}_{t+1}(s_{t+1})
        \leq (1+\delta^{\textup{GD}}_{t+1})^{-1} \frac{p_{t}^{2}}{p_{t-1}^{2}} V^{\textup{GD}}_{t}(s_{t}).
        \label{ineq:sm-gc-Lyapgd-main-step}
    \end{align}
    We proceed with an inductive argument based on \eqref{ineq:sm-gc-Lyapgd-main-step}.
    To establish the base case, we apply \eqref{ineq:descent_lemma} with \(y=y_{0}\) and \(x=x^{\star}\), obtaining \(\tilde{f}(y_{0}) \leq (L/2)\Vert y^{\star}_{0} \Vert^{2}\).
    Then, since \(y_{0}=x_{0}\), it follows that \(z^{\star}_{0}=x^{\star}_{0}=y^{\star}_{0}\), and
    \begin{alignat*}{3}
        W_{0}(s_{0})
        &= \tilde{f}(y_{0}) + (m/2)\Vert x^{\star}_{0} \Vert^{2}
        \leq ((L+m)/2)\Vert x^{\star}_{0} \Vert^{2}
        \leq L\Vert x^{\star}_{0} \Vert^{2},
        \\
        U_{0}(s_{0})
        &= \tilde{f}(y_{0}) + (L_{0}/2)\Vert y^{\star}_{0} \Vert^{2}
        \leq \max(L, L_{0})\Vert x^{\star}_{0} \Vert^{2}.
    \end{alignat*}
    Since \(\bar{\alpha}_{0}/\sqrt{p_{0}}\in\mathopen{[}0,1\mathclose{]}\), the above inequalities imply
    \begin{align}
        V^{\textup{GD}}_{0}(s_{0})
        &= W_{0}(s_{0}) 
        + (\bar{\alpha}_{0}/\sqrt{p_{0}})U_{0}(s_{0})
        \leq 2\max(L, L_{0})\Vert x^{\star}_{0} \Vert^{2}.
        \label{ineq:sm-gc-Lyapgd-main-0-ub}
    \end{align}
    Moreover, \(W_{1}=W_{0}\) and \(U_{1}=U_{0}\), so that \(V^{\textup{GD}}_{1}=V^{\textup{GD}}_{0}\). 
    Hence, if \(m_{1} \geq m\), then combining \eqref{ineq:sm-gc-Lyapgd-main-descent} with \eqref{ineq:sm-gc-Lyapgd-main-0-ub}, we obtain
    \begin{align}
        V^{\textup{GD}}_{1}(s_{1})
        \leq 
        2\max(L, L_{0})\Vert x^{\star}_{0} \Vert^{2}(1+\delta^{\textup{GD}}_{1})^{-1}.
        \label{ineq:sm-gc-Lyapgd-main-induction-base-case}
    \end{align}
    Having established the base case \eqref{ineq:sm-gc-Lyapgd-main-induction-base-case}, suppose that
    \begin{align}
        V^{\textup{GD}}_{j+1}(s_{j+1})
        \leq 
        2\max(L, L_{0})\frac{p_{j}^{2}}{p_{0}^{2}}
        \Vert x^{\star}_{0} \Vert^{2}
        \prod\limits_{i=1}^{j+1}(1+\delta^{\textup{GD}}_{i})^{-1},
        \label{ineq:sm-gc-Lyapgd-main-induction-hypothesis}
    \end{align}
    holds for all \(0\leq j\leq t-1\) such that \(m_{j} \geq m\).
    Then, suppose \(m_{t} \geq m\). 
    Since the \(m\) estimates are nonincreasing, it follows that \(m_{j} \geq m\) for all \(0\leq j \leq t\).
    Hence, plugging the induction hypothesis \eqref{ineq:sm-gc-Lyapgd-main-induction-hypothesis} with \(j=t-1\) into \eqref{ineq:sm-gc-Lyapgd-main-step}, the \(p_{t-1}^{2}\) term on the numerator of \eqref{ineq:sm-gc-Lyapgd-main-induction-hypothesis} and on the denominator of \eqref{ineq:sm-gc-Lyapgd-main-step} cancel each other and we obtain
    \begin{align*}
        V^{\textup{GD}}_{t+1}(s_{t+1})
        \leq 
        2\max(L, L_{0})\frac{p_{t}^{2}}{p_{0}^{2}}
        \Vert x^{\star}_{0} \Vert^{2}
        \prod\limits_{i=1}^{t+1}(1+\delta^{\textup{GD}}_{i})^{-1}.
    \end{align*}
    Therefore, we conclude by induction that \eqref{ineq:sm-gc-Lyapgd-main-induction-hypothesis} holds for all \(j \geq 0\), proving \eqref{ineq:sm-gc-Lyapgd-main}.
\end{proof}

The same arguments above directly imply \cref{cor:nag}.

\begin{proof}[Proof of \cref{cor:nag}]
    The proof of \cref{thm:gc} does not use of the particular dynamics of \(m_{t}\) induced by NAG-free (\cref{alg:nag-free}), and the initialization of \(m_{0}\) is also not important as long as \(m_{0} \in \mathopen{[}m,L\mathclose{]}\).
    Hence, the same arguments also apply to the analysis of NAG \labelcref{def:nag-descent,def:nag-momentum}.
    In this case, \(L_{t}\equiv L\) and \(m_{t}\equiv m\) for all \(t\).
    Therefore, denoting by \(s_{t}=(x_{t},y_{t})\) the iterates of NAG \labelcref{def:nag-descent,def:nag-momentum} and using the definition of \(V^{\textup{GD}}_{t}\), \eqref{def:sm-gc-Lyapgd}, it follows that
    \begin{align*}
        f(y_{t+1}) - f(x^{\star})
        \leq V^{\textup{GD}}_{t+1}(s_{t+1})
        \leq 
        \Big(
        1 - \frac{1}{\kappa}
        \Big)^{t}
        2L\kappa^{2}
        \Vert x^{\star}_{0} \Vert^{2}.
    \end{align*}
\end{proof}

\subsection{Case 2: \texorpdfstring{\(m_{t} < m\)}{mt < m}}
\label{gc:case_2}

In the previous section, we analyzed iterations in which \(m_{t} \geq m\).
Now, we analyze iterations in which \(m_{t} < m\) and also the iteration \(t\) in which \(m_{t} \geq m\) and \(m_{t+1} < m\).
Since \(m_{t}\) are nonincreasing, there is a most one such transition iteration.

In \cref{thm:sm-gc-Lyapgd-main}, we proved that if \(m_{t} \geq m\), then \cref{alg:nag-free} converges at least as fast as GD.
Now, we prove that if \(m_{t} \geq m\), then \cref{alg:nag-free} converges evvn faster.
Specifically, if \(m_{t} < m\), then \cref{alg:nag-free} achieves the accelerated rate \((1+\delta(\sqrt{\bar{\kappa}_{t})})^{-1}\), where
\begin{align}
    \hat{\delta}^{\textup{ACC}}_{t}
    = 
    \begin{cases}
        1/(\sqrt{q_{0}}-1), & t =0,
        \\
        1/(\sqrt{q_{t-1}}-1), & t \geq 1.
    \end{cases}
    \label{def:sm-gc-Lyapacc-deltahat}
\end{align}

The proof once again consists in an inductive argument based on descending and ascending bounds on a Lyapunov function. 
The function we work with this time is \(V^{\textup{ACC}}_{t}\), given by
\begin{align}
    V^{\textup{ACC}}_{t}(s_{t}) = 
    \begin{cases}
        \tilde{f}(y_{0}) + (m_{0}/2)\Vert w^{\star}_{0} \Vert^{2}, & t = 0,
        \\
        \tilde{f}(y_{t}) + (m_{t-1}/2)\Vert w^{\star}_{t} \Vert^{2}, & t \geq 1,
    \end{cases}
    \label{def:sm-gc-Lyapacc}
\end{align}
where the pseudo-state \(w^{\star}_{t}\), analogous to \(z^{\star}_{t}\), is given by
	\begin{align}
		w^{\star}_{t} = w_{t}-x^{\star},
		&&
		w_{t} =
		\begin{cases}
			x_{0} + \sqrt{q_{0}}(x_{0}-y_{0}), & t = 0,
			\\
			x_{t} + \sqrt{q_{t-1}}(x_{t}-y_{t}), & t \geq 1.
		\end{cases}
		\label{def:sm-gc-Lyapacc-w}
	\end{align}
We first prove the descending bound and then prove the ascending one.
	
\begin{lemma}\label[lemma]{lem:sm-gc-Lyapacc-descent}
    Let \(f\in\mathcal{S}(L,m)\), and let \(s_{t+1}\) be generated by \cref{alg:nag-free}.
    If \(m_{t} \leq m\), then
    \begin{align}
        (1+\hat{\delta}^{\textup{ACC}}_{t+1})V^{\textup{ACC}}_{t+1}(s_{t+1}) -V^{\textup{ACC}}_{t+1}(s_{t})\leq 0.
        \label{ineq:sm-gc-Lyapacc-descent}
    \end{align}
\end{lemma}

\begin{proof}
    The difference \((1+\hat{\delta}^{\textup{ACC}}_{t+1})V^{\textup{ACC}}_{t+1}(s_{t+1})-V^{\textup{ACC}}_{t+1}(s_{t})\) is the sum of a difference involving \(\tilde{f}\) and another one involving 2-norms.
    We first analyze the difference involving \(\tilde{f}\), splitting it into three further differences:
    \begin{align*}
        (1+\hat{\delta}^{\textup{ACC}}_{t+1})\tilde{f}(y_{t+1})-\tilde{f}(y_{t})
        =&\ (1+\hat{\delta}^{\textup{ACC}}_{t+1})(f(y_{t+1})-f(x_{t}))
        \\
        &+ \hat{\delta}^{\textup{ACC}}_{t+1}(f(x_{t})-f(x^{\star}))
        \\
        &+ f(x_{t})-f(y_{t}).
    \end{align*}
    Since \(y_{t+1}\) produced by \cref{alg:nag-free} satisfies \eqref{ineq:sm-gc-NEST-dl}, we have that
    \begin{align}
        (1+\hat{\delta}^{\textup{ACC}}_{t+1})(f(y_{t+1})-f(x_{t}))
        \leq&\ -(1+\hat{\delta}^{\textup{ACC}}_{t+1})(1/2L_{t})\Vert g_{t} \Vert^{2}.
        \label{ineq:sm-gc-Lyapacc-fdiff1}
    \end{align}
    Moreover, if \(m_{t}\leq m\), then \eqref{ineq:strong_convexity} implies that for all \(x\) and \(y\)
    \begin{align}
        f(x) + \langle \nabla f(x),y-x \rangle + (m_{t}/2)\Vert x-y \Vert^{2}
        \leq f(y).
        \label{ineq:sm-gc-Lyapacc-sc}
    \end{align}
    Hence, plugging \(x=x_{t}\) and \(y=x^{\star}\) in \eqref{ineq:sm-gc-Lyapacc-sc} and using the fact that \(f\) is convex, we obtain
    \begin{align}
        \hat{\delta}^{\textup{ACC}}_{t+1}(f(x_{t})-f(x^{\star}))
        \leq&\ \hat{\delta}^{\textup{ACC}}_{t+1}\left\langle g_{t},x^{\star}_{t} \right\rangle -\hat{\delta}^{\textup{ACC}}_{t+1}(m_{t}/2)\Vert x^{\star}_{t} \Vert^{2},
        \label{ineq:sm-gc-Lyapacc-fdiff2}
        \\
        f(x_{t})-f(y_{t})
        \leq&\ \left\langle g_{t},x^{y}_{t} \right\rangle.
        \label{ineq:sm-gc-Lyapacc-fdiff3}
    \end{align}
    
    Next, we address the 2-norm difference in \((1+\hat{\delta}^{\textup{ACC}}_{t+1})V^{\textup{ACC}}_{t+1}(s_{t+1})-V^{\textup{ACC}}_{t+1}(s_{t})\) by expanding the pseudo-states inside 2-norms.
    One pseudo-state is \(w^{\star}_{t+1}\) which, using \labelcref{def:sm-gc-beta-q,def:sm-gc-Lyapacc-w}, we express as
    \begin{align*}
        w^{\star}_{t+1}
        =&\ x_{t+1} + \sqrt{q_{t}}(x_{t+1}-y_{t+1}) - x^{\star}
        \\
        =&\ y_{t+1}+\beta_{t}(y_{t+1}-y_{t}) + \sqrt{q_{t}}\beta_{t}(y_{t+1}-y_{t}) -x^{\star}
        \\
        =&\ -(1/L_{t})(1+\beta_{t}(1+\sqrt{q_{t}}))g_{t} + \beta_{t}(1+\sqrt{q_{t}})x^{y}_{t} + x^{\star}_{t}
        \\
        =&\ -(1/L_{t})\sqrt{q_{t}}g_{t} + (\sqrt{q_{t}}-1)x^{y}_{t} + x^{\star}_{t}.
    \end{align*}
    After expanding \(w^{\star}_{t+1}\) inside the 2-norm, we use the following identities after colons to simplify the coefficients of terms before colons:
    \begin{align*}
        \Vert g_{t} \Vert^{2}&:
        &
        (q_{t}/L_{t}^{2})(m_{t}/2)
        =&\ 1/2L_{t},
        \\
        \langle g_{t},x^{y}_{t} \rangle&:
        &
        m_{t}(1+\hat{\delta}^{\textup{ACC}}_{t+1})\sqrt{q_{t}}(\sqrt{q_{t}}-1)/L_{t}
        =&\ 1,
        \\
        \langle g_{t},x^{\star}_{t} \rangle&:
        &
        m_{t}(1+\hat{\delta}^{\textup{ACC}}_{t+1})\sqrt{q_{t}}/L_{t}
        =&\ \delta^{\textup{ACC}}_{t},
        \\
        \Vert x^{y}_{t} \Vert^{2}&:
        &
        (1+\hat{\delta}^{\textup{ACC}}_{t+1})(\sqrt{q_{t}}-1)^{2}
        =&\ \sqrt{q_{t}}(\sqrt{q_{t}}-1),
        \\
        \langle x^{y}_{t},x^{\star}_{t} \rangle&:
        &
        (1+\hat{\delta}^{\textup{ACC}}_{t+1})(\sqrt{q_{t}}-1)
        =&\ \sqrt{q_{t}}.
    \end{align*}
    Thus, the 2-norm difference in \((1+\hat{\delta}^{\textup{ACC}}_{t+1})V^{\textup{ACC}}_{t+1}(s_{t+1})-V^{\textup{ACC}}_{t+1}(s_{t})\) reduces to
    \begin{align}
        &(1+\hat{\delta}^{\textup{ACC}}_{t+1})\frac{m_{t}}{2}\Vert w^{\star}_{t+1} \Vert^{2} 
        -\frac{m_{t}}{2}\Vert x^{\star}_{t} + \sqrt{q_{t}}x^{y}_{t} \Vert^{2}
        \nonumber\\
        =& \frac{1+\hat{\delta}^{\textup{ACC}}_{t+1}}{2L_{t}}\Vert g_{t} \Vert^{2}
        -\left\langle g_{t},x^{y}_{t} \right\rangle
        -\delta^{\textup{ACC}}_{t}\left\langle g_{t},x^{\star}_{t} \right\rangle
        +\frac{m_{t}}{2}\sqrt{q_{t}}(\sqrt{q_{t}}-1)\Vert x^{y}_{t} \Vert^{2}
        \nonumber\\
        &+\frac{m_{t}}{2}(2\sqrt{q_{t}}\langle x^{y}_{t},x^{\star}_{t} \rangle
        +(1+\hat{\delta}^{\textup{ACC}}_{t+1})\Vert x^{\star}_{t} \Vert^{2})
        -\frac{m_{t}}{2}(q_{t}\Vert x^{y}_{t} \Vert^{2} + 2\sqrt{q_{t}}\langle x^{y}_{t},x^{\star}_{t} \rangle
        +\Vert x^{\star}_{t} \Vert^{2})
        \nonumber\\
        =& \frac{1+\hat{\delta}^{\textup{ACC}}_{t+1}}{2L_{t}}\Vert g_{t} \Vert^{2}
        -\left\langle g_{t},x^{y}_{t} \right\rangle
        -\delta^{\textup{ACC}}_{t}\left\langle g_{t},x^{\star}_{t} \right\rangle
        -\frac{m_{t}}{2}\sqrt{q_{t}}\Vert x^{y}_{t} \Vert^{2}
        + \delta^{\textup{ACC}}_{t}\frac{m_{t}}{2}\Vert x^{\star}_{t} \Vert^{2}.
        \label{id:sm-gc-Lyapacc-2normdiff}
    \end{align}
    Finally, combining \labelcref{ineq:sm-gc-Lyapacc-fdiff1,ineq:sm-gc-Lyapacc-fdiff2,ineq:sm-gc-Lyapacc-fdiff3,id:sm-gc-Lyapacc-2normdiff}, cancelling terms and then using the assumption that \(m_{t}>0\), we obtain
    \begin{align*}
        (1+\hat{\delta}^{\textup{ACC}}_{t+1})V^{\textup{ACC}}_{t+1}(s_{t+1})-V^{\textup{ACC}}_{t+1}(s_{t})
        \leq -(m_{t}/2)\sqrt{q_{t}}\Vert x^{y}_{t} \Vert^{2} 
        \leq 0.
    \end{align*}
\end{proof}

\begin{lemma}\label[lemma]{lem:sm-gc-Lyapacc-ascent}
    Let \(f\in\mathcal{S}(L,m)\). 
    If \(L_{t}\geq L_{t-1}\geq m_{t-1}\geq m_{t}\) and \(m_{t}\leq m\), then
    \begin{align}
        V^{\textup{ACC}}_{t+1}
        \leq 
        \frac{q_{t}^{2}}{q_{t-1}^{2}}
        V^{\textup{ACC}}_{t}.
        \label{ineq:sm-gc-Lyapacc-ascent}
    \end{align}
\end{lemma}

\begin{proof}
    If \(m_{t} \leq m_{t-1}\), then
    \begin{align}
        V^{\textup{ACC}}_{t+1}(s_{t})-V^{\textup{ACC}}_{t}(s_{t})
        =&\ \frac{m_{t}}{2}\Vert x^{\star}_{t} + \sqrt{q_{t}}x^{y}_{t} \Vert^{2}
        -\frac{m_{t-1}}{2}\Vert w^{\star}_{t} \Vert^{2}
        \nonumber\\
        \leq&\ \frac{m_{t}}{2}(\Vert x^{\star}_{t} + \sqrt{q_{t}}x^{y}_{t} \Vert^{2}-\Vert w^{\star}_{t} \Vert^{2}).
        \label{ineq:sm-gc-Lyapacc-ascent-V-gap}
    \end{align}
    Hence, to prove \eqref{ineq:sm-gc-Lyapacc-ascent}, we express bounds on \eqref{ineq:sm-gc-Lyapacc-ascent-V-gap} in terms of \(V^{\textup{ACC}}_{t+1}\) and \(V^{\textup{ACC}}_{t}\).
    To this end, we first note that the term in parenthesis on the right-hand side of \eqref{ineq:sm-gc-Lyapacc-ascent-V-gap} can be expressed as
    \begin{align}
        \Vert x^{\star}_{t} + \sqrt{q_{t}}x^{y}_{t} \Vert^{2}-\Vert w^{\star}_{t} \Vert^{2}
        =&\ 2(\sqrt{q_{t}}-\sqrt{q_{t-1}})\langle x^{\star}_{t},x^{y}_{t} \rangle
        + (q_{t}-q_{t-1})\Vert x^{y}_{t} \Vert^{2}.
        \label{id:sm-gc-Lyapacc-ascent-2norm-gap}
    \end{align}
    We consider two cases, each representing a possible sign of \(\langle x^{y}_{t},x^{\star}_{t} \rangle\). 
    
    First, suppose \(\langle x^{y}_{t},x^{\star}_{t} \rangle\geq 0\). 
    If \(L_{t}\geq L_{t-1}\), then \(\sqrt{q_{t-1}}/\sqrt{q_{t}}\leq 1\), so that
    \begin{align}
        \sqrt{q_{t}}-\sqrt{q_{t-1}}
        \leq q_{t}/\sqrt{q_{t}} - \sqrt{q_{t-1}}(\sqrt{q_{t-1}}/\sqrt{q_{t}})
        = (q_{t}-q_{t-1})/\sqrt{q_{t}}.
        \label{ineq:sm-gc-Lyapacc-ascent-case-aux1}
    \end{align}
    Plugging \eqref{ineq:sm-gc-Lyapacc-ascent-aux1} into \eqref{id:sm-gc-Lyapacc-ascent-2norm-gap} and then adding a nonnegative \(\Vert x^{\star}_{t} \Vert^{2}\) term, we get
    \begin{align}
        \Vert x^{\star}_{t} + \sqrt{q_{t}}x^{y}_{t} \Vert^{2}-\Vert w^{\star}_{t} \Vert^{2}
        \leq&\ 2\frac{q_{t}-q_{t-1}}{\sqrt{q_{t}}}\langle x^{\star}_{t},x^{y}_{t} \rangle
        + (q_{t}-q_{t-1})\Vert x^{y}_{t} \Vert^{2}
        + \frac{q_{t}-q_{t-1}}{q_{t}}\Vert x^{\star}_{t} \Vert^{2}
        \nonumber\\
        =&\ \frac{q_{t}-q_{t-1}}{q_{t}}\Vert x^{\star}_{t} + \sqrt{q_{t}}x^{y}_{t} \Vert^{2}.
        \label{ineq:sm-gc-Lyapacc-ascent-case1-aux2}
    \end{align}
    Then, plugging \eqref{ineq:sm-gc-Lyapacc-ascent-case1-aux2} back into \eqref{ineq:sm-gc-Lyapacc-ascent-V-gap} yields
    \begin{align}
        V^{\textup{ACC}}_{t+1}(s_{t})-V^{\textup{ACC}}_{t}(s_{t})
        \leq \frac{q_{t}-q_{t-1}}{q_{t}}\frac{m_{t}}{2}\Vert x^{\star}_{t} + \sqrt{q_{t}}x^{y}_{t} \Vert^{2}
        \leq \frac{q_{t}-q_{t-1}}{q_{t}}V^{\textup{ACC}}_{t+1}(s_{t}),
        \label{ineq:sm-gc-Lyapacc-ascent-case1-aux3}
    \end{align}
    where the last inequality follows from the definition of \(V^{\textup{ACC}}_{t}\), \eqref{def:sm-gc-Lyapacc}, as \(\tilde{f}\geq 0\) implies
    \begin{align}
        V^{\textup{ACC}}_{t+1}(s_{t})
        = \tilde{f}(y_{t}) + \frac{m_{t}}{2}\Vert x^{\star}_{t} + \sqrt{q_{t}}x^{y}_{t} \Vert^{2}
        \geq \frac{m_{t}}{2}\Vert x^{\star}_{t} + \sqrt{q_{t}}x^{y}_{t} \Vert^{2}.
        \label{ineq:sm-gc-Lyapacc-ascent-lb}
    \end{align}
    Thus, rearranging terms in \eqref{ineq:sm-gc-Lyapacc-ascent-case1-aux3} and then multiplying both sides by \(q_{t}/q_{t-1}\), we obtain
    \begin{align*}
        V^{\textup{ACC}}_{t+1}(s_{t})
        \leq \frac{q_{t}}{q_{t-1}}V^{\textup{ACC}}_{t}(s_{t})
        \leq \frac{q_{t}^{2}}{q_{t-1}^{2}}V^{\textup{ACC}}_{t}(s_{t}),
    \end{align*}
    where the second inequality holds because \(q_{t}/q_{t-1}\geq 1\).
    
    Now, suppose \(\langle x^{y}_{t},x^{\star}_{t} \rangle < 0\).
    As in the previous case, we start by bounding the gap \eqref{id:sm-gc-Lyapacc-ascent-2norm-gap}. 
    But given the negative sign of \(\langle x^{y}_{t},x^{\star}_{t} \rangle\) term, we bound the \(\Vert x^{y}_{t} \Vert^{2}\) term instead.
    To this end, we first use the assumption that \(\langle x^{y}_{t},x^{\star}_{t} \rangle < 0\) to establish that
    \begin{align}
        \Vert y^{\star}_{t} \Vert^{2}
        = \Vert y^{\star}_{t} \mp x^{\star}_{t} \Vert^{2}
        = \Vert x^{\star}_{t} - x^{y}_{t} \Vert^{2}
        = \Vert x^{\star}_{t} \Vert^{2} -2\langle x^{\star}_{t},x^{y}_{t} \rangle + \Vert x^{y}_{t} \Vert^{2}
        \geq \Vert x^{\star}_{t} \Vert^{2}.
        \label{ineq:sm-gc-Lyapacc-ascent-yk-ub}
    \end{align}
    To use the above inequality on \eqref{id:sm-gc-Lyapacc-ascent-2norm-gap}, first we rewrite \eqref{id:sm-gc-Lyapacc-ascent-2norm-gap} as
    \begin{align}
        \Vert x^{\star}_{t} + \sqrt{q_{t}}x^{y}_{t} \Vert^{2}-\Vert w^{\star}_{t} \Vert^{2}
        =&\ 2\frac{\sqrt{q_{t}}-\sqrt{q_{t-1}}}{\sqrt{q_{t}}}\langle x^{\star}_{t},\sqrt{q_{t}} x^{y}_{t} \rangle		
        + \sqrt{q_{t}}(\sqrt{q_{t}}-\sqrt{q_{t-1}})\Vert x^{y}_{t} \Vert^{2}
        \nonumber\\
        &+ \sqrt{q_{t-1}}(\sqrt{q_{t}}-\sqrt{q_{t-1}})\Vert x^{y}_{t} \Vert^{2}
        \pm \frac{\sqrt{q_{t}}-\sqrt{q_{t-1}}}{\sqrt{q_{t}}}\Vert x^{\star}_{t} \Vert^{2}
        \nonumber\\
        =&\ \frac{\sqrt{q_{t}}-\sqrt{q_{t-1}}}{\sqrt{q_{t}}}\Vert x^{\star}_{t} + \sqrt{q_{t}}x^{y}_{t} \Vert^{2}
        + \sqrt{q_{t-1}}(\sqrt{q_{t}}-\sqrt{q_{t-1}})\Vert x^{y}_{t} \Vert^{2}
        \nonumber\\
        &- \frac{\sqrt{q_{t}}-\sqrt{q_{t-1}}}{\sqrt{q_{t}}}\Vert x^{\star}_{t} \Vert^{2}.
        \label{id:sm-gc-Lyapacc-ascent-aux1}
    \end{align}
    Then, we apply \eqref{ineq:sm-gc-Lyap1-ascent-elementary} with \(a=w^{\star}_{t}\), \(b=x^{\star}_{t}\) and \(c^{2}=\sqrt{q_{t-1}}/\sqrt{q_{t}}\) to bound the \(\Vert x^{y}_{t} \Vert^{2}\) term on \eqref{id:sm-gc-Lyapacc-ascent-aux1}, as in
    \begin{align}
        \sqrt{q_{t-1}}(\sqrt{q_{t}}-\sqrt{q_{t-1}})\Vert x^{y}_{t} \Vert^{2}
        =&\ \sqrt{q_{t-1}}(\sqrt{q_{t}}-\sqrt{q_{t-1}})\Vert x^{y}_{t} \pm x^{\star}_{t}/\sqrt{q_{t-1}}\Vert^{2}
        \nonumber\\
        =&\ \frac{\sqrt{q_{t}}-\sqrt{q_{t-1}}}{\sqrt{q_{t-1}}}\Vert w^{\star}_{t} - x^{\star}_{t}\Vert^{2}
        \nonumber\\
        \leq&\ \frac{\sqrt{q_{t}}-\sqrt{q_{t-1}}}{\sqrt{q_{t-1}}}\Bigl( 1+\frac{\sqrt{q_{t}}}{\sqrt{q_{t-1}}} \Bigr)\Vert w^{\star}_{t} \Vert^{2}
        \nonumber\\
        & + \frac{\sqrt{q_{t}}-\sqrt{q_{t-1}}}{\sqrt{q_{t-1}}}\Bigl( 1+\frac{\sqrt{q_{t-1}}}{\sqrt{q_{t}}} \Bigr)\Vert x^{\star}_{t} \Vert^{2}
        \nonumber\\
        =&\ \frac{q_{t}-q_{t-1}}{q_{t-1}}\Vert w^{\star}_{t} \Vert^{2}
        +\frac{\sqrt{q_{t}}-\sqrt{q_{t-1}}}{\sqrt{q_{t}}}\frac{\sqrt{q_{t}}+\sqrt{q_{t-1}}}{\sqrt{q_{t-1}}}\Vert x^{\star}_{t} \Vert^{2}.
        \label{ineq:sm-gc-Lyapacc-ascent-aux1}
    \end{align}
    Plugging \eqref{ineq:sm-gc-Lyapacc-ascent-aux1} back into \eqref{id:sm-gc-Lyapacc-ascent-aux1} and then using \eqref{ineq:sm-gc-Lyapacc-ascent-yk-ub}, we get
    \begin{align}
        \Vert x^{\star}_{t} + \sqrt{q_{t}}x^{y}_{t} \Vert^{2}-\Vert w^{\star}_{t} \Vert^{2}
        \leq&\ \frac{\sqrt{q_{t}}-\sqrt{q_{t-1}}}{\sqrt{q_{t}}}\Vert x^{\star}_{t} + \sqrt{q_{t}}x^{y}_{t} \Vert^{2}
        +\frac{q_{t}-q_{t-1}}{q_{t-1}}\Vert w^{\star}_{t} \Vert^{2}
        \nonumber\\
        &+ \frac{\sqrt{q_{t}}-\sqrt{q_{t-1}}}{\sqrt{q_{t}}}\Bigl( \frac{\sqrt{q_{t}}+\sqrt{q_{t-1}}}{\sqrt{q_{t-1}}} -1 \Bigr)\Vert x^{\star}_{t} \Vert^{2}
        \nonumber\\
        \leq&\ \frac{\sqrt{q_{t}}-\sqrt{q_{t-1}}}{\sqrt{q_{t}}}\Vert x^{\star}_{t} + \sqrt{q_{t}}x^{y}_{t} \Vert^{2}
        +\frac{q_{t}-q_{t-1}}{q_{t-1}}\Vert w^{\star}_{t} \Vert^{2}
        \nonumber\\
        &+ \frac{\sqrt{q_{t}}-\sqrt{q_{t-1}}}{\sqrt{q_{t-1}}}\Vert y^{\star}_{t} \Vert^{2}.
        \label{ineq:sm-gc-Lyapacc-ascent-aux2}
    \end{align}
    In turn, since \(m_{t-1}\geq m_{t}\) and \(m_{t}\leq m\), plugging \eqref{ineq:sm-gc-Lyapacc-ascent-aux2} back into \eqref{ineq:sm-gc-Lyapacc-ascent-V-gap} we obtain
    \begin{align}
        V^{\textup{ACC}}_{t+1}(s_{t})-V^{\textup{ACC}}_{t}(s_{t})
        \leq&\ \frac{m_{t}}{2}\frac{\sqrt{q_{t}}-\sqrt{q_{t-1}}}{\sqrt{q_{t}}}\Vert x^{\star}_{t} + \sqrt{q_{t}}x^{y}_{t} \Vert^{2}
        + \frac{m_{t}}{2}\frac{q_{t}-q_{t-1}}{q_{t-1}}\Vert w^{\star}_{t} \Vert^{2}
        \nonumber\\
        &+ \frac{m_{t}}{2}\frac{\sqrt{q_{t}}-\sqrt{q_{t-1}}}{\sqrt{q_{t-1}}}\Vert y^{\star}_{t} \Vert^{2}
        \nonumber\\
        \leq&\ \frac{\sqrt{q_{t}}-\sqrt{q_{t-1}}}{\sqrt{q_{t}}}\frac{m_{t}}{2}\Vert x^{\star}_{t} + \sqrt{q_{t}}x^{y}_{t} \Vert^{2}
        + \frac{q_{t}-q_{t-1}}{q_{t-1}}\frac{m_{t-1}}{2}\Vert w^{\star}_{t} \Vert^{2}
        \nonumber\\
        &+ \frac{\sqrt{q_{t}}-\sqrt{q_{t-1}}}{\sqrt{q_{t-1}}}\frac{m}{2}\Vert y^{\star}_{t} \Vert^{2}.
        \label{ineq:sm-gc-Lyapacc-ascent-aux3}
    \end{align}
    Now, as in \eqref{ineq:sm-gc-Lyapacc-ascent-lb}, applying \(\tilde{f}\geq 0\) to the definition of \(V^{\textup{ACC}}_{t}\), we get
    \begin{align}
        V^{\textup{ACC}}_{t}(s_{t})
        = \tilde{f}(y_{t}) + \frac{m_{t-1}}{2}\Vert w^{\star}_{t}\Vert^{2}
        \geq \frac{m_{t-1}}{2}\Vert w^{\star}_{t}\Vert^{2}.
        \label{ineq:sm-gc-Lyapacc-ascent-lb2}
    \end{align}
    In the same vein, applying \eqref{ineq:strong_convexity} with \(x=x^{\star}\) and \(y=y_{t}\) to the definition of \(V^{\textup{ACC}}_{t}\), we obtain
    \begin{align}
        V^{\textup{ACC}}_{t}(s_{t})
        = \tilde{f}(y_{t}) + \frac{m_{t-1}}{2}\Vert w^{\star}_{t}\Vert^{2}
        \geq \frac{m}{2}\Vert y^{\star}_{t}\Vert^{2}.
        \label{ineq:sm-gc-Lyapacc-ascent-lb3}
    \end{align}
    Plugging in \eqref{ineq:sm-gc-Lyapacc-ascent-lb}, \eqref{ineq:sm-gc-Lyapacc-ascent-lb2} and \eqref{ineq:sm-gc-Lyapacc-ascent-lb3} back into \eqref{ineq:sm-gc-Lyapacc-ascent-aux3}, and then moving all \(V^{{acc}}_{t+1}(s_{t})\) terms to the left-hand side and all \(V^{\textup{ACC}}_{t}(s_{t})\) to the right-hand side, we obtain
    \begin{align}
        \frac{\sqrt{q_{t-1}}}{\sqrt{q_{t}}}V^{\textup{ACC}}_{t+1}(s_{t})
        \leq&\ \Bigl( \frac{q_{t}}{q_{t-1}} + \frac{\sqrt{q_{t}}-\sqrt{q_{t-1}}}{\sqrt{q_{t-1}}} \Bigr)V^{\textup{ACC}}_{t}(s_{t})
        \label{ineq:sm-gc-Lyapacc-ascent-aux4}
    \end{align}
    Multiplying both sides of \eqref{ineq:sm-gc-Lyapacc-ascent-aux4} by \(\sqrt{q_{t}}/\sqrt{q_{t-1}}\), and then using the fact that \(\sqrt{q_{t}}\geq \sqrt{q_{t-1}}\) yields
    \begin{align*}
        V^{\textup{ACC}}_{t+1}(s_{t})
        \leq \frac{\sqrt{q}_{t}}{\sqrt{q_{t-1}}}\Bigl(\frac{q_{t}}{q_{t-1}} + \frac{\sqrt{q_{t}}-\sqrt{q_{t-1}}}{\sqrt{q_{t-1}}} \Bigr)V^{\textup{ACC}}_{t}(s_{t})
        \leq \frac{q_{t}^{2}}{q_{t-1}^{2}}V^{\textup{ACC}}_{t}(s_{t}),
    \end{align*}
    where the last inequality above is a consequence of the following equivalences:
    \begin{align*}
        \frac{q_{t}}{q_{t-1}} + \frac{\sqrt{q_{t}}-\sqrt{q_{t-1}}}{\sqrt{q_{t-1}}}
        \leq \frac{q_{t}^{3/2}}{q_{t-1}^{3/2}}
        \iff&\
        \sqrt{q_{t-1}}q_{t} + q_{t-1}(\sqrt{q_{t}}-\sqrt{q_{t-1}})
        \leq q_{t}^{3/2},
        \\
        \iff&\
        q_{t-1}(\sqrt{q_{t}}-\sqrt{q_{t-1}})
        \leq q_{t}(\sqrt{q_{t}}-\sqrt{q_{t-1}}),
    \end{align*}
    which hold since \(q_{t}\geq q_{t-1}\).
    Therefore, whether \(\langle x^{\star}_{t},x^{y}_{t} \rangle\geq 0\) or \(\langle x^{\star}_{t},x^{y}_{t} \rangle< 0\), we have that
    \begin{align*}
        V^{\textup{ACC}}_{t+1}(s_{t})
        \leq 
        \frac{q_{t}^{2}}{q_{t-1}^{2}}
        V^{\textup{ACC}}_{t}(s_{t})
    \end{align*}
    for all \(s_{t}\), proving \eqref{ineq:sm-gc-Lyapacc-ascent}.
\end{proof}

\begin{theorem}\label{thm:sm-gc-Lyapacc-main}
    Let \(f\in\mathcal{S}(L,m)\) and let \(s_{t+1}\) be generated by \cref{alg:nag-free}.
    If \(m_{N} \leq m\), then for all \(t\geq N\) we have that
    \begin{align}
        f(y_{t+1})-f(x^{\star})
        \leq \frac{q_{t}^{2}}{q_{N}^{2}}\prod\limits_{i=N}^{t+1}(1+\hat{\delta}^{\textup{ACC}}_{i})^{-1}V^{\textup{ACC}}_{N}(s_{N}).
        \label{ineq:sm-gc-Lyapacc-main}
    \end{align}
\end{theorem}

\begin{proof}
    Since \(m_{t}\) is nonincreasing, if \(m_{N} \leq m\), then \(m_{t} \leq m\) for all \(t \geq N\).
    Therefore, if \(m_{N}\leq m\), then \cref{lem:sm-gc-Lyapacc-descent,lem:sm-gc-Lyapacc-ascent} hold for all \(t\geq N\).
    So, plugging \eqref{ineq:sm-gc-Lyapacc-ascent} into \eqref{ineq:sm-gc-Lyapacc-descent} and proceeding with a simple inductive argument, it follows that
    \begin{align*}
        f(y_{t+1})-f(x^{\star})
        \leq V^{\textup{ACC}}_{t+1}(s_{t+1})
        \leq \frac{q_{t}^{2}}{q_{N}^{2}}\prod\limits_{i=N}^{t+1}(1+\hat{\delta}^{\textup{ACC}}_{i})^{-1}V^{\textup{ACC}}_{N}(s_{N}),
    \end{align*}
    where the first inequality follows directly from \eqref{def:sm-gc-Lyapacc}, the definition of \(V^{\textup{ACC}}_{t}\), since
    \begin{align*}
        V^{\textup{ACC}}_{t+1}(s_{t+1})
        = \tilde{f}(y_{t+1}) 
        + (m_{t}/2)\Vert w^{\star}_{t+1} \Vert^{2}
        \geq \tilde{f}(y_{t+1}).
    \end{align*}
\end{proof}

\subsection*{The transition iteration from \(m_{t} \geq m\) to \(m_{t} < m\)}

We now have analyzed iterations of \cref{alg:nag-free} in which \(m_{t}\geq m\) and iterations in which \(m_{t} < m\).
To prove \cref{thm:gc}, it remains to join the two analyses, considering the transition from the first kind of iteration to the second kind.
Since \(m_{t}\) is nonincreasing, there can be at most one such transition.
We start by bounding \(V^{\textup{ACC}}_{t}\) in terms of \(V^{\textup{GD}}_{t}\).

\begin{lemma}\label[lemma]{lem:sm-gc-Lyapgd-2-Lyapacc-ascent}
    If \(f\in \mathcal{S}(L,m)\) and \(L_{t-1}\geq m_{t-1}\geq m>0\), then
    \begin{align}
        V^{\textup{ACC}}_{t} 
        \leq (m_{t-1}/m) V^{\textup{GD}}_{t}.
        \label{ineq:sm-gc-Lyapgd-2-Lyapacc-ascent}
    \end{align}
\end{lemma}

\begin{proof}
    To prove \eqref{ineq:sm-gc-Lyapgd-2-Lyapacc-ascent}, we split the analysis according to the sign of \(\langle x^{\star}_{t},x^{y}_{t} \rangle\) in
    \begin{align}
        V^{\textup{ACC}}_{t}(s_{t})-V^{\textup{GD}}_{t}(s_{t})
        =&\ \frac{m_{t-1}}{2}\Vert x^{\star}_{t} + \sqrt{q_{t-1}}x^{y}_{t} \Vert^{2}
        -\frac{m}{2}\Vert x^{\star}_{t} + \sqrt{p_{t-1}}x^{y}_{t} \Vert^{2}
        \nonumber\\
        =&\ \frac{m_{t-1}-m}{2}\Vert x^{\star}_{t} \Vert^{2}
        +2\frac{\sqrt{L_{t-1}}(\sqrt{m_{t-1}}-\sqrt{m})}{2}\langle x^{\star}_{t},x^{y}_{t} \rangle
        \label{id:sm-gc-Lyapgd-2-Lyapacc-ascent-V-gap}
    \end{align}
    in terms of \(V^{\textup{GD}}_{t}\) or \(V^{\textup{ACC}}_{t}\). 
    We consider the case \(\langle x^{\star}_{t},x^{y}_{t} \rangle \geq 0\) first.
    
    Multiplying the \(\langle x^{\star}_{t},x^{y}_{t} \rangle\) coefficient on \eqref{id:sm-gc-Lyapgd-2-Lyapacc-ascent-V-gap} by \((\sqrt{m_{t-1}}+\sqrt{m})/\sqrt{m_{t-1}} \geq 1\), we obtain
    \begin{align}
        \sqrt{L_{t-1}}(\sqrt{m_{t-1}}-\sqrt{m})
        \leq \sqrt{q_{t-1}}(m_{t-1}-m).
        \label{ineq:sm-gc-Lyapgd-2-Lyapacc-ascent-aux}
    \end{align}
    Hence, if \(\langle x^{\star}_{t},x^{y}_{t} \rangle \geq 0\), then plugging \eqref{ineq:sm-gc-Lyapgd-2-Lyapacc-ascent-aux} into \eqref{id:sm-gc-Lyapgd-2-Lyapacc-ascent-V-gap}, adding a nonnegative \(\Vert x^{y}_{t} \Vert^{2}\) term, completing a square to form a \(\Vert w^{\star}_{t} \Vert^{2}\) term and then applying \eqref{ineq:sm-gc-Lyapacc-ascent-lb2}, we get
    \begin{align*}
        V^{\textup{ACC}}_{t}(s_{t})-V^{\textup{GD}}_{t}(s_{t})
        &\leq \frac{m_{t-1}-m}{m_{t-1}}\frac{m_{t-1}}{2}(\Vert x^{\star}_{t} \Vert^{2}
        + 2\sqrt{q_{t-1}}\langle x^{\star}_{t},x^{y}_{t} \rangle
        + q_{t-1}\Vert x^{y}_{t} \Vert^{2})
        \\
        &\leq \frac{m_{t-1}-m}{m_{t-1}}V^{\textup{ACC}}_{t}(s_{t}).
    \end{align*}		
    Moving terms around and then multiplying both sides by \(m_{t-1}/m > 0\), we get
    \begin{align*}
        V^{\textup{ACC}}_{t}(s_{t}) \leq (m_{t-1}/m)V^{\textup{GD}}_{t}(s_{t}).
    \end{align*}
    Now, suppose \(\langle x^{\star}_{t},x^{y}_{t} \rangle<0\).
    In this case, we cannot increase the \(\langle x^{\star}_{t},x^{y}_{t} \rangle\) coefficient to complete a square as we did before. 
    Instead, we complete a square with the given \(\langle x^{\star}_{t},x^{y}_{t} \rangle\) coefficient by splitting the \(\Vert x^{\star}_{t} \Vert^{2}\) term using the following identity:
    \begin{align}
        \frac{m_{t-1}-m}{m}
        =\frac{\sqrt{m_{t-1}} - \sqrt{m}}{\sqrt{m}}
        \frac{\sqrt{m_{t-1}} + \sqrt{m}}{\sqrt{m}}
        =\frac{\sqrt{m_{t-1}} - \sqrt{m}}{\sqrt{m}}
        \Bigl( 1 + \frac{\sqrt{m_{t-1}}}{\sqrt{m}} \Bigr).
        \label{id:sm-gc-Lyapgd-2-Lyapacc-ascent-xastk-coeff}
    \end{align}
    To handle the \(\Vert x^{\star}_{t} \Vert^{2}\) term that stays out of the square, we use the fact that
    \begin{align}
        \Vert y^{\star}_{t} \Vert^{2}
        = \Vert y^{\star}_{t} \pm x^{\star}_{t} \Vert^{2}
        = \Vert x^{\star}_{t}-x^{y}_{t} \Vert^{2}
        = \Vert x^{\star}_{t} \Vert^{2} -2\langle x^{\star}_{t},x^{y}_{t} \rangle + \Vert x^{y}_{t} \Vert^{2}
        \geq \Vert x^{\star}_{t} \Vert^{2},
        \label{ineq:sm-gc-Lyapgd-2-Lyapcc-ascent-x-ub}
    \end{align}
    which follows since \(\langle x^{\star}_{t},x^{y}_{t} \rangle<0\).
    By the definition of \(\bar{\alpha}_{t-1}\), the assumption that \(m_{t-1}\geq m\) yields \(\bar{\alpha}_{t-1}\geq 0\), thus \(\bar{\alpha}_{t-1}/\sqrt{p_{t-1}}\geq 0\). 
    Moreover, \(U_{t}\geq 0\) because of the assumption that \(L_{t-1}>0\). Since \(\tilde{f}\) is also nonnegative, from \eqref{def:sm-gc-Lyapgd} we obtain
    \begin{align}
        V^{\textup{GD}}_{t}(s_{t})
        \geq W_{t}(s_{t})
        = \tilde{f}(y_{t})
        +(m/2)\Vert z^{\star}_{t} \Vert^{2}
        \geq (m/2)\max( \Vert z^{\star}_{t} \Vert^{2}, \Vert y^{\star}_{t} \Vert^{2} ),
        \label{ineq:sm-gc-Lyapgd-lb}
    \end{align}
    where the right-hand side follows from applying \eqref{ineq:strong_convexity} with \(x=x^{\star}\) and \(y=y_{t}\).
        
    Hence, splitting the coefficient of \(\Vert x^{\star}_{t} \Vert^{2}\) on \eqref{id:sm-gc-Lyapgd-2-Lyapacc-ascent-V-gap} according to \eqref{id:sm-gc-Lyapgd-2-Lyapacc-ascent-xastk-coeff}, adding a positive \(\Vert x^{y}_{t} \Vert^{2}\) term to form a \(\Vert z^{\star}_{t} \Vert^{2}\) term, applying \eqref{ineq:sm-gc-Lyapgd-2-Lyapcc-ascent-x-ub} and then using \eqref{ineq:sm-gc-Lyapgd-lb}, we obtain
    \begin{align*}
        V^{\textup{ACC}}_{t}(s_{t})-V^{\textup{GD}}_{t}(s_{t})
        =&\ \frac{\sqrt{m_{t-1}}-\sqrt{m}}{\sqrt{m}}\frac{m}{2}
        \Bigl( \Bigl( 1 + \frac{\sqrt{m_{t-1}}}{\sqrt{m}} \Bigr)\Vert x^{\star}_{t} \Vert^{2} + 2\sqrt{p_{t-1}}\langle x^{\star}_{t},x^{y}_{t} \rangle \Bigr)
        \\
        \leq&\ \frac{\sqrt{m_{t-1}}-\sqrt{m}}{\sqrt{m}}\frac{m}{2}\Vert z^{\star}_{t} \Vert^{2}
        + \frac{\sqrt{m_{t-1}}-\sqrt{m}}{\sqrt{m}}\frac{\sqrt{m_{t-1}}}{\sqrt{m}}\frac{m}{2}\Vert y^{\star}_{t} \Vert^{2}
        \\
        \leq&\ \frac{m_{t-1}-m}{m}V^{\textup{GD}}_{t}(s_{t}).
    \end{align*}    
    Finally, moving terms around, we get
    \begin{align*}
        V^{\textup{ACC}}_{t}(s_{t})
        \leq (m_{t-1}/m)V^{\textup{GD}}_{t}(s_{t}).
    \end{align*}
    Therefore, both when \(\langle x^{\star}_{t},x^{y}_{t} \rangle\geq 0\) and when \(\langle x^{\star}_{t},x^{y}_{t} \rangle<0\), the inequality
    \begin{align*}
        V^{\textup{ACC}}_{t}(s_{t})
        \leq&\ (m_{t-1}/m)V^{\textup{GD}}_{t}(s_{t})
    \end{align*}
    holds generically for all \(s_{t}\), and we recover \eqref{ineq:sm-gc-Lyapgd-2-Lyapacc-ascent}.
\end{proof}

Now, we are ready to prove \cref{thm:gc}, which combines \cref{thm:sm-gc-Lyapgd-main,thm:sm-gc-Lyapacc-main} into a single result that holds for all iterations.
First, we note that if \(L_{0}\geq m\), then the NAG-free initialization implies that \(m_{0}=L_{0}\geq m\).
Hence, since \(c_{t}\in \mathopen{[}m,L\mathclose{]}\) by \eqref{ineq:effective-curvature}, it follows that \(m_{t}\geq m/\gamma\) for all \(t\geq 0\).
If \(L_{0} \geq L\), then since \eqref{ineq:descent_lemma} holds for all \(\bar{L}\geq L\), it follows that \(L_{t} = L_{0}\) for all \(t\).
Otherwise, if \(L_{0} \leq m\), then since \(L_{t}\) is adjusted by a factor of \(\gamma_{L}\leq 2\) every time \eqref{ineq:descent_lemma} is violated, and \eqref{ineq:descent_lemma} holds for all \(\bar{L}\geq L\), it follows that \(L_{t}\leq 2L\) for all \(t\).
With that in mind, let \(\bar{L}=\max(L_{0},2L)\) and \(\bar{\kappa}=\bar{L}/m\).
Then, we have that
\begin{align}
    p_{t}
    = (L_{t}/m)
    \leq \bar{\kappa}
    &&
    \text{and}
    &&
    q_{t}
    = (L_{t}/m_{t})
    \leq \gamma\bar{\kappa}.
    \label{ineq:sm-gc-p-q-bounds}
\end{align}
Plugging \eqref{ineq:sm-gc-p-q-bounds} into the definitions of \(\delta^{\textup{GD}}_{t}\) and \(\hat{\delta}^{\textup{ACC}}_{t}\), we obtain
\begin{align}
    \delta^{\textup{GD}}_{t+1}
    =&\ 1/(p_{t}-1)
    \geq 1/(\bar{\kappa}-1)
    = \delta(\bar{\kappa}),
    \label{ineq:sm-gc-deltagd-ub}
    \\
    \hat{\delta}^{\textup{ACC}}_{t+1}
    =&\ 1/(\sqrt{q_{t}}-1)
    \geq 1/(\sqrt{}-1)
    = \delta(\sqrt{\gamma\bar{\kappa}}).
    \label{ineq:sm-gc-deltaacc-ub}
\end{align}
Moreover, by design \(m_{0} \geq L_{0}\) and \(m_{t}\) is nonincreasing, therefore \(m_{t} \leq L_{0}\) for all \(t\), and it follows from \eqref{ineq:sm-gc-p-q-bounds} that
\begin{align}
    \frac{q_{t}^{2}}{q_{j}^{2}}\frac{p_{j}^{2}}{p_{0}^{2}}
    = q_{t}^{2}\frac{m_{j}^{2}}{L_{j}^{2}}\frac{L_{j}^{2}}{L_{0}^{2}}
    = q_{t}^{2}\frac{m_{j}^{2}}{L_{0}^{2}}
    \leq q_{t}^{2}
    \leq \gamma^{2}\bar{\kappa}^{2}.
    \label{ineq:sm-gc-qp-prod-ub}
\end{align}

\begin{proof}[Proof of \cref{thm:gc}]
    By \eqref{ineq:sm-gc-deltagd-ub}, we have that
    \begin{align*}
        (1+\delta^{\textup{GD}}_{t+1})^{-1}
        \leq (1+\delta(\bar{\kappa}))^{-1}
        = (\bar{\kappa}-1)/\bar{\kappa}
    \end{align*}
    for all \(t\) such that \(m_{t} \geq m\).
    Hence, by \cref{thm:sm-gc-Lyapgd-main} and \eqref{ineq:sm-gc-p-q-bounds}, it follows that
    \begin{align*}
        f(y_{t})-f(x^{\star})
        \leq 
        \Bigl( 
        1 - \frac{1}{\bar{\kappa}} 
        \Bigr)^{t}
        2\max(L_{0},L)\bar{\kappa}^{2}\Vert x^{\star}_{0} \Vert^{2}
    \end{align*}
    for all \(t\) such that \(m_{t} \geq m\).
    If \(m_{t} \geq m\) for all \(t\), then \eqref{ineq:gc} holds for all \(t\).
    Otherwise, let \(N\) be the first iteration for which \(m_{N+1} \leq m\). 
    By simple manipulations, we have that
    \begin{align*}
        (\sqrt{\gamma\bar{\kappa}} - 1)/\sqrt{\gamma\bar{\kappa}}
        = (1+\delta(\sqrt{\gamma\bar{\kappa}}))^{-1}
        \leq
        (1+\delta(\bar{\kappa}))^{1}
        = (\bar{\kappa} - 1)/\bar{\kappa}
    \end{align*}
    if and only if \(\gamma \leq \bar{\kappa}\).
    Hence, since by assumption \(\bar{\kappa} \geq \kappa \geq 2 = \gamma\), from \eqref{ineq:sm-gc-p-q-bounds} it follows that
    \begin{align*}
        (1+\delta^{\textup{ACC}}_{t+1})^{-1}
        \leq
        (1+\delta(\sqrt{\gamma\bar{\kappa}}))^{-1}
        \leq
        (1+\delta(\bar{\kappa}))^{-1}
    \end{align*}
    for all \(t \geq N\).
    Combining \cref{lem:sm-gc-Lyapgd-2-Lyapacc-ascent} with \cref{thm:sm-gc-Lyapacc-main}, then applying \cref{thm:sm-gc-Lyapgd-main} and plugging in \eqref{ineq:sm-gc-qp-prod-ub} with \(\gamma\leq 2\) and \(m_{N}\leq L_{0}\) into the result, it follows that for all \(t \geq N\),
    \begin{align*}
        f(y_{t+1})-f(x^{\star})
        \leq&\ 
        \frac{q_{t}^{2}}{q_{N}^{2}}
        \Bigl( 
        1-\frac{1}{\bar{\kappa}} 
        \Bigr)^{t-N}
        \frac{m_{N}}{m}
        2\max(L_{0},L)\frac{p_{N}^{2}}{p_{0}^{2}}
        \Bigl( 
        1-\frac{1}{\bar{\kappa}} 
        \Bigr)^{N}
        \Vert x^{\star}_{0} \Vert^{2}
        \\
        \leq&\ 
        \Bigl( 
        1-\frac{1}{\bar{\kappa}} 
        \Bigr)^{t}
        8\max(L_{0},L)\bar{\kappa}^{3}
        \Vert x^{\star}_{0} \Vert^{2}.
    \end{align*}
    Therefore, \eqref{ineq:gc} holds for all \(t \geq 0\).
\end{proof}

\newpage
\clearpage
\section{Local Acceleration}
\label{app:local_acceleration}

In this section, we prove \cref{thm:la}, establishing that NAG-free (\cref{alg:nag-free}) converges at an accelerated rate to \(x^{\star}\), the minimum of \(f\in\mathcal{S}(L,m)\), when its iterates get sufficiently close to \(x^{\star}\).
\begin{align}
    \rnag(z)=\frac{\sqrt{z}-1}{\sqrt{z}}.
    \label{la:def:racc}
\end{align}

\begin{remark}[Deriving most results assuming \(L_{0}=L\).]
    To prove \cref{thm:la}, we assume that some \(L_{0}>L\) is known and can be used to initialize \cref{alg:nag-free}.
    If \(L_{0}>L\), then \(L_{0}\) also satisfies \eqref{ineq:descent_lemma}.
    In turn, if \cref{alg:nag-free} is initialized with \(L_{0}\), then the descent lemma condition given by \(f(y_{t+1}) -f(x_{t}) \leq -(1/2L_{t})\Vert \nabla f(x_{t}) \Vert^{2}\) is always satisfied.
    Hence, \(L_{t}\equiv L_{0}\) for all \(t\).
    Another consequence is that \(m_{0}=L_{0}\leq L\) for all \(t\).
    However, we derive most of the results in this section using \(L\) to avoid working with the cluttered notation \(L_{0}\), since essentially all of the results below hold for any \(L_{0} \geq L\).
    Once we prove the local acceleration results, we plug \(L_{0}>L\) back in.
\end{remark}

To prove \cref{thm:la}, we first consider the simplified case where \(f\) is quadratic, and then analyze the general case as a perturbation of the quadratic case.
To this end, we use the fact established by \cref{thm:gc} that the iterates of \cref{alg:nag-free} converge to \(x^{\star}\) at a rate no worse than that of gradient descent, regardless of the initial point \(x_{0}\).
By that we mean 
\begin{align*}
    f(y_{t}) - f(x^{\star})
    \leq 
    \rgd(\bar{\kappa})^{t}
    8\max(L_{0},L)\bar{\kappa}^{3}
    \Vert x^{\star}_{0} \Vert^{2},
\end{align*}
where \(\bar{\kappa}=\max(L_{0},2L)/m\) and \(\rgd\) is defined over \(\mathopen{[}1,+\infty\mathclose{)}\) as
\begin{align}
    \rgd(z)=\frac{z-1}{z}.
    \label{la:def:rgd}
\end{align}

\subsection{Quadratic Case}\label{la:sub:setup}

First, we assume the objective function is given by \(f(x)=(1/2)(x-x^{\star})^{\T}H(x-x^{\star})\), with \(H\in\mathbb{R}^{d\times d}\).
Every quadratic function \((1/2)x^{\T}Hx + x^{\T}g + f(0)\) can be expressed\footnote{Since \(H\) is strongly convex, \(H\) is invertible and the first-order condition \(Hx^{\star} + g = 0\) admits a unique solution \(x^{\star}\). Plugging \(x=x^{\star}\) back into \(f(x)\), and solving for \(f(0)\), we get that \(f(0)=-\frac{1}{2}x^{\star,\T}Hx^{\star}\). Then, plugging \(f(0)\) back into \(f(x)\) and replacing the inner-product \(g^{\T}x\) with \(g^{\T}x=-x^{\star,\T}Hx=-\frac{1}{2}x^{\star,\T}Hx-\frac{1}{2}x^{\T}Hx^{\star}\) yields the desired form of \(f(x)\).} in the form \((1/2)(x-x^{\star})^{\T}H(x-x^{\star}) + f(x^{\star})\), and minimizing the latter is equivalent to minimizing \((1/2)(x-x^{\star})^{\T}H(x-x^{\star})\).
Thus, \(\nabla f(x) = H(x-x^{\star})\).
Moreover, since \(f \in \mathcal{S}(L,m)\), \(H\) must be positive definite with all \(d\) eigenvalues \(\lambda_{i}\) inside \(\mathopen{[}m,L\mathclose{]}\). 
Hence, assuming \(\lambda_{i}\) ordered by their indices, we have that
\begin{align*}
    m =\lambda_{1} \leq \dots \leq \lambda_{d} = L.
\end{align*}
Since \(\nabla^{2}f\) is locally smooth at \(x^{\star}\), it is also continuous at \(x^{\star}\).
Hence, \(H=\nabla^{2}f(x^{\star})\) is real symmetric in general, not only in the case where \(f\) is quadratic.
Therefore, by the spectral theorem \cite{Conway2019} we can pick eigenvectors \(v_{i}\) associated with \(\lambda_{i}\) such that \(\{v_{i}\}_{i=1}^{d}\) form an orthonormal basis for \(\mathbb{R}^{d}\). 
Then, \(x_{t}-x^{\star}\) and \(y_{t+1}-x^{\star}\) can be uniquely decomposed in this eigenbasis as
\begin{align}
    x_{t} -x^{\star}
    &= \sum\limits_{i=1}^{d} x_{t,i}v_{i},
    \label{la:eq:eigendecomposition}
    \\
    y_{t+1} -x^{\star}
    &= x_{t}-\frac{1}{L}\nabla f(x_{t}) -x^{\star}
    = \sum\limits_{i=1}^{d}\Bigl( 1-\frac{\lambda_{i}}{L} \Bigr)x_{t,i}v_{i}.
    \label{la:id:GD-step}
\end{align}
Substituting \labelcref{la:id:GD-step} for the descent steps yields
\begin{align}
    \sum\limits_{i=1}^{d}x_{t+1,i}v_{i}
    &= x_{t+1}-x^{\star}
    \nonumber\\
    &= (1+\beta_{t})y_{t+1} -\beta_{t}y_{t} -x^{\star} \mp\beta_{t}x^{\star}
    \nonumber\\
    &= (1+\beta_{t})(y_{t+1}-x^{\star}) -\beta_{t}(y_{t}-x^{\star})
    \nonumber\\
    &= \sum\limits_{i=1}^{d}\left[ (1+\beta_{t})\Bigl( 1 -\frac{\lambda_{i}}{L} \Bigr) x_{t,i} -\beta_{t}\Bigl(1 -\frac{\lambda_{i}}{L}\Bigr)x_{t-1,i} \right]v_{i},
    \label{la:eq:momentum-step}
\end{align}
where \(\beta_{t}=\beta(m_{t})\) is a particular value taken by the function \(\beta:\mathopen{(}0,L\mathclose{]}\to \mathopen{[}0,1\mathclose{)}\) defined by
\begin{align}
    \beta(m_{t})
    =\frac{\sqrt{L}-\sqrt{m_{t}}}{\sqrt{L}+\sqrt{m_{t}}}.
    \label{la:def:momentum-map}
\end{align}

That is, each component \(x_{t,i}\) of \(x_{t}-x^{\star}\) behaves as an LTV system \cite{Hespanha2009}. 
But if \(\gamma>1\), then by design \(m_{t}\) decreases by a factor of at least \(\gamma\) every time it is updated, which implies \(m_{t}\) only changes finitely many times.
Hence, each \(x_{t,i}\) behaves as a sequence of linear time-invariant (LTI) systems described by
\begin{align}
    X_{t+1,i}
    = G_{i}(m_{t})
    X_{t,i},
    \label{la:eq:Xit-dynamics}
\end{align}
where \(X_{t,i}\) denote the vectors of current and past coordinates stacked together as in
\begin{align}
    X_{t,i} 
    =\begin{cases}
        \left[ x_{0,i} \quad x_{0,i} \right]^{\T}, & t=0,
        \\[2mm]
        \left[ x_{t-1,i} \quad x_{t,i} \right]^{\T}, & t>0,
    \end{cases}
    \label{la:eq:Xit}
\end{align}
and \(G_{i}:\mathopen{(}0,L\mathclose{]}\to\mathbb{R}^{2\times 2}\) map estimates \(m_{t}\) to system matrices given by
\begin{align}
    G_{i}(m_{t}) 
    = 
    \begin{bmatrix} 
        0 & 1 \\ 
        -\beta(m_{t})\Bigl( 1-\dfrac{\lambda_{i}}{L} \Bigr) & (1+\beta(m_{t}))\Bigl( 1-\dfrac{\lambda_{i}}{L} \Bigr) 
    \end{bmatrix}.
    \label{la:def:Git}
\end{align}

Hence, the dynamics of \eqref{la:eq:Xit-dynamics} is determined by the eigenvalues of \(G_{i}(m_{t})\), which are given by
\begin{align}
    \lambda(G_{i}(m_{t}))
    =
    \frac{1+\beta(m_{t})}{2}\Bigl( 1-\dfrac{\lambda_{i}}{L} \Bigr)
    \pm \sqrt{\dfrac{(1+\beta(m_{t}))^{2}}{4}\Bigl( 1-\dfrac{\lambda_{i}}{L} \Bigr)^{2} - \beta(m_{t})\Bigl( 1-\dfrac{\lambda_{i}}{L} \Bigr)}.
    \label{la:id:eig}
\end{align}
The greatest between the two eigenvalues given by \eqref{la:id:eig} defines the so-called spectral radius \cite{Golub2013} of \(G_{i}\), captured by the function \(\rho:\mathopen{(}0,L\mathclose{]}\times \mathbb{R}_{\geq 0}\to\mathbb{R}_{\geq 0}\) defined by
\begin{align}
    \rho(s,\ell)
    =\max
    \left\lvert 
    \frac{1+\beta(m_{t})}{2}\Bigl( 1-\dfrac{\ell}{L} \Bigr)
    \pm\sqrt{\dfrac{(1+\beta(m_{t}))^{2}}{4}\Bigl( 1-\dfrac{\ell}{L} \Bigr)^{2} - \beta(m_{t})\Bigl( 1-\dfrac{\ell}{L} \Bigr)}
    \right\rvert.
    \label{la:def:rho}
\end{align}
We also define a function \(\varrho:\mathopen{(}0,L\mathclose{]}\times \mathbb{R}_{\geq 0}\to\mathbb{R}_{\geq 0}\) for the least of the two eigenvalues:
\begin{align}
    \varrho(s,\ell)
    =\min
    \left\lvert 
    \frac{1+\beta(m_{t})}{2}\Bigl( 1-\dfrac{\ell}{L} \Bigr)
    \pm\sqrt{\dfrac{(1+\beta(m_{t}))^{2}}{4}\Bigl( 1-\dfrac{\ell}{L} \Bigr)^{2} - \beta(m_{t})\Bigl( 1-\dfrac{\ell}{L} \Bigr)}
    \right\rvert.
    \label{la:def:varrho}
\end{align}
Note that \(\rho\) and \(\varrho\) take an argument ``\(\ell\)'' that need not be an actual eigenvalue \(\lambda_{i}\) of \(H\), which will be convenient later on. 
Next, we derive several auxiliary results on \(\rho\) and \(\varrho\) that will be useful later on.

\subsection*{Properties of the Spectral Radius \texorpdfstring{\(\rho\)}{rho}}

\begin{lemma}\label{la:lem:rho-expression}
    Let \(s,\ell\in\mathopen{(}0,L\mathclose{]}\). 
    The two numbers
    \begin{align}
        \frac{1+\beta(s)}{2}\Bigl( 1-\dfrac{\ell}{L} \Bigr)
        \pm\sqrt{\dfrac{(1+\beta(s))^{2}}{4}\Bigl( 1-\dfrac{\ell}{L} \Bigr)^{2} - \beta(s)\Bigl( 1-\dfrac{\ell}{L} \Bigr)}
        \label{la:def:rho-expression-eigs}
    \end{align}
    have nonzero imaginary part if and only if \(s < \ell < L\). If \eqref{la:def:rho-expression-eigs} have zero imaginary part, then
    \begin{align}
        \rho(s,\ell)
        = \frac{1+\beta(s)}{2}\Bigl( 1-\dfrac{\ell}{L} \Bigr)
        +\sqrt{\dfrac{(1+\beta(s))^{2}}{4}\Bigl( 1 - \dfrac{\ell}{L} \Bigr)^{2} - \beta(s)\Bigl( 1-\dfrac{\ell}{L} \Bigr)},
        \label{la:id:rho-expression-real}
    \end{align}
    otherwise, if \eqref{la:def:rho-expression-eigs} have nonzero imaginary part, then
    \begin{align}
        \rho(s,\ell) = \sqrt{\beta(s)\Bigl( 1-\dfrac{\ell}{L} \Bigr)}.
        \label{la:id:rho-expression-complex}
    \end{align}
\end{lemma}

\begin{proof}
    Let \(r_{+}\) be defined by
    \begin{align*}
        r_{+}
        = \frac{1+\beta(s)}{2}\Bigl( 1-\dfrac{\ell}{L} \Bigr)
        +\sqrt{\dfrac{(1+\beta(s))^{2}}{4}\Bigl( 1-\dfrac{\ell}{L} \Bigr)^{2} - \beta(s)\Bigl( 1-\dfrac{\ell}{L} \Bigr)},
    \end{align*}
    and let \(r_{-}\) be defined by
    \begin{align*}
        r_{-}
        = \frac{1+\beta(s)}{2}\Bigl( 1-\dfrac{\ell}{L} \Bigr)
        -\sqrt{\dfrac{(1+\beta(s))^{2}}{4}\Bigl( 1-\dfrac{\ell}{L} \Bigr)^{2} - \beta(s)\Bigl( 1-\dfrac{\ell}{L} \Bigr)}.
    \end{align*}
    Also, let \(\Delta\) be defined by
    \begin{align}
        \Delta(s,\ell)
        = \dfrac{(1+\beta(s))^{2}}{4}\Bigl( 1-\dfrac{\ell}{L} \Bigr)^{2} - \beta(s)\Bigl( 1-\dfrac{\ell}{L} \Bigr).
        \label{la:def:Delta}
    \end{align}
    
    If \(\ell=0\), then \(\ell\leq s\) since \(s\geq 0\), because \(s\in \mathopen{(}0,L\mathclose{]}\). 
    Moreover, plugging \(\ell=0\) into \eqref{la:def:Delta}, we obtain
    \begin{align*}
        \Delta(s,\ell) = \frac{(1+\beta(s))^{2}}{4} - \beta(s) < 0
        \iff (1-\beta(s))^{2} = (1+\beta(s))^{2} -4\beta(s) < 0.
    \end{align*}
    Hence, \(\Delta(s,\ell)\geq 0\) because \((1-\beta)^{2}\geq 0\). 
    Furthermore, \(1-\ell/L=1\) and \(\rho(s,\ell)\) trivially reduces to form \eqref{la:id:rho-expression-real}.
    
    Now, suppose \(\ell>0\). 
    If \(\ell=L\), then \(1-\ell/L=0\) and \(\rho(s,\ell)=0\) trivially has zero imaginary part and takes the form \eqref{la:id:rho-expression-real}.
    Otherwise, if \(\ell < L\), then \(1 - \ell/L > 0\) and \(\Delta < 0\) if and only if
    \begin{align}
        (1+\beta)^{2}\Bigl( 1-\frac{\ell}{L} \Bigr) - 4\beta < 0
        \iff (1-\beta)^{2}L < (1+\beta)^{2}\ell
        \iff \frac{L}{\ell} < \left( \frac{1+\beta}{1-\beta} \right)^{2},
        \label{la:iff:Delta=0}
    \end{align}
    where \(L/\ell\) is well-defined since \(\ell>0\), by assumption, while \((1-\beta)^{-1}\) is well-defined because \(0\leq \beta(s)<1\) for all \(s\in\mathopen{(}0,L\mathclose{]}\). 
    Plugging \eqref{la:def:momentum-map} into \(\beta\), the squared factor on the right-hand side of \eqref{la:iff:Delta=0} turns into
    \begin{align}
        \frac{1+\beta}{1-\beta}
        = \frac{2\sqrt{L}/(\sqrt{L} + \sqrt{s})}{2\sqrt{s}/(\sqrt{L}+\sqrt{s})}
        = \sqrt{L/s}.
        \label{la:id:ratio-beta}
    \end{align}
    Thus, by \eqref{la:iff:Delta=0}, \(\Delta(s,\ell)\) is negative if and only if \( s < \ell\). 
    Hence, if \(s \geq \ell\), then \(\Delta \geq 0\), which combined with the the assumption that \(L > \ell\) implies
    \begin{align*}
        1-\dfrac{\ell}{L}
        = \left\lvert 1-\dfrac{\ell}{L} \right\rvert
        >0,
    \end{align*}
    so that
    \begin{align*}
        \frac{1+\beta}{2}\Bigl( 1-\dfrac{\ell}{L} \Bigr)
        \geq \sqrt{\frac{(1+\beta)^{2}}{4}\Bigl( 1 - \dfrac{\ell}{L} \Bigr)^{2} -\beta\Bigl( 1-\dfrac{\ell}{L} \Bigr)}
        =\sqrt{\Delta}.
    \end{align*}
    Plugging the above inequality back into \(r_{+}\) and \(r_{-}\), we obtain
    \begin{align*}
        \lvert r_{+} \rvert
        &= r_{+}
        \\
        &= \frac{1+\beta}{2}\Bigl( 1-\dfrac{\ell}{L} \Bigr)
        +\sqrt{\Delta}
        \\
        &\geq \frac{1+\beta}{2}\Bigl( 1- \dfrac{\ell}{L}\Bigr) 
        -\sqrt{\Delta}
        \\
        &= \left\lvert \frac{1+\beta}{2}\Bigl( 1-\dfrac{\ell}{L} \Bigr) -\sqrt{\Delta} \right\rvert
        \\
        &= \lvert r_{-} \rvert.
    \end{align*}
    That is, \(\rho(s,\ell)\) takes the form \eqref{la:id:rho-expression-real}.
    
    Finally, if \(s< \ell\), then \(\Delta(s,\ell)\) is negative, so \(r_{+}\) and \(r_{-}\) are complex conjugates with the same norm given by
    \begin{align*}
        \lvert r_{+} \rvert
        = \sqrt{\frac{1+\beta(s)^{2}}{4}\Bigl( 1-\dfrac{\ell}{L} \Bigr)^{2}
        + \beta(s)\Bigl( 1-\dfrac{\ell}{L} \Bigr) -\dfrac{(1+\beta(s))^{2}}{4}\Bigl( 1-\dfrac{\ell}{L} \Bigr)^{2}}
        = \sqrt{\beta(s)\Bigl( 1-\dfrac{\ell}{L} \Bigr)}.
    \end{align*}
    Therefore, \(\rho(s,\ell)\) takes the form \eqref{la:id:rho-expression-complex}.
\end{proof}

\begin{corollary}\label[corollary]{la:cor:Gi-eig}
    If \(m_{t} \in \mathopen{(}0,L\mathclose{]}\), then the eigenvalues of \(G_{i}(m_{t})\) have nonzero imaginary part if and only if \(m_{t}<\lambda_{i}<L\). 
    Moreover, if \(\lambda_{i}<L\), then the eigenvalues of \(G_{i}(m_{t})\) coincide if and only if \(m_{t}=\lambda_{i}\).
    Furthermore, if \(\lambda_{i} < m_{t}\), then the eigenvalues of \(G_{i}(m_{t})\) are positive and distinct.
\end{corollary}

\begin{proof}
    Plugging \(s=m_{t}\) and \(\ell=\lambda_{i}\) into \eqref{la:def:rho-expression-eigs}, we recover the two eigenvalues of \(G_{i}(m_{t})\) which, by \cref{la:lem:rho-expression}, have nonzero imaginary part if and only if \(m_{t} < \lambda_{i} < L\). 

    Moreover, the eigenvalues of \(G_{i}(m_{t})\) coincide if and only if the discriminant \eqref{la:def:Delta} is zero for \(\ell=\lambda_{i}\) and \(s=m_{t}\).
    In turn, by \eqref{la:iff:Delta=0} and \eqref{la:id:ratio-beta}, the discriminant \eqref{la:def:Delta} is zero for \(\ell=\lambda_{i}\) and \(s=m_{t}\) if and only if \(m_{t}=\lambda_{i}\).

    Furthermore, if \(\lambda_{i} < m_{t}\), then for all \(\lambda_{i}\in\mathopen{(}0,L\mathclose{]}\), we have that
    \begin{align*}
        \frac{1+\beta}{2}\Bigl( 1 - \frac{\lambda_{i}}{L} \Bigr)
        \geq \sqrt{\Delta(m_{t},\lambda_{i})}
        > 0.
    \end{align*}
    Therefore, the eigenvalues of \(G_{i}(m_{t})\) are positive and distinct.
\end{proof}

\begin{lemma}\label{la:lem:top-rho}
    Given \(a\) and \(b\) such that \(0\leq a<b\leq L\), then \(\rho(s,b) < \rho(s,a)\) for all \(s \in \mathopen{(}0,L\mathclose{]}\).
    In particular, if \(b \in \mathopen{(}m,L\mathclose{]}\), then \(\rho(s,b) < \rho(s,m)\) for all \(s \in \mathopen{(}0,L\mathclose{]}\).
\end{lemma}

\begin{proof}
    Consider the following two cases:
    
    \hypertarget{case1-la:lem:top-spectral-radii}{\textbf{case 1 (\(b \leq s \)).}} 
    By assumption, \(s \in \mathopen{(}0,L\mathclose{]}\), hence \(s \leq L\) and if \(b \leq s\), then \(1 - b/L \geq 0\). 
    Moreover, \(a < b \leq s\), so \cref{la:lem:rho-expression} implies \(\rho(s,a)\) and \(\rho(s,b)\) both take form \eqref{la:id:rho-expression-real}.
    If, in addition \(s=L\), then \(\beta=0\), which when substituted back into \eqref{la:id:rho-expression-real} yields
    \begin{align*}
        \rho(s,b) 
        = 1 - b/L
        < 1 - a/L
        = \rho(s,a).
    \end{align*}
    Otherwise, if \(s<L\), then \(\beta>0\). Moreover, \(a < b \leq s\), so that \(b-a>0\), therefore
    \begin{align*}
        \Delta(s,b)
        < \Delta(s,a)
        \iff&&\ (1+\beta)^{2}\frac{b^{2}-a^{2}}{L}
        &< 2((1+\beta)^{2}-2\beta)(b-a) 
        \\
        \iff&&\ (1+\beta)^{2}\frac{b+a}{L}
        &< 2(1+\beta^{2}) 
        \\
        \iff&&\ \frac{4(L/s)}{(\sqrt{L/s}+1)^{2}}\frac{b+a}{L}
        &< 2(1+\beta^{2}),
    \end{align*}
    where the last equivalence follows at once from \eqref{la:def:momentum-map}.
    Furthermore, \(a < b \leq s < L\), thus \(\sqrt{L/s}+1> 2\) and
    \begin{align*}
        \frac{4L/s}{(\sqrt{L/s}+1)^{2}}\frac{b+a}{L}
        =\frac{4}{(\sqrt{L/s}+1)^{2}}\frac{b+a}{s}
        \leq \frac{8}{(\sqrt{L/s}+1)^{2}}
        < 2(1+\beta^{2}).
    \end{align*}
    Thus, \(\Delta(s,b) < \Delta(s,a)\). 
    Hence, since \(\rho(s,a)\) and \(\rho(s,b)\) are given by \eqref{la:id:rho-expression-real} and \(1-b/L < 1-a/L\), it follows that
    \begin{align*}
        \rho(s,b) 
        = \frac{1+\beta}{2}\Bigl( 1 - \dfrac{b}{L}\Bigr)
        + \sqrt{\Delta(s,b)}
        < \frac{1+\beta}{2}\Bigl( 1 - \dfrac{a}{L}\Bigr)
        + \sqrt{\Delta(s,a)}
        = \rho(s,a).
    \end{align*}
    
    \hypertarget{case2-la:lem:top-spectral-radii}{\textbf{case 2 (\(s < b \)).}} 
    By assumption \(b \leq L\), so \(a < b \leq L\) and it follows that
    \begin{align*}
        \dfrac{(1+\beta)^{2}}{4}\Bigl( 1 - \dfrac{a}{L} \Bigr)^{2} - \beta\Bigl( 1 - \dfrac{b}{L}\Bigr)
        >
        \dfrac{(1+\beta)^{2}}{4}\Bigl( 1 - \dfrac{a}{L} \Bigr)^{2} - \beta\Bigl( 1 - \dfrac{a}{L}\Bigr)
        \geq 0,
    \end{align*}
    that is
    \begin{align*}
        0\leq \beta\Bigl( 1 - \dfrac{b}{L}\Bigr) < \dfrac{(1+\beta)^{2}}{4}\Bigl( 1 - \dfrac{a}{L} \Bigr)^{2}.
    \end{align*}
    If, in addition \(b=L\), then \(\rho(s,b)=0\) the above inequality implies \(\rho(s,b)<\rho(s,a)\).
    Otherwise, it must be that \(s < b\), in which case \(\rho(s,b)\) takes the form \eqref{la:id:rho-expression-complex} by \cref{la:lem:rho-expression} and the above inequality yields
    \begin{align*}
        \rho(s,b)
        = \sqrt{\beta\Bigl( 1 - \dfrac{b}{L}\Bigr)}
        < \dfrac{1+\beta}{2}\Bigl( 1 - \dfrac{a}{L} \Bigr)
        \leq \rho(s,a).
    \end{align*}
\end{proof}

\begin{lemma}\label{la:lem:rho-lt-1}
    For every \(s\in \mathopen{(}0,L\mathclose{]}\) and every \(\ell \in \mathopen{[}m,L\mathclose{]}\), \(\rho(s,\ell) \leq \rgd(\kappa) < 1\).
\end{lemma}

\begin{proof}
    Let \(s\in \mathopen{(}0,L\mathclose{]}\) and \(\ell \in \mathopen{[}m,L\mathclose{]}\).
    By \cref{la:lem:top-rho}, \(\rho(s,\ell) \leq \rho(s,m)\), so it suffices to show \(\rho(s,m) \leq \rgd(\kappa)\).
    If \(m < m_{t}\), then by \cref{la:lem:rho-expression}, the eigenvalues of \(G_{1}(s)\) have zero imaginary part and, omitting the argument \(s\) in \(\beta=\beta(s)\), \(\rho(s,m)\) is given by
    \begin{align*}
        \rho(s,m) = \frac{1+\beta}{2}\Bigl( 1 - \frac{m}{L} \Bigr) + \sqrt{\frac{(1+\beta)^{2}}{4}\Bigl(1 - \frac{m}{L}\Bigr)^{2} - \beta\Bigl( 1 - \frac{m}{L} \Bigr)}.
    \end{align*}
    Hence, after simple manipulations, we obtain the equivalences
    \begin{align*}
        \rho(s,m) \leq \rgd(\kappa)
        \iff&& 
        \sqrt{\frac{(1+\beta)^{2}}{4}\Bigl(1 - \frac{1}{\kappa}\Bigr)^{2} - \beta\Bigl(1 - \frac{1}{\kappa}\Bigr)} 
        &\leq \frac{1-\beta}{2}\frac{\kappa-1}{\kappa}
        \\
        \iff&& \frac{(1+\beta)^{2}}{4}\Bigl(\frac{\kappa-1}{\kappa}\Bigr)^{2}
        &\leq \frac{(1-\beta)^{2}}{4}\Bigl(\frac{\kappa-1}{\kappa}\Bigr)^{2} + \beta\frac{\kappa-1}{\kappa}.
    \end{align*}
    Since \((1+\beta)^{2} = (1-\beta)^{2} + 4\beta\), \(\beta \geq 0\) and \((\kappa-1) < \kappa\), it follows that
    \begin{align*}
        \frac{(1+\beta)^{2}}{4}\Bigl(\frac{\kappa-1}{\kappa}\Bigr)^{2}
        = \frac{(1-\beta)^{2} + 4\beta}{4}\Bigl(\frac{\kappa-1}{\kappa}\Bigr)^{2}
        \leq \frac{(1-\beta)^{2}}{4}\Bigl(\frac{\kappa-1}{\kappa}\Bigr)^{2} + \beta\frac{\kappa-1}{\kappa}.
    \end{align*}
    Therefore, \(\rho(s,m)\leq \rgd(\kappa)\).
    Otherwise, if \(s\leq m\), then by \cref{la:lem:rho-expression} the eigenvalues of \(G_{1}(s)\) are complex, so that
    \begin{align*}
        \rho(s,m) = \sqrt{\beta\Bigl( 1 - \frac{m}{L} \Bigr)}.
    \end{align*}
    Hence, after simple manipulations, we obtain the equivalences
    \begin{align*}
        \rho(s,m) \leq \rgd(\kappa)
        \iff
        \beta\frac{\kappa-1}{\kappa}
        \leq \Bigl( \frac{\kappa-1}{\kappa} \Bigr)^{2}
        \iff
        \frac{\sqrt{\kappa}-1}{\sqrt{\kappa}+1}
        \leq \frac{\sqrt{\kappa}-1}{\sqrt{\kappa}}\frac{\sqrt{\kappa}+1}{\sqrt{\kappa}}.
    \end{align*}
    Since the right-hand side inequality above holds, so does \(\rho(s,m)\leq \rgd(\kappa)\) and we are done.
\end{proof}

\begin{lemma}\label{la:lem:Gi-eigenvalue-separation}
    If \cref{ass:mt-separation-from-eigs} holds, then \(\vert \zeta_{i} -\xi_{i} \vert \geq \sqrt{\delta_{L}\delta_{\lambda}}\), where \(\zeta_{i}=\zeta_{i}(m_{t})\) and \(\xi_{i}=\xi_{i}(m_{t})\) denote the eigenvalues of \(G_{i}(m_{t})\) and \(\delta_{L}=(L_{0}-L)/L_{0}\).
\end{lemma}

\begin{proof}
    If \cref{ass:known-lipschitz-upper-bound,ass:mt-separation-from-eigs} hold, then there exists some \(\delta_{\lambda}>0\) such that \(\vert m_{t} - \lambda_{i} \vert \geq \delta_{\lambda}\).
    Moreover, since \(L_{0}>L\), we have that \(\delta_{L}=(L_{0}-L)/L_{0}\).
    Moreover, whether \(\zeta_{i}=\zeta_{i}(m_{t})\) and \(\xi_{i}=\xi_{i}(m_{t})\) are complex or real, we have that
    \begin{align*}
        \vert \zeta_{i} -\xi_{i} \vert
        = 2\biggl\vert 
        \frac{(1+\beta)^{2}}{4}\frac{(L_{0}-\lambda_{i})^{2}}{L_{0}^{2}} 
        -\beta\frac{L_{0}-\lambda_{i}}{L_{0}}
        \biggr\vert^{1/2}
        = \frac{1+\beta}{L_{0}}\vert (L_{0}-\lambda_{i})(m_{t}-\lambda_{i}) \vert^{1/2}
        \geq \sqrt{\delta_{L}\delta_{\lambda}},
    \end{align*}
    where in the last equality we have replaced \(L\) with \(L_{0}\) in the identity
    \begin{align}
        \frac{4\beta L}{(1+\beta)^{2}}
        = 4\frac{\sqrt{L} - \sqrt{s}}{\sqrt{L} + \sqrt{s}}\frac{(\sqrt{L} + \sqrt{s})^{2}}{4L}L
        = L - s.
        \label{la:id:beta-ratio}
    \end{align}
\end{proof}

\subsection*{Sufficiently Accurate \texorpdfstring{\(m_{t}\)}{mt} Estimates}

In this section, we determine how good the estimate \(m_{t}\) must be for \(x_{t}\) to converge to \(x^{\star}\) at an accelerated rate.
From \cref{la:lem:top-rho}, it follows that \(\rho(m_{t},m)\) dominates the convergence of \(x_{t}\), therefore our goal is to characterize \(\sigma=\sigma(m_{t})\) such that \(\rho(m_{t},m)\leq \rnag(\sigma \kappa)\), where \(\sigma\) represents a suboptimality factor relative to the optimal convergence rate of \(\rnag(\kappa)\).

By \cref{thm:sm-gc-Lyapacc-main}, if \(m_{t} < m\), then the iterates converge at an accelerated rate.
So, in this section, we focus on \(m_{t} \in \mathopen{[}m,(1+\delta_{m})m\mathclose{]}\), where \(\delta_{m}>0\) is a small number.
We proceed in two steps.
First, we bound \(\rho(m_{t},m)\) for \(m_{t} \in \mathopen{[}m,(1+\delta)m\mathclose{]}\) in terms of a rate \(r_{\delta}\) that depends on the relative precision \(\delta>0\) and the condition number \(\kappa\).
Second, given some \(\sigma > 0\), we characterize \(\delta_{\sigma}\) for which \(r_{\delta}(\kappa)\leq \rnag(\sigma\kappa)\) holds for all \(\delta \in \mathopen{(}0,\delta_{\sigma}\mathclose{]}\) and \(\kappa \geq 1 + \delta\).
The rate \(r_{\delta}\) is parameterized by \(\delta\in\mathopen{(}0,1\mathclose{)}\) and defined over \(z\geq 1+\delta\) as
\begin{align}
    r_{\delta}(z)
    = \frac{1+\beta_{\delta}(z)}{2}\dfrac{z-1}{z}
    + \sqrt{\dfrac{(1+\beta_{\delta}(z))^{2}}{4}\Bigl( \dfrac{z-1}{z} \Bigr)^{2} - \beta_{\delta}(z)\dfrac{z-1}{z}},
\end{align}
where \(\beta_{\delta}\) is also defined over \(z\geq 1+\delta\) as
\begin{align}
    \beta_{\delta}(z)
    = \frac{\sqrt{z}-\sqrt{1+\delta}}{\sqrt{z}+\sqrt{1+\delta}}.
\end{align}

\begin{lemma}\label{la:lem:rho-leq-rhoeps}
    If \(m\leq s\leq (1+\delta)m \leq L\), then \(\rho(s,m) \leq r_{\delta}(\kappa)\) for all \(\kappa \geq 1 + \delta\).
\end{lemma}

\begin{proof}
    Let \(m\leq s\leq (1+\delta)m \leq L\).
    Since \(m \leq s \leq L\), then by \cref{la:lem:rho-expression}
    \begin{align*}
        \rho(s,m)
        = \frac{1+\beta(L,s)}{2}\Bigl( 1-\dfrac{m}{L} \Bigr)
        +\sqrt{\dfrac{(1+\beta(L,s))^{2}}{4}\Bigl( 1 - \dfrac{m}{L} \Bigr)^{2} - \beta(L,s)\Bigl( 1-\dfrac{m}{L} \Bigr)}.
    \end{align*}
    Omitting the arguments in \(\beta=\beta(L,s)\) and using the identity \eqref{la:id:beta-ratio}, the discriminant above can be expressed as
    \begin{align*}
        \frac{(1+\beta)^{2}}{4}\Bigl( 1-\dfrac{m}{L} \Bigr)^{2} - \beta\Bigl( 1-\dfrac{m}{L} \Bigr)
        = \frac{4L(L-m)(s-m)}{4L^{2}(\sqrt{L}+\sqrt{s})^{2}}
        =&\ \frac{(L-m)(s-m)}{L(\sqrt{L}+\sqrt{s})^{2}}.
    \end{align*}
    Plugging the above expression back into \(\rho(s,m)\), we obtain
    \begin{align*}
        \rho(s,m)
        = \frac{\sqrt{L}}{\sqrt{L}+\sqrt{s}}\frac{L-m}{L} + \frac{\sqrt{L-m}}{\sqrt{L}}\frac{\sqrt{s-m}}{\sqrt{L}+\sqrt{s}}
        = \frac{\sqrt{L-m}}{\sqrt{L}}\frac{\sqrt{L-m}+\sqrt{s-m}}{\sqrt{L}+\sqrt{s}}.
    \end{align*}
    The right-hand side above is increasing in \(s\geq m\) since \(Ls > (L-m)(s-m)\), which implies that
    \begin{align*}
        \frac{\partial}{\partial s}\frac{\sqrt{L-m}+\sqrt{s-m}}{\sqrt{L}+\sqrt{s}}
        &= \frac{1}{2\sqrt{s-m}}\frac{1}{\sqrt{L}+\sqrt{s}}-\frac{\sqrt{L-m}+\sqrt{s-m}}{2\sqrt{s}(\sqrt{L}+\sqrt{s})^{2}}
        \\
        &= \frac{m+\sqrt{Ls}- \sqrt{(L-m)(s-m)}}{2\sqrt{s}\sqrt{s-m}(\sqrt{L}+\sqrt{s})^{2}}
        \\
        &>0.
    \end{align*}
    Therefore, for all \(m\leq s \leq (1+\delta)m\) and \(\kappa \geq 1+\delta\), we have that
    \begin{align*}
        \rho(s,m)
        \leq \rho((1+\delta)m,m)
        = r_{\delta}(\kappa).
    \end{align*}
\end{proof}

Next, we bound \(r_{\delta}\) in terms of \(\rnag\).
We start with an identity involving \(\beta_{\delta}(\kappa)\), analogous to \eqref{la:id:beta-ratio}:
\begin{align*}
    4\frac{\beta_{\delta}(\kappa)}{(1+\beta_{\delta}(\kappa))^{2}}
    =4\frac{\sqrt{\kappa}-\sqrt{1+\delta}}{\sqrt{\kappa}+\sqrt{1+\delta}}
    \frac{(\sqrt{\kappa}+\sqrt{1+\delta})^{2}}{4\kappa}
    =\frac{\kappa-(1+\delta)}{\kappa}.
\end{align*}
Plugging the above identity into the discriminant of \(r_{\delta}(\kappa)\) yields
\begin{align*}
    \frac{(1+\beta_{\delta}(\kappa))^{2}}{4}\Bigl( \frac{\kappa-1}{\kappa} \Bigr)^{2}
    -\beta_{\delta}(\kappa)\frac{\kappa-1}{\kappa}
    &= \frac{\kappa}{(\sqrt{\kappa}+\sqrt{1+\delta})^{2}}
    \frac{\kappa-1}{\kappa}\Bigl( \frac{\kappa-1}{\kappa}-\frac{\kappa-(1+\delta)}{\kappa} \Bigr)
    \\
    &= \frac{\kappa-1}{(\sqrt{\kappa}+\sqrt{1+\delta})^{2}}
    \frac{\delta}{\kappa}.
\end{align*}
In turn, plugging the above expression for the discriminant back into \(r_{\delta}(\kappa)\), we obtain an alternative expression for \(r_{\delta}(\kappa)\):
\begin{align}
    r_{\delta}(\kappa)
    =\frac{\sqrt{\kappa}}{\sqrt{\kappa}+\sqrt{1+\delta}}\frac{\kappa-1}{\kappa}
    +\frac{\sqrt{\kappa-1}}{\sqrt{\kappa}+\sqrt{1+\delta}}\frac{\sqrt{\delta}}{\sqrt{\kappa}}
    =\frac{\sqrt{\kappa-1}}{\sqrt{\kappa}}
    \frac{\sqrt{\kappa-1}+\sqrt{\delta}}{\sqrt{\kappa}+\sqrt{1+\delta}}.
    \label{la:id:rhogeps-short}
\end{align}
Using this alternative expression, we obtain the following.

\begin{lemma}\label{la:lem:rhoeps-leq-racc}
    Given \(\sigma > 1\), then \(r_{\delta}(\kappa)\leq \rnag(\sigma'\kappa)\) for all \(\delta\in \mathopen{(}0,\delta_{\sigma}\mathclose{]}\), \(\sigma'\geq \sigma\) and \(\kappa \geq 1 + \delta\), where \(\delta_{\sigma}=(\sigma-1)^{2}/4\sigma\).
    Conversely, given \(\delta>0\), then \(r_{\delta'}(\kappa)\leq \rnag(\sigma\kappa)\) for all \(\delta'\in \mathopen{(}0,\delta\mathclose{]}\), \(\sigma\geq \sigma_{\delta}\) and \(\kappa\geq 1 + \delta'\), where \(\sigma_{\delta}=1+2\delta+2\sqrt{\delta(1+\delta)}\).
\end{lemma}

\begin{proof}
    Let \(\sigma > 1\).
    From \eqref{la:id:rhogeps-short} and \eqref{la:def:racc}, it follows that the condition that \(r_{\delta}(\kappa) \leq \rnag(\sigma\kappa)\) for some \(\delta>0\) and \(\kappa \geq 1 + \delta\) is equivalent to
    \begin{align}
        \frac{\sqrt{\kappa-1}}{\sqrt{\kappa}}
        \frac{\sqrt{\kappa-1}+\sqrt{\delta}}{\sqrt{\kappa}+\sqrt{1+\delta}}
        \leq \frac{\sqrt{\sigma\kappa}-1}{\sqrt{\sigma\kappa}}.
        \label{la:ineq:rdelta-leq-racc}
    \end{align}
    By successively manipulating \eqref{la:ineq:rdelta-leq-racc}, it follows that
    \begin{align}
        r_{\delta}(\kappa)
        \leq \rnag(\sigma\kappa)
        &\iff \sqrt{\kappa-1}(\sqrt{\kappa-1} + \sqrt{\delta})\sqrt{\sigma}
        \leq (\sqrt{\sigma\kappa} - 1)(\sqrt{\kappa} + \sqrt{1+\delta})
        \nonumber\\
        &\iff 
        \sqrt{\kappa} + \sqrt{1+\delta}
        \leq (1 + \sqrt{(1+\delta)\kappa} - \sqrt{\delta(\kappa-1)})\sqrt{\sigma}
        \nonumber\\
        &\iff
        \frac{\sqrt{\kappa} + \sqrt{1+\delta}}{1 + \sqrt{(1+\delta)\kappa} -\sqrt{\delta(\kappa-1)}}
        \leq \sqrt{\sigma}.
        \label{la:ineq:rhoeps-leq-racc-equiv}
    \end{align}
    Taking the derivative of the left-hand side of the \eqref{la:ineq:rhoeps-leq-racc-equiv} with respect to \(\kappa\), we obtain
    \begin{align*}
        \frac{\partial}{\partial \kappa}\frac{\sqrt{\kappa} + \sqrt{1+\delta}}{1 + \sqrt{(1+\delta)\kappa} -\sqrt{\delta(\kappa-1)}}
        =& \frac{\delta\sqrt{\kappa} + \delta\kappa\sqrt{(1+\delta)} - \delta\sqrt{\delta(\kappa-1)\kappa}}{2\kappa\sqrt{\delta(\kappa-1)}(1+\sqrt{(1+\delta)\kappa} -\sqrt{\delta(\kappa-1)})^{2}}
        \\
        =& \frac{\delta}{2\sqrt{\delta\kappa(\kappa-1)}(1+\sqrt{(1+\delta)\kappa} -\sqrt{\delta(\kappa-1)})}> 0.
    \end{align*}
    That is, the left-hand side of \eqref{la:ineq:rhoeps-leq-racc-equiv} is increasing in \(\kappa \geq 1 + \delta\) and it follows that
    \begin{align*}
        \frac{\sqrt{\kappa} + \sqrt{1+\delta}}{1 + \sqrt{(1+\delta)\kappa} -\sqrt{\delta(\kappa-1)}}
        \leq \lim_{k\to+\infty}\frac{\sqrt{\kappa} + \sqrt{1+\delta}}{1 + \sqrt{(1+\delta)\kappa} -\sqrt{\delta(\kappa-1)}}
        = \frac{1}{\sqrt{1+\delta}-\sqrt{\delta}}.
    \end{align*}
    Moreover, \(1/(\sqrt{1+\delta} - \sqrt{\delta})\) is increasing in \(\delta>0\).
    Therefore, if \(\delta_{\sigma}=(\sigma-1)^{2}/4\sigma\), then for all \(\delta\in \mathopen{(}0,\delta_{\sigma}\mathclose{]}\), \(\kappa \geq 1 + \delta\) and \(\sigma'\geq\sigma\), we have that
    \begin{align*}
        \frac{\sqrt{\kappa} + \sqrt{1+\delta}}{1 + \sqrt{(1+\delta)\kappa} -\sqrt{\delta(\kappa-1)}}
        &\leq \frac{1}{\sqrt{1+\delta} -\sqrt{\delta}}
        \\
        &\leq \frac{1}{\sqrt{1+\delta_{\sigma}} -\sqrt{\delta_{\sigma}}}
        \\
        &= \frac{2\sqrt{\sigma}}{\sqrt{(1+\sigma)^{2}} - \sqrt{(\sigma-1)^{2}}}
        \\
        &= \sqrt{\sigma}
        \\
        &\leq \sqrt{\sigma'}.
    \end{align*}
    Conversely, given \(\delta>0\), if \(\delta' \in \mathopen{(}0,\delta\mathclose{]}\) and \(\sigma\geq \sigma_{\delta}\), where \(\sigma_{\delta}=1+2\delta+2\sqrt{\delta(1+\delta)}\), then
    \begin{align*}
        \frac{1}{\sqrt{1+\delta'} -\sqrt{\delta'}}
        \leq \frac{1}{\sqrt{1+\delta} -\sqrt{\delta}}
        = \sqrt{\sigma_{\delta}}
        \leq \sqrt{\sigma}.
    \end{align*}
    Therefore, \(r_{\delta'}(\kappa)\leq \rnag(\sigma\kappa)\) for all \(\delta'\leq \delta\), \(\kappa\geq 1 + \delta'\) and \(\sigma\geq \sigma_{\delta}\).
\end{proof}

\begin{corollary}\label{la:cor:rho-leq-racc}
    Given \(\sigma > 1\), then \(\rho(s,m) \leq \rnag(\sigma'\kappa)\) for all \(s\in\mathopen{[}m,(1+\delta)m\mathclose{]}\), \(\delta\in \mathopen{(}0,\delta_{\sigma}\mathclose{]}\), \(\sigma'\geq \sigma\) and \(\kappa \geq 1 + \delta\), where \(\delta_{\sigma}=(\sigma-1)^{2}/4\sigma\).
    Conversely, given \(\delta>0\), then \(\rho(s,m)\leq \rnag(\sigma\kappa)\) for all \(s\in\mathopen{[}m,(1+\delta)m\mathclose{]}\), \(\delta'\in \mathopen{(}0,\delta\mathclose{]}\), \(\sigma\geq \sigma_{\delta}\) and \(\kappa\geq 1 + \delta'\), where \(\sigma_{\delta}=1+2\delta+2\sqrt{\delta(1+\delta)}\).
\end{corollary}

\begin{proof}
    The corollary follows by combining \cref{la:lem:rho-leq-rhoeps,la:lem:rhoeps-leq-racc}.
\end{proof}

\subsection*{Iterate Dynamics Between \texorpdfstring{\(m_{t}\)}{mt} Updates}

Through \(G_{i}(m_{t})\), the dynamics of \(X_{t,i}\) are determined by \(m_{t}\), which is updated by \cref{alg:nag-free} \emph{after} the \(t\)-th iterate is computed.
Moreover, if \(\gamma>1\), then \(m_{t}\) is updated at most \(\log_{\gamma}\kappa + 1\) times.
So, suppose the estimates \(m_{t}\) take \(M+1 \leq \log_{\gamma}\kappa + 1\) values.
Then, let \(t_{j}\) denote the iteration in which \(m_{t}\) is adjusted to its \(j\)-th value \(\mu_{j}\), \(j=0,\ldots,M\).
Since NAG-free computes the iterate \(x_{t}\) and then adjusts \(m_{t}\) in iteration \(t\), this means that \(t_{j}+1\) is the first iteration in which the estimate \(\mu_{j}\) takes effect, and \cref{alg:nag-free} computes iterates for \(t \in \mathopen{(}t_{j},t_{j+1}\mathclose{]}\) using \(m_{t} = \mu_{j}\).
For example, \(t_{0}=0\) and \(m_{t}=\mu_{0}=m_{0}\) for all \(t<t_{1}\).
Therefore, given \(t\) and \(t'\) such that \(t_{j} < t' \leq t_{j+1} \leq t_{J} < t \leq t_{J+1}\),
\begin{align}
    X_{t,i} 
    = \prod_{k=0}^{t-1}G_{i}(\mu_{k})X_{0,i}
    = G_{i}(\mu_{J})^{t-t_{J}}\Biggl( \prod_{k=j+1}^{J-1}G_{i}(\mu_{k})^{t_{k+1}-t_{k}} \Biggr)G_{i}(\mu_{j})^{t_{j+1}-t'}X_{t',i}.
    \label{la:eq:Xit-piecewise-lti}
\end{align}
Now, if \(m_{t} > m\), then under \cref{ass:mt-separation-from-eigs}, \cref{la:cor:Gi-eig} implies that the eigenvalues of \(G_{i}(m_{t})\) are distinct.
So, letting \(\zeta_{i}=\zeta_{i}(m_{t})\) and \(\xi_{i}=\xi_{i}(m_{t})\) denote the eigenvalues of \(G_{i}(m_{t})\), we define
\begin{align}
    T_{i}(m_{t})
    = \begin{bmatrix}
        1 & 1
        \\
        \zeta_{i} & \xi_{i}
    \end{bmatrix}.
    \label{la:def:diagonalizing-matrix-comp}
\end{align}
It can be checked that the columns of \(T_{i}(m_{t})\) are eigenvectors of \(G_{i}(m_{t})\), therefore \(T_{i}(m_{t})\) diagonalizes \(G_{i}(m_{t})\):
\begin{align}
    G_{i}(m_{t}) = T_{i}(m_{t})D_{i}(m_{t})T_{i}(m_{t})^{-1}.
    \label{la:eq:Git-diagonal-form}
\end{align}
That is, \(D_{i}(m_{t})\) is a diagonal matrix whose diagonal entries are the eigenvalues of \(G_{i}(m_{t})\):
\begin{align}
    D_{i}(m_{t})
    = \begin{bmatrix}
        \zeta_{i} & 0
        \\
        0 & \xi_{i}
    \end{bmatrix}
    .
    \label{la:def:Dit}
\end{align}
Combining \labelcref{la:eq:Xit-piecewise-lti,la:eq:Git-diagonal-form,la:def:Dit}, then applying \cref{la:lem:top-rho} it follows that for every \(t_{j} <  t \leq t_{j+1}\)
\begin{align}
    \Vert X_{t,i} \Vert^{2}
    &\leq 
    \overline{C}_{i} \rho(\mu_{j},\lambda_{i})^{2(t-t_{j})}
    \biggl(
    \prod_{k=0}^{j-1}\rho(\mu_{k},\lambda_{i})^{2(t_{k+1}-t_{k})}
    \biggr)
    x_{0,i}^{2},
    \label{la:ineq:Xit-quad-norm-sq-bound}
\end{align}
where the constant \(\overline{C}_{i}\) that is uniformly bounded, since
\begin{align*}
    \Vert X_{t,i} \Vert^{2}
    &= \biggl\Vert
    T_{i}(\mu_{j})D_{i}(\mu_{j})^{t-t_{j}}T_{i}(\mu_{j})^{-1}
    \Biggl(
    \prod_{k=0}^{j-1}
    T_{i}(\mu_{k})D_{i}(\mu_{k})^{t_{k+1}-t_{k}}T_{i}(\mu_{k})^{-1}
    \Biggr)
    X_{0,i}
    \biggr\Vert^{2}
    \\
    &\leq 
    \Vert T_{i}(\mu_{j}) D_{i}(\mu_{j})^{t-t_{j}} T_{i}(\mu_{j})^{-1} \Vert^{2}
    \Biggl(
    \prod_{k=0}^{j-1}
    \Vert T_{i}(\mu_{k}) D_{i}(\mu_{k})^{t_{k+1}-t_{k}}T_{i}(\mu_{k})^{-1} \Vert^{2}
    \Biggr)
    x_{0,i}^{2}
    \\
    &\leq 
    \Biggl(
    \prod_{k=0}^{M} 
    \Vert T_{i}(\mu_{k}) \Vert^{2}
    \Vert T_{i}(\mu_{k})^{-1} \Vert^{2}
    \Biggr)
    \Vert D_{i}(\mu_{j})^{t-t_{j}} \Vert^{2}
    \Biggl(
    \prod_{k=0}^{j-1}
    \Vert D_{i}(\mu_{k})^{t_{k+1}-t_{k}} \Vert^{2}
    \Biggr)
    2x_{0,i}^{2}
    \\
    &\leq 
    \Biggl(
    2\prod_{k=0}^{\log_{\gamma}\kappa + 1} 
    \Vert T_{i}(\mu_{k}) \Vert^{2}
    \Vert T_{i}(\mu_{k})^{-1} \Vert^{2}
    \Biggr)
    \rho(\mu_{j},\lambda_{i})^{2(t-t_{j})}
    \Biggl(
    \prod_{k=0}^{j-1}
    \rho(\mu_{k},\lambda_{i})^{2(t_{k+1}-t_{k})}
    \Biggr)
    x_{i.0}^{2}
\end{align*}
and, by applying \cref{la:lem:rho-lt-1,la:lem:Gi-eigenvalue-separation} to \eqref{la:def:diagonalizing-matrix-comp}, for all \(\mu_{k}\) we obtain
\begin{align*}
    \Vert T_{i}(\mu_{k}) \Vert^{2}
    \leq 4,
    &&
    \Vert T_{i}(\mu_{k})^{-1} \Vert^{2}
    = \frac{1}{\vert \zeta_{i} - \xi_{i} \vert^{2}}
    \biggl\Vert
    \begin{bmatrix}
        \xi_{i} & -1
        \\
        -\zeta_{i} & 1
    \end{bmatrix}
    \biggr\Vert^{2}
    \leq \frac{4}{\sqrt{\delta_{\lambda}\delta_{L}}},
\end{align*}
where \(\delta_{L}=(L_{0}-L)/L_{0}\) and \(\delta_{\lambda}\) is given by \cref{ass:mt-separation-from-eigs}.
Furthermore, omitting the \(m_{t}\) arguments, for \(t \in \mathopen{(}t_{j},t_{j+1}\mathclose{]}\), we have that
\begin{align*}
    X_{t,i}
    &= G_{i}^{t-t_{j}}X_{t_{j},i}
    \\
    &= T_{i}
    D_{i}^{t-t_{j}}
    T_{i}^{-1}X_{t_{j},i}
    \\
    &= \begin{bmatrix}
        1 & 1
        \\
        \zeta_{i} & \xi_{i}
    \end{bmatrix}
    \begin{bmatrix}
        \zeta_{i}^{t-t_{j}} & 0
        \\
        0 & \xi_{i}^{t-t_{j}}
    \end{bmatrix}
    \frac{1}{\xi_{i} - \zeta_{i}}
    \begin{bmatrix}
        \xi_{i} & -1
        \\
        -\zeta_{i} & 1
    \end{bmatrix}
    \begin{bmatrix}
        x_{t_{j},i-1}
        \\
        x_{t_{j},i}
    \end{bmatrix}
    \\
    &= \frac{1}{\xi_{i} - \zeta_{i}}
    \begin{bmatrix}
        \zeta_{i}^{t-t_{j}} & \xi_{i}^{t-t_{j}}
        \\
        \zeta_{i}^{t+1-t_{j}} & \xi_{i}^{t+1-t_{j}}
    \end{bmatrix}
    \begin{bmatrix}
        \xi_{i} & -1
        \\
        -\zeta_{i} & 1
    \end{bmatrix}
    \begin{bmatrix}
        x_{t_{j},i-1}
        \\
        x_{t_{j},i}
    \end{bmatrix}
    \\
    &= \frac{1}{\xi_{i} - \zeta_{i}}
    \begin{bmatrix}
        \xi_{i}\zeta_{i}^{t-t_{j}}-\zeta_{i}\xi_{i}^{t-t_{j}} & \xi_{i}^{t-t_{j}}-\zeta_{i}^{t-t_{j}}
        \\
        \xi_{i}\zeta_{i}^{t+1-t_{j}}-\zeta_{i}\xi_{i}^{t+1-t_{j}} & \xi_{i}^{t+1-t_{j}}-\zeta_{i}^{t+1-t_{j}}
    \end{bmatrix}
    \begin{bmatrix}
        x_{t_{j},i-1}
        \\
        x_{t_{j},i}
    \end{bmatrix}
    \\
    &= \frac{1}{\xi_{i}-\zeta_{i}}
    \begin{bmatrix}
        (\xi_{i}x_{t_{j},i-1}-x_{t_{j},i})\zeta_{i}^{t-t_{j}} + (x_{t_{j},i}-\zeta_{i}x_{t_{j},i-1})\xi_{i}^{t-t_{j}}
        \\
        (\xi_{i}x_{t_{j},i-1}-x_{t_{j},i})\zeta_{i}^{t+1-t_{j}} + (x_{t_{j},i}-\zeta_{i}x_{t_{j},i-1})\xi_{i}^{t+1-t_{j}}
    \end{bmatrix}.
\end{align*}
Therefore, \(X_{t,i}\) can be decomposed into two modes:
\begin{align}
    X_{t,i}
    = A_{t_{j},i}\zeta_{i}^{t-t_{j}} + B_{t_{j},i}\xi^{t-t_{j}},
    \label{la:eq:Xit-quad-modal-decomposition}
\end{align}
where \(A_{i}\) and \(B_{i}\) are two-dimensional vectors given by
\begin{align}
    A_{t_{j},i} 
    = \frac{x_{t_{j},i}-\xi_{i}x_{t_{j},i-1}}{\zeta_{i}-\xi_{i}}
    \begin{bmatrix}
        1
        \\
        \zeta_{i}
    \end{bmatrix}
    &&
    \text{and}
    &&
    B_{t_{j},i}
    = \frac{\zeta_{i}x_{t_{j},i-1}-x_{t_{j},i}}{\zeta_{i}-\xi_{i}}
    \begin{bmatrix}
        1
        \\
        \xi_{i}
    \end{bmatrix},
    \label{la:def:Ai-Bi-quad}
\end{align}
which are well-defined, by \cref{la:lem:Gi-eigenvalue-separation}.
In particular, for \(t_{0} < t \leq t_{1}\), we have that
\begin{align*}
    X_{t,i}
    = \frac{(1-\xi_{i})x_{0,i}}{\zeta_{i}-\xi_{i}}
    \begin{bmatrix}
        1
        \\
        \zeta_{i}
    \end{bmatrix}\zeta_{i}^{t} 
    + \frac{(\zeta_{i}-1)x_{0,i}}{\zeta_{i}-\xi_{i}}
    \begin{bmatrix}
        1
        \\
        \xi_{i}
    \end{bmatrix}
    \xi_{i}^{t}.
\end{align*}
In turn, if without loss of generality we assume \(x_{0,1}>0\), then
\begin{align*}
    x_{t,1}-x_{t-1,1}
    &= \frac{(1-\xi_{1})(\zeta_{1}-1)\zeta_{1}^{t}x_{0,1} + (\zeta_{1}-1)(\xi_{1}-1)\xi_{1}^{t}x_{0,1}}{\zeta_{1}-\xi_{1}}
    \leq \kappa^{-1}\zeta_{1}^{t-1}x_{0,1} < 0,
\end{align*}
where in first inequality above we used the fact that \(0 < \xi_{1} < \zeta_{1}\) and the identity
\begin{align*}
    (1-\zeta_{i})(1-\xi_{i})
    = \biggl( 
    1 -\frac{1+\beta}{2}\Bigl( 1 - \frac{\lambda_{i}}{L}\Bigr)
    \biggr)^{2}
    -\frac{(1+\beta)^{2}}{4}\Bigl( 1 - \frac{\lambda_{i}}{L}\Bigr)^{2}
    +\beta\Bigl( 1 - \frac{\lambda_{i}}{L}\Bigr)
    = \frac{\lambda_{i}}{L}.
\end{align*}
Moreover, for \(t_{0} < t \leq t_{1}\) we also have that
\begin{align*}
    x_{t,1}-\xi_{1}x_{t-1,1}
    &= \frac{(1-\xi_{1})(\zeta_{1}-\xi_{1})\zeta_{1}^{t}x_{0,1} + (\zeta_{1}-1)(\xi_{1}-\xi_{1})\xi_{1}^{t}x_{0,1}}{\zeta_{1}-\xi_{1}}
    = (1-\xi_{1})\zeta_{1}^{t}x_{0,1} < 0,
    \\
    \zeta_{1}x_{t-1,1}-x_{t,1}
    &= \frac{(1-\xi_{1})(\zeta_{1}-\zeta_{1})\zeta_{1}^{t}x_{0,1} + (\zeta_{1}-1)(\zeta_{1}-\xi_{1})\xi_{1}^{t}x_{0,1}}{\zeta_{1}-\xi_{1}}
    = (\zeta_{1}-1)\xi_{1}^{t}x_{0,1} > 0.
\end{align*}
It follows that, for \(t_{1} < t \leq t_{2}\)
\begin{align*}
    x_{t,1} - x_{t-1,1}
    &= \frac{
    (x_{1,t_{j}}-\xi_{1}x_{1,t_{1}-1})(\zeta_{1}-1)\zeta_{1}^{t-t_{1}} 
    + (\zeta_{1}x_{1,t_{1}-1}-x_{1,t_{1}})(\xi_{1}-1)\xi_{1}^{t-t_{1}}
    }{\zeta_{1}-\xi_{1}}
    \\
    &\leq 
    \zeta_{1}^{t-t_{1}}
    \frac{
    (x_{1,t_{j}}-\xi_{1}x_{1,t_{1}-1})(\zeta_{1}-1) 
    + (\zeta_{1}x_{1,t_{1}-1}-x_{1,t_{1}})(\xi_{1}-1)
    }{\zeta_{1}-\xi_{1}}
    \\
    &= \zeta_{1}^{t-t_{1}}(x_{1,t_{1}}-x_{1,t_{1}-1})
    \\
    &\leq \kappa^{-1}\zeta_{1}(\mu_{1})^{t-t_{1}}\zeta_{1}(\mu_{0})^{t_{1}-1}x_{0,1}
    \\
    &< 0,
\end{align*}
since \(0<\xi_{1}(m_{1})<\zeta_{1}(m_{1})\), and moreover
\begin{align*}
    x_{t,1} - \xi_{1}x_{t-1,1}
    &= \zeta_{1}^{t-t_{1}}(x_{1,t_{1}} -\xi_{1}x_{1,t_{1}-1})
    = \zeta_{1}(\mu_{1})^{t-t_{1}}(1-\xi_{1}(\mu_{0}))\zeta_{1}(\mu_{0})^{t}x_{0,1} < 0,
    \\
    \zeta_{1}x_{t-1,1}-x_{t,1}
    &= \xi_{1}^{t-t_{1}}(\zeta_{1}x_{1,t_{1}-1}-x_{1,t_{1}})
    = \xi_{1}(\mu_{1})^{t-t_{1}}(\zeta_{1}(\mu_{0})-1)\xi_{1}(\mu_{0})^{t}x_{0,1} > 0.
\end{align*}
Therefore, using the fact that \(\zeta_{1}(m_{t})=\rho(m_{t},m)\), it follows by induction that for \(t_{j} < t \leq t_{j+1}\)
\begin{align}
    (x_{t+1,1}-x_{t,1})^{2}
    \geq
    \underline{C}_{1}\rho(\mu_{j},m)^{2t-t_{j}}
    \Biggl(
    \prod_{k=0}^{j-1}\rho(\mu_{k},m)^{2(t_{k+1} - t_{k})}
    \Biggr)
    x_{0,1}^{2}
    \geq 0,
    \label{la:ineq:x1t-diff-quad-lower-bound}
\end{align}
for some \(\underline{C}_{1}\geq\kappa^{-2}\).


\subsection*{The Dynamics of \texorpdfstring{\(c_{t}\)}{ct}}

\cref{la:lem:rhoeps-leq-racc} bounds the suboptimality factor in the convergence rate of \(x_{t}\) when \(m_{t}\in\mathopen{[}m,(1+\delta)m\mathclose{]}\), for a given \(\delta>0\).
Now, we determine how long \(m_{t}\) takes to reach the interval \(\mathopen{[}m,(1+\delta)m\mathclose{]}\).
Our starting point is to determine the dynamics of \(c_{t+1}\). 
To this end, we plug \eqref{la:eq:eigendecomposition} and \eqref{la:eq:momentum-step} into \eqref{ineq:effective-curvature}, obtaining\footnote{Note that \(x_{t+1}-x_{t}=(x_{t+1}-x^{\star}) -(x_{t}-x^{\star})=\sum_{i=1}^{d}(x_{t+1,i}-x_{t,i})v_{i}\).}
\begin{align}
    c^{2}_{t+1}
    = \left\Vert \frac{\nabla f(x_{t+1}) -\nabla f(x_{t})}{x_{t+1} - x_{t}} \right\Vert^{2}
    = \left\Vert \frac{\sum_{i=1}^{d}(x_{t+1,i} - x_{t,i})\lambda_{i}v_{i}}{\sum_{i=1}^{d}(x_{t+1,i} - x_{t,i})v_{i}} \right\Vert^{2}
    = \frac{\sum_{i=1}^{d}(x_{t+1,i} - x_{t,i})^{2}\lambda_{i}^{2}}{\sum_{i=1}^{d}(x_{t+1,i} - x_{t,i})^{2}}
    .
    \label{la:eq:curvature-weighted-average}
\end{align}
The identity \eqref{la:eq:curvature-weighted-average} reveals that \(c^{2}_{t+1}\) can be expressed as an average of the squared eigenvalues \(\lambda_{i}^{2}\) weighted by \((x_{t+1,i} - x_{t,i})^{2}\).
Since the weights are a static map of \(x_{t,i}\), the dynamics of \(x_{t,i}\) determine the dynamics of the estimated effective curvature \(c_{t+1}\).
In particular, \(x_{t,i}\) determine if one weight can outweigh the others, in which case \(c_{t+1}\) tends to \(\lambda_{i}\).

By \cref{la:lem:top-rho}, \(\rho(s,\lambda_{i}) < \rho(s,m)\) for all \(\lambda_{i} \in \mathopen{(}m,L\mathclose{]} \).
Hence, from \eqref{la:ineq:Xit-quad-norm-sq-bound} and \eqref{la:ineq:x1t-diff-quad-lower-bound}, we conclude that the weight associated with \(m\) eventually dominates the other weights, so that \(c_{t+1}\) converges to \(m\).
In the following, we show that this happens at an accelerated rate.
To this end, we define\footnote{Note that \(\rho < 0\) cannot occur by the definition of \(\rho\), \eqref{la:def:rho}.} \(\phi:\mathcal{D} \to \mathopen{[}0,1\mathclose{]}\) as
\begin{align}
    \phi(s,a,b) = 
    \begin{cases}
        \min \left\{ 1, \dfrac{\rho(s,a)}{\rho(s,b)} \right\}, & \rho(s,b) > 0,
        \\
        1, & \rho(s,b) = 0,
    \end{cases}
    \label{la:def:phi}
\end{align}
where the domain \(\mathcal{D}\) is given by
\begin{align}
    \mathcal{D} = \mathopen{(}0,L\mathclose{]} \times \left\{ (a,b)\in \mathbb{R}^{2}_{> 0}: a\neq b \right\},
    \label{la:def:phi-domain}
\end{align}
\(\mathbb{R}_{>0}\) being the set of positive real numbers.
With \(\phi\), we can bound how fast \(c_{t+1}\) takes to decrease below \((1+\delta)\ell\) for a given \(\ell\in\mathopen{[}m,L\mathclose{]}\), not necessarily an eigenvalue of \(H\), where \(\delta>0\) represents some estimate precision relative to \(\ell\).
To this end, we characterize \(\phi((1+\delta)\ell, \ell, m)^{2}\), first showing that it is decreasing in \(\ell\).

\begin{lemma}\label{la:lem:phi-decreasing}
    If \(\delta \in \mathopen{(}0,1\mathclose{]}\) and \(\kappa \geq 2\), then \(\phi((1+\delta)\ell, \ell, m)\) is decreasing in \(\ell \geq m > 0\).
\end{lemma}

\begin{proof}
    Let \(L>m>0\).
    Given \(\ell\) and \(\delta>0\) such that \(m \leq \ell < (1+\delta)\ell \leq L\), by \eqref{la:def:phi} and \cref{la:lem:rho-expression}, we have that
    \begin{align*}
        \phi((1+\delta)\ell, \ell, m)
        =& \frac{L - \ell + \sqrt{(L - \ell)\delta\ell}}{L - m + \sqrt{(L - m)((1+\delta)\ell - m)}}.
    \end{align*}
    Letting \(\phi_{\ell}\) the derivative of \(\phi((1+\delta)\ell, \ell, m)\) with respect to \(\ell\), we obtain
    \begin{align*}
        \phi_{\ell}
        =& \frac{-(L-m)\delta^{2}\ell - (L-m)\sqrt{(L-\ell)\delta\ell}(L + \ell + \sqrt{(L-m)((1+\delta)\ell - m)} - 2m)}{2\sqrt{(L-\ell)\delta\ell}\sqrt{(L-m)((1+\delta)\ell - m)}(L + \sqrt{(L-m)((1+\delta)\ell - m)} - m)^{2}}
        \\
        &- \frac{(L-m)\delta(\ell^{2} + \ell(\sqrt{(L-\ell)\delta\ell} + 2\sqrt{(L-m)((1+\delta)\ell - m)} - 2m))}{2\sqrt{(L-\ell)\delta\ell}\sqrt{(L-m)((1+\delta)\ell - m)}(L + \sqrt{(L-m)((1+\delta)\ell - m)} - m)^{2}}
        \\
        &- \frac{(L-m)\delta L(\sqrt{(L-\ell)\delta\ell} - \sqrt{(L-m)((1+\delta)\ell - m)} + m)}{2\sqrt{(L-\ell)\delta\ell}\sqrt{(L-m)((1+\delta)\ell - m)}(L + \sqrt{(L-m)((1+\delta)\ell - m)} - m)^{2}}
        \\
        \leq& -\frac{(L-m)( (L-m)\sqrt{(L-\ell)\delta\ell} - \delta(L-2\ell)\sqrt{(L-m)((1+\delta)\ell - m)} )}{2\sqrt{(L-\ell)\delta\ell}\sqrt{(L-m)((1+\delta)\ell - m)}(L + \sqrt{(L-m)((1+\delta)\ell - m)} - m)^{2}}.
    \end{align*}
    So, to show \(\phi((1+\delta)\ell, \ell, m)\) is decreasing in \(\ell\), it suffices to show the numerator above is positive.
    To this end, since \(L>m\), it suffices to show that the second factor is positive:
    \begin{align}
        &(L-m)\sqrt{(L-\ell)\delta\ell} + \sqrt{(L-\ell)\delta\ell}\sqrt{(L-m)((1+\delta)\ell - m)} 
        \nonumber\\
        &- \delta(L-2\ell)\sqrt{(L-m)((1+\delta)\ell - m)}
        > 0.
        \label{la:ineq:phi-decreasing}
    \end{align}
    The negative term on the left-hand side above is maximized at the critical point characterized by
    \begin{align*}
        \frac{\partial}{\partial \ell}(L - 2\ell)\sqrt{((1+\delta)\ell - m)}
        &= -2\sqrt{(1+\delta)\ell - m} + \frac{(L - 2\ell)(1+\delta)}{2\sqrt{(1+\delta)\ell - m}}
        \\
        &= \frac{(1+\delta)(L - 2\ell) - 4((1+\delta)\ell -m)}{2\sqrt{(1+\delta)\ell - m}}
        \\
        &= 0.
    \end{align*}
    Taking \(\ell\) at this critical point, \(\ell = \frac{1}{6}L + \frac{2}{3(1+\delta)}m \), and using the assumptions that \(\kappa \geq 2\) and \(\delta\leq 1\), it follows that
    \begin{align*}
        (L-\ell)\ell
        \geq \frac{5L-4m}{6}\frac{(1+\delta)L + 4m}{6(1+\delta)}
        = \frac{5(1+\delta)L^{2} + 4(5-(1+\delta))Lm -16m^{2}}{36(1+\delta)}
        \geq \frac{5}{36}L^{2}.
    \end{align*}
    Hence, plugging \(\ell = \frac{1}{6}L + \frac{2}{3(1+\delta)}m\) back into \eqref{la:ineq:phi-decreasing} and using the assumptions that \(\delta\leq 1\) and \(\kappa\geq 2\) yields
    \begin{align*}
        &(\delta(L-2\ell) - \sqrt{(L-\ell)\delta\ell})\sqrt{(L-m)((1+\delta)\ell - m)}
        \\
        &\leq \delta\frac{(4 -\sqrt{5})L}{6}
        \sqrt{(L-m)\frac{1+\delta}{6}\biggl(L - \frac{2m}{1+\delta}\biggr)}
        \\
        &\leq \frac{2\delta\sqrt{1+\delta}}{6\sqrt{6}}L(L-m),
    \end{align*}
    and, similarly
    \begin{align*}
        \sqrt{(L-\ell)\delta\ell}(L-m)
        \geq \sqrt{\delta\frac{5L-4m}{6}\frac{L}{6}}(L-m)
        \geq \frac{\sqrt{\delta}}{3\sqrt{2}}L(L-m).
    \end{align*}
    Hence, canceling the common factor \(\sqrt{\delta}L(L-m)\) above and then rearranging, we conclude that \eqref{la:ineq:phi-decreasing} holds if
    \begin{align*}
        \sqrt{\delta}\sqrt{1 + \delta}
        \leq \sqrt{3},
    \end{align*}
    which is true since \(\sqrt{\delta}\leq 1\).
\end{proof}

In fact, \(\phi((1+\delta)\ell,\ell,m)\) is decreasing for any \(\delta>0\), which can be seen in its graph, but the case where \(\delta \in \mathopen{(}0,1\mathclose{]}\) suffices for the upcoming results.
Namely, given \(\delta_{\ell}>0\) and \(\delta_{u} \in \mathopen{(}0,1\mathclose{]}\), by \cref{la:lem:phi-decreasing} we have that for every \(\ell \in \mathopen{[}(1+\delta_{\ell})m,L\mathclose{]}\)
\begin{align}
    \phi((1+\delta_{u})\ell, \ell, m)
    =& \frac{L - \ell + \sqrt{(L - \ell)\delta_{u}\ell}}{L - m + \sqrt{(L - m)((1+\delta_{u})\ell - m)}}
    \nonumber\\
    \leq& \frac{L - (1+\delta_{\ell})m + \sqrt{(L - (1+\delta_{\ell})m)\delta_{u}(1+\delta_{\ell})m}}{L - m + \sqrt{(L - m)((1+\delta_{u})(1+\delta_{\ell})m - m)}}
    \nonumber\\
    =& \frac{\kappa - (1+\delta_{\ell}) + \sqrt{(\kappa - (1+\delta_{\ell}))\delta_{u}(1+\delta_{\ell})}}{\kappa - 1 + \sqrt{(\kappa - 1)(\delta_{u} + \delta_{\ell} + \delta_{u}\delta_{\ell})}}
    \nonumber\\
    \eqqcolon& r_{\phi}(\delta_{u}, \delta_{\ell}, \kappa).
    \label{la:ineq:phi-leq-rphi}
\end{align}
Hence, to show \(\phi((1+\delta_{u})\ell, \ell, m)^{2}\) is an accelerated rate, suffices to show that \(r_{\phi}(\delta_{u}, \delta_{\ell}, \kappa)^{2}\) is an accelerated rate for appropriate \(\kappa, \delta_{u}\) and \(\delta_{\ell}\), which we do in the next result.
The function \(r_{\phi}\) is well-defined for \(\delta_{\ell}>0\), \(\delta_{u}>0\) and \(k\geq 1 + \delta_{\ell}\) and, by simple inspection, it follows that \(r_{\phi}(\delta_{u}, \delta_{\ell}, \kappa)\in\mathopen{(}0,1\mathclose{)}\).

\begin{lemma}\label{la:lem:rphi-sq-leq-racc}
    Given \(\delta_{u}>0\), \(\delta_{\ell} > 0\) and \(\kappa \geq 1 + \delta_{\ell}\), there is a \(\sigma_{\phi}=\sigma_{\phi}(\delta_{u}, \delta_{\ell}, \kappa)\) such that 
    \begin{align*}
        r_{\phi}(\delta_{u},\delta_{\ell},\kappa)^{2}
        \leq \rnag(\sigma_{\phi}\kappa).
    \end{align*}
    Moreover, the function \(\kappa \mapsto \sigma_{\phi}(\delta_{u}, \delta_{\ell}, \kappa)\) is bounded and satisfies
    \begin{align*}
        \lim_{\kappa\to+\infty} \sigma_{\phi}(\delta_{u}, \delta_{\ell}, \kappa) 
        = \frac{1}{4(\sqrt{\delta_{u} + \delta_{\ell} + \delta_{u}\delta_{\ell}} - \sqrt{\delta_{u}(1+\delta_{\ell})})^{2}}.
    \end{align*}
\end{lemma}

\begin{proof}
    Let \(\delta_{u}>0\), \(\delta_{\ell}>0\) and \(\kappa \geq 1 + \delta_{\ell}\).
    By direct algebraic manipulation, we obtain
    \begin{align}
        r_{\phi}(\delta_{u}, \delta_{\ell}, \kappa)^{2} 
        \leq \rnag(\sigma_{\phi}\kappa)
        =\frac{\sqrt{\sigma_{\phi}\kappa}-1}{\sqrt{\sigma_{\phi}\kappa}}
        \iff
        \frac{1}{(1-r_{\phi}(\delta_{u}, \delta_{\ell}, \kappa)^{2})^{2}\kappa} \leq \sigma_{\phi}.
        \label{la:iff:rphi-sq-leq-racc}
    \end{align}
    For such \(\delta_{u}\), \(\delta_{\ell}\) and \(\kappa\), we have \(r_{\phi}(\delta_{u}, \delta_{\ell}, \kappa)\in\mathopen{(}0,1\mathclose{)}\), so that \(1-r_{\phi}^{2}>0\).
    Therefore, the lower bound of the inequality on the right-hand side of \eqref{la:iff:rphi-sq-leq-racc} is well-defined.
    So, let \(\sigma_{\phi}\) be defined such that \eqref{la:iff:rphi-sq-leq-racc} holds with equality:
    \begin{align*}
        \sigma_{\phi}(\delta_{u}, \delta_{\ell}, \kappa)
        = \frac{1}{(1-r_{\phi}(\delta_{u}, \delta_{\ell}, \kappa)^{2})^{2}\kappa}.
    \end{align*}
    For fixed \(\delta_{u}>0\) and \(\delta_{\ell}>0\), the map \(r_{\phi}(\delta_{u}, \delta_{\ell}, \kappa)\) is continuous in \(\kappa > 1 + \delta_{\ell}\) and right-continuous at \(\kappa=1+\delta_{\ell}\), hence so is \((1 - r_{\phi}(\delta_{u}, \delta_{\ell}, \kappa)^{2})\sqrt{\kappa}\).
    Moreover, \((1 - r_{\phi}(\delta_{u}, \delta_{\ell}, \kappa)^{2})\sqrt{\kappa} > 0\) for \(\kappa \geq 1 + \delta_{\ell}\).
    Therefore, \(1/((1 - r_{\phi}(\delta_{u}, \delta_{\ell}, \kappa)^{2})^{2}\kappa)\) is continuous in \(\kappa > 1 + \delta_{\ell}\) and right-continuous at \(\kappa=1+\delta_{\ell}\).
    Furthermore, \(\lim_{\kappa \to +\infty} 1 + r_{\phi}(\delta_{u}, \delta_{\ell}, \kappa) = 2\) and
    \begin{align*}
        \lim_{\kappa\to+\infty}(1-r_{\phi}(\delta_{u},\delta_{\ell},\kappa))\sqrt{\kappa}
        &= \lim_{\kappa\to+\infty}\frac{\delta_{\ell}\sqrt{\kappa} + \sqrt{\kappa(\kappa-1)\delta_{s}} - \sqrt{\kappa(\kappa - (1+\delta_{\ell}))\delta_{u}(1+\delta_{\ell})}}{\kappa-1 + \sqrt{(\kappa - 1)\delta_{s}}}
        \\
        &= \sqrt{\delta_{s}} - \sqrt{\delta_{u}(1+\delta_{\ell})},
    \end{align*}
    where \(\delta_{s}=\delta_{\ell}+\delta_{u}+\delta_{u}\delta_{\ell}\).
    It follows that
    \begin{align*}
        \lim_{\kappa\to +\infty} \sigma_{\phi}(\delta_{u},\delta_{\ell},\kappa)
        = \lim_{\kappa \to +\infty} \frac{1}{((1 - r_{\phi}(\delta_{u}, \delta_{\ell}, \kappa)^{2})^{2}\kappa}
        = \frac{1}{4(\sqrt{\delta_{s}} - \sqrt{\delta_{u}(1+\delta_{\ell})})^{2}}.
    \end{align*}
    Hence, \(\kappa \mapsto \sigma_{\phi}(\delta_{u}, \delta_{\ell}, \kappa)\) attains a maximum on \(\mathopen{[}1 + \delta_{\ell},\infty\mathclose{)}\) and is bounded.
\end{proof}

\cref{la:fig:rphi_bounds} shows a plot of the map \(\kappa \mapsto 1/((1-r_{\phi}(\kappa)^{2})^{2}\kappa)\) for \(\kappa=10, \ldots, 10^{9}\) and the asymptotic value of \(\sigma_{\phi}\),
\begin{align*}
    \lim_{\kappa \to +\infty} \sigma_{\phi}(\delta_{u}, \delta_{\ell}, \kappa)
    = \frac{1}{4(\sqrt{\delta_{s}} - \sqrt{\delta_{u}(1+\delta_{\ell})})^{2}}
    \approx 2.31,
\end{align*}
for \(\delta_{u}=0.01\) and \(\delta_{\ell}=0.18\).
We see that the asymptotic value of \(\sigma_{\phi}\) is slightly less than the peak value of \(\sigma_{\phi}\), but the first still provides a good approximation to the second.

\begin{figure}[tb]
    \begin{center}
        \centerline{\includegraphics[width=0.75\columnwidth]{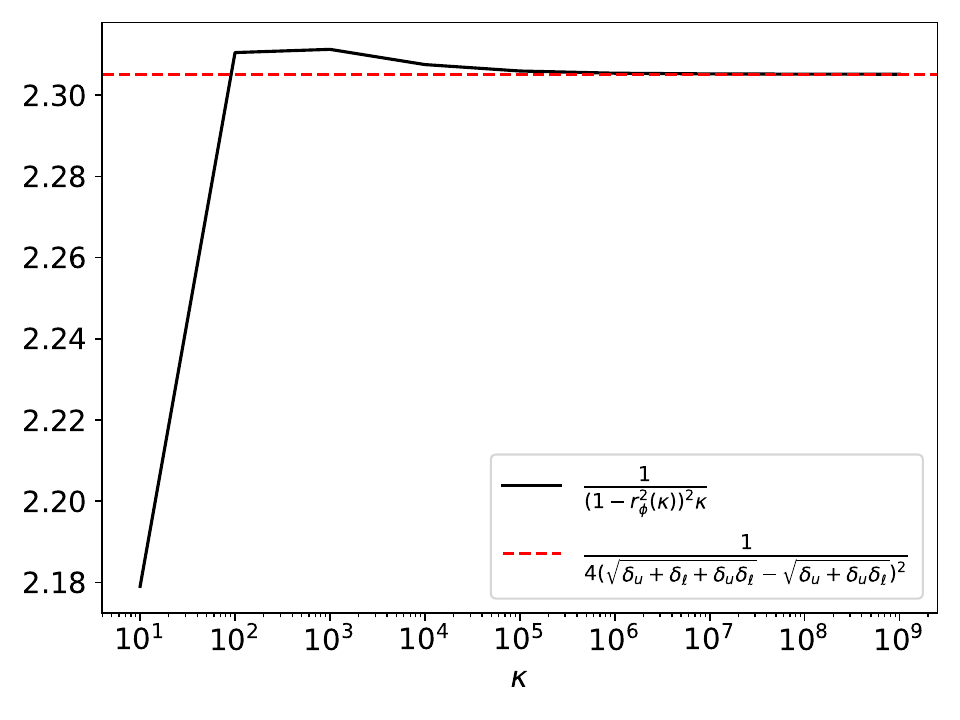}}
        \caption{Numerical (black solid line) lower bound on and asymptotic value (dashed red line) of \(\sigma_{\phi}\) such that \(r_{\phi}^{2}\leq \rnag(\sigma_{\phi} \kappa)\) holds for all \(\kappa \geq 1 + 1/\delta_{s}\), with \(\delta_{u}=0.01\) and \(\delta_{\ell}=0.18\).}
        \label{la:fig:rphi_bounds}
    \end{center}
\end{figure}

Building upon the two lemmas above, we now establish that \(\phi((1+\delta)\ell, \ell, m)\) is actually much faster than \(\rnag(\sigma_{\phi}\kappa)\) for most values of \(\ell\).

\begin{lemma}\label{la:lem:phi-sq-leq-racc-pwr}
    Given \(\delta_{u} \in \mathopen{(}0,1\mathclose{]}\), \(\delta_{\ell}\in \mathopen{(}0,1\mathclose{]}\) and \(\kappa \geq 2\), there exist \(\sigma_{\phi}=\sigma_{\phi}(\delta_{u},\delta_{\ell},\kappa)>0\) and \(\alpha_{\phi}=\alpha_{\phi}(\delta_{u},\delta_{\ell},m)>0\) such that for all \(\ell \in \mathopen{[}(1+\delta_{\ell})m,L/(1+\delta_{u})\mathclose{]}\)
    \begin{align}
        \phi((1+\delta_{u})\ell,\ell,m)^{2}
        \leq \rnag(\sigma_{\phi}\kappa)^{1+\alpha_{\phi}(\ell-(1+\delta_{\ell})m)},
        \label{la:ineq:phi-sq-leq-racc-pwr}
    \end{align}
    where the function \(\kappa \mapsto \sigma_{\phi}(\delta_{u}, \delta_{\ell}, \kappa)\) is bounded and satisfies
    \begin{align*}
        \lim_{\kappa\to+\infty} \sigma_{\phi}(\delta_{u}, \delta_{\ell}, \kappa) 
        = \frac{1}{4(\sqrt{\delta_{u} + \delta_{\ell} + \delta_{u}\delta_{\ell}} - \sqrt{\delta_{u}(1+\delta_{\ell})})^{2}}.
    \end{align*}
\end{lemma}

\begin{proof}
    Combining \cref{la:lem:phi-decreasing,la:lem:rphi-sq-leq-racc}, we have that
    \begin{align*}
        \phi((1+\delta_{u})\ell,\ell,m)^{2}
        \leq \rnag(\sigma_{\phi}\kappa)
    \end{align*}
    for all \(\ell \in \mathopen{[}(1+\delta_{\ell})m,L/(1+\delta_{u})\mathclose{]}\).
    Moreover, \(\phi((1+\delta_{u})\ell,\ell,m)\) is decreasing and continuously differentiable with respect to \(\ell\).
    So, consider the maximum slope of \(\phi((1+\delta_{u}\ell,\ell,m))\) over the interval \(\mathopen{[}(1+\delta_{\ell})m,L/(1+\delta_{u})\mathclose{]}\):
    \begin{align*}
        s=\max_{\ell \in \mathopen{[}(1+\delta_{\ell})m,L/(1+\delta_{u})\mathclose{]}} 
        \frac{\partial }{\partial \ell}\phi((1+\delta_{u})\ell,\ell,m)
        < 0.
    \end{align*}
    Then, it follows that for all \(\mathopen{[}(1+\delta_{\ell})m,L/(1+\delta_{u})\mathclose{]}\)
    \begin{align*}
        \frac{\partial }{\partial \ell}\phi((1+\delta_{u}\ell,\ell,m))
        \leq s
        \leq a \sqrt{\rnag(\sigma_{\phi})}
        \leq a \phi((1+\delta_{u}\ell,\ell,m))
        < 0,
    \end{align*}
    where \(a = s / \rnag(\sigma_{\phi}\kappa) < 0\).
    Hence, Gr\"onwall's inequality implies that
    \begin{align}
        \phi((1+\delta_{u}\ell,\ell,m))^{2}
        \leq \exp(2a(\ell-(1+\delta_{\ell})m))\rnag(\sigma_{\phi}\kappa)
        = \rnag(\sigma_{\phi}\kappa)^{1+\alpha_{\phi}(\ell-(1+\delta_{\ell})m)},
    \end{align}
    where, since \(a<0\) and \(\log_{\rnag(\sigma_{\phi}\kappa)}e<0\), \(\alpha_{\phi}\) is a positive constant given by
    \begin{align*}
        \alpha_{\phi}
        = 2a\log_{\rnag(\sigma_{\phi}\kappa)}e
        >0,
    \end{align*}
    which proves the claim.
\end{proof}

\cref{la:fig:alpha_phi} illustrates \cref{la:lem:phi-sq-leq-racc-pwr} numerically with \(L=10,m=1,\delta_{u}=1,\sigma_{\phi}=4\) and \(\alpha_{\phi}=10\).
We see that \(\phi((1+\delta_{u})\ell,\ell,m)^{2}\) becomes significantly smaller than \(\rnag(\sigma_{\phi}\kappa)\) as \(\ell\) approaches \(L\).
In fact, \(\phi(L,L,m)=0\), therefore the estimate adjustments take place extremely fast when the estimate is large and gradually slow down as the estimate improves, but always at an accelerated rate.
As we now show, this implies drastic estimate convergence speed-up.

\begin{figure}[tb]
    \begin{center}
        \centerline{\includegraphics[width=0.75\columnwidth]{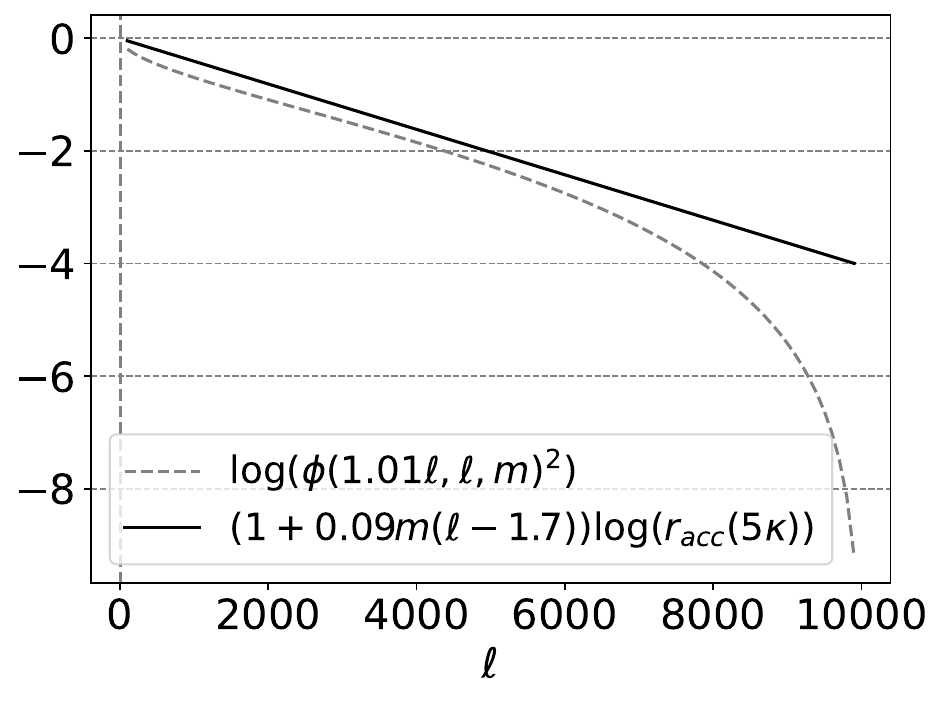}}
        \caption{Numerical illustration of \cref{la:lem:phi-sq-leq-racc-pwr} with \(L=10^{4},m=1,\delta_{\ell}=0.01,\sigma_{\phi}=5\) and \(\alpha_{\phi}=0.09\).}
        \label{la:fig:alpha_phi}
    \end{center}
\end{figure}

\begin{lemma}\label{la:lem:ct}
    Let \(f\in\mathcal{S}(L,m)\) be a quadratic function, suppose that \cref{ass:known-lipschitz-upper-bound} holds for some \(L_{0}>L\) and let \(\kappa=L_{0}/m\).
    Also, let \(\delta_{m}\) and \(\delta_{u}\) be positive numbers such that \(\delta_{m}>\delta_{u}>0\) and let \(r'=r'(\delta_{u},\delta_{\ell},\kappa)\) be a function such that \(\rnag(\sigma_{\phi}\kappa) \leq r' < 1\) for all \(\kappa \geq 1+\delta_{\ell}\), where \(\delta_{\ell} = ((1+\delta_{m})/(1+\delta_{u}))-1>0\) and \(\sigma_{\phi}=\sigma_{\phi}(\delta_{u},\delta_{\ell},\kappa)\) is given by \cref{la:lem:rphi-sq-leq-racc}.
    Then, there exists some \(\nu \geq 1\) such that the estimates \(m_{t}\) of \cref{alg:nag-free} reach \([m/\gamma,(1+\delta_{m})m]\) after no more than \(\tau\) iterations, where
    \begin{align}
        \tau
        =
        \nu \frac{-\log(4\kappa^{2}M_{1}\omega/\delta_{u})}{\log r'}
        \label{la:ineq:ct}
    \end{align}
    and \(M_{1}=\max_{i}\overline{C}_{i}/\underline{C}_{1}\), with \(\omega\) given by \cref{ass:xi0-lower-bound}.
\end{lemma}

\begin{proof}
    Suppose that the last value of \(m_{t}\) before reaching the interval \([m/\gamma,(1+\delta_{m})m)\) is \((1+\delta_{m})m\).
    Then, suppose that the value before last is \(\gamma(1+\delta_{m})m\), and so on, up to \(\gamma^{K}(1+\delta_{m})m\) for some \(K\) such that \(\gamma^{K}(1+\delta_{m})m\leq L < \gamma^{K+1}(1+\delta_{m})m\).
    Using this \(m_{t}\) schedule, we bound the number of iterations that \(m_{t}\) takes to reach the interval \([m/\gamma,(1+\delta_{m})m]\), and then we argue that no other \(m_{t}\) schedule can lead to a worse bound.
    
    Let \(\ell_{j}=\gamma^{j}(1+\delta_{m})m/(1+\delta_{u})\).
    Then, we have that \(\ell_{j} \geq (1+\delta_{\ell})m\) for \(\delta_{\ell}=((1+\delta_{m})/(1+\delta_{u}))-1\).
    Since \(\delta_{m}>\delta_{u}\), then \(\delta_{\ell}>0\), and \cref{la:lem:phi-sq-leq-racc-pwr} applies.
    Now, let \(I_{j} = \min\left\{ i: \lambda_{i} \geq \ell_{j} \right\}\).
    That is, \(\lambda_{i} \geq \ell_{j}\) if and only if \(i \geq I_{j}\).
    Then, using this fact and separating the terms indexed by \(i < I_{0}\) from those indexed by \(i \geq I_{0}\) in \eqref{la:eq:curvature-weighted-average} into two sums yields
    \begin{align*}
        c^{2}_{t+1}
        < \ell_{0}^{2}\frac{\sum_{i=1}^{I_{0}-1}(x_{t+1,i}-x_{t,i})^{2}}{\sum_{i=1}^{d}(x_{t+1,i}-x_{t,i})^{2}}
        +\frac{\sum_{i=I_{0}}^{d}\lambda_{i}^{2}(x_{t+1,i}-x_{t,i})^{2}}{\sum_{i=1}^{d}(x_{t+1,i}-x_{t,i})^{2}}.
    \end{align*}
    In turn, plugging the above inequality into the identity \((c_{t+1} + \ell_{0})(c_{t+1} - \ell_{0})=c^{2}_{t+1} - \ell_{0}^{2}\), and then using the fact that \(\lambda_{i} \leq L\) and \(\ell_{0}\geq m\), we obtain
    \begin{align}
        c_{t+1} - \ell_{0}
        =\frac{c^{2}_{t+1}-\ell_{0}^{2}}{c_{t+1} + \ell_{0}}
        < \frac{\sum_{i=I_{0}}^{d}(\lambda_{i}^{2} - \ell_{0}^{2})(x_{t+1,i}-x_{t,i})^{2}}{\ell_{0}\sum_{i=1}^{d}(x_{t+1,i}-x_{t,i})^{2}}
        \leq \ell_{0} \kappa^{2}\sum_{i=I_{0}}^{d}\frac{(x_{t+1,i}-x_{t,i})^{2}}{(x_{t+1,1}-x_{t,1})^{2}}.
        \label{la:ineq:conv-2x-rate-aux1}
    \end{align}
    Moreover, using \eqref{la:ineq:Xit-quad-norm-sq-bound}, we have that
    \begin{align}
        (x_{t+1,i}-x_{t,i})^{2}
        =(\begin{bmatrix}
            -1 & 1
        \end{bmatrix}X_{t+1,i})^{2}
        \leq 2\Vert X_{t+1,i} \Vert^{2}
        \leq 2\overline{C}_{i} \rho(\mu_{0},\lambda_{i})^{2t}
        x_{0,i}^{2}.
        \label{la:ineq:xit-diff-norm}
    \end{align}
    To address the terms in the sum in \eqref{la:ineq:conv-2x-rate-aux1}, we combine \eqref{la:ineq:xit-diff-norm} and \eqref{la:ineq:x1t-diff-quad-lower-bound}, assuming \(t_{K} \leq t < t_{K+1}\).
    That is, we consider the last adjustment before \(m_{t}\) reaches the interval \([m/\gamma,(1+\delta_{m})m)\).
    Then, we apply \cref{la:lem:top-rho} twice, to get \(\rho(m_{k},\lambda_{i})\leq \rho(m_{k},\ell_{0})\) and \(\rho(m_{k},\ell_{0})<\rho(m_{k},m)\) for all \(i\geq I_{0}\), which gives
    \begin{align}
        \sum_{i=I_{0}}^{d}\frac{(x_{t+1,i}-x_{t,i})^{2}}{(x_{t+1,1}-x_{t,1})^{2}}
        &\leq 
        2M_{1}\sum_{i=I_{0}}^{d}
        \frac{x_{0,i}^{2}}{x_{0,1}^{2}}
        \prod_{k=t_{K}}^{t}\frac{\rho(m_{k},\lambda_{i})^{2}}{\rho(m_{k},m)^{2}}
        \prod_{k=1}^{\tau_{K}}\frac{\rho(m_{k},\lambda_{i})^{2}}{\rho(m_{k},m)^{2}}
        \nonumber\\
        &\leq
        2M_{1}\omega\phi((1+\delta_{u})\ell_{0},\ell_{0},m)^{2(t-t_{K}+1)}.
        \label{la:ineq:conv-2x-rate-aux2}
    \end{align}
    where \(M_{1}=2\max_{i}\overline{C}_{i}/\underline{C}_{1}\).
    Next, we put \eqref{la:ineq:conv-2x-rate-aux1} and \eqref{la:ineq:conv-2x-rate-aux2} together, and since \(\ell_{0}\geq m\), we get
    \begin{align*}
        c_{t+1} - \ell_{0}
        <
        2\kappa^{2}M_{1}\omega\phi((1+\delta_{m})m,\ell_{0},m)^{2(t-t_{K}+1)}\ell_{0}
        \leq \delta_{u}\ell_{0}/2,
    \end{align*}
    for all \(t\geq t_{K} + \Delta t_{0}\), where
    \begin{align*}
        \Delta t_{0}
        = -\frac{\log (4\kappa^{2}M_{1}\omega/\delta_{u})}{\log \rnag(\sigma_{\phi}\kappa)}.
    \end{align*}
    Therefore, \(c_{t+1}<(1+\delta_{u})\ell_{0}=(1+\delta_{m})m\) for \(t\geq t_{K}+\Delta t_{0}\) or, equivalently,
    \begin{align*}
        t_{K+1}-t_{K}
        \leq \Delta t_{0}.
    \end{align*}
    Note that for every \(m_{t}\) schedule, if \(\mu_{K}\) denotes the last value of \(m_{t}\) before reaching \([m/\gamma,(1+\delta_{m})m]\), then \(\mu_{K}\geq (1+\delta_{m})m\), by definition.
    Hence, letting \(\ell_{0}'=\mu_{K}/(1+\delta_{u})\) and \(I_{0}'=\{i:\lambda_{i}\geq \ell_{0}'\}\), then \(I_{0}'\geq I_{0}\), and it follows that
    \begin{align*}
        \sum_{i=I_{0}'}^{d}\frac{(x_{t+1,i}-x_{t,i})^{2}}{(x_{t+1,1}-x_{t,1})^{2}}
        &\leq 
        2M_{1}\sum_{i=I_{0}'}^{d}
        \frac{x_{0,i}^{2}}{x_{0,1}^{2}}
        \prod_{k=t_{K}}^{t}\frac{\rho(m_{k},\lambda_{i})^{2}}{\rho(m_{k},m)^{2}}
        \prod_{k=1}^{\tau_{K}}\frac{\rho(m_{k},\lambda_{i})^{2}}{\rho(m_{k},m)^{2}}
        \\
        &\leq 
        2M_{1}\sum_{i=I_{0}}^{d}
        \frac{x_{0,i}^{2}}{x_{0,1}^{2}}
        \prod_{k=t_{K}}^{t}\frac{\rho(m_{k},\lambda_{i})^{2}}{\rho(m_{k},m)^{2}}
        \prod_{k=1}^{\tau_{K}}\frac{\rho(m_{k},\lambda_{i})^{2}}{\rho(m_{k},m)^{2}}
        \\
        &\leq
        2M_{1}\omega\phi((1+\delta_{u})\ell_{0},\ell_{0},m)^{2(t-t_{K}+1)}.
    \end{align*}
    Therefore, the last adjustment cannot take more than \(\Delta t_{0}\) iterations for any \(m_{t}\) schedule.
    
    Then, let \(\Delta t_{j}\) be quantities analogous to \(\Delta t_{0}\), defined for \(j=0,\ldots,K\) as
    \begin{align*}
        \Delta t_{j}
        = 
        \frac{-\log (4\kappa^{2}M_{1}\omega/\delta_{u})}{\log \rnag(\sigma_{\phi}\kappa)}
        \frac{1}{(1 + \alpha_{\phi}(\ell_{j}-(1+\delta_{\ell})m))}.
    \end{align*}
    If \(\gamma>1\), then \(m_{t}\) decreases by a factor of at least \(\gamma\) every time it is adjusted to a new value.
    Hence, \(\mu_{K-1}\geq \gamma\mu_{K}\) for every \(m_{t}\) schedule, which implies that \(\ell_{1}\leq \mu_{K-1}/(1+\delta_{u})\) for every \(m_{t}\) schedule.
    Hence, by the same rationale above, it cannot take more than \(\Delta t_{1}\) for \(m_{t}\) to be adjusted to its second last value before reaching the interval \([m/\gamma,(1+\delta_{m})m]\).
    It follows by induction that it cannot take more than \(\Delta t_{j}\) for \(m_{t}\) to be adjust to its \(K-j\)-th to last value before reaching the interval \([m/\gamma,(1+\delta_{m})m]\).
    Moreover, since by design \(m_{t}\leq L\), it cannot more than \(K\leq \log_{\gamma}(\kappa/(1+\delta_{m}))\) adjustments before \(m_{t}\) reaches the interval \([m/\gamma,(1+\delta_{m})m]\).
    Therefore, letting
    \begin{align}
        \nu
        = \sum_{j=0}^{+\infty}
        \frac{1}{1 + \alpha_{\phi}m(1+\delta_{\ell})(\gamma^{j}-1)},
        \label{la:def:nu}
    \end{align}
    we conclude that \(m_{t+1}\leq (1+\delta_{m})m\) for all \(t\geq \tau\), where
    \begin{align*}
        \tau 
        = 
        \nu
        \frac{-\log(4\kappa^{2}M_{1}\omega/\delta_{u})}{\log r'}
        \geq 
        \frac{-\log(4\kappa^{2}M_{1}\omega/\delta_{u})}{\log \rnag(\sigma_{\phi}\kappa)}
        \sum_{j=0}^{K}
        \frac{1}{1 + \alpha_{\phi}m(1+\delta_{\ell})(\gamma^{j}-1)},
    \end{align*}
    because \(\log\) is monotone and \(\rnag(\sigma_{\phi}\kappa) \leq r' < 1\), and \(1 + \alpha_{\phi}m(1+\delta_{\ell})(\gamma^{j}-1) > 0\).
\end{proof}

\subsection*{Main Result in the Quadratic Case}

We now prove the main local convergence result for NAG-free when the objective function is quadratic.
There is no difference between local and global convergence in this case, but it will be the foundation to derive the main local convergence in the general case later.
To this end, we first establish that for every \(G_{i}(m_{t})\), there is a quadratic Lyapunov function certifying convergence of \(X_{t,i}\) at rate \(\rho(m_{t},\lambda_{i})\) up to arbitrary precision, at the expense of worse condition numbers.

\begin{lemma}\label{la:lem:lyap-eps}
    Let \(m_{t}\in \mathopen{[}m/\gamma,L\mathclose{]}\) for some \(\gamma>1\), and let \(\rho(G_{i}(m_{t}))\) denote the spectral radius of \(G_{i}(m_{t})\).
    Then, given \(r\in \mathopen{[}\rho(G_{i}(m_{t})),1\mathclose{)}\) and \(\delta>0\) such that \((1+\delta)r<1\), there is some \(P=P(G_{i}(m_{t}),r,\delta) \in \mathbb{R}^{d\times d}\) such that \(G_{i}(m_{t})^{\T}PG_{i}(m_{t}) \prec (1+\delta)^{2}r^{2}P\) and \(P \succeq I\).
    Moreover, letting \(\lambda_{\min}(P)\) and \(\lambda_{\max}(P)\) denote the least and the greatest eigenvalues of \(P\), then
    \begin{align}
        \max_{m_{t}\in \mathopen{[}m/\gamma,L\mathclose{]}}\Vert P(G_{i}(m_{t}), r, \delta) \Vert
        <\ \frac{1 + (1+\delta)^{-2}}{1-(1+\delta)^{-2}}
        + \frac{2M_{2}^{2}}{(1+\delta)^{2}r^{2}}
        \frac{1+(1+\delta)^{-2}}{(1-(1+\delta)^{-2})^{3}},
        \label{la:ineq:max-cond-P-bound}
    \end{align}
    where \(M_{2}\) is an appropriate constant that does not depend on neither \(\delta\) nor \(r\).
\end{lemma}

\begin{proof}
    By \cref{la:lem:rho-lt-1}, \(\rho(m_{t},\lambda_{i})<1\) for all \(m_{t}\in \mathopen{(}0,L\mathclose{]}\) and \(i=1,\ldots,d\).
    Thus, \(\rho(G_{i}(m_{t}))<1\), where \(\rho(G_{i}(m_{t}))\) denotes the spectral radius of \(G_{i}(m_{t})\).
    Therefore, the interval \(\mathopen{[}\rho(G_{i}(m_{t})),1\mathclose{)}\) is nonempty.
    So, let \(r\) and \(\delta\) be two positive numbers such that \(r\in \mathopen{[}\rho(G_{i}(m_{t})),1\mathclose{)}\) and \((1+\delta)r<1\).
    Then, take \(r_{\delta}=(1+\delta)r\) and \(P=\sum_{k=0}^{+\infty}(G_{i}(m_{t})^{\T}/r_{\delta})^{k}(G_{i}(m_{t})/r_{\delta})^{k}\).
    The matrix \(P\) is well-defined because \(\rho(G_{i}(m_{t})/r_{\delta})\leq 1/(1+\delta)<1\), and \(P \succeq I\), by construction.
    Moreover, it satisfies
    \begin{align*}
        (G_{i}(m_{t}))/r_{\delta})^{\T}P(G_{i}(m_{t})/r_{\delta})
        = \sum_{k=1}^{+\infty}(G_{i}(m_{t}))^{\T}/r_{\delta})^{k}(G_{i}(m_{t}))/r_{\delta})^{k}
        = P - I.
    \end{align*}
    Therefore, \(G_{i}(m_{t})^{\T}PG_{i}(m_{t}) \prec (1+\delta)^{2}r^{2}P\), which proves the first claim.
    
    To prove the second claim, we first express \(G_{i}(m_{t})\) in Schur form \cite[section~7.1.3]{Golub2013}.
    To this end, we construct a two-by-two orthogonal matrix \(Q_{i}(m_{t})\), whose first column is a unit eigenvector \(q_{1,i}\) associated with \(\zeta_{i}=\zeta_{i}(m_{t})\), the top eigenvalue of \(G_{i}(m_{t})\), as in
    \begin{align*}
        q_{1,i} 
        = \frac{1}{\sqrt{1 + \vert \zeta_{i} \vert^{2}}}
        \begin{bmatrix}
            1 \\ \zeta_{i}
        \end{bmatrix}.
    \end{align*}
    To determine the second column of \(Q_{i}(m_{t})\), we apply the Gram-Schmidt orthogonalization procedure \cite[section~5.2.7]{Golub2013} to obtain from \(e_{1}\) a vector orthonormal to \(q_{1,i}\):
    \begin{align*}
        e_{1} 
        -\frac{\langle e_{1}, q_{1,i} \rangle}{\langle q_{1,i}, q_{1,i} \rangle}q_{1,i}
        = \begin{bmatrix}
            1 
            \\ 
            0
        \end{bmatrix}
        -\frac{1}{1 + \vert\zeta_{i}\vert^{2}}
        \begin{bmatrix}
            1 
            \\ 
            \zeta_{i}
        \end{bmatrix}
        = \frac{1}{1 + \vert\zeta_{i}\vert^{2}}
        \begin{bmatrix}
            \vert \zeta_{i} \vert^{2}
            \\
            -\zeta_{i}
        \end{bmatrix}.
    \end{align*}
    Normalizing the vector above, we obtain
    \begin{align*}
        q_{2,i}
        = \frac{1}{\sqrt{1 + \vert\zeta_{i}\vert^{2}}}
        \frac{1}{\vert \zeta_{i} \vert}
        \begin{bmatrix}
            \vert \zeta_{i} \vert^{2}
            \\
            -\zeta_{i}
        \end{bmatrix}.
    \end{align*}
    So, letting \(Q_{i}(m_{t})\) be the orthogonal matrix given by
    \begin{align}
        Q_{i}(m_{t})
        = 
        \begin{bmatrix} q_{1,i} & q_{2,i} \end{bmatrix},
        \label{la:def:Qi}
    \end{align}
    and letting \(T_{i}(m_{t})\) be the matrix given by
    \begin{align}
        T_{i}(m_{t})
        =
        Q_{i}(m_{t})^{\ct}G_{i}(m_{t})Q_{i}(m_{t}),
        \label{la:def:Ti}
    \end{align}
    where \(Q_{i}(m_{t})^{\ct}\) denotes the conjugate-transpose of \(Q_{i}(m_{t})\), it follows that
    \begin{align*}
        T_{i}(m_{t})
        &= \begin{bmatrix}
            q_{1,i}(m_{t})^{\ct}G_{i}(m_{t})q_{1,i}(m_{t}) & q_{1,i}(m_{t})^{\ct}G_{i}(m_{t})q_{2,i}(m_{t})
            \\
            q_{2,i}(m_{t})^{\ct}G_{i}(m_{t})q_{1,i}(m_{t}) & q_{2,i}(m_{t})^{\ct}G_{i}(m_{t})q_{2,i}(m_{t})
        \end{bmatrix}
        \\
        &= \begin{bmatrix}
            \zeta_{i} & q_{1,i}(m_{t})^{\ct}G_{i}(m_{t})q_{2,i}(m_{t})
            \\
            \zeta_{i}q_{2,i}(m_{t})^{\ct}q_{1,i}(m_{t}) & q_{2,i}(m_{t})^{\ct}G_{i}(m_{t})q_{2,i}(m_{t})
        \end{bmatrix}
        \\
        &= \begin{bmatrix}
            \zeta_{i} & q_{1,i}(m_{t})^{\ct}G_{i}(m_{t})q_{2,i}(m_{t})
            \\
            0 & q_{2,i}(m_{t})^{\ct}G_{i}(m_{t})q_{2,i}(m_{t})
        \end{bmatrix}
    \end{align*}
    because \(q_{1,i}(m_{t})\) is a unit eigenvector of \(G_{i}(m_{t})\) associated with \(\zeta_{i}\), and is orthogonal to \(q_{2,i}(m_{t})\).
    Moreover, the product \(Q_{i}(m_{t})^{\ct}G_{i}(m_{t})Q_{i}(m_{t})\) preserves the eigenvalues of \(G_{i}(m_{t})\) because \(Q_{i}(m_{t})\) is orthogonal, therefore
    \begin{align}
        T_{i}(m_{t})
        = \begin{bmatrix}
            \zeta_{i} & q_{1,i}(m_{t})^{\ct}G_{i}(m_{t})q_{2,i}(m_{t})
            \\
            0 & \xi_{i}
        \end{bmatrix},
        \label{la:eq:Ti}
    \end{align}
    where \(\xi_{i}\) denotes the other eigenvalue of \(G_{i}(m_{t})\).
    Now, \(G_{i}(m_{t})\) and, therefore, \(Q_{i}(m_{t})\) are continuous functions of \(m_{t}\), thus
    \begin{align*}
        M_{2}
        = \max_{m_{t}\in \mathopen{[}m/\gamma,L\mathclose{]}} q_{1,i}(m_{t})^{\ct}G_{i}(m_{t})q_{2,i}(m_{t})
        < +\infty
    \end{align*}
    is well-defined.
    Moreover, left-multiplying and right-multiplying \eqref{la:eq:Ti} by \(Q_{i}(m_{t})\) and \(Q_{i}(m_{t})^{\ct}\), respectively, yields
    \begin{align*}
        Q_{i}(m_{t})T_{i}(m_{t})Q_{i}(m_{t})^{\ct}
        = T_{i}(m_{t}).
    \end{align*}
    Substituting the above for \(G_{i}(m_{t})\), using submultiplicativity of the Euclidean norm, the fact that \(Q_{i}(m_{t})\) are orthogonal, and the fact that \(\rho(m_{t},\lambda_{i})/r_{\delta} \leq 1/(1+\delta)\), where \(r_{\delta}=(1+\delta)r\), we get
    \begin{align*}
        \Vert P \Vert
        &\leq 
        \sum_{k=0}^{+\infty}
        \Vert ((G_{i}(m_{t})/r_{\delta})^{k})^{\T} (G_{i}(m_{t}^{k})/r_{\delta}) \Vert
        \\
        &\leq
        \sum_{k=0}^{+\infty}
        \Vert Q_{i}(m_{t})(T_{i}(m_{t})/r_{\delta})^{k}Q_{i}(m_{t})^{\ct} \Vert^{2}
        \\
        &\leq 
        1
        + \sum_{k=0}^{+\infty}
        \Vert Q_{i}(m_{t})^{\ct} \Vert^{2}
        \Vert T_{i}(m_{t})/r_{\delta} \Vert^{2k}
        \Vert Q_{i}(m_{t}) \Vert^{2}
        \\
        &\leq 
        1
        + \sum_{k=0}^{+\infty}
        r_{\delta}^{-2(k+1)}
        \biggl\Vert
        \begin{bmatrix}
            \rho(m_{t},\lambda_{i})^{k+1} & (k+1)M_{2}\rho(m_{t},\lambda_{i})^{k}
            \\
            0 & \rho(m_{t},\lambda_{i})^{k+1}
        \end{bmatrix}
        \biggr\Vert^{2}
        \\
        &\leq 
        1
        + \sum_{k=0}^{+\infty}
        r_{\delta}^{-2(k+1)}
        \Bigl(
        \rho(m_{t},\lambda_{i})^{k+1}
        + (k+1)M_{2}\rho(m_{t},\lambda_{i})^{k}
        \Bigr)^{2}
        \\
        &=
        1
        + \sum_{k=0}^{+\infty}
        \biggl(
        (1+\delta)^{-(k+1)} + \frac{M_{2}}{r_{\delta}}(k+1)(1+\delta)^{-k}
        \biggr)^{2}.
    \end{align*}
    Then, using the fact that \((a+b)^{2}\leq 2a^{2} + 2b^{2}\) for any \(a\) and \(b\), yields
    \begin{align*}
        \Vert P \Vert
        &\leq 
        1 
        + \sum_{k=0}^{+\infty}
        \biggl(
        2(1+\delta)^{-2(k+1)} + \frac{2M_{2}^{2}}{r_{\delta}^{2}}(k+1)^{2}(1+\delta)^{-2k}
        \biggr)
        \\
        &\leq
        1
        + \frac{2(1+\delta)^{-2}}{1-(1+\delta)^{-2}}
        + \frac{2M_{2}^{2}}{r_{\delta}^{2}}
        \frac{1+(1+\delta)^{-2}}{(1-(1+\delta)^{-2})^{3}}
        \\
        &=
        \frac{1+(1+\delta)^{-2}}{1-(1+\delta)^{-2}}
        + \frac{2M_{2}^{2}}{r_{\delta}^{2}}
        \frac{1+(1+\delta)^{-2}}{(1-(1+\delta)^{-2})^{3}},
    \end{align*}
    where the second inequality follows by noting that for any \(\alpha\) such that \(0<\alpha<1\), we have that
    \begin{align*}
        \sum_{k=0}^{+\infty}(k+1)^{2}\alpha^{k}
        = \frac{1}{\alpha}\sum_{k=1}^{+\infty}k^{2}\alpha^{k}
        = \frac{1}{\alpha}\frac{\alpha(1+\alpha)}{(1-\alpha)^{3}}
        = \frac{1+\alpha}{(1-\alpha)^{3}},
    \end{align*}
    and then plugging \((1+\delta)^{-2}\) into \(\alpha\).
\end{proof}

\begin{proposition}
\label{la:prop:main-quad}
    Let \(f\in\mathcal{S}(L,m)\) be a quadratic function with \(\kappa = L/m \geq 2\), and let \(L_{0}>L\).
    Suppose that \cref{ass:xi0-lower-bound} holds for some \(\omega>0\), and that \cref{ass:mt-separation-from-eigs} holds as well.
    Also, let \(\delta_{m}\) and \(\delta_{u}\) be positive numbers such that \(\delta_{u} < \min\{\delta_{m}, 1/2\}\) and \(\delta_{m}\leq \gamma-1\).
    If \cref{alg:nag-free} is initialized with \(L_{0}\) as input, then its iterates \(x_{t}\) satisfy
    \begin{align}
        \Vert x_{t+1} - x^{\star} \Vert
        \leq C\bar{\kappa}^{7/2 + 2\nu}\rnag(2\sigma\bar{\kappa})^{t} \Vert x_{0} - x^{\star} \Vert,
        \label{la:ineq:main-quad}
    \end{align}
    where \(\bar{\kappa}=L_{0}/m > \kappa\), \(\sigma = \max\{\gamma,\sigma_{m},\sigma_{\phi}\}\), \(\sigma_{m}=1+2\delta_{m}+2\sqrt{\delta_{m}(1+\delta_{m})}\), \(\sigma_{\phi}=\sigma_{\phi}(\delta_{u},\delta_{\ell},\bar{\kappa})\) is a function of \(\delta_{\ell}=(1+\delta_{m})/(1+\delta_{u})-1\) and is bounded in \(\bar{\kappa}\geq 1 + \delta_{\ell}\), such that
    \begin{align*}
        \lim_{\bar{\kappa} \to + \infty}\sigma_{\phi}(\delta_{u}, \delta_{\ell}, \bar{\kappa}) 
        = \frac{1}{4(\sqrt{\delta_{u}(1+\delta_{\ell}) + \delta_{\ell}} - \sqrt{\delta_{u}(1+\delta_{\ell})})^{2}},
    \end{align*}
    and \(C\) and \(\nu\) are constants that depend on \(\gamma, \delta_{u}, \sigma\) and \(\omega\).
\end{proposition}

\begin{proof}
    Let \(r'=r'(\delta_{u},\delta_{\ell},\kappa)\) be a function such that \(\rnag(\sigma_{\phi}\kappa) \leq r' < 1\) for all \(\kappa\geq 1+\delta_{\ell}\), where \(\sigma_{\phi}=\sigma_{\phi}(\delta_{u}, \delta_{\ell}, \kappa)\) is given by \cref{la:lem:rphi-sq-leq-racc}.
    Then, by \cref{la:lem:ct} we have that \(m_{t}\leq (1+\delta_{m})m\) for all \(t\geq \tau\), where 
    \begin{align}
        \tau
        =
        \nu\frac{-\log (4\kappa^{2}M_{1}\omega/\delta_{u})}{\log r'},
        \label{la:ineq:tau-prime}
    \end{align}
    \(\nu\) is given by \eqref{la:def:nu} and \(M_{1}=\max_{i}\overline{C}_{i}/\underline{C}_{1}\).
    
    By design, we have that \(m_{t}\geq m/\gamma\).
    If \(m_{t}<m\), then by \cref{la:lem:rho-expression}, and unsing the fact that \((\sqrt{L/m_{t}}-1)/(\sqrt{L/m_{t}}+1)\) is decreasing in \(m_{t}\) and \((\kappa-1)/\kappa\) is increasing in \(\kappa\), we get
    \begin{align*}
        \rho(m_{t},m)
        = \sqrt{\frac{\sqrt{L/m_{t}}-1}{\sqrt{L/m_{t}}+1}\frac{\kappa-1}{\kappa}}
        \leq \sqrt{\frac{\sqrt{\gamma\kappa}-1}{\sqrt{\gamma\kappa}+1}\frac{\gamma\kappa-1}{\gamma\kappa}}
        = \rnag(\gamma\kappa).
    \end{align*}
    Otherwise, if \(m_{t} \in \mathopen{[}m,(1+\delta_{m})m\mathclose{]}\), then by \cref{la:cor:rho-leq-racc}, we have that \(\rho(m_{t},m)\leq \rnag(\sigma_{m}\kappa)\) for all \(\kappa\geq 1+ \delta_{m}\), where \(\sigma_{m}=\sigma_{m}(\delta_{m})=1+2\delta_{m}+2\sqrt{\delta_{m}(1+\delta_{m})}\).
    Hence, \(\rho(m_{t},m) \leq \rnag(\sigma_{1}\kappa)\) for all \(t\geq \tau\), where \(\sigma_{1}=\max\{\gamma,\sigma_{m}\}\).
    
    Now, by \cref{la:lem:top-rho}, we have that \(\rho(m_{t},\lambda_{i}) \leq \rnag(\sigma_{1}\kappa)\) for all \(\lambda_{i}\).
    Hence, given \(\delta_{\sigma}\) such that \((1+\delta_{\sigma})\rnag(\sigma_{1}\kappa)<1\), by \cref{la:lem:lyap-eps} there is a \(P_{i}(m_{t})=P_{i}(m_{t},\delta_{\sigma})\succeq I\) for each \(\lambda_{i}\) such that \(G_{i}(m_{t})^{\T}P_{i}(m_{t})G_{i}(m_{t})\preceq (1+\delta_{\sigma})^{2}\rnag(\sigma_{1}\kappa)^{2}P_{i}(m_{t})\).
    Hence, if \(t_{j} \leq t < t_{j+1}\) and \(t \geq \tau\), then
    \begin{align*}
        X_{t+1,i}^{\T}P_{i}(m_{t})X_{t+1,i}
        &= X_{t,i}^{\T}G_{i}(\mu_{j})^{\T}P_{i}(\mu_{j})G_{i}(\mu_{j})X_{t,i}
        \\
        &\leq (1+\delta_{\sigma})^{2}\rnag(\sigma_{1}\kappa)^{2}X_{t,i}^{\T}P_{i}(\mu_{j})X_{t,i},
    \end{align*}
    since \(m_{t}=\mu_{j}\).
    Consecutively applying this inequality, we obtain
    \begin{align*}
        \lambda_{\min}(P_{i}(m_{t}))\Vert X_{t,i} \Vert^{2}
        \leq X_{t,i}^{\T}P_{i}(m_{t})X_{t,i}
        &\leq ((1+\delta_{\sigma})\rnag(\sigma_{1}\kappa))^{2(t-t_{j})}X_{t_{j},i}^{\T}P_{i}(\mu_{j})X_{t_{j},i}
        \\
        &\leq \lambda_{\max}(P_{i}(\mu_{j})) ((1+\delta_{\sigma})\rnag(\sigma_{1}\kappa))^{2(t-t_{j})} \Vert X_{t_{j},i} \Vert^{2}.
    \end{align*}
    Rearranging the above yields
    \begin{align*}
        \Vert X_{t,i} \Vert^{2}
        \leq \frac{\lambda_{\max}(P_{i}(\mu_{j}))}{\lambda_{\min}(P_{i}(\mu_{j}))}
        ((1+\delta_{\sigma})\rnag(\sigma_{1}\kappa))^{2(t-t_{j})} \Vert X_{t_{j},i} \Vert^{2}.
    \end{align*}
    Moreover, since by assumption \(1+\delta_{m} < \gamma\), \(m_{t}\) is adjusted at most once if \(m_{t}\leq (1+\delta_{m})m\), therefore denoting by \(\mu_{-1}\) and \(\mu_{-2}\) respectively the last and before last values taken by \(m_{t}\), for all \(t\geq \tau\) we have that
    \begin{align*}
        \Vert X_{t,i} \Vert^{2}
        \leq 
        \frac{\lambda_{\max}(P_{i}(\mu_{-1}))}{\lambda_{\min}(P_{i}(\mu_{-1}))}
        \frac{\lambda_{\max}(P_{i}(\mu_{-2}))}{\lambda_{\min}(P_{i}(\mu_{-2}))}
        ((1+\delta_{\sigma})\rnag(\sigma_{1}\kappa))^{2(t-\ceil{\tau})} \Vert X_{\ceil{\tau},i} \Vert^{2},
    \end{align*}
    In turn, the above bound yields
    \begin{align*}
        \Vert X_{t} \Vert^{2}
        &= \sum_{i=1}^{d}\Vert X_{t,i} \Vert^{2}
        \\
        &\leq 
        \sum_{i=1}^{d} 
        \frac{\lambda_{\max}(P_{i}(\mu_{-1}))}{\lambda_{\min}(P_{i}(\mu_{-1}))}
        \frac{\lambda_{\max}(P_{i}(\mu_{-2}))}{\lambda_{\min}(P_{i}(\mu_{-2}))}
        ((1+\delta_{\sigma})\rnag(\sigma_{1}\kappa))^{2(t-\ceil{\tau})} \Vert X_{\ceil{\tau},i} \Vert^{2}.
    \end{align*}
    Since \(\rnag\) is monotone and \(\sum_{i=1}^{d}\Vert X_{t,i} \Vert^{2} = \Vert X_{t} \Vert^{2}\), defining \(\sigma = \max\{\gamma, \sigma_{m}, \sigma_{\phi}\}\), it follows that
    \begin{align}
        \Vert X_{t} \Vert^{2}
        \leq 
        M_{3}^{2}
        ((1+\delta_{\sigma})\rnag(\sigma\kappa))^{2(t-\ceil{\tau})} 
        \Vert X_{\ceil{\tau}} \Vert^{2},
        \label{la:ineq:xt-norm-sq-ub}
    \end{align}
    where \(M_{3}^{2}\) is given by the product of the worst condition numbers of all \(P_{i}(\mu_{-1})\) and \(P_{i}(\mu_{-2})\):
    \begin{align*}
        M_{3}
        = 
        \sqrt{\max_{i=1,\ldots,d}\frac{\lambda_{\max}(P_{i}(\mu_{-1}))}{\lambda_{\min}(P_{i}(\mu_{-1}))}
        \max_{i=1,\ldots,d}\frac{\lambda_{\max}(P_{i}(\mu_{-2}))}{\lambda_{\min}(P_{i}(\mu_{-2}))}}.
    \end{align*}
    
    Plugging \(r' = (1+\delta_{\sigma})\rnag(\sigma\kappa)\) in \eqref{la:ineq:tau-prime}, it follows that \(m_{t}\in\mathopen{[}m/\gamma,(1+\delta_{m})m\mathclose{]}\) for all \(t\geq \tau\), where
    \begin{align*}
        \tau
        &= 
        \nu\frac{-\log(4\kappa^{2}M_{1}\omega/\delta_{u})}{\log (1+\delta_{\sigma})\rnag(\sigma\kappa)}.
    \end{align*}
    Therefore, since \(\rnag(\sigma\kappa) \geq \rnag(\kappa) \geq (\sqrt{2}-1)/\sqrt{2} \geq 1/4\) and \(\ceil{\tau} \leq \tau + 1\), it follows that
    \begin{align}
        ((1+\delta_{\sigma})\rnag(\sigma\kappa))^{-\ceil{\tau}}
        \leq 4M_{4},
        \label{la:ineq:racc-pow-tau}
    \end{align}
    for a constant \(M_{4}\) given by
    \begin{align*}
        M_{4} = (4\kappa^{2}M_{1}\omega/\delta_{u})^{\nu}.
    \end{align*}
    Then, plugging \eqref{la:ineq:racc-pow-tau} into \eqref{la:ineq:xt-norm-sq-ub}, we obtain
    \begin{align*}
        \Vert x_{t} \Vert
        \leq 
        \Vert X_{t} \Vert
        \leq 
        4M_{3}M_{4}
        ((1+\delta_{\sigma})\rnag(\sigma\kappa))^{t}
        \Vert X_{\ceil{\tau}} \Vert.
    \end{align*}
    To establish \eqref{la:ineq:main-quad}, it remains to bound \(\Vert X_{\ceil{\tau}} \Vert\).
    To this end, we plug \(x=x^{\star}\) and \(y=y_{t+1}\) into \eqref{ineq:strong_convexity}, and then use the global convergence bound from \cref{thm:gc} to get
    \begin{align*}
        \Vert y_{t+1} - x^{\star} \Vert^{2}
        \leq (2/m)( f(y_{t+1}) - f(x^{\star}) )
        \leq\rgd(\kappa)^{t} 16 \kappa^{4}\Vert x_{0} - x^{\star} \Vert^{2}.
    \end{align*}
    Then, substituting \(x_{t+1}\) with its definition from \cref{alg:nag-free}, summing \(\pm\beta_{t}x^{\star} = 0\), using the above bound and then the fact that \(\beta_{t}\in \mathopen{[}0,1\mathclose{)}\) and that \(\rgd\in \mathopen{(}0,1\mathclose{)}\), we obtain
    \begin{align}
        \Vert x_{t+1} - x^{\star} \Vert^{2}
        = \Vert (1+\beta_{t})y_{t+1} -\beta_{t}y_{t} - x^{\star} \pm \beta_{t}x^{\star}\Vert^{2}
        &= \Vert (1+\beta_{t})(y_{t+1}-x^{\star}) -\beta_{t}(y_{t} - x^{\star})\Vert^{2}
        \nonumber\\
        &\leq (2\Vert y_{t+1} - x^{\star} \Vert + \Vert y_{t} - x^{\star} \Vert)^{2}
        \nonumber\\
        &\leq \rgd(\kappa)^{t-1}
        144 \kappa^{4}
        \Vert x_{0} - x^{\star} \Vert^{2},
        \label{la:ineq:xt-norm-gd-bound}
    \end{align}
    which implies that
    \begin{align*}
        \Vert X_{\ceil{\tau}} \Vert
        \leq \Vert x_{\ceil{\tau}} - x^{\star} \Vert 
        + \Vert x_{\ceil{\tau}-1} - x^{\star} \Vert
        &\leq 
        \rgd(\kappa)^{(\ceil{\tau}-3)/2}
        24 \kappa^{2} \Vert x_{0} - x^{\star} \Vert
        \\
        &\leq 
        48\sqrt{2} \kappa^{2} \Vert x_{0} - x^{\star} \Vert,
    \end{align*}
    since \(r_{\textup{GD}}(\kappa)\geq r_{\textup{GD}}(2)=2\).
    Combining the two bounds above yields
    \begin{align*}
        ((1+\delta_{\sigma})\rnag(\sigma\kappa))^{-\ceil{\tau}}
        \rgd(\kappa)^{-3/2}
        \leq 3\cdot 2^{6} \sqrt{2} M_{4}
        \leq 5\cdot 2^{6} M_{4}.
    \end{align*}
    Therefore, we have that
    \begin{align}
        \Vert x_{t+1} - x^{\star} \Vert
        \leq 
        5\cdot 2^{6} M_{3}M_{4}\sqrt{\kappa}
        ((1+\delta_{\sigma})\rnag(\sigma\kappa))^{t} \Vert x_{0} - x^{\star} \Vert.
        \label{la:ineq:main-quad-aux}
    \end{align}
    Our next step is to express the rate \((1+\delta_{\sigma})\rnag(\sigma\kappa)\) in terms of \(\rnag(\sigma_{2}\kappa)\) for some \(\sigma_{2}\), as in
    \begin{align*}
        (1+\delta_{\sigma})\rnag(\sigma\kappa)
        = (1+\delta_{\sigma})\frac{\sqrt{\sigma\kappa}-1}{\sqrt{\sigma\kappa}}
        = \frac{\sqrt{\sigma_{2}\kappa}-1}{\sqrt{\sigma_{2}\kappa}}.
    \end{align*}
    Solving the above identity for \(\sigma_{2}\), we obtain
    \begin{align*}
        \sigma_{2}
        = \frac{\sigma}{(1+\delta_{\sigma}-\delta_{\sigma}\sqrt{\sigma\kappa})^{2}}
        \leq \frac{\sigma}{(1-\delta_{\sigma}\sqrt{\sigma\kappa})^{2}}.
    \end{align*}
    That is, \(\sigma_{2}=(1+\delta)\sigma\), where
    \begin{align*}
        \delta 
        = \frac{\delta_{\sigma}\sqrt{\sigma\kappa}(2-\delta_{\sigma}\sqrt{\sigma\kappa})}{(1-\delta_{\sigma}\sqrt{\sigma\kappa})^{2}}.
    \end{align*}
    So, if \(\delta_{\sigma}=1/(4\sqrt{\sigma\kappa})\), then \(\delta\leq 7/9\), which implies that \(1 + \delta \leq 2\) and
    \begin{align*}
        (1+\delta_{\sigma})\rnag(\sigma\kappa)
        \leq \rnag(2\sigma\kappa).
    \end{align*}
    Moreover, since \(\sigma \geq 2\) and \(\kappa \geq 2\), it follows that \(\delta_{\sigma}=1/(4\sqrt{\sigma\kappa}) \leq 1/8\) and
    \begin{align*}
        \frac{1}{1-(1+\delta_{\sigma})^{-2}}
        &= \frac{(1+\delta_{\sigma})^{2}}{\delta_{\sigma}(2+\delta_{\sigma})}
        \leq \frac{(1+1/8)^{2}}{2}\frac{1}{\delta_{\sigma}}
        = \frac{9^2}{2 \cdot 8^2}\frac{1}{\delta_{\sigma}},
        \\
        1+(1+\delta_{\sigma})^{-2}
        &= \frac{2+\delta_{\sigma}(2+\delta_{\sigma})}{1+\delta_{\sigma}}
        \leq 2(1+\delta_{\sigma}).
    \end{align*}
    In the same vein, using the fact that \(\lambda_{\min}(P_{i}(\mu_{-1}))\geq 1 \) and that \(\lambda_{\max}(P_{i}(\mu_{-1}))\leq \Vert P_{i}(\mu_{-1})\Vert\), plugging \(\delta_{\sigma}=1/(4\sqrt{\sigma\kappa})\) into \eqref{la:ineq:max-cond-P-bound} and using the fact that \(\rnag(\sigma\kappa)\geq \rnag(4)=1/2\) yields
    \begin{align*}
        \max_{i=1,\ldots,d}\frac{\lambda_{\max}(P_{i}(\mu_{-1}))}{\lambda_{\min}(P_{i}(\mu_{-1}))}
        &\leq \max_{i=1,\ldots,d} \Vert P_{i}(\mu_{-1})\Vert
        \\
        &< 1
        + 2\frac{(1+\delta_{\sigma})^{-2}}{1-(1+\delta_{\sigma})^{-2}}
        + 2\frac{M_{2}^{2}}{(1+\delta_{\sigma})^{2}\rnag(\sigma\kappa)^{2}}
        \frac{1+(1+\delta_{\sigma})^{-2}}{(1-(1+\delta_{\sigma})^{-2})^{3}}
        \\
        &< 1 
        + \frac{9^2}{2 \cdot 8^2}\frac{(1+\delta_{\sigma})^{-2}}{\delta_{\sigma}}
        + 4\frac{M_{2}^{2}}{(1+\delta_{\sigma})^{2}\rnag(\sigma\kappa)^{2}}
        \frac{1+\delta_{\sigma}}{\delta_{\sigma}^{3}}
        \\
        &< 
        \biggl(
        \frac{1}{8^{3}}
        + \frac{9^{2}}{2 \cdot 8^{4}}
        + \frac{4}{\rnag(\sigma\kappa)^{2}}
        \biggr)
        \frac{M_{2}^{2}}{\delta_{\sigma}^{3}}
        \\
        &< 18 \frac{M_{2}^{2}}{\delta_{\sigma}^{3}}.
    \end{align*}
    Using the above bound twice yields
    \begin{align}
        M_{3}
        = 
        \sqrt{\max_{i=1,\ldots,d}\frac{\lambda_{\max}(P_{i}(\mu_{-1}))}{\lambda_{\min}(P_{i}(\mu_{-1}))}
        \max_{i=1,\ldots,d}\frac{\lambda_{\max}(P_{i}(\mu_{-2}))}{\lambda_{\min}(P_{i}(\mu_{-2}))}}
        <
         \frac{M_{2}^{2}}{\delta_{\sigma}^{3}}
        = 18 \cdot 4^{3}M_{2}^{2}\sigma^{3/2}\kappa^{3/2}.
        \label{la:ineq:M3-bound}
    \end{align}
    Finally, we prove \eqref{la:ineq:main-quad} by plugging \eqref{la:ineq:M3-bound} into \eqref{la:ineq:main-quad-aux}, and then replacing \(\kappa\) with \(\bar{\kappa}\), so that
    \begin{align*}
        \Vert x_{t+1} - x^{\star} \Vert
        \leq C\bar{\kappa}^{7/2+2\nu}\rnag(2\sigma\bar{\kappa})^{t} \Vert x_{0} - x^{\star} \Vert,
    \end{align*}
    where the constant \(C\) is given by
    \begin{align*}
        C = 9 \cdot 2^{13} M_{2}(4M_{1}\omega/\delta_{u})^{\nu}\sigma^{3/2}.
    \end{align*}
\end{proof}

\subsection{General Case}
\label{la:general-case}

We now build on the quadratic case to prove that the iterates \(x_{t}\) of \cref{alg:nag-free} also converge to the optimum \(x^{\star}\) at an accelerated rate when the objective function \(f\) is not necessarily quadratic.
Our approach is to show that if \(x_{t}\) is sufficiently close to \(x^{\star}\), then \(x_{t}-x^{\star}\) consists of a perturbation of the iterate when the objective is given by the local quadratic approximation of \(f\) at \(x^{\star}\).

\subsection*{Iterate Dynamics in the General Case}

Under \cref{ass:locally-smooth-hessian}, it follows that \(f\) is twice continuously differentiable at \(x^{\star}\).
Hence, by Taylor's theorem \cite[theorem~2.1]{Nocedal2006}, the gradient at \(x_{t}\) can be expressed as
\begin{align}
    \nabla f(x_{t})
    =& \nabla f(x^{\star}) + \int_{0}^{1}\nabla^{2}f(x^{\star} + s(x_{t}-x^{\star}))(x_{t}-x^{\star})ds
    \nonumber\\
    =& \nabla^{2}f(x^{\star})x_{t} + \int_{0}^{1}(\nabla^{2}f(x^{\star} + s(x_{t}-x^{\star})) - \nabla^{2}f(x^{\star}))(x_{t}-x^{\star})ds
    \nonumber\\
    =& (H+\tilde{H}_{t})(x_{t}-x^{\star}),
    \label{la:eq:gradient-perturbation}
\end{align}
where the Hessian error term \(\tilde{H}_{t} = \tilde{H}_{t}(x_{t})\) is given by
\begin{align}
    \tilde{H}_{t}    
    = \int_{0}^{1}(\nabla^{2}f(x^{\star} + s(x_{t}-x^{\star})) - \nabla^{2}f(x^{\star}))ds.
    \label{la:def:hessian-error}
\end{align}
Moreover, by \eqref{la:ineq:xt-norm-gd-bound}, we have that \(\Vert x_{t+1} - x^{\star}\Vert \leq \sqrt{144\bar{\kappa}^{4} \rgd(\kappa)^{t-1}}\Vert x_{0} - x^{\star}\Vert\).
Hence, since \(\rgd(\kappa)\in \mathopen{(}0,1\mathclose{)}\), if \(\Vert x_{0} - x^{\star}\Vert \leq \epsilon^{\max(1,1/\alpha_{H})}\sqrt{\rgd(\kappa)/144\bar{\kappa}^{4}}\) and \(\epsilon \leq \delta_{H}\), then for all \(t\geq 0\), we have that \(\Vert x_{t} - x^{\star}\Vert \leq \delta_{H}\), and it follows from \cref{ass:locally-smooth-hessian} that
\begin{align}
    \Vert \tilde{H}_{t} \Vert
    &\leq \int_{0}^{1}\Vert \nabla^{2}f(x^{\star} + s(x_{t}-x^{\star})) - \nabla^{2}f(x^{\star})) \Vert ds
    \nonumber\\
    &\leq L_{H}\int_{0}^{1}s\Vert x_{t} - x^{\star} \Vert ds
    \nonumber\\
    &\leq \epsilon L_{H}\rgd(\kappa)^{t\alpha_{H}/2}.
    \label{la:ineq:hessian-error-norm-bound}
\end{align}
Since \(v_{j}\) form an eigenbasis for \(\mathbb{R}^{d}\), \(\tilde{H}_{t}v_{j}\) can be expressed in \(v_{j}\)-coordinates, \(\tilde{h}_{t,i,j}\), as
\begin{align}
    \tilde{H}_{t}v_{j} 
    = \sum_{i=1}^{d}\tilde{h}_{t,i,j}v_{i},
    && j=1,\ldots,d.
    \label{la:Ht-vj-eigendecomposition}
\end{align}
Then, using \eqref{la:Ht-vj-eigendecomposition} and the decomposition \(x_{t}-x^{\star}=\sum_{j=1}^{d}x_{j,t}v_{j}\) yields
\begin{align}
    \tilde{H}_{t}(x_{t}-x^{\star})
    = \tilde{H}_{t}\sum_{j=1}^{d}x_{j,t}v_{j}
    = \sum_{j=1}^{d}x_{j,t}\tilde{H}_{t}v_{j}
    = \sum_{j=1}^{d}x_{j,t}\sum_{i=1}^{d}\tilde{h}_{t,i,j}v_{i}
    = \sum_{i=1}^{d}\sum_{j=1}^{d}\tilde{h}_{t,i,j}x_{j,t}v_{i}.
    \label{la:eq:hessian-xt-eigendecomposition}
\end{align}
In turn, combining the decomposition \(x_{t}-x^{\star}=\sum_{j=1}^{d}x_{j,t}v_{j}\) with \labelcref{la:eq:gradient-perturbation,la:eq:hessian-xt-eigendecomposition}, we obtain
\begin{align*}
    y_{t+1}-x^{\star}
    &= x_{t} - (1/L) \nabla f(x_{t}) -x^{\star}
    \\
    &= (I - H/L - \tilde{H}_{t}/L )(x_{t}-x^{\star})
    \\
    &= \sum_{i=1}^{d}\Bigl[(1 - \lambda_{i}/L)x_{t,i} + \sum_{j=1}^{d}(\tilde{h}_{t,i,j}/L)x_{j,t}\Bigr]v_{i},
\end{align*}
from which it follows that
\begin{align*}
    \sum_{j=1}^{d}x_{j,t+1}v_{j}
    =&\ x_{t+1}-x^{\star}
    \\
    =&\ (1+\beta_{t})y_{t+1} - \beta_{t}y_{t} -x^{\star} \mp \beta_{t}x^{\star}
    \\
    =&\ \sum_{i=1}^{d}
    \biggl[ 
    (1+\beta_{t})\Bigl(1 - \frac{\lambda_{i}}{L}\Bigr)x_{t,i} -\beta_{t}\Bigl(1 - \frac{\lambda_{i}}{L}\Bigr)x_{t-1,i} 
    \\
    &+ \sum_{j=1}^{d}
    \Bigl(
    (1+\beta_{t})\frac{\tilde{h}_{t,i,j}}{L}x_{j,t} -\beta_{t}\frac{\tilde{h}_{t,i,j}}{L}x_{j,t-1}
    \Bigr)
    \biggr] 
    v_{i}.
\end{align*}
Therefore, we have that
\begin{align}
    X_{t+1}
    = (G(m_{t}) + \tilde{G}_{t}) X_{t},
    \label{la:eq:Xt-dynamics-perturbed}
\end{align}
where \(X_{t}\) is the vector with ``stacked'' \(X_{t,i}\), as in 
\begin{align}
    X_{t} = 
    \begin{bmatrix}
        X_{t,1} \\ \vdots \\ X_{d,t}
    \end{bmatrix},
    \label{la:def:Xt}
\end{align}
while \(G(m_{t})\) and \(\tilde{G}_{t}\) are matrices given by
\begin{align}
    G(m_{t})
    &= \diag(G_{1}(m_{t}), \ldots,G_{d}(m_{t})),
    \label{la:def:Gt}\\
    \tilde{G}_{t}
    &= \frac{1}{L}
    \begin{bmatrix}
        0 & 0 & \ldots & 0 & 0
        \\
        -\beta_{t}\tilde{h}_{1,1,t-1} & (1+\beta_{t})\tilde{h}_{1,1,t} & \ldots & -\beta_{t}\tilde{h}_{1,d,t-1} & (1+\beta_{t})\tilde{h}_{1,d,t}
        \\
        \vdots & \vdots & \ddots & \vdots & \vdots
        \\
        0 & 0 & \ldots & 0 & 0
        \\
        -\beta_{t}\tilde{h}_{d,1,t-1} & (1+\beta_{t})\tilde{h}_{d,1,t} & \ldots & -\beta_{t}\tilde{h}_{d,d,t-1} & (1+\beta_{t})\tilde{h}_{d,d,t}
    \end{bmatrix},
    \label{la:def:Gtildet}
\end{align}
where \(G_{i}(m_{t})\), defined by \eqref{la:def:Git}, are the system matrices governing the dynamics of each \(X_{t,i}\) in the quadratic case where \(f(x)=(x-x^{\star})^{\T}H_{t}(x-x^{\star})\).
Using the fact that \(\tilde{h}_{t,i,j}=v_{i}^{\T}\tilde{H}_{t}v_{j}\), the matrix \(\tilde{G}_{t}\) given by \eqref{la:def:Gtildet} can be expressed as
\begin{align}
    \tilde{G}_{t}
    &= 
    \frac{1+\beta_{t}}{L} W_{1}^{\T} V^{\T}\tilde{H}_{t} V W_{1}
    -\frac{\beta_{t}}{L} W_{1}^{\T} V^{\T} \tilde{H}_{t} V W_{2},
    \label{la:eq:Gtilde-decomposition}
\end{align}
where the matrices \(V\in \mathbb{R}^{d\times d}\), \(W_{1}, W_{2} \in \mathbb{R}^{d\times 2d}\) are given by
\begin{align*}
    V &= \begin{bmatrix} v_{1}^{\T} \\ \vdots \\ v_{d}^{\T} \end{bmatrix}^{\T},
    &&
    W_{1} =
    \begin{bmatrix}
        0 & 1 & 0 & 0 & \ldots & 0 & 0
        \\
        0 & 0 & 0 & 1 & \ldots & 0 & 0
        \\
        \vdots & \vdots & \vdots & \vdots & \ddots & \vdots & \vdots
        \\
        0 & 0 & 0 & 0 & \ldots & 0 & 0
        \\
        0 & 0 & 0 & 0 & \ldots & 0 & 1
    \end{bmatrix},
    &&
    W_{2}
    =
    \begin{bmatrix}
        1 & 0 & 0 & 0 & \ldots & 0 & 0
        \\
        0 & 0 & 1 & 0 & \ldots & 0 & 0
        \\
        \vdots & \vdots & \vdots & \vdots & \ddots & \vdots & \vdots
        \\
        0 & 0 & 0 & 0 & \ldots & 0 & 0
        \\
        0 & 0 & 0 & 0 & \ldots & 1 & 0
    \end{bmatrix}.
\end{align*}
Since \(v_{i}\) are orthonormal, so is \(V\).
Thus, \(V\) has unitary norm, as do \(W_{1}\) and \(W_{2}\).
Therefore, applying the triangle inequality and norm submultiplicativity to \eqref{la:eq:Gtilde-decomposition}, then using the fact that \(\beta_{t} \in \mathopen{[}0,1\mathclose{)}\) and lastly plugging in \eqref{la:ineq:hessian-error-norm-bound}, we obtain
\begin{align}
    \Vert \tilde{G}_{t} \Vert
    &\leq \frac{2}{L}\Vert W_{1}^{\T} \Vert \Vert V^{\T} \Vert \Vert \tilde{H}_{t} \Vert \Vert V \Vert \Vert W_{1} \Vert 
    + \frac{1}{L}\Vert W_{1}^{\T} \Vert \Vert V^{\T} \Vert \Vert \tilde{H}_{t} \Vert \Vert V \Vert \Vert W_{2} \Vert 
    \nonumber\\
    &= \frac{3}{L}\Vert \tilde{H}_{t} \Vert
    \nonumber\\
    &\leq \epsilon\frac{L_{H}}{L}\rgd(\kappa)^{t\alpha_{H}/2}.
    \label{la:ineq:Gtilde-norm-bound}
\end{align}

We continue by noting that if \cref{ass:mt-separation-from-eigs} holds for some \(\delta_{\lambda}>0\), then it also holds for every \(\delta_{\lambda}'<\delta_{\lambda}\).
So, without loss of generality, suppose that \cref{ass:mt-separation-from-eigs} holds for some \(\delta_{\lambda}\leq \delta_{m}m\).
Then, while \(m_{t} > (1+\delta_{m})m\), we have that \(\vert m_{t} - \lambda_{i} \vert \geq \delta_{\lambda}\) for all \(i=1,\ldots,d\).
Hence, noting that for all \(\lambda_{i}\) we have that \(\lambda_{i} < L_{0}\), then from \cref{la:cor:Gi-eig}, it follows that the two eigenvalues \(\zeta_{i}(m_{t})\) and \(\xi_{i}(m_{t})\) of each \(G_{i}(m_{t})\) are distinct.
Therefore, because \(\zeta_{i}(m_{t})\) and \(\xi_{i}(m_{t})\) are continuous in \(m_{t}\), we have that
\begin{align}
    \delta_{T}
    = \min_{m_{t}\in \mathcal{S}}
    \min_{i=1,\ldots,d}\vert \zeta_{i}(m_{t}) - \xi_{i}(m_{t}) \vert
    \label{la:def:delta_T}
    > 0,
\end{align}
where \(\mathcal{S}=\mathcal{S}(\delta_{\lambda})\) is a compact set defined in terms
\begin{align*}
    \mathcal{S}
    = \mathopen{[}(1+\delta_{m})m,L\mathclose{]}\setminus \cup_{i=1}^{d}B(\lambda_{i},\delta_{\lambda}),
\end{align*}
and \(B(\lambda_{i}, \delta_{\lambda}) = \{ x: \vert x - \lambda_{i} \vert < \delta_{\lambda} \}\) is the open ball of radius \(\delta_{\lambda}\) centered at \(\lambda_{i}\).
In the same vein, since \(T_{i}\) defined by \eqref{la:def:diagonalizing-matrix-comp} are continuous in \(\zeta_{i}\) and \(\xi_{i}\), and \(\Vert \cdot \Vert\) is continuous, it follows that
\begin{align*}
    \max_{m_{t}\in \mathcal{S}}
    \max_{i=1,\ldots,d} \Vert T_{i}(m_{t}) \Vert
    < \infty.
\end{align*}
Furthermore, explicitly computing the inverse of \(T_{i}\) for \(m_{t}\in\mathcal{S}\) yields
\begin{align}
    \Vert T_{i}(m_{t})^{-1} \Vert
    = \frac{1}{\vert \zeta_{i}-\xi_{i} \vert}
    \biggl\Vert
    \begin{bmatrix}
        \xi_{i} & -1
        \\
        -\zeta_{i} & 1
    \end{bmatrix}
    \biggr\Vert
    \leq \frac{1}{\delta_{T}}
    \biggl\Vert
    \begin{bmatrix}
        \xi_{i} & -1
        \\
        -\zeta_{i} & 1
    \end{bmatrix}
    \biggr\Vert
    \label{la:ineq:inv-diagonalizing-matrix-norm}.
\end{align}
Hence, since both sides of \eqref{la:ineq:inv-diagonalizing-matrix-norm} are continuous in \(m_{t}\), it follows that
\begin{align*}
    \max_{m_{t}\in \mathcal{S}}
    \max_{i=1,\ldots,d} \Vert T_{i}(m_{t})^{-1} \Vert
    < \infty.
\end{align*}
Therefore, we have that
\begin{align}
    M_{T}
    =
    \max_{m_{t}\in \mathcal{S}}
    \max_{i=1,\ldots,d} \Vert T_{i}(m_{t}) \Vert \Vert T_{i}(m_{t})^{-1} \Vert
    < + \infty.
    \label{la:ineq:diagonalizing-matrix-comp-norm-bound}
\end{align}
Then, let \(T\) denote the coordinate transformation given by
\begin{align}
    T(m_{t})
    = \diag(T_{1}(m_{t}),\ldots,T_{d}(m_{t})).
    \label{la:def:diagonalizing-matrix}
\end{align}
The block-diagonal structure of \(T\) combined with \eqref{la:ineq:diagonalizing-matrix-comp-norm-bound} implies that
\begin{align}
    \max_{m_{t}\in \mathcal{S}}
    \Vert T(m_{t}) \Vert \Vert T(m_{t})^{-1} \Vert
    \leq M_{T}.
    \label{la:ineq:diagonalizing-matrix-norm-bound}
\end{align}
Furthermore, \(T(m_{t})\) diagonalizes \(G(m_{t})\), as in
\begin{align}
    G(m_{t})
    = T(m_{t})D(m_{t})T(m_{t})^{-1},
    \label{la:eq:Gt-diagonal-form}
\end{align}
where \(D(m_{t})\) is the block-diagonal matrix defined as \(D(m_{t}) = \diag(D_{1}(m_{t}), \ldots, D_{d}(m_{t}))\) and \(D_{i}(m_{t})\) are the diagonal matrices given by \eqref{la:eq:Git-diagonal-form}.
So, defining the state \(Z_{t}=T^{-1}(\mu_{0})X_{t}\) for \(t\in \mathopen{[}t_{0},t_{1}\mathclose{]}\) and plugging \(Z_{t}\) and \eqref{la:eq:Gt-diagonal-form} into \eqref{la:eq:Xt-dynamics-perturbed}, since \(m_{t}\equiv\mu_{0}\) for \(t\in \mathopen{[}t_{0},t_{1}\mathclose{)}\), it follows that
\begin{align}
    Z_{t+1}
    &= T^{-1}(\mu_{0})X_{t+1}
    \nonumber\\
    &= T^{-1}(\mu_{0})(G(m_{t}) + \tilde{G}_{t})X_{t}
    \nonumber\\
    &= T^{-1}(G(\mu_{0}) + \tilde{G}_{t})T(\mu_{0})Z_{t}
    \nonumber\\
    &= (D(\mu_{0}) + \tilde{D}_{t})Z_{t},
    \label{la:eq:Zt-dynamics}
\end{align}
where \(\tilde{D}_{t}\) is a perturbation matrix given by
\begin{align}
    \tilde{D}_{t}
    = T^{-1}(m_{t})\tilde{G}_{t}T(m_{t}).
    \label{la:eq:Dt-perturbation}
\end{align}
Using submultiplicativity and then combining \eqref{la:ineq:Gtilde-norm-bound} with \eqref{la:ineq:diagonalizing-matrix-norm-bound}, yields
\begin{align}
    \Vert \tilde{D}_{t} \Vert
    \leq \Vert T^{-1}(m_{t}) \Vert \Vert \tilde{G}_{t} \Vert \Vert T(m_{t}) \Vert
    \leq \epsilon M_{T}\frac{L_{H}}{L}\rgd(\kappa)^{t\alpha_{H}/2}.
    \label{la:ineq:Dtilde-norm-bound}
\end{align}
Then, summing \eqref{la:ineq:Dtilde-norm-bound}, we obtain
\begin{align}
    \sum_{t=0}^{+\infty}\Vert \tilde{D}_{t} \Vert
    \leq \epsilon M_{T} \frac{L_{H}}{L}\sum_{t=0}^{+\infty}\rgd(\kappa)^{t\alpha_{H}/2}
    \leq 
    \epsilon
    M_{T}
    \frac{L_{H}}{L}
    \frac{1}{1-\rgd(\kappa)^{\alpha_{H}/2}}.
    \label{la:ineq:Dtilde-norm-sum-bound}
\end{align}
Moreover, since \(\lambda_{i}\leq L < L_{0}\), it follows that \(G(m_{t})\) are nonsingular.
This fact combined with \eqref{la:ineq:Dtilde-norm-sum-bound} allows us to use results from the theory of asymptotic integration of difference equations \cite{Bodine2015} to establish that the solutions to \eqref{la:eq:Zt-dynamics} are perturbed solutions of the particular case when \(\tilde{D}_{t}\equiv 0\).
Namely, by \cite[theorem~3.4]{Bodine2015}, for \(t\in \mathopen{[}t_{0},t_{1}\mathclose{)}\) we have that
\begin{align*}
    Z_{t+1} = [I+O(\epsilon)]D(\mu_{0})^{t}Z_{0},
\end{align*}
which implies that for \(t\in \mathopen{[}t_{0},t_{1}\mathclose{)}\), we have that
\begin{align*}
    X_{t+1} 
    = T(\mu_{0})[I+O(\epsilon)]D(\mu_{0})^{t}Z_{0}
    = T(\mu_{0})[I+O(\epsilon)]D(\mu_{0})^{t}T(\mu_{0})^{-1}X_{0},
\end{align*}
which can also be written as a perturbation of the solution of the quadratic case \(G(\mu_{0})^{t}X_{0}\):
\begin{align*}
    X_{t+1}
    = G(\mu_{0})^{t}X_{0} + T(\mu_{0})D_{0}^{t}O(\epsilon)T(\mu_{0})^{-1}X_{0}.
\end{align*}

By repeatedly following the above procedure, we conclude that
\begin{align}
    X_{t+1} 
    =& 
    T(\mu_{J})[I+O(\epsilon)]D(\mu_{J})^{t-t_{J}}T(\mu_{J})^{-1}
    \biggl(
    \prod_{j=0}^{J-1}T(\mu_{j})[I+O(\epsilon)]D(\mu_{j})^{t_{j+1}-t_{j}}T(\mu_{j})^{-1}
    \biggr)
    X_{0}
    \nonumber
    \\
    =& 
    T(\mu_{J})D(\mu_{J})^{t-t_{J}}T(\mu_{J})^{-1}
    \biggl(
    \prod_{j=0}^{J-1}T(\mu_{j})D(\mu_{j})^{t_{j+1}-t_{j}}T(\mu_{j})^{-1}
    \biggr)
    X_{0}
    \nonumber
    \\
    &+ T(\mu_{J})O(\epsilon)D(\mu_{J})^{t-t_{J}}T(\mu_{J})^{-1}
    \biggl(
    \prod_{j=0}^{J-1}T(\mu_{j})D(\mu_{j})^{t_{j+1}-t_{j}}T(\mu_{j})^{-1}
    \biggr)
    X_{0}
    + \ldots
    \nonumber
    \\
    &+ T(\mu_{J})O(\epsilon)D(\mu_{J})^{t-t_{J}}T(\mu_{J})^{-1}
    \biggl(
    \prod_{j=0}^{J-1}T(\mu_{j})O(\epsilon)D(\mu_{j})^{t_{j+1}-t_{j}}T(\mu_{j})^{-1}
    \biggr)
    X_{0},
    \label{la:eq:Xt-perturbed-solution}
\end{align}
where \(t\in \mathopen{[}t_{J},t_{J+1}\mathclose{)}\).

\subsection*{The Dynamics of \texorpdfstring{\(c_{t}\)}{ct} in the General Case}

Having established that the components \(X_{t,i}\) in the general case behave like perturbed components of the quadratic case, we can also derive the dynamics of \(c_{t}\) in the general case.
To this end, we establish bounds on the differences \(x_{t+1,i}-x_{t,i}\).
First, we notice that
\begin{align*}
    X_{t+1,i}
    = \begin{bmatrix}
        0 & \ldots & 0 & I & 0 & \ldots & 0
    \end{bmatrix}
    X_{t+1},
\end{align*}
where \(\begin{bmatrix} 0 & \ldots & 0 & I & 0 & \ldots & 0 \end{bmatrix}\in \mathbb{R}^{2\times 2d}\) is a matrix made of a row of two-by-two blocks, where the \(i\)-th block is \(I\in \mathbb{R}^{2\times 2}\) and all other blocks are \(0\in\mathbb{R}^{2\times 2}\).
Also, we have that
\begin{align*}
    &\begin{bmatrix}
        0 & \ldots & 0 & I & 0 & \ldots & 0
    \end{bmatrix}
    T(\mu_{J})D(\mu_{J})^{t-t_{J}}T(\mu_{J})^{-1}
    \biggl(
    \prod_{j=0}^{J-1}T(\mu_{j})D(\mu_{j})^{t_{j+1}-t_{j}}T(\mu_{j})^{-1}
    \biggr)
    X_{0}
    \\
    &= \begin{bmatrix}
        0 & \ldots & 0 & I & 0 & \ldots & 0
    \end{bmatrix}
    G(\mu_{J})^{t-t_{J}}
    \prod_{j=0}^{J-1}G(\mu_{j})^{t_{j+1}-t_{j}}
    X_{0}
    \\
    &= \begin{bmatrix}
        0 & \ldots & 0 & G_{i}^{t-t_{J}}\prod_{j=0}^{J-1}G_{i}(\mu_{j})^{t_{j+1}-t_{j}} & 0 & \ldots & 0
    \end{bmatrix}
    X_{0}
    \\
    &= G_{i}^{t-t_{J}}\prod_{j=0}^{J-1}G_{i}(\mu_{j})^{t_{j+1}-t_{j}}X_{0,i},
\end{align*}
since \(G_{i}(\mu_{j})=T_{i}(\mu_{j})D_{i}(\mu_{j})T_{i}(\mu_{j})^{-1}\), by \eqref{la:eq:Git-diagonal-form}, and \(G,T\) and \(D\) are block-diagonal matrices with blocks given by \(G_{i},T_{i}\) and \(D_{i}\), respectively.
Then, we notice that all but the first term in \eqref{la:eq:Xt-perturbed-solution} are \(O(\epsilon)\), and \(\rho(\mu_{j},\lambda_{i})\leq \rho(\mu_{j},m)\) for all the eigenvalues \(\rho(\mu_{j},\lambda_{i})\) of \(D(\mu_{j})\), by \cref{la:lem:top-rho}.
Therefore, combining the above remarks with \cref{ass:xi0-lower-bound}, it follows that for \(t\in \mathopen{[}t_{j},t_{j+1}\mathclose{)}\)
\begin{align*}
    \Vert X_{t+1,i} \Vert^{2}
    \leq & \
    \overline{C}_{i} \rho(\mu_{j},\lambda_{i})^{2(t-t_{j})}
    \Biggl(
    \prod_{k=0}^{j-1}\rho(\mu_{k},\lambda_{i})^{2(t_{k+1}-t_{k})}
    \Biggr)
    x_{0,i}^{2}
    \\
    &+ O(\epsilon) \rho(\mu_{j},m)^{2(t-t_{j})}
    \Biggl(
    \prod_{k=0}^{j-1}\rho(\mu_{k},m)^{2(t_{k+1}-t_{k})}
    \Biggr)
    x_{0,1}^{2},
\end{align*}
for some \(\overline{C}_{i}\).
The above bound is analogous to \eqref{la:ineq:Xit-quad-norm-sq-bound}, but with an additional \(O(\epsilon)\) term accounting for the perturbation of the quadratic solution.
Thus, for \(t\in \mathopen{[}t_{j},t_{j+1}\mathclose{)}\), we have that
\begin{align}
    (x_{t+1,i}-x_{t,i})^{2}
    =&\ (\begin{bmatrix}
        -1 & 1
    \end{bmatrix}
    X_{t+1,i})^{2}
    \nonumber\\
    \leq &\  2\overline{C}_{i} \rho(\mu_{j},\lambda_{i})^{2(t-t_{j})}
    \Biggl(
    \prod_{k=0}^{j-1}\rho(\mu_{k},\lambda_{i})^{2(t_{k+1}-t_{k})}
    \Biggr)
    x_{0,i}^{2}
    \nonumber\\
    &+ O(\epsilon) \rho(\mu_{j},m)^{2(t-t_{j})}
    \Biggl(
    \prod_{k=0}^{j-1}\rho(\mu_{k},m)^{2(t_{k+1}-t_{k})}
    \Biggr)
    x_{0,1}^{2}.
    \label{la:ineq:xit-diff-norm-perturbed}
\end{align}
In the same vein, combining \eqref{la:eq:Xt-perturbed-solution} with \cref{la:lem:top-rho}, we have that for \(t\in \mathopen{[}t_{j},t_{j+1}\mathclose{)}\)
\begin{align}
    \Vert X_{t+1} \Vert^{2}
    \leq 2(1+O(\epsilon))\overline{C} \rho(\mu_{j},m)^{2(t-t_{j})}
    \Biggl(
    \prod_{k=0}^{j-1}\rho(\mu_{k},m)^{2(t_{k+1}-t_{k})}
    \Biggr)
    \Vert x_{0} \Vert^{2},
    \label{la:ineq:Xt-norm-sq-aux-perturbed}
\end{align}
where \(\overline{C}=\max_{i=1,\ldots,d}\overline{C}_{i}\).
Similarly, combining the derivation of \eqref{la:ineq:x1t-diff-quad-lower-bound} with \eqref{la:eq:Xt-perturbed-solution} and \cref{ass:xi0-lower-bound}, for \(t\in \mathopen{[}t_{j},t_{j+1}\mathclose{)}\) we have that
\begin{align}
    (x_{t+1,1}-x_{t,1})^{2}
    \geq (1-O(\epsilon))
    \underline{C}_{1}\rho(\mu_{j},m)^{2(t-t_{j})}
    \Biggl(
    \prod_{k=0}^{j-1}\rho(\mu_{k},m)^{2(t_{k+1} - t_{k})}
    \Biggr)
    x_{0,1}^{2},
    \label{la:ineq:x1t-diff-norm-perturbed}
\end{align}
for some \(\underline{C}_{1}\).

Our next step is to also express \(\nabla f(x_{t+1}) - \nabla f(x_{t})\) as a perturbation of the quadratic case.
To this end, we substitute \eqref{la:eq:gradient-perturbation} for \(\nabla f(x_{t+1})\) and \(\nabla f(x_{t})\), and obtain
\begin{align*}
    \nabla f(x_{t+1}) - \nabla f(x_{t})
    &= (H+\tilde{H}_{t+1})(x_{t+1}-x^{\star})
    - (H+\tilde{H}_{t})(x_{t}-x^{\star})
    \\
    &= H(x_{t+1}-x_{t})
    + \tilde{H}_{t+1}(x_{t+1}-x^{\star})
    - \tilde{H}_{t}(x_{t}-x^{\star}).
\end{align*}
Using \(x_{t}-x^{\star}=\sum_{j=1}^{d}x_{j,t}v_{j}\), the terms of \(\nabla f(x_{t+1}) - \nabla f(x_{t})\) above can be written as
\begin{align*}
    H(x_{t+1}-x_{t})
    &= \sum_{i=1}^{d}(x_{t+1,i}-x_{t,i})\lambda_{i}v_{i},
    \\
    \tilde{H}_{t+1}(x_{t+1}-x^{\star})
    - \tilde{H}_{t}(x_{t}-x^{\star})
    &= \sum_{i=1}^{d}
    \Biggl(
    \sum_{j=1}^{d}(\tilde{h}_{t,i,j+1}x_{j,t+1}-\tilde{h}_{t,i,j}x_{j,t})
    \Biggr)
    v_{i}.
\end{align*}
In turn, using the above expressions, it follows that
\begin{align*}
    \Vert \nabla f(x_{t+1}) - \nabla f(x_{t}) \Vert^{2}
    =&\ \sum_{i=1}^{d}\lambda_{i}^{2}(x_{t+1,i}-x_{t,i})^{2}
    \\
    &+ 2\sum_{i=1}^{d}\lambda_{i}(x_{t+1,i}-x_{t,i})
    \sum_{j=1}^{d}(\tilde{h}_{t,i,j+1}x_{j,t+1}-\tilde{h}_{t,i,j}x_{j,t})
    \\
    &+ \sum_{i=1}^{d}
    \Biggl(
    \sum_{j=1}^{d}(\tilde{h}_{t,i,j+1}x_{j,t+1}-\tilde{h}_{t,i,j}x_{j,t})
    \Biggr)^{2}.
\end{align*}
Our next step is to bound the third term above in \(\Vert \nabla f(x_{t+1}) - \nabla f(x_{t}) \Vert^{2}\).
Combining the identities \(\Vert \tilde{H}_{t+1} \Vert_{F}^{2} + \Vert \tilde{H}_{t} \Vert_{F}^{2} = \sum_{i=1}^{d}\sum_{j=1}^{d}(\tilde{h}_{t,i,j+1}^{2} + \tilde{h}_{t,i,j}^{2})\) and \(\Vert X_{t+1} \Vert^{2} = \sum_{j=1}^{d}(x_{j,t+1}^{2} + x_{j,t}^{2})\) with the bound \eqref{la:ineq:hessian-error-norm-bound}, it follows that
\begin{align}
    &\sum_{i=1}^{d}
    \Biggl(
    \sum_{j=1}^{d}(\tilde{h}_{t,i,j+1}x_{j,t+1}-\tilde{h}_{t,i,j}x_{j,t})
    \Biggr)^{2}
    \nonumber\\
    &\leq
    \sum_{i=1}^{d}
    \Biggl(
    \sum_{j=1}^{d}(\vert \tilde{h}_{t,i,j+1} \vert + \vert \tilde{h}_{t,i,j} \vert)
    (\vert x_{j,t+1} \vert + \vert x_{j,t} \vert)
    \Biggr)^{2}
    \nonumber\\
    &\leq
    \sum_{i=1}^{d}
    \Biggl(
    \sum_{j=1}^{d}(\vert \tilde{h}_{t,i,j+1} \vert + \vert \tilde{h}_{t,i,j} \vert)^{2}
    \Biggr)
    \Biggl(
    \sum_{j=1}^{d}(\vert x_{j,t+1} \vert + \vert x_{j,t} \vert)^{2}
    \Biggr)
    \nonumber\\
    &\leq
    \sum_{i=1}^{d}
    \Biggl(
    2\sum_{j=1}^{d}(\tilde{h}_{t,i,j+1}^{2} + \tilde{h}_{t,i,j}^{2})
    \Biggr)
    \Biggl(
    2\sum_{j=1}^{d}(x_{j,t+1}^{2} + x_{j,t}^{2})
    \Biggr)
    \nonumber\\
    &= 4(\Vert \tilde{H}_{t+1} \Vert_{F}^{2} + \Vert \tilde{H}_{t} \Vert_{F}^{2})\Vert X_{t+1} \Vert^{2}
    \nonumber\\
    &\leq 4d(\Vert \tilde{H}_{t+1} \Vert^{2} + \Vert \tilde{H}_{t} \Vert^{2})\Vert X_{t+1} \Vert^{2}
    \nonumber\\
    &\leq 4\epsilon^{2} L_{H}^{2}d\Vert X_{t+1} \Vert^{2}.
    \label{la:ineq:grad-diff-term-3-bound}
\end{align}
In turn, we address the second term of \(\Vert \nabla f(x_{t+1}) - \nabla f(x_{t}) \Vert^{2}\) using \eqref{la:ineq:grad-diff-term-3-bound}, which gives
\begin{align}
    &\sum_{i=1}^{d}(x_{t+1,i}-x_{t,i})
    \sum_{j=1}^{d}(\tilde{h}_{t,i,j+1}x_{j,t+1}-\tilde{h}_{t,i,j}x_{j,t})
    \nonumber\\
    &\leq
    \sqrt{\sum_{i=1}^{d}(x_{t+1,i}-x_{t,i})^{2}}
    \sqrt{\sum_{i=1}^{d}\Biggl( \sum_{j=1}^{d}(\tilde{h}_{t,i,j+1}x_{j,t+1}-\tilde{h}_{t,i,j}x_{j,t}) \Biggr)^{2}}
    \nonumber\\
    &\leq \sqrt{2\sum_{i=1}^{d}(x_{t+1,i}^{2}+x_{t,i}^{2})}
    \sqrt{4\epsilon^{2} L_{H}^{2}d\Vert X_{t+1} \Vert^{2}}
    \nonumber\\
    &\leq 2\epsilon L_{H}\sqrt{2d}\Vert X_{t+1} \Vert^{2}.
    \label{la:ineq:grad-diff-term-2-bound}
\end{align}
Then, combining \labelcref{la:ineq:Xt-norm-sq-aux-perturbed,la:ineq:grad-diff-term-3-bound,la:ineq:grad-diff-term-2-bound} we obtain
\begin{align}
    \Vert \nabla f(x_{t+1}) - \nabla f(x_{t}) \Vert^{2}
    &\leq 
    \sum_{i=1}^{d}\lambda_{i}^{2}(x_{t+1,i}-x_{t,i})^{2}
    \nonumber\\
    &+ O(\epsilon) \rho(\mu_{j},m)^{2(t-t_{j})}
    \Biggl(
    \prod_{k=0}^{j-1}\rho(\mu_{k},m)^{2(t_{k+1}-t_{k})}
    \Biggr)
    x_{0,1}^{2}.
    \label{la:ineq:grad-diff-aux}
\end{align}
In turn, plugging \eqref{la:ineq:grad-diff-aux} into \eqref{ineq:effective-curvature}, and then using \eqref{la:ineq:x1t-diff-norm-perturbed}, it follows that
\begin{align}
    c_{t+1}^{2}
    = c(x_{t+1},x_{t})^{2}
    &\leq \frac{\sum_{i=1}^{d}\lambda_{i}^{2}(x_{t+1,i}-x_{t,i})^{2}}{\sum_{i=1}^{d}(x_{t+1,i}-x_{t,i})^{2}}
    + O(\epsilon).
    \label{la:ineq:curvature-weighted-average-perturbed}
\end{align}

\begin{lemma}\label{la:lem:ct-perturbed}
    Let \(f\in\mathcal{S}(L,m)\), suppose that \cref{ass:locally-smooth-hessian,ass:known-lipschitz-upper-bound,ass:mt-separation-from-eigs,ass:xi0-lower-bound} hold and let \(\kappa=L_{0}/m\).
    Also, let \(\delta_{m}\) and \(\delta_{u}\) be positive numbers such that \(\delta_{m}>\delta_{u}>0\) and let \(r'=r'(\delta_{u},\delta_{\ell},\kappa)\) be a function such that \(\rnag(\sigma_{\phi}\kappa) \leq r' < 1\) for all \(\kappa \geq 1+\delta_{\ell}\), where \(\delta_{\ell} = ((1+\delta_{m})/(1+\delta_{u}))-1>0\) and \(\sigma_{\phi}=\sigma_{\phi}(\delta_{u},\delta_{\ell},\kappa)\) is given by \cref{la:lem:rphi-sq-leq-racc}.
    Then, there exist \(\nu\) and \(\epsilon>0\) such that if \(\Vert x_{0} -x^{\star}\Vert \leq \epsilon\), then the estimates \(m_{t}\) of \cref{alg:nag-free} reach \([m/\gamma,(1+\delta_{m})m]\) after no more than \(\tau\) iterations, where
    \begin{align}
        \tau
        =
        \nu
        \frac{-\log(8\kappa^{2}M_{1}\omega/\delta_{u})}{\log r'}
        \label{la:ineq:ct-perturbed}
    \end{align}
    \(M_{1}=\max_{i}\overline{C}_{i}/\underline{C}_{1}\), with \(\omega\) given by \cref{ass:locally-smooth-hessian,ass:mt-separation-from-eigs,ass:xi0-lower-bound}.
\end{lemma}

\begin{proof}
    Suppose that the last value of \(m_{t}\) before reaching the interval \([m/\gamma,(1+\delta_{m})m)\) is \((1+\delta_{m})m\).
    Then, suppose that the value before last is \(\gamma(1+\delta_{m})m\), and so on, up to \(\gamma^{K}(1+\delta_{m})m\) for some \(K\) such that \(\gamma^{K}(1+\delta_{m})m\leq L < \gamma^{K+1}(1+\delta_{m})m\).
    Using this \(m_{t}\) schedule, we bound the number of iterations that \(m_{t}\) takes to reach the interval \([m/\gamma,(1+\delta_{m})m]\), and then we argue that no other \(m_{t}\) schedule can lead to a worse bound.
    
    Let \(\ell_{j}=\gamma^{j}(1+\delta_{m})m/(1+\delta_{u})\).
    Then, we have that \(\ell_{j} \geq (1+\delta_{\ell})m\) for \(\delta_{\ell}=((1+\delta_{m})/(1+\delta_{u}))-1\).
    Since \(\delta_{m}>\delta_{u}\), then \(\delta_{\ell}>0\), and \cref{la:lem:phi-sq-leq-racc-pwr} applies.
    Now, let \(I_{j} = \min\left\{ i: \lambda_{i} \geq \ell_{j} \right\}\).
    That is, \(\lambda_{i} \geq \ell_{j}\) if and only if \(i \geq I_{j}\).
    Then, using this fact and separating the terms indexed by \(i < I_{0}\) from those indexed by \(i \geq I_{0}\) in \eqref{la:ineq:curvature-weighted-average-perturbed} into two sums yields
    \begin{align*}
        c^{2}_{t+1}
        < \ell_{0}^{2}\frac{\sum_{i=1}^{I_{0}-1}(x_{t+1,i}-x_{t,i})^{2}}{\sum_{i=1}^{d}(x_{t+1,i}-x_{t,i})^{2}}
        +\frac{\sum_{i=I_{0}}^{d}\lambda_{i}^{2}(x_{t+1,i}-x_{t,i})^{2}}{\sum_{i=1}^{d}(x_{t+1,i}-x_{t,i})^{2}}
        +O(\epsilon).
    \end{align*}
    In turn, plugging the above inequality into the identity \((c_{t+1} + \ell_{0})(c_{t+1} - \ell_{0})=c^{2}_{t+1} - \ell_{0}^{2}\), and then using the fact that \(\lambda_{i} \leq L\) and \(\ell_{0}\geq m\), we obtain
    \begin{align}
        c_{t+1} - \ell_{0}
        =\frac{c^{2}_{t+1}-\ell_{0}^{2}}{c_{t+1} + \ell_{0}}
        &< \frac{\sum_{i=I_{0}}^{d}(\lambda_{i}^{2} - \ell_{0}^{2})(x_{t+1,i}-x_{t,i})^{2}}{\ell_{0}\sum_{i=1}^{d}(x_{t+1,i}-x_{t,i})^{2}}
        + O(\epsilon)
        \nonumber\\
        &\leq \ell_{0} \kappa^{2}\sum_{i=I_{0}}^{d}\frac{(x_{t+1,i}-x_{t,i})^{2}}{(x_{t+1,1}-x_{t,1})^{2}}
        +O(\epsilon).
        \label{la:ineq:conv-2x-rate-aux1-perturbed}
    \end{align}
    To address the terms in the sum in \eqref{la:ineq:conv-2x-rate-aux1-perturbed}, we combine  \eqref{la:ineq:xit-diff-norm-perturbed} and \eqref{la:ineq:x1t-diff-norm-perturbed}, assuming \(t_{K} \leq t < t_{K+1}\).
    That is, we consider the last adjustment before \(m_{t}\) reaches the interval \([m/\gamma,(1+\delta_{m})m)\).
    Then, we apply \cref{la:lem:top-rho} twice, to get \(\rho(m_{k},\lambda_{i})\leq \rho(m_{k},\ell_{0})\) and \(\rho(m_{k},\ell_{0})<\rho(m_{k},m)\) for all \(i\geq I_{0}\), which gives
    \begin{align}
        \sum_{i=I_{0}}^{d}\frac{(x_{t+1,i}-x_{t,i})^{2}}{(x_{t+1,1}-x_{t,1})^{2}}
        &\leq 
        2(1+O(\epsilon))M_{1}\sum_{i=I_{0}}^{d}
        \frac{x_{0,i}^{2}}{x_{0,1}^{2}}
        \prod_{k=t_{K}}^{t}\frac{\rho(m_{k},\lambda_{i})^{2}}{\rho(m_{k},m)^{2}}
        \prod_{k=1}^{\tau_{K}}\frac{\rho(m_{k},\lambda_{i})^{2}}{\rho(m_{k},m)^{2}}
        +O(\epsilon)
        \nonumber\\
        &\leq
        2(1+O(\epsilon))M_{1}\omega\phi((1+\delta_{u})\ell_{0},\ell_{0},m)^{2(t-t_{K}+1)}
        +O(\epsilon).
        \label{la:ineq:conv-2x-rate-aux2-perturbed}
    \end{align}
    where \(M_{1}=2\max_{i}\overline{C}_{i}/\underline{C}_{1}\).
    Next, we put \eqref{la:ineq:conv-2x-rate-aux1-perturbed} and \eqref{la:ineq:conv-2x-rate-aux2-perturbed} together, and since \(\ell_{0}\geq m\), by choosing \(\epsilon\) sufficiently small, we get
    \begin{align*}
        c_{t+1} - \ell_{0}
        \leq
        2(1+O(\epsilon))\kappa^{2}M_{1}\omega\phi((1+\delta_{m})m,\ell_{0},m)^{2(t-t_{K}+1)}\ell_{0}
        + O(\epsilon)
        \leq 
        \delta_{u}\ell_{0}/2,
    \end{align*}
    for all \(t\geq t_{K} + \Delta t_{0}\), where
    \begin{align*}
        \Delta t_{0}
        = -\frac{\log (8\kappa^{2}M_{1}\omega/\delta_{u})}{\log \rnag(\sigma_{\phi}\kappa)}.
    \end{align*}
    Therefore, \(c_{t+1}<(1+\delta_{u})\ell_{0}=(1+\delta_{m})m\) for \(t\geq t_{K}+\Delta t_{0}\) or, equivalently,
    \begin{align*}
        t_{K+1}-t_{K}
        \leq \Delta t_{0}.
    \end{align*}
    Note that for every \(m_{t}\) schedule, if \(\mu_{K}\) denotes the last value of \(m_{t}\) before reaching \([m/\gamma,(1+\delta_{m})m]\), then \(\mu_{K}\geq (1+\delta_{m})m\), by definition.
    Hence, letting \(\ell_{0}'=\mu_{K}/(1+\delta_{u})\) and \(I_{0}'=\{i:\lambda_{i}\geq \ell_{0}'\}\), then \(I_{0}'\geq I_{0}\), and by applying \cref{la:lem:top-rho} before, it follows that
    \begin{align*}
        \sum_{i=I_{0}'}^{d}\frac{(x_{t+1,i}-x_{t,i})^{2}}{(x_{t+1,1}-x_{t,1})^{2}}
        &\leq 
        2(1+O(\epsilon))M_{1}\sum_{i=I_{0}'}^{d}
        \frac{x_{0,i}^{2}}{x_{0,1}^{2}}
        \prod_{k=t_{K}}^{t}\frac{\rho(m_{k},\lambda_{i})^{2}}{\rho(m_{k},m)^{2}}
        \prod_{k=1}^{\tau_{K}}\frac{\rho(m_{k},\lambda_{i})^{2}}{\rho(m_{k},m)^{2}}
        +O(\epsilon)
        \\
        &\leq 
        2(1+O(\epsilon))M_{1}\sum_{i=I_{0}}^{d}
        \frac{x_{0,i}^{2}}{x_{0,1}^{2}}
        \prod_{k=t_{K}}^{t}\frac{\rho(m_{k},\lambda_{i})^{2}}{\rho(m_{k},m)^{2}}
        \prod_{k=1}^{\tau_{K}}\frac{\rho(m_{k},\lambda_{i})^{2}}{\rho(m_{k},m)^{2}}
        +O(\epsilon)
        \\
        &\leq
        2(1+O(\epsilon))M_{1}\omega\phi((1+\delta_{u})\ell_{0},\ell_{0},m)^{2(t-t_{K}+1)}
        + O(\epsilon).
    \end{align*}
    Therefore, the last adjustment cannot take more than \(\Delta t_{0}\) iterations for any \(m_{t}\) schedule.
    
    Then, let \(\Delta t_{j}\) be quantities analogous to \(\Delta t_{0}\), defined for \(j=0,\ldots,K\) as
    \begin{align*}
        \Delta t_{j}
        = 
        \frac{-\log (8\kappa^{2}M_{1}\omega/\delta_{u})}{\log \rnag(\sigma_{\phi}\kappa)}
        \frac{1}{1 + \alpha_{\phi}(\ell_{j}-(1+\delta_{\ell})m)}.
    \end{align*}
    If \(\gamma>1\), then \(m_{t}\) decreases by a factor of at least \(\gamma\) every time it is adjusted to a new value.
    Hence, \(\mu_{K-1}\geq \gamma\mu_{K}\) for every \(m_{t}\) schedule, which implies that \(\ell_{1}\leq \mu_{K-1}/(1+\delta_{u})\) for every \(m_{t}\) schedule.
    Hence, by the same rationale above, it cannot take more than \(\Delta t_{1}\) for \(m_{t}\) to be adjusted to its second last value before reaching the interval \([m/\gamma,(1+\delta_{m})m]\).
    It follows by induction that it cannot take more than \(\Delta t_{j}\) for \(m_{t}\) to be adjust to its \(K-j\)-th to last value before reaching the interval \([m/\gamma,(1+\delta_{m})m]\).
    Moreover, since by design \(m_{t}\leq L\), it cannot more than \(K\leq \log_{\gamma}(\kappa/(1+\delta_{m}))\) adjustments before \(m_{t}\) reaches the interval \([m/\gamma,(1+\delta_{m})m]\).
    Therefore, letting \(\nu\) be given by \eqref{la:def:nu}, as
    \begin{align*}
        \nu
        = \sum_{j=0}^{+\infty}
        \frac{1}{1 + \alpha_{\phi}m(1+\delta_{\ell})(\gamma^{j}-1)},
    \end{align*}
    we conclude that \(m_{t+1}\leq (1+\delta_{m})m\) for all \(t\geq \tau\), where
    \begin{align*}
        \tau 
        = 
        \nu \frac{-\log(8\kappa^{2}M_{1}\omega/\delta_{u})}{\log r'}
        \geq 
        \frac{-\log(8\kappa^{2}M_{1}\omega/\delta_{u})}{\log \rnag(\sigma_{\phi}\kappa)}
        \sum_{j=0}^{K}
        \frac{1}{1 + \alpha_{\phi}m(1+\delta_{\ell})(\gamma^{j}-1)},
    \end{align*}
    because \(\log\) is monotone and \(\rnag(\sigma_{\phi}\kappa) \leq r' < 1\), and \(1 + \alpha_{\phi}m(1+\delta_{\ell})(\gamma^{j}-1) > 0\).
\end{proof}

\subsection{Main Result}

At last, we are ready to prove \cref{thm:la}, establishing that \cref{alg:nag-free} achieves acceleration around the minimum.

\begin{proof}[Proof of \cref{thm:la}]
    Let \(\delta_{m}\) and \(\delta_{u}\) be positive numbers such that \(\delta_{u} < \min\{\delta_{m},1/2\}\) and \(\delta_{m}\leq \gamma-1\).
    Then, define \(\delta_{\ell}=(1+\delta_{m})/(1+\delta_{u})-1\), and let \(r'=r'(\delta_{u},\delta_{\ell},\kappa)\) be a function such that \(\rnag(\sigma_{\phi}\kappa) \leq r' < 1\) for all \(\kappa\geq 1+\delta_{\ell}\), where \(\sigma_{\phi}=\sigma_{\phi}(\delta_{u}, \delta_{\ell}, \kappa)\) is given by \cref{la:lem:rphi-sq-leq-racc}.
    By \cref{la:lem:ct-perturbed}, there is some \(\nu\) such that \(m_{t}\leq (1+\delta_{m})m\) for all \(t\geq \tau\), where 
    \begin{align}
        \tau
        =
        \nu
        \frac{-\log (8\kappa^{2}M_{1}\omega/\delta_{u})}{\log r'}
        \label{la:ineq:tau-prime-perturbed}
    \end{align}
    \(M_{1}=\max_{i}\overline{C}_{i}/\underline{C}_{1}\).
    In turn, as in the proof of \cref{la:prop:main-quad}, \cref{la:cor:rho-leq-racc} then implies that \(\rho(m_{t},m)\leq \rnag(\sigma_{1}\kappa)\) for all \(t\geq \tau\), where \(\sigma_{1}=\max\{\gamma,\sigma_{m}\}\), and \(\sigma_{m}=1+2\delta_{m}+2\sqrt{\delta_{m}(1+\delta_{m})}\).

    Now, by \cref{la:lem:top-rho}, we have that \(\rho(m_{t},\lambda_{i}) \leq \rnag(\sigma_{1}\kappa)\) for all \(\lambda_{i}\).
    Hence, given \(\delta_{\sigma}\) such that \((1+\delta_{\sigma})\rnag(\sigma_{1}\kappa)<1\), by \cref{la:lem:lyap-eps} there is a \(P_{i}(m_{t})=P_{i}(m_{t},\delta_{\sigma})\succeq I\) for each \(\lambda_{i}\) such that \(G_{i}(m_{t})^{\T}P_{i}(m_{t})G_{i}(m_{t})\preceq (1+\delta_{\sigma})^{2}\rnag(\sigma_{1}\kappa)^{2}P_{i}(m_{t})\).
    Using \(P_{i}(m_{t})\) as diagonal blocks, we define the matrix \(P(m_{t})=P(m_{t},\delta_{\sigma})=\diag(P_{1}(m_{t}), \ldots, P_{d}(m_{t}))\).
    The block diagonal structure of \(P\) and \(G\) implies that \(P(m_{t})\succeq I\), and that \(G(m_{t})^{\T}P(m_{t})G(m_{t})\preceq (1+\delta_{\sigma})^{2}\rnag(\sigma_{1}\kappa)^{2}P(m_{t})\).
    Hence, if \(t_{j} \leq t < t_{j+1}\) and \(t \geq \tau\), then \eqref{la:eq:Xt-dynamics-perturbed} yields
    \begin{align}
        X_{t+1}^{\T}P(m_{t})X_{t+1}
        &= X_{t}^{\T}(G(\mu_{j}) + \tilde{G}_{t})^{\T}P(\mu_{j})(G(\mu_{j}) + \tilde{G}_{t})X_{t}
        \nonumber\\
        &\leq (1+\delta_{\sigma})^{2}\rnag(\sigma_{1}\kappa)^{2}X_{t}^{\T}P(\mu_{j})X_{t}
        + X_{t}^{\T}\tilde{P}_{t}X_{t},
        \label{la:ineq:Lyap-aux-perturbed}
    \end{align}
    since \(m_{t}=\mu_{j}\), where
    \begin{align*}
        \tilde{P}_{t}
        = \tilde{G}_{t}^{\T}P(m_{t})G(m_{t})
        + G(m_{t})^{\T}P(m_{t})\tilde{G}_{t}
        + \tilde{G}_{t}^{\T}P(m_{t})\tilde{G}_{t}.
    \end{align*}
    By \cref{la:def:Git}, we have that
    \begin{align}
        \Vert G_{i}(m_{t}) \Vert
        &\leq 
        \left\Vert
        \begin{bmatrix} 
            0 & 1 \\ 
            -\beta(m_{t})\Bigl( 1-\dfrac{\lambda_{i}}{L} \Bigr) & 0
        \end{bmatrix}
        \right\Vert
        + \left\Vert
        \begin{bmatrix} 
            0 & 0 \\ 
            0 & (1+\beta(m_{t}))\Bigl( 1-\dfrac{\lambda_{i}}{L} \Bigr) 
        \end{bmatrix}
        \right\Vert
        \nonumber\\
        &=
        \max
        \left\{
        1, \beta(m_{t})\Bigl( 1-\dfrac{\lambda_{i}}{L} \Bigr)
        \right\}
        + (1+\beta(m_{t}))\Bigl( 1-\dfrac{\lambda_{i}}{L} \Bigr)
        \nonumber\\
        &\leq 3.
        \label{la:ineq:Gi-norm}
    \end{align}
    Furthermore, the block diagonal structure of \(P\) combined with \eqref{la:ineq:max-cond-P-bound} yields
    \begin{align}
        \Vert P(m_{t}) \Vert
        \leq \frac{1+(1+\delta_{\sigma})^{-2}}{1-(1+\delta_{\sigma})^{-2}}
        +\frac{2M_{2}^{2}(1+(1+\delta_{\sigma})^{-2})}{\rnag(\sigma_{1}\kappa)^{2}(1-(1+\delta_{\sigma})^{-2})^{3}}
        = M_{\delta}.
        \label{la:ineq:P-norm-bound}
    \end{align}
    Therefore, combining \labelcref{la:ineq:Gtilde-norm-bound,la:ineq:Gi-norm,la:ineq:P-norm-bound}, and taking \(\epsilon\) such that \(\epsilon L_{H}/L < 1\), we obtain
    \begin{align}
        \Vert \tilde{P}_{t} \Vert
        \leq (2\Vert G(m_{t}) \Vert + \Vert \tilde{G}_{t} \Vert)\Vert \tilde{G}_{t} \Vert \Vert P(m_{t}) \Vert
        \leq 7\epsilon M_{\delta} L_{H}/L.
        \label{la:ineq:Ptilde-norm-bound}
    \end{align}
    Then, since \(P(m_{t})\succeq I\), from \labelcref{la:ineq:Lyap-aux-perturbed,la:ineq:Ptilde-norm-bound} it follows that
    \begin{align}
        X_{t+1}^{T}PX_{t+1}
        \leq ((1+\delta)^{2}r^{2} + 7\epsilon M_{\delta}L_{H}/L)X_{t}^{T}PX_{t}
        = \tilde{r}^{2}X_{t}^{T}PX_{t},
        \label{la:ineq:Lyap-aux2-perturbed}
    \end{align}
    where we conveniently use a perturbed rate \(\tilde{r}\), given by
    \begin{align}
        \tilde{r} 
        = \sqrt{(1+\delta_{\sigma})^{2}r^{2} + 7\epsilon M_{\delta}L_{H}/L}.
        \label{la:def:rtilde}
    \end{align}
    Consecutively applying \eqref{la:ineq:Lyap-aux2-perturbed} and reproducing the steps in the proof of \cref{la:prop:main-quad}, we get
    \begin{align}
        \Vert x_{t+1} - x^{\star} \Vert
        \leq C'\bar{\kappa}^{7/2 + 2\nu}\tilde{r}^{t} \Vert x_{0} - x^{\star} \Vert,
        \label{la:ineq:xt-norm-perturbed-aux}
    \end{align}
    where the constant \(C\) is given by
    \begin{align*}
            C' = 9\cdot2^{13}M_{2}(8M_{1}\omega/\delta_{u})^{\nu}\sigma_{2}^{3/2},
    \end{align*}
    with \(\nu\), \(\sigma_{2}=\max\{\gamma, \sigma_{m}, \sigma_{\phi}\}\), and  \(\delta_{\sigma}=1/(4\sqrt{\sigma_{2}}\kappa)\) is chosen such that
    \begin{align*}
        (1+\delta_{\sigma})\rnag(\sigma_{2}\kappa)
        < \rnag(2\sigma_{2} \kappa).
    \end{align*}
    Hence, choosing \(\epsilon\) sufficiently small such that
    \begin{align*}
        \sqrt{(1+\delta_{\sigma})^{2}r^{2} + 7\epsilon M_{\delta}L_{H}/L}
        \leq \rnag(2\sigma_{2} \kappa),
    \end{align*}
    and then plugging this choice of \(\epsilon\) back into \eqref{la:ineq:xt-norm-perturbed-aux}, we obtain
    \begin{align*}
        \Vert x_{t+1} - x^{\star} \Vert
        \leq C\rnag(\sigma\bar{\kappa})^{t} \Vert x_{0} - x^{\star} \Vert,
    \end{align*}
    where \(\sigma=2\sigma_{2}\) and \(C=C'\bar{\kappa}^{7/2 + 2\nu}\), which concludes the proof.
\end{proof}

To conclude this section, we make a few remarks on \(C\) and \(\sigma\) in \cref{thm:la}.

One of the factors of \(C\) involves a power \(\nu\), which is defined by \eqref{la:def:nu} and implicitly involves some quantities that are arbitrarily set in the local analysis such as \(\delta_{\ell}\).
\cref{la:fig:alpha_phi} illustrates a numerical example for a particular choice of these quantities, in which case \(\nu\approx 1.9\).
In reality, \(\nu\) is an artifact of a conservative analysis and does not play a role in practical performance.
Indeed, in \cref{la:lem:ct,la:lem:ct-perturbed} we bound the number of iterates that \(m_{t}\) takes to be updated to a new value disregarding that the \(\lambda_{i}\)-coordinates for \(\lambda_{i}\geq m_{t}\) have already been reduced substantially relative to the others in the previous update, which is reflected in the value of \(c_{t}\).
In other words, we analyze the convergence of \(m_{t}\) as if it was starting from the same initial conditions every time for every update.

Now consider the suboptimality factor \(\sigma\).
In the proof of \cref{thm:la}, we work with \(\sigma=2\sigma_{2}\), where \(\sigma_{2}=\max\{\gamma,\sigma_{m},\sigma_{\phi}\}\), which is a function of \(\sigma_{m}=1+2\delta_{m}+2\sqrt{\delta_{m}(1+\delta_{m})}\) and \(\sigma_{\phi} \lesssim 1/4(\sqrt{\delta_{u} + \delta_{\ell}(1+\delta_{u})}-\sqrt{\delta_{u}(1+\delta_{\ell})})^{2}\).
The suboptimality factor \(\sigma_{m}\) is decreasing in \(\delta_{m}\), which determines the gap in the upper bound of \(\mathopen{[}m/\gamma,(1+\delta_{m})m\mathclose{]}\), the interval that contains \(m_{t}\) in the final convergence regime of NAG-free around \(x^{\star}\).
Intuitively, the smaller \(\delta_{m}\) is, the better the estimate \(m_{t}\) is in the final regime, therefore a smaller suboptimality rate.
On the other hand, if \(\delta_{m}\) is small, then so is \(\delta_{\ell}<\delta_{m}\), therefore \(\sigma_{\phi}\) increases.
Intuitively, \(\sigma_{\phi}\) represents that the time that \(m_{t}\) takes to become sufficiently accurate.
Thus, a smaller \(\delta_{m}\) means that \(m_{t}\) takes longer to reach the interval \(\mathopen{[}m/\gamma,(1+\delta_{m})m\mathclose{]}\).
More importantly, as \(m_{t}\) approaches \(m\), the rate at which \(m_{t}\) converges to \(m\) slows down.
The factor 2 in \(\sigma=2\sigma_{2}\) is a result of a compromise to obtain a reasonable condition number for the matrix \(P\) in the Lyapunov analysis of the final regime of NAG-free, when \(m_{t}\) is sufficiently accurate for accelerated convergence.
In reality, this compromise is an artifact of any Lyapunov analysis of linear systems, whose solutions are linear combinations of some \(t^{k}r^{t}\) terms, rather than purely exponential solutions \(r^{t}\).
Hence, this compromise is typically ignored, e.g. as in \cite{Lessard2016}, in which case the convergence rate is \(\max\{\gamma,\sigma_{m},\sigma_{\phi}\}\).
Then, for example, letting \(\gamma=2,\delta_{m}=0.2, \delta_{u}=0.01\) and \(\delta_{\ell}=(1+\delta_{m})/(1+\delta_{u})-1\), we obtain \(\max\{\gamma,\sigma_{m},\sigma_{\phi}\}\leq 2.4\).
Therefore, the convergence rate in this case would not be worse than \(\rnag(2.4\kappa)\), which is competitive with restart schemes, where ``the convergence rate is slowed down by roughly a factor four'' \cite[page 167]{dAspremont2021}.
Notwithstanding, \(2.4\) is still conservative and, in the next section, we present experiments in which we see that the suboptimality factor is much closer to 1.

\newpage
\clearpage
\section{Numerical experiments}
\label{app:ne}

In this appendix, we present additional experiments.

\subsection{Logistic regression}

\cref{app:ne:tab:logreg-mts} presents the values of \(m_{\infty}(L_{0})\) held by NAG-free at the last iteration for a given \(L_{0}\).
\cref{app:ne:fig:logreg_summary_2} shows the results for three more datasets from LIBSVM \cite{Chang2011}.

\setlength{\tabcolsep}{6pt} 
\begin{table}[h]
    \centering
    \caption{Regularization parameter \(\eta\), \(\eta/\gamma\) and the estimates \(m_{\infty}(L_{0})\) held by NAG-free at the last iteration solving the logistic regression problem for a given \(L_{0}\).
    For most datasets, \(m_{\infty}\) fall within the interval \(\mathopen{[}\eta/\gamma,\eta\mathclose{]}\), except for \textsc{phishing}, which is highlighted in gray.}
    \begin{tabular}{lcccc}
        \toprule
        Dataset & \(\eta/\gamma\) & \(m_{\infty}(\bar{L})\) & \(m_{\infty}(0.01\bar{L})\) & \(\eta\) \\ 
        \hline
        a9a & 3.22e-6 & 4.09e-6 & 4.73e-6 & 4.83e-6 \\
        gisette{\_}scale & 9.37e-3 & 1.24e-2 & 1.28e-2 & 1.40e-2 \\
        mushrooms & 2.12e-5 & 2.33e-5 & 2.62e-5 & 3.18e-5 \\
        \rowcolor{gray!30} phishing & 9.80e-7 & 8.02e-6 & 8.35e-6 & 1.47e-6 \\
        svmguide1 & 1.90e-1 & 1.76e-1 & 2.01e-1 & 2.85e-1 \\
        w1a & 1.67e-5 & 2.26e-5 & 2.13e-5 & 2.51e-5 \\
        \bottomrule
    \end{tabular}
    \label{app:ne:tab:logreg-mts}
\end{table}

\begin{figure}[hb]
    \centering
    \includegraphics[width=.9\linewidth]{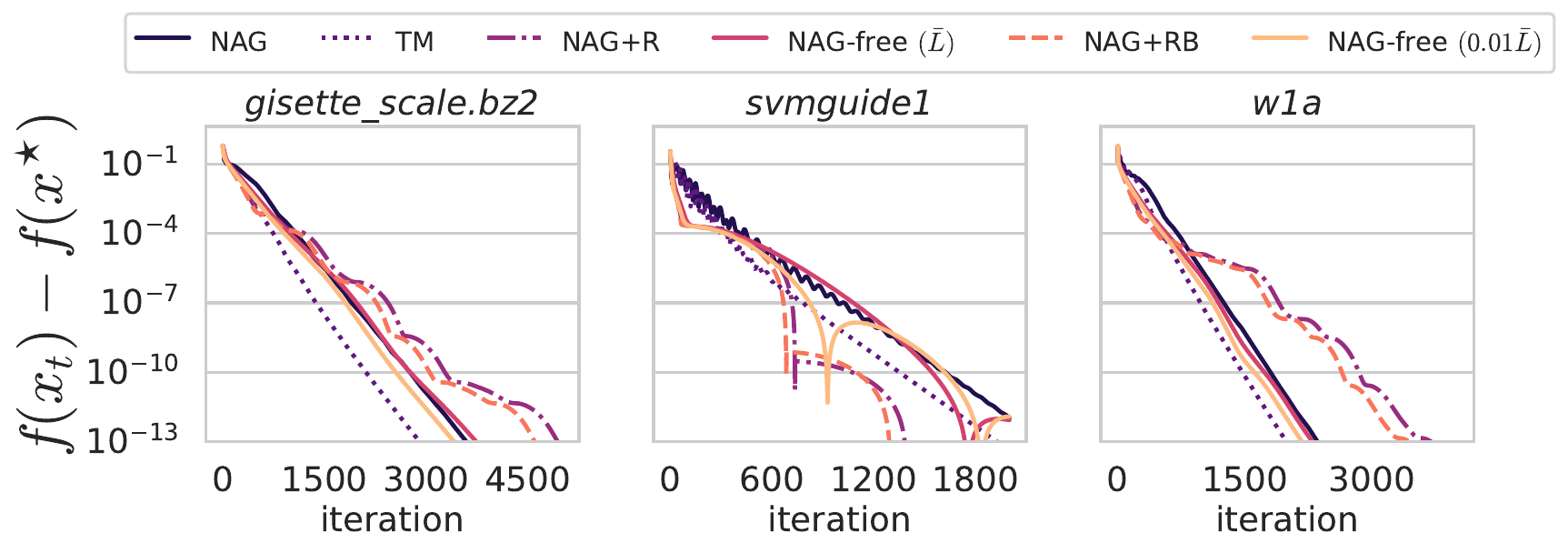}
    \caption{Suboptimality gap for logistic regression on three datasets.}
    \label{app:ne:fig:logreg_summary_2}
\end{figure}

\subsection{Log-sum-exp}

\cref{app:ne:fig:log-sum-exp_summary_0,app:ne:fig:log-sum-exp_summary_1,app:ne:fig:log-sum-exp_summary_2} show the suboptimality gap for several \(\eta,\eta\) settings.

\begin{figure}[htbp]
    \centering
    \includegraphics[width=\linewidth]{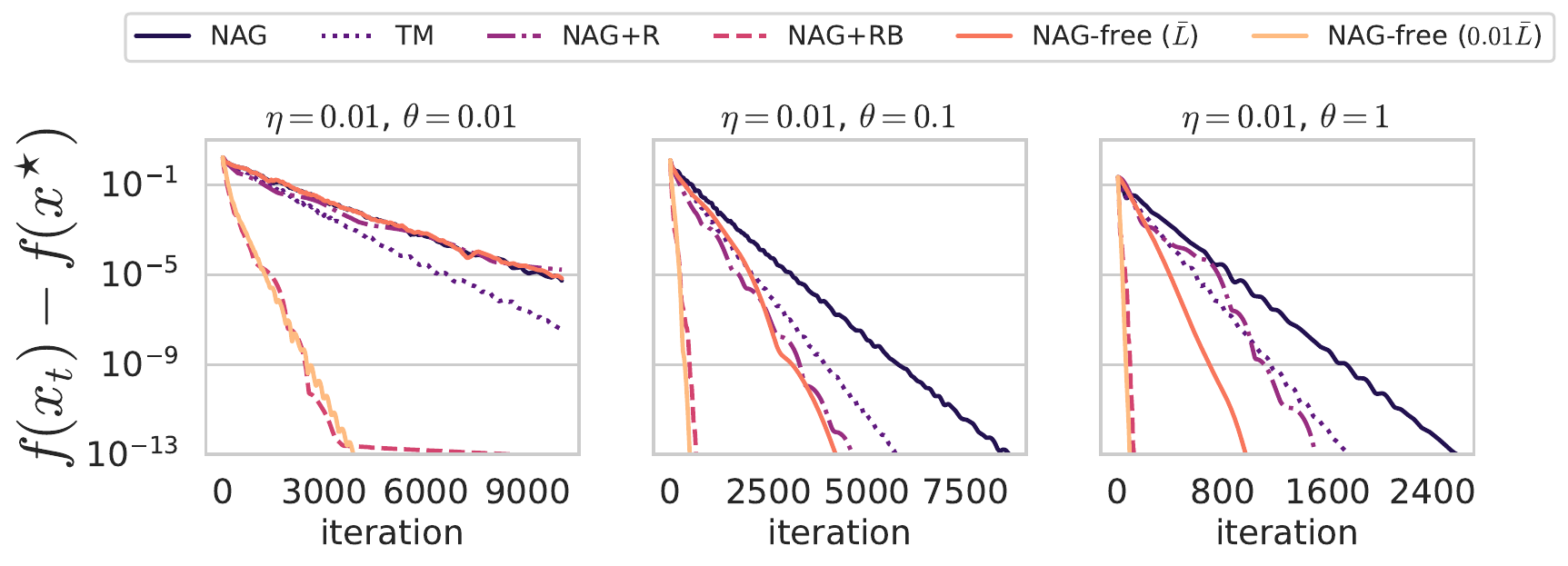}
    \caption{Suboptimality gap \(f(x_{t})-f(x^{\star})\) for log-sum-exp under different \((\eta,\theta)\) settings.}
    \label{app:ne:fig:log-sum-exp_summary_0}
\end{figure}

\begin{figure}
    \centering
    \includegraphics[width=\linewidth]{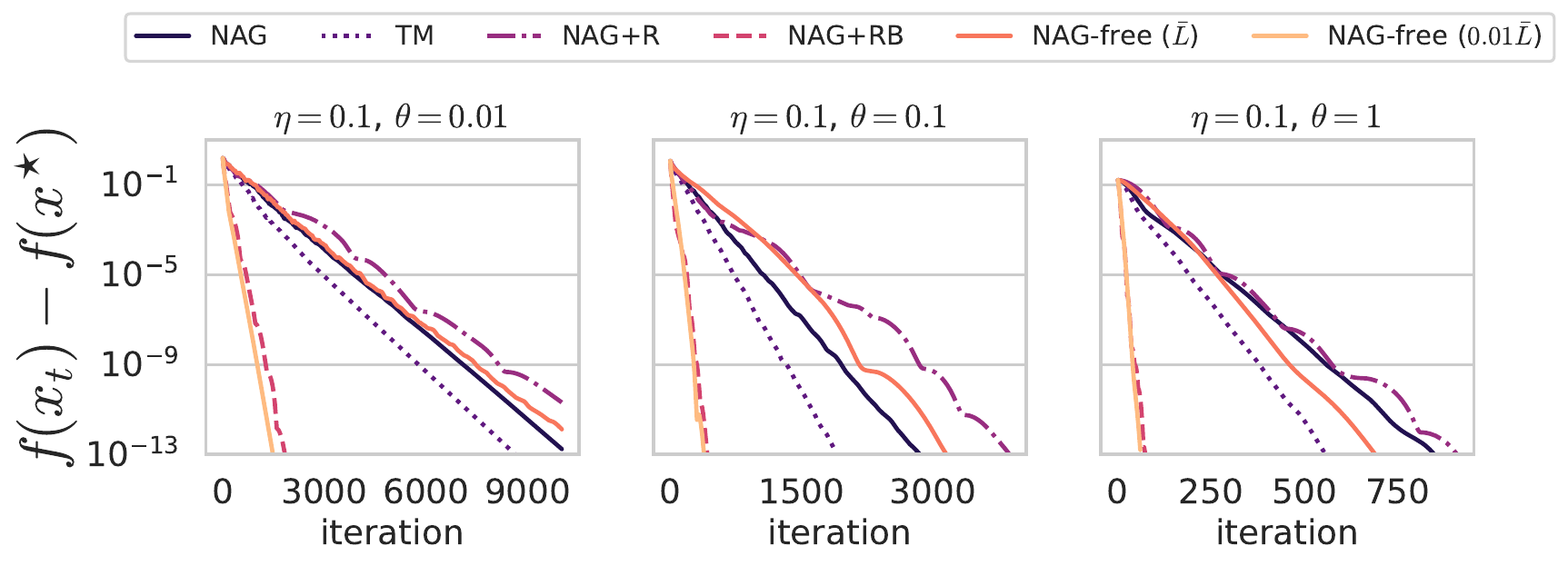}
    \caption{Suboptimality gap \(f(x_{t})-f(x^{\star})\) for log-sum-exp under different \((\eta,\theta)\) settings.}
    \label{app:ne:fig:log-sum-exp_summary_1}
\end{figure}

\begin{figure}
    \centering
    \includegraphics[width=\linewidth]{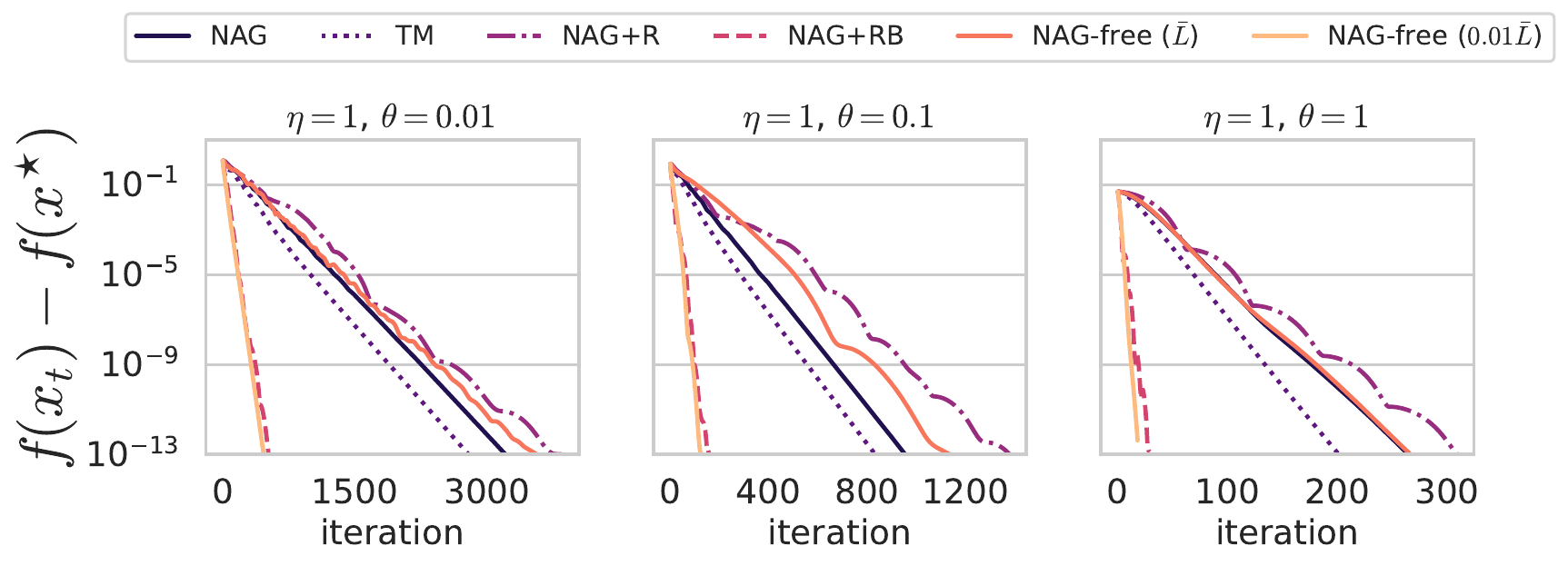}
    \caption{Suboptimality gap \(f(x_{t})-f(x^{\star})\) for log-sum-exp under different \((\eta,\theta)\) settings.}
    \label{app:ne:fig:log-sum-exp_summary_2}
\end{figure}

\setlength{\tabcolsep}{5pt} 
\begin{table}
    \centering
    \caption{Smoothing parameter \(\theta\), regularization parameter \(\eta\), and the estimates \(m_{\infty}\) and \(L_{\infty}\) held by different methods at the last iteration.
    Rows highlighted in gray indicate that \(m_{\infty} > \eta\).}
    \begin{tabular}{ccccc}
        \toprule
        \(\eta\) & \(\theta\) & \(m_{\infty} (\bar{L})\) & \(m_{\infty} (0.01\bar{L})\)\\
        \hline
        \num{1e-2} & \num{1e-2} & \num{1e-2} & \num{1e-2}
        \\
        \rowcolor{gray!30} \num{1e-2} & \num{1e-1} & \num{6.13e-2} & \num{5.92e-2}
        \\
        \rowcolor{gray!30} \num{1e-2} & \num{1e0} & \num{6.68e-2} & \num{6.57e-2}
        \\
        \num{1e-1} & \num{1e-2} & \num{6.67e-2} & \num{1e-1}
        \\
        \num{1e-1} & \num{1e-1} & \num{9.26e-2} & \num{8.94e-2}
        \\
        \rowcolor{gray!30} \num{1e-1} & \num{1e0} & \num{1.74e-1} & \num{1.72e-1}
        \\
        \num{1e0} & \num{1e-2} & \num{6.67e-1} & \num{1e0}
        \\
        \num{1e0} & \num{1e-1} & \num{7.48e-1} & \num{7.57e-1}
        \\
        \num{1e0} & \num{1e0} & \num{9.69e-1} & \num{9.65e-1}
        \\
        \bottomrule
    \end{tabular}
    \label{app:ne:tab:log-sum-exp}
\end{table}


\subsection{Triple momentum extension}
\label{app:ne:tm}

In \cref{conclusion}, we briefly discussed extending the triple momentum method (TM) with an \(m\) estimator, which is shown by \cref{app:ne:alg:tm-free}.

Also in \cref{conclusion}, we mentioned that even in the worst case represented by \textsc{a9a} in \cref{fig:logreg_tm_free}, when \(m\) can essentially be estimated offline, TM-free simply recovers an accurate estimate of \(m\) without compromising performance.
The exact values of the offline and TM-free estimates for this experiment were \(\eta=\num{4.83e-6}\) and \(m_{\infty}=\num{4.69e-6}\), respectively.
We note that as for NAG-free, the TM-free estimate falls in the interval \([\eta/\gamma,\eta]\), where \(\gamma=1.5\) was used.

\begin{algorithm2e}
    \caption{TM-free, an extension of TM that estimates the strong convexity parameter.}
    \label{app:ne:alg:tm-free}
    \SetAlgoLined
    \KwData{\(T>0, xi_{0}, L > 0, \gamma > 1, \gamma_{L} > 1\)}
    \KwResult{\(x_{T}\)}
    $\xi_{-1} \gets \xi_{0}$
    \tcp*{initialization}
    $y_{0} \gets \xi_{0}$
    \\
    $m_{0} \gets L$
    \\
    \For{$t=0,1,\ldots,T-1$}{
        $\rho_{t} \gets 1 - \sqrt{m_{t}/L}$
        \tcp*{update hyperparameters}
        $\alpha_{t} \gets (1+\rho_{t})/L$
        \\
        $\beta_{t} \gets \rho_{t}^{2}/(2-\rho_{t})$
        \\
        $\gamma_{t} \gets \rho_{t}^{2}/((1+\rho_{t})(2-\rho_{t}))$
        \\
        $\xi_{t+1} \gets (1+\beta_{t})\xi_{t} -\beta_{t}\xi_{t-1} - \alpha_{t}\nabla f(y_{t})$
        \tcp*{compute iterates}
        $y_{t+1} \gets (1+\gamma_{t})\xi_{t+1} -\gamma_{t}\xi_{t}$
        \\
        $x_{t+1} \gets (1+\delta_{t})\xi_{t+1} -\delta_{t}\xi_{t}$
        \\
        $c_{t+1} \gets \Vert \nabla f(y_{t+1}) - \nabla f(y_{t}) \Vert/\Vert y_{t+1} - y_{t} \Vert$\tcp*{estimate \(m\)}
        \eIf{$c_{t+1} < m_{t}$}{
            $m_{t+1} \gets \min(m_{t}/\gamma,c_{t+1})$\\
        }
        {
            $m_{t+1} \gets m_{t}$\\
        }
    }
\end{algorithm2e}

\end{document}